\documentclass[12pt,a4paper]{report}

\usepackage[T1]{fontenc}
\usepackage[latin1]{inputenc}
\usepackage{latexsym,amsfonts,amssymb}
\usepackage{wrapfig}
\usepackage{graphicx}
\usepackage{eepic}
\usepackage{color}
\usepackage{hyperref}

\newtheorem{theorem}{Theorem}[section]
\newtheorem{lemma}{Lemma}[section]
\newtheorem{corollary}{Corollary}[section]
\newtheorem{proposition}{Proposition}[section]
\newenvironment{definition}

\newcommand{\overdub}[1]{\overline{\overline{#1 }}}
\newcommand{\beq}{\begin{equation}}
\newcommand{\eeq}{\end{equation}}
\begin{document}
\begin{titlepage}
\parbox[t]{0.46\linewidth}{
\begin{flushleft}
Universit\'e Libre de Bruxelles\\
Facult\'e des Sciences\\
D\'epartement de Math\'ematiques
\end{flushleft}
}
\hfill
\parbox[t]{0.46\linewidth}{
\begin{flushright}
Ann\'ee Acad\'emique 2002--2003
\end{flushright}
}
\vspace{2cm}

\begin{center}
%%%%%%%%%%%%%%
%\textbf{\huge Classification of}\\
{\Huge Classification of Group Extensions
\smallskip
\smallskip
\bigskip
}

{\Large Claude Archer}
\vspace{1cm}
%%%%%%%%%%%%%%
%{\Huge
%\smallskip
%Classification of Group Extensions
%\smallskip
%}
%Claude \textsc{Archer}
%\vspace{2cm}
\begin{center}
\includegraphics[width=6.5cm,height=9cm]{couv.eps}
\end{center}
\end{center}

\vspace{1cm}
\parbox[t]{0.46\linewidth}{
\begin{flushleft}
Th\`ese pr\'esent\'ee en vue\\
de l'obtention du grade de\\
Docteur en Sciences
\end{flushleft}
}
\hfill
\parbox[t]{0.46\linewidth}{
\begin{flushright}
Promoteur :\\
Anne Delandtsheer\\
\end{flushright}
}
\end{titlepage}
%\vspace{1cm}
%\newpage
%\thispagestyle{empty}

%\begin{flushright}
%\begin{minipage}{10cm}
%\vspace{15cm}
%Je tiens \`a remercier A. \textsc{Delandtsheer},
%Professeur \`a l'Uni\-ver\-si\-t\'e Libre de Bruxelles,
%pour ses conseils et pour l'attention minutieuse qu'elle a bien
 %voulu accorder \`a la r\'edaction de cette th\`ese.
%Je tiens \`a remercier particuli\`erement Bettina Eick pour ses commentaires
%judicieux. Je tiens \`a remercier tout mes proches, Catherine et
%mon fiston Julien. Je suis aussi très reconnaissant à ma
%  mamam, à Samuel et Sofyane pour leur aide.

%Ce travail doit beaucoup \`a K. \textsc{Lux} et B. \textsc{Eick},
% que ces s\'ejours en Allemagne m'ont permis de rencontrer.

%\end{minipage}
%\end{flushright}

\chapter*{Remerciements}
Je tiens \`a remercier mon promoteur Anne Delandtsheer, pour ses commentaires et ses conseils.
% d'avoir \'et\'e   mon promoteur %\textsc{Delandtsheer},
%Professeur \`a l'Uni\-ver\-si\-t\'e Libre de Bruxelles,
%pour ses conseils et pour l'attention minutieuse qu'elle a bien
% voulu accorder \`a la r\'edaction de cette th\`ese.
%%et pour ses commentaires.
Je tiens \'egalement \`a remercier Bettina Eick pour son accueil \`a Braunschweig et
pour  ses remarques toujours judicieuses et constructives.\\

Je ne remercierai jamais assez mon \'epouse Catherine pour son soutien, sa solidarit\'e et son amour
durant ces quatre ann\'ees ; sans elle ce travail n'aurait pas abouti \`a sa forme actuelle.
Je remercie ma maman Henriette, pour son soutien et ses encouragements sans mesure tout au long
de mon parcours.
Je remercie la patience du petit Julien qui se demande pourquoi son papa reste
chaque soir immobile devant un carr\'e de lumi\`ere au lieu de venir jouer avec lui.
Gr\^{a}ce \`a lui j'ai compris qu'une th\`ese devait se terminer un jour.
Je remercie aussi, pour sa compr\'ehension, mon p\`ere venu de loin pour nous rendre visite et qui n'a vu de moi
qu'un dos face a un \'ecran.
Je remercie enfin mes amis Samuel Fiorini et Sofyan Iblisdir pour leur soutien et leur
d\'evouement dans la derni\`ere ligne droite.

La recherche math\'ematique m'\'emerveille par l'enthousiasme qu'elle peut insuffler mais la vie
simultan\'ement, m'\'emerveille tout autant pour tout ce qu'elle a d'autre \`a nous faire d\'ecouvrir.

\tableofcontents
%\printindex

% espacement entre paragraphes
\setlength{\parskip}{2ex plus 0.5ex minus 0.5ex}

\chapter{Introduction}
%Introduction : Known results are explicitly attributed to theirs authors.
%All the others results are from the author .
%The main application is the construction of finite nonsolvable groups.
%The family of groups for which these methods are the most useful are non solvable groups
%that are much larger than their Perfect Residuum $P$. Or more precisely such that
%the centralizer of of $P$ is larger than the centre of $P$
%But even if the centralizer is resumed to the centre of $P$ (automorphic extensions),
%my methods is efficient to classify these groups and to compute their automorphism group.
%and their maximal subgroups
%The essence is to reduce computation to computation in a small solvable extensions of
%a small nilpotent (or even a p-subgroup, by subdirect product) subgroup of $P$.
%we have completely solved
%the isomorphism problem for extension of $G$ by a group $F$ : it has been reduced to the
% classification of $Z\phi$-extensions of $Z(G)$ by $F$ (see Theorem \ref{autsplit}).
A group $E$ is an extension of a group $G$ by a group $H$  if $E$ contains a normal
subgroup $G$ such that $E/G\cong H$. For given groups $G$ and $H$, %% isomorphic to 
 the extension problem consists of classifying all groups $E$ up to isomorphism.

% that contains a normal subgroup $G$ is an extension of $G$.
The aim of this thesis is to propose a method to reduce
the classification of the extensions of a group $G$ by $H$ to the classification of the
extensions of a subgroup $U<G$ by $H$. For a finite group $G$, we prove that $U$ can be chosen
to be nilpotent, so that for $H$ solvable, the extension problem can be reduced
to a classification of solvable groups. We then present some applications of this reduction
method in particular, we provide tables to classify finite nonsolvable groups up to order
$23,039$.
%among others, the classification of finite nonsolvable groups up to order
%$23,040$ is achieved.
%%%% SEE WHAT I HAVE ALREADY WRITTEN AS SENTENCES
% We prove that if $G$ is finite then there exists a nilpotent subgroup $U$
% anf if $H$ is solvable all groups involved are solvable

%\subsection{The role of extensions into Finite Group Theory}
\section{Origins and motivations}
\subsubsection{Origins and motivations of the extension problem for
finite group theory}
 The notion of homomorphism is central in algebra. It
consists of mapping an object $O$ onto a simpler object $O'$
so that the mapping preserves the algebraic structure. Further
mappings on objects $O'',O''',\ldots$ can be obtained until we
reach an "atom" that cannot be simplified further. In the case of groups
the atoms are called simple groups and this approach by mappings
leads to the notion of composition series.

A composition series of a group
 $G$ is a finite sequence $G=A_{0}\rhd A_{1}\rhd A_{2}\rhd\ldots A_{n-1}\rhd A_{n}=1$
 of subgroups of $G$, where $A_{i}$ is a maximal normal subgroup of $A_{i-1}$.
 The factor groups $A_{i-1}/A_{i}$ are simple groups called composition factors of $G$.
 Every finite group posseses a composition series and the \emph{Jordan-H\"older's theorem}
states that every composition series of a given finite group produces the same composition factors up to
reordering. % mais  peut-\^etre dans un autre ordre.
 Every subgroup $A_{i-1}$ is an extension of $A_{i}$ by the simple group $A_{i-1}/A_{i}$.
 The group $G$ is an extension of $A_{1}$ which is an extension of $A_{2}$, etc.
Thus, if all finite simple groups are known and if we could solve the extension problem
it would be in principle possible to classify all finite groups. These two problems are more than
a century old ; since 1982, specialists consider that finite simple groups are "almost"
classified, but no full general and practical solution of the extension problem is expected 
so far, because it would imply the classification of all groups. Actually, the extension problem is divided
into two different subproblems. The first one is to construct explicitly all extensions of a group $G$ by a group
$H$. But since a lot of isomorphic copies of the same group can be obtained in this way, the second problem
is to classify these extensions up to isomorphism.

%%%%%%%%%%%%%%%%%%%%%%
%%%%%%%%%%%%%%%%%%%%%%%%

%%%%%%%%%%%%%%%%%%%%%%%%%
\section{The extension problem : historical survey}
\subsubsection{Schreier's description of extensions}
The first systematic analysis of group extensions is due to Otto Schreier in 1926 (\cite{Schreier1} and \cite{Schreier2}).
Schreier described an extension $E$ of $G$ by $H$ by a function $\xi:H\to Aut(G)$
that corresponds to the conjugation action on $G$ of a transversal
to $G$ in $E$, and by a function $\varphi:H\times H\to G$ called
the "factor set" associated to the transversal. These functions
define a multiplication on the cartesian  product of the set $G$ and
the set $H$, and Schreier gave necessary and sufficient conditions on
$\xi$ and $\varphi$ for $E(\xi,\varphi)$ to be an extension.

However these conditions as such cannot be used in practice to
determine in a reasonable ammount of time all extensions of $G$ by $H$.
 There are three well-known cases where the construction of extensions can be
achieved in practice. The case where $H$ is cyclic, the case where
$Z(G)=1$ and the case where the orders of $G$ and $H$ are coprime.
 The latter result is known as the Schur-Zassenhaus theorem (see chapter \ref{basic}).
A change of coset representatives for $G$ in $E(\xi,\varphi)$
gives rise to a new pair of associated functions $(\xi ',\varphi ')$ that defines an extension
$E(\xi ',\varphi ')$. An extension produced in this way is said to be \emph{equivalent} to $E(\xi,\varphi)$.
Since equivalent extensions are isomorphic, a first reduction of the isomorphism problem is obtained
by considering representatives of equivalence classes but generally, isomorphic extensions are not
necessarily equivalent.

%(importance of extensions for finite group theory and what has been done ...)
\subsection{Extensions, cohomology and the birth of category theory}
In 1942 Eilenberg and MacLane discovered unexpected connections
between group extensions and cohomology, a discipline born from topology.
 At the beginning of their collaboration, they expected to solve
topological problems by solving their translations into group
extension language. But the role of this connection will become of
first importance to the future of all mathematics. It was
generalized into a theory of connections between the different
areas of mathematics, namely category theory.

Here is a citation (\cite{Maclane}) of Saunders MacLane, cofounder with Samuel
Eilenberg of category theory.

"Sammy's idea was to dig deeper and deeper till he got to the bottom of each issue. This I
learned when I lectured at Ann Arbor about group extensions. I had calculated an example of
group extensions for an interesting factor group involving a prime number p. When I told
Sammy this result, he immediately saw that it answered a question of Steenrod about the
regular cycles of the p-adic solenoid (inside a solid torus, wrap another one p times around,
 and so on, ad infinitum). So Sammy and I stayed up all night to find out the reason for
 this unexpected appearance of group extensions. We found out more: it rested on a
"universal coefficient theorem" which gave cohomology with any coefficient group G in terms
 of homology and an exact sequence involving Ext, the group of group extensions. Thus Sammy
insisted on understanding this unexpected connection between algebra and topology. There was
 more there: the connection involved mapping topology into algebra, so we were forced to
invent functors, natural transformations, and categories to describe this. All told, this
 led to our fifteen joint papers. "

\subsubsection{Applications of cohomology to group theory}
It was not clear just after 1945 whether cohomology would be more than a convenient
language to reformulate the extension problem. The first purely
group theoretical theorems proved by a cohomological method were
established only during the sixties. One of the most famous
such applications is the proof of the existence of outer automorphisms
 of $p$-groups by Gasch\"{u}tz in 1965-66. A survey of other
applications of cohomology to group theory can be found in \cite{Robinson}.

 \subsection{The abelian case}
%%%%%%%%%%%%%%%%%%%%%%%%%%%%%%%%%%%%%%%%%%%%%%%
The applications of cohomology to extensions of groups
are based on the following result. Let $A$ be an abelian group and
let $\Phi:H\to Aut(A)$ be an homomorphism. The equivalence classes
of extensions of $A$ by $H$ associated with $\Phi$ are in one-to-one 
correspondence with the elements of the second cohomology group $H_{\Phi} ^{2}(H,A)$. But
even if $H_{\Phi} ^{2}(H,A)$ is known, it remains to detemine
whether two elements of $H_{\Phi} ^{2}(H,A)$ define isomorphic
extensions.

If one considers only isomorphisms that preserve $A$ ($A$-isomorphisms), a solution of the $A$-isomorphism
problem has been proposed by Eick and Besche in \cite{Eick_Ulrich} in the case where $H$ is
finite and where $A$ is elementary abelian. Let $\omega\in Aut(H)$ and
$\pi\in Aut(G)$. The pair $(\omega,\pi)$ is compatible if
$\pi^{-1} \Phi (h) \pi=\Phi (h^{\omega})$ for every $h\in H$. The group
of all compatible pairs is denoted by $Comp(\Phi)$. Eick and Besche
proved that there is a group action $\Gamma$ of $Comp(\Phi)$ on
$H_{\Phi} ^{2}(H,A)$ and that each $A$-isomorphism class is the
union of the equivalence classes of extensions $\mathcal{E}(x)$,
where $x$ ranges over an orbit of $\Gamma$. We will generalize here
this action to any groups $H$ and $A$ (not necessarily finite or abelian).\\

%%%%%%%%%%%%%%%%%%%%%%%%%%%%%%%%%%%%%%%%%%%%%%
\subsection{The nonabelian case}
Contrary to the abelian case, if $G$ is nonabelian it does not necessarily exist an extension
of $G$ by $H$ associated with a given homomophism $\Phi:H\to Out(G)$. % (\cite{Atlas} page )non voir basic.tex
%there do not necessarily exist an  extension of $G$ by $H$ associated with $\Phi$.
The key result of MacLane and Eilenberg
for the nonabelian case is the following. Assume that there
exists an extension $E(\xi,\varphi)$ associated with $\Phi$, then
all the other equivalence classes of such extensions can be
obtained as $E(\xi,\varphi+\zeta_i)$ where $\{\zeta_i\}$ is a set
of representatives for the elements of $H_{\Phi} ^{2}(H,Z(G))$.
Together with the existence problem for $E(\xi,\varphi)$, the
isomorphism problem remains once again to be solved. It was not known whether or not two
representatives $\zeta_1$ and $\zeta_2$ define $G$-isomorphic extensions
and we will treat this problem successfully in chapter \ref{General}.

The result of Mac-Lane and Eilenberg is very important in order to know
which type of classification can be expected for extensions. Even for an abelian group $A$, the
construction of the extensions (that is the determination of $H_{\Phi}
^{2}(H,A)$) is a hard problem. We also know that the construction
of the extensions of a group $G$ implies the construction of the
extensions of the abelian group $Z(G)$ ( that is the determination of
$H_{\Phi} ^{2}(H,Z(G))$). Hence we cannot hope to do better than
what is known for abelian groups. This puts a limitation to the
classification of extensions of nonabelian groups and we will
consider that the problem is completely solved if it can be
reduced to a problem of isomorphism between the extensions of $Z(G)$,
 since better results can not be achieved. That is why %%we consider the
the reduction of the extension and isomorphism problem for a group $G$
to the extension and isomorphism problem for $Z(G)$ is a kind of
"Holy Grail quest" for extension theory.

%%%%%%%%%%%%%%%%%%%%%%%%%%%%%%%%%%%%%%%%%%%%%%%%%%%%
\subsection{Presently available classifications}
\begin{itemize}
\item The most famous classification of finite groups is certainly the classification of the finite simple groups.
An overview of this collective work that has lasted more than a century can be found in \cite{solomon2001}
and a more detailled description of these groups is presented in \cite{Atlas}.

\item Perfect groups. The importance of these groups comes from the fact that every
finite nonsolvable group is an extension of a perfect group by a solvable group (see section \ref{solv+nonsolv}).
A partial classification of the perfect groups of order up to $1,000,000$ was achieved
by Holt and Plesken in 1989 (\cite{Holt_Plesken}). Some orders are missing ; for instance the orders $2^{10}.|S|$
where $S$ is a nonabelian simple groups.

\item Solvable groups. "Most" of the finite groups are $p$-groups (groups whose order is a power of a prime number).
% The groups of order 128 and 256 where respectively obtained by O'Brien in 1990 and 1990,512
The reference method for classifying $p$-groups is the \emph{$p$-group generation algorithm} of
 Newman and O'Brien \cite{O_Brien}. Eick and O'brien have developed a method to enumerate $p$-groups
 \cite{Eick_O_Brien}, that has been used to enumerate the 49 487 365 422 groups of order $2^{10}$.
Finite solvable groups  whose order is not a power of a prime were obtained up to order 2000 by
Besche and Eick \cite{Eick_Ulrich} in 2000, using the Frattini extension method.

\item Finite groups of order at most 2000 (except for the order 1024) are explicitely known and electronically
available for computer algebra systems. These groups form the Small Groups Library that has been
developed by Besche, Eick and O'brien.  To list all such finite groups, Besche and Eick simply added the
nonsolvable groups of order at most 2000 to the already obtained solvable group (see \cite{Millenium}).
They have constructed these nonsolvable groups by the method of iterated cyclic extension
(see subsection \ref{iteratedcyclic}).
We have classified in 1998 \cite{Archermem} the nonsolvable groups of order less than 3840.

A bibliography of the various determinations of groups of a given order is provided in \cite{O_Brien_Short} and
has been updated in \cite{Millenium}.
\end{itemize}
%
%>Wilbert Dijkhof  <s9400842@mail.student.utwente.nl> asked for a reference
%>for this result I had posted:
%>	Let  a_n  be the number of non-isomorphic groups of order  p^n.
%>	Then  log_p (a_n) is asymptotic to  (2/27)n^3.
%>
%>Graham Higman showed (Proc London Math Soc 10 (1960) 24-30) that
%>the number b_n = log_p (a_n) / n^3  satisfies
%>	2/27 - o(n) <=  b_n  <= 2/15 + o(n)
%>where the implied constant is independent of  p. Later, Charles Sims
%>(Proc London Math Soc 15 (1965) 151-166) improved the upper bound to
%>	b_n <= 2/27 + o(n)
%>(Again this is independent of  p, although Higman's lower bound is
%>a little tighter than Sims' upper bound.)%
%I am a little confused by your o(n) terms, since you have already divided
%through by n^3.
%More precisely,  Higman proved
%        n^3b_n >= 2n^3/27 + O(n^2)
%and Sims proved
%        n^3b_n <= 2n^3/27 + O(n^(8/3))
\subsubsection{Computational aspects of the different group representations }
A group may be specified in a number of different ways : as a group of permutations or matrices, in
terms of a finite presentation with given relations on generators (Fp-groups) or as the group of automorphisms of
a combinatorial structure.  Experience has shown that, as a general rule, the most powerful algorithms
 are those designed for permutation groups (see \cite{Cannon_Havas}). At the opposite side ,
 the less efficient algorithms are those designed for an arbitrary finitely presented groups.
There are Fp-groups for which one can not decide whether two different words represent the same element and
this is called \emph{the word problem}. A similar undecidability problem exists to determine whether a
given presentation defines a finite or an infinite group.

However, for some groups, there are particular presentations for which the undecidability of the word
problem vanishes. For such presentations, there exists a set of rules to rewrite a word into a
%%normal form such that no matter in which order the rules are applied, every word is rewrited into a
unique normal form. The word problem is then solved by a comparison of normal forms  (see \cite{Sims}).

Every finite solvable group admits a presentation called \emph{polycyclic} ( see subsection \ref{polycyclic})
that has such a nice rewriting property. It appears that the polycyclic presentation can be used to
determine efficiently almost all properties of a finite solvable group (see \cite{Neubuser}).
The  algorithms developed for finite polycyclic groups are among the most powerful and can be
compared to the situation of permutation groups algorithms (see \cite{seress2001}) .
In this work we show that it is possible to reduce the classification of finite nonsolvable groups to a
classification of finite solvable groups.
%%This is the reason why a reduction to finite solvable groups is an important improvement.
%\subsubsection{Computational Group Theory}
% From a computational point of view, all group
%presentations are not equivalent. It is much more efficient to
%work on a group $G$ presented as a permutation group or as a
%polycyclic group than if it is presented as a general finitely
%presented group. Very efficient algorithms have been developed for
%polycyclic groups and that is the reason why a reduction to finite
%solvable groups (hence polycyclic) is an important improvement.
%\subsubsection{Finite Nonsolvable groups}
%Every finite non solvable group $E$ is an extension of a Perfect
%group $P$ by a solvable group $H$. The derived series shows that
%$E$ is obtained from $P$ as a succession of cyclic extensions.
%%%%%%%%%%%%%%%%%%%%%%%%%%%%%%%%%%%%%%%%%%%%%
%\subsubsection{survey by chapters}
\section{Survey of the main results obtained in this thesis}
%\subsection{$G$-isomorphism problem}
\subsection{A general solution to the $G$-isomorphism problem}
We recall that for a nonabelian group $G$ and a group $H$, if there is an extension of $G$ by $H$ associated
with a given homomorphism from $H$ into $Out(G)$, then the equivalences classes of such extensions are
in one-to-one correspondence with the elements of the cohomology group $H^{2}:=H_{\Phi} ^{2}(H,Z(G))$.
Let $E(\zeta_1)$ and $E(\zeta_2)$ denote the extensions that correspond repectively to an elements $\zeta_1$
 and $\zeta_2$ of $H ^{2}$. We prove that there is a group action $\Gamma$ of the group of compatible pairs
 $Comp(\Phi)$ on $H^{2}$. One of our main results states that $E(\zeta_1)$ and $E(\zeta_2)$ are $G$-isomorphic if
 and only if $\zeta_2$ belongs to the orbit $\zeta_1 ^{\Gamma}$ (see chapter \ref{General}).
 Hence the $G$-isomorphism
classes of extensions are in one-to-one correspondence with the orbits of $\Gamma$ in $H^{2}$.

%In their work on construction of finite solvable groups, Eick and Besche (\cite{Eick_Ulrich})
%are led to classify extensions of an elementary abelian group $A$ by a finite group $H$ up to
% $A$-isomorphism ( isomorphisms that preserve $A$). They describe an action $\Gamma$ of
% a subgroup of $Aut(H)\times Aut(G)$ on the elements of the second cohomology group
% $H^2 (H,A)$ % of extensions of $G$ by $H$ associated to a given homomorphism $H\to Aut(G)$.
%and they prove that there is a one to one correspondence between the $A$-isomorphism
%classes of extensions and the orbits under $\Gamma$. In chapter \ref{General}, we generalize this idea
%to any group $G$. We prove that there is a similar action $\Gamma$ on the set
%$H^2 (H,A)$ where $A$ is the abelian group $Z(G)$ and that there is also
%a one to one correspondence between orbits and $G$-isomorphism classes of extensions of $G$
%by $H$.
%
%We propose in chapter \ref{General} a solution of the $G$-isomorphism problem,
%that generalizes to any two groups $G$ and $H$ the action $\Gamma$
%described by  Eick and Besche. We prove that there is a group
%action $\Gamma$ of $Comp(\Phi)$ on $H_{\Phi} ^{2}(H,Z(G))$ such
%that $E(\xi,\varphi+ \zeta_1)$ and $E(\xi,\varphi+ \zeta_2)$ are
%$G$-isomorphic if and only if $\zeta_2$ belongs to the orbit
%$\zeta_1 ^{\Gamma}$. Hence the $G$-isomorphism classes are in one-to-one
% correspondence with the orbits of $\Gamma$ in $H_{\Phi}
%^{2}(H,Z(G))$.
However, to make this method to classify group extensions practical and efficient we combine it with another
method that we developed in order to reduce the extension problem for $G$, to the extension problem for
a subgroup of $G$ (see section \ref{constsup}). For instance, if $G$ is finite and $H$ is solvable, we prove
that it is possible to reduce the computations involved for the determination of the $\Gamma$-orbits,
into solvable groups only.
%This is an application of the results mentionned in the subsection below.

We finally apply this method to the construction
 of all finite nonsolvable groups up to the order $23,039$ (there are more than $8.5\,10^{6}$ such groups).
For the nonsolvable groups of order at most $2000$, we confirm the existence of 1024 such groups, as
established by Eick and Besche. In subsection \ref{perfo}, we compare some performances of
different extension methods used to classify nonsolvable groups.

\subsection{Reduction to solvable groups}\label{reduc}
%Supplement function
First, the construction of an extension $E$ of $G$ will be reduced to the
construction of a smaller group $S_E$ which is a supplement to $G$ in
$E$ (that is $E=GS_E$). Next, we introduce the concept of a
\emph{supplement function preserved by a $G$-isomorphism}. We prove that
if such a function exists for a group $G$ then the isomorphism problem for the extensions of $G$ is also
reduced to an isomorphism problem for extensions of a subgroup of $G$.

We will prove stepwise in the next chapters that if $G$ is finite, then there exist such a function
$E\to S_E$ (where $E$ is an extension of $G$ and $S_E$ is a supplement to $G$ in $E$)
 such that the following three properties hold
\begin{enumerate}
\item $S_{E}\cap G$ is a nilpotent subgroup of $G$.
\item $E$ is completely described by the conjugation action of $S_{E}$ on $G$.
\item If $i:E_{1}\rightarrow E_{2}$ is an isomorphism between the extensions $E_1$ and $E_2$ of $G$ and
if $i$ preserves $G$, then there exists an isomorphism
$j:E_{1}\rightarrow E_{2}$ that maps $S_{E_{1}}$ onto $S_{E_{2}}$.
\end{enumerate}
Moreover, if $E_{1}$ and $E_{2}$ are two extensions of $G$ we will
prove necessary and sufficient conditions under which an
isomorphism from $S_{E_{1}}$ onto $S_{E_{2}}$ can be lifted into
an isomorphism from $E_{1}$ onto $E_{2}$. Section \ref{constsup} provides more details
about this reduction method.

 We will also describe explicitly in chapter \ref{findsup}, how to find such supplement functions $S_{E}$.
We provide an efficient algorithm for finding some functions $S_{E}$ and we give several examples
of the performance of this algorithm.

%\subsection{Algorithms}
%In chapter \ref{findsup}, we provide an efficient algorithm to find the
%supplements $S_{E}$ through a function $N_{\vec{p}}$. We produce examples
%of the performance of this algorithm.
\subsection{Full classification in particular cases}
We have explained previously that to reduce the extension
problem for a group $G$ to the extension problem for $Z(G)$ can be
considered as a full solution to the extension problem for $G$.
In chapter \ref{classiftheo} we describe families of groups for
which such a full classification of  their extensions can be achieved.
For the groups $G$ in these families, there exists an extension of
$G$ by $H$ for every homomorphism $\Phi:H\to Out(G)$ and the
isomorphism problem can be completely reduced to isomorphisms
between extensions of $Z(G)$ by $H$.

Here are two examples of such families of groups $G$
\begin{enumerate}
\item $|Z(G)|$ and $|\Phi (H)|$ are coprime. The case
where $|Z(G)|$ and $|Out(G)|$ are coprime and the case where $|Z(G)|$ and $|H|$ are coprime
are interesting subcases.
\item $Aut(G)$ is a split extension of $Inn(G)$;
%%%\item If $|Z(G)|$ and $|H|$ are coprime.
\end{enumerate}
Our result in the case where $|Z(G)|$ and $|H|$ are coprime generalizes in some sense the
Schur-Zassenhaus theorem  that allows to classify
extensions of a group $G$ by a group $H$ when $|G|$ and $|H|$ are
coprime.
\subsubsection{Cross-check of the supplement method}
For the groups $G$ of these families, this alternative method
provides us with a cross-check of our main method that uses
supplements. For instance for all but two perfect groups
(namely $G\neq SL_{2}(9)$ and $G\neq 3.SL_{2}(9)$)  this alternative method can be used to enumerate all
the nonsolvable groups of order less than $58,320$.
%%%%%%%%%%%%%%%%%%%%%%%%%%%%%%%%%%%%%%%%%%%%%%%%%%%%
\subsection{Examples of performances for the different methods}\label{perfo}
%\emph{\bf{JE DOIS ABSOLUMENT DIRE QUE JE COMBINE MA METHODE DES SUPPLEMENTS AVEC
%LA METHODE DU $H^2$ SINON ON NE LE SAIT PAS !!!. CE SERAIT AUSSI UNE BONNE IDEE DE
%MONTRER LE TABLEAU DE COMPARAISON ENTRE MA METHODE ET L'AUTRE. DANS MON MEMOIRE
%CA IMPRESSIONNAIT. SINON LE LECTEUR DOIT ATTENDRE TROP LONGTEMPS}}
\emph{\bf{Extensions of $G=PerfectGroup(1080,1)$ by $H$ :}}\\

We have $|Z(G)|=3$ and $Out(G)\cong 2^2$. The first column shows performances of the iterated cyclic
extension method. The second column do the same for the computation of the $\Gamma$-orbits by our
"reduction to supplement" method. The last column is the cross-check by our $z\phi$-classes method described
in the previous subsection. The mention "SG" means that the Small Groups Library has been used.
 The supplement method used it to find the groups of order $|H|
 $.
The $z\phi$-classe method used it to find the groups of order $3|H|$.

\begin{tabular}{|c|c|c|c|}
\hline
& iter. cyclic &$\Gamma$+supp. +SG & $z\phi$-classes + SG \\
\hline\hline
extensions  for $|H|=6$ & 12 & 12 & 12  \\
 time &  17 min. & 7 sec  & 1.3 sec.  \\
\hline
extensions  for $|H|=8$ & 34 & 34 & 34  \\
 time & 1 h 3 min.  & 1 min 20 sec. & 24 sec.  \\
\hline
extensions  for $|H|=12$ & 34  & 34& 34  \\
 time & 3 h. 9 min.  & 1 min. 40 sec.  & 29 sec.  \\
\hline
extensions  for $|H|=16$ & ? & 151 &  151 \\
 time &  $?>10\,days$ & 1 h. 44 min.  &  2 min. 27 sec. \\
\hline
extensions  for $|H|=24$ &  ?& 159 & 159  \\
 time & $?>10\,days$  & 1 h. 36 min. &  2 min. 30 sec. \\
\hline
\end{tabular}\label{tab3a6}
%%%%%%%%%%%%%%%%%%%
\section{The main method : construction of supplements}\label{constsup}
 Our aim
is to classify the extensions $E$ of a finite normal subgroup $G$
by a factor group $F$. The main idea is to reduce this problem to
a classification of extensions $S$ of a nilpotent subgroup $U$ of
$G$ by $F$. If $F$ is solvable then $S$ is solvable even if $G$ is
nonsolvable. For instance, the classification of nonsolvable
groups as extensions of a perfect group $G$ by a solvable group
$F$ is reduced to a classification  of smaller solvable groups.
Moreover it turns out that in many applications the order of
$U$ is quite small compared with the order of $G$ (for instance
$|U|\leq 4|Z(G)|$).

\subsubsection*{Outline of the general method}
We describe now with more detail the various steps of this method in a slightly simplified
version.% Further refinements will be added later on.
% We assume throughout this summary that
We will assume here that every isomorphism between two extensions of $G$
preserves $G$.
%if extensions $E_{1}$ and $E_{2}$ are isomorphic, then there exists an isomorphism between
%$E_{1}$ and $E_{2}$ that preserves $G$.
This is the case if we aim to construct
nonsolvable groups as extensions of a perfect group $G$ by a solvable group $H$.

Here are the different steps

\begin{enumerate}
%\item Choose a representative for each conjugacy class of subgroup in $Out(G)$
\item We first classify the possible actions of $E$ on $G$ by conjugation.
Let $L\leq Aut(G)$ with $Inn(G)\leq L$.
Let $\Psi$ be the set of all homomorphisms of $H$ onto the subgroup $L/Inn(G)$ of
$Aut(G)/Inn(G)=Out(G)$.
%Here, $L$ denotes a subgroup of $Aut(G)$ that contains $Inn(G)$.
 In Chapter \ref{General}
we define an action of $Aut(H)\times Out(G)$ on the elements of $\Psi$; let
$\{\psi_{i}\,|\, i\in I\}$ be a system of representatives for the orbits in $\Psi$ under
this action. Let $\mathbf{E}_{i}$ be the set of extensions %($I$ is a family of indices)
of $G$ by $F$ associated with $\psi_{i}$.
We prove that for an extension $E$ of $G$ by $H$ there is a unique $i$ such that $E$ is
isomorphic to an extension contained in $\mathbf{E}_{i}$ and therefore the sets $\mathbf{E}_{i}$ are
the classes of a partition of all extensions of $G$ by $H$.% into pairwise disjoint families.
%is isomorphic to an extension associated to exactly one of the $\psi_{i}$.

\item From now on, we are reduced to classify extensions associated to a given
$\psi_{i}:H\rightarrow L/Inn(G)$. Suppose that there exists such an extension $E$ and let
$f:E\rightarrow L\leq Aut(G)$ be the conjugation action of $E$ on $G$. We show in Chapter
\ref{findsup} how to find a supplement $B$ of $Inn(G)$ in $L$ such that $S:=f^{-1}(B)$ has
the properties listed in section \ref{reduc}.
\mbox{\setlength{\unitlength}{3947sp}%
\begingroup\makeatletter\ifx\SetFigFont\undefined%
\gdef\SetFigFont#1#2#3#4#5{%
  \reset@font\fontsize{#1}{#2pt}%
  \fontfamily{#3}\fontseries{#4}\fontshape{#5}%
  \selectfont}%
\fi\endgroup%
\begin{picture}(6184,3161)(1484,-4389)
\thicklines
\put(2926,-1861){\circle{168}}

\put(1651,-2161){\line( 4, 1){1200}}
 
\put(2926,-3136){\circle{168}}
 
\put(1576,-3473){\circle{168}}
 
\put(1651,-3436){\line( 4, 1){1200}}
 
\put(6226,-3061){\circle{168}}
 
\put(4876,-3398){\circle{168}}

\put(4951,-3361){\line( 4, 1){1200}}
 
\put(2926,-3136){\circle{168}}
 
\put(1576,-3473){\circle{168}}
 
\put(1651,-3436){\line( 4, 1){1200}}

\put(1576,-2198){\circle{168}}

%\put(7576,-1486){\circle{168}}
\put(7126,-1620){\circle{168}}

\put(6226,-1861){\circle{168}}
 
\put(4876,-2198){\circle{168}}

\put(1576,-3923){\circle{168}}

\put(4876,-3923){\circle{168}}
 
\put(1576,-4298){\circle{168}}

%\put(6301,-1786){\line( 4, 1){1200}}
\put(6310,-1826){\line( 4, 1){745}}
 
\put(4951,-2161){\line( 4, 1){1200}}

\put(2926,-1936){\line( 0,-1){1125}}
 
%\put(1576,-2311){\line( 0,-1){1050}}
\put(1576,-2275){\line( 0,-1){1110}}
 
%\put(4876,-2311){\line( 0,-1){975}}
\put(4876,-2275){\line( 0,-1){1000}}

\put(6226,-1936){\line( 0,-1){1050}}
 
\multiput(1651,-2236)(120.58824,0.00000){26}{\line( 1, 0){ 60.294}}
\put(4726,-2236){\vector( 1, 0){0}}

\put(4188,-1750){\makebox(0,0)[lb]{\smash{\SetFigFont{10}{12.0}{\rmdefault}{\mddefault}{\updefault}$f$
}}}
\multiput(3071,-1861)(119.38776,0.00000){25}{\line( 1, 0){ 59.694}}
\put(5926,-1861){\vector( 1, 0){0}}
\put(3061,-1861){\vector( -1, 0){0}}

\multiput(3041,-3061)(121.27660,0.00000){24}{\line( 1, 0){ 60.638}}
\put(5851,-3061){\vector( 1, 0){0}}
\put(3001,-3061){\vector( -1, 0){0}}

% \put(1576,-3586){\line( 0,-1){225}}
 \put(1576,-3540){\line( 0,-1){300}}

%\put(4876,-3511){\line( 0,-1){300}}
\put(4876,-3475){\line( 0,-1){370}}
 
% \put(1576,-4036){\line( 0,-1){150}}
\put(1576,-4000){\line( 0,-1){220}}
 
\multiput(1651,-3886)(120.58824,0.00000){26}{\line( 1, 0){ 60.294}}
\put(4726,-3886){\vector( 1, 0){0}}

\multiput(1876,-3436)(118.08511,0.00000){24}{\line( 1, 0){ 59.043}}
\put(4651,-3436){\vector( 1, 0){0}}

\put(1501,-2011){\makebox(0,0)[lb]{\smash{\SetFigFont{10}{12.0}{\rmdefault}{\mddefault}{\updefault}$G$
}}}
\put(2851,-1711){\makebox(0,0)[lb]{\smash{\SetFigFont{10}{12.0}{\rmdefault}{\mddefault}{\updefault}$E$
}}}
\put(1651,-3286){\makebox(0,0)[lb]{\smash{\SetFigFont{10}{12.0}{\rmdefault}{\mddefault}{\updefault}$\tilde{U}$
}}}
\put(2776,-2986){\makebox(0,0)[lb]{\smash{\SetFigFont{10}{12.0}{\rmdefault}{\mddefault}{\updefault}$S$
}}}
\put(6001,-2911){\makebox(0,0)[lb]{\smash{\SetFigFont{10}{12.0}{\rmdefault}{\mddefault}{\updefault}$B$
}}}
\put(4576,-2086){\makebox(0,0)[lb]{\smash{\SetFigFont{10}{12.0}{\rmdefault}{\mddefault}{\updefault}$Inn(G)$
}}}
\put(6076,-1711){\makebox(0,0)[lb]{\smash{\SetFigFont{10}{12.0}{\rmdefault}{\mddefault}{\updefault}$L$
}}}
%\put(7426,-1336){\makebox(0,0)[lb]{\smash{\SetFigFont{10}{12.0}{\rmdefault}{\mddefault}{\updefault}$Aut(G)$
%}}}
\put(6926,-1461){\makebox(0,0)[lb]{\smash{\SetFigFont{10}{12.0}{\rmdefault}{\mddefault}{\updefault}$Aut(G)$
}}}

\put(4651,-3286){\makebox(0,0)[lb]{\smash{\SetFigFont{10}{12.0}{\rmdefault}{\mddefault}{\updefault}$U$
}}}
\put(4651,-3736){\makebox(0,0)[lb]{\smash{\SetFigFont{10}{12.0}{\rmdefault}{\mddefault}{\updefault}$1$
}}}
\put(1651,-3811){\makebox(0,0)[lb]{\smash{\SetFigFont{10}{12.0}{\rmdefault}{\mddefault}{\updefault}$Z(G)$
}}}
\put(1651,-4186){\makebox(0,0)[lb]{\smash{\SetFigFont{10}{12.0}{\rmdefault}{\mddefault}{\updefault}$1$
}}}
\end{picture}
}
\label{picturesupplement}
\item The major reduction amounts to construct only $S$ together with its action $\phi$ on $G$, instead
of $E$. The supplement $S$ is an extension of $\tilde{U}:=G\cap S$ by $H$.
We will provide the conditions under which a homomorphism
$\phi:S\rightarrow B\leq Aut(G)$ and an extension $S$ of $\tilde{U}$ correspond to a
supplement to $G$ in an extension $E$. We finally obtain the corresponding extension as
$E\cong S\ltimes_{\phi}G/K$, where $S\ltimes_{\phi}G$ is the semi-direct product associated with $\phi$
and where $K$ is its subgroup $\{(u,u^{-1})\,|\, u\in \tilde{U}\}$ ( we identify $\tilde{U}$
in $S$ and in $G$). If $H$ is finite and solvable, $S$ has the same property and a
polycyclic generating sequence
 (pcgs) for $S$ can be constructed by extending the pcgs
 of $\tilde{U}$ with the pcgs of $H$.

\item Reduction to isomorphism. Suppose that the extension $E_{k}$ ($k=1,2$) is produced
by the supplement $S_{k}$, the action $\phi_{k}$  on $G$ and the normal subgroup $K_{k}$ as in
the above paragraph.
 We prove that if $E_1$ and $E_2$ are isomorphic then there exists an isomorphism $E_1\to E_2$ that can be
 lifted to an isomorphism of the corresponding semi-direct products  :
 i.e $E_{1}=S^{1}\ltimes _{\phi_{1}}G/K_{1}$ is isomorphic to $E_{2}=S^{2}\ltimes _{\phi_{2}}G/K_{2}$
 if and only if there is an isomorphism $i$ between the corresponding semi-direct product
$S^{1}\ltimes _{\phi_{1}}G$ and $S^{2}\ltimes _{\phi_{2}}G$
 which preserves $G$ and maps $S^{1}$ onto $S^{2}$ and $K_{1}$ onto $K_{2}$.
We will give necessary and sufficient conditions under which such an isomorphism exists.
% Such isomorphism exists if and only if for a pair $(j,b)$, where $b\in B$ and
% $j:S^{1}\rightarrow S^{2}$ is an isomorphism, the equations of Proposition \ref{isogh},
%  is fulfilled.
 %% then by $(\star)$, there is an isomorphism
% We will also use this result to compute the automorphism group of an extension and to
% reduce it to computing automorphism of the supplement $f^{-1}(B)$.
%semi-direct product %$S\rtimes G$ latex ne le reconnait pas
\end{enumerate}

\subsubsection{Examples of computing times}
We have implemented our method in GAP 4.2. We have checked that for
nonsolvable groups of order $\leq 2000$ our list of groups is
exactly the same as the one obtained by Eick and Besche in
\cite{Eick_Ulrich}. Unfortunately, for groups of orders between
$2000$ and $23,040$, there is no such list that can be compared to
ours. Nevertheless, we can compare our method to the method of
iterated cyclic extensions, which is implemented in GAP. But
this can only be done up to orders around $10,000$ because after
this bound, the time and memory  ressources needed become too large.
%of actual common personnal computer ( 900 mega of ram memory and a 900 Mhz cpu-speed) are
%unsufficient.
We have observed computing time beyond a day for a
fixed extension  by the iterated cyclic extension method, whereas our method
only took minutes ($<20$) to complete on the same example and is able
to handle much larger groups.
%%is also a very strong limit for an everyday-use of the

%\chapter{Basic notions}
\chapter{Elementary notions}\label{basic}
\subsubsection{Isomorphisms Theorems}
\begin{proposition}(Isomorphisms theorems)\label{isotheorems}
Let $G$ be a group and let $N$, $S$ and $K$ be subgroup of $G$.
\begin{enumerate}
\item If $f$ is a homomorphism from a group $G$ onto the group $f(G)$ then $f(G)\cong G/Ker f$.
\item If $N\unlhd G$ then $NS$ is a subgroup of $G$. If moreover $S\unlhd G$ then $NS\unlhd G$.
\item  If $N\unlhd G$ then $NS/N\cong S/(S\cap N)$. %%Let $N$ and $S$ be subgroups of $G$.
\item If $K$ and $N$ are normal in $G$ and if $K\leq N$ then $\frac{G/K}{N/K}\cong G/N$.
\item If $f:G\rightarrow f(G)$ is an homomorphism, then there is a one-to-one
correspondence between the subgroups of $f(G)$ and
the subgroups of $G$ that contain $Ker f$. Each subgroup of $f(G)$ is the image
of exactly one subgroup $S\leq G$ such that $Ker f \leq S$. The same holds for
normal subgroups of $f(G)$.
\end{enumerate}
\end{proposition}
Proofs can be found in the first chapters of any standard reference for group theory (as \cite{Hall} or \cite{Suzuki_vol1}).

\textbf{Definition.} Let $G$ be a subgroup of a group $E$.  A \emph{\textbf{right transversal for G in E}} is a
set $\tau$ that contains exactly one representative of each right coset of $G$ in $E$. A left transversal is defined
similarly for left cosets. If $G$ is normal in $E$, right transversals are left transversals (and conversely),
 so that $\tau$ is just called a transversal.

\subsubsection{The exponential notation} In this work, we use the exponential notation for various situations :
\begin{enumerate}
\item Let $g$ and $t$ be elements of a group $G$. The conjugate of $g$ under $t$ is  denoted by  $g^{t}:=t^{-1}gt$.
The same notation is used for the conjugate $S^t$ of a subgroup $S\leq G$.
\item Let $\pi_1$ and $\pi_2$ be automorphims of $G$. We denote the image of $g$ under $\pi_1$ by
$g^{\pi_1}$. The exponential notation is handy for multiplying automorphims as $(g^{\pi_1})^{\pi_2}=g^{\pi_1 \pi_2}$.
If the conjugation automorphism by $t$ is composed with $\pi$ then parentheses remain to denote the
image $(g^{t})^{\pi}$ of $g$. A notation as $g^{t\pi}$ is a misuse since the multiplication between exponents is
not defined.
\index{check everywhere that parentheses are respected}
\item Let $f$ be a function from a set $X\neq G$ to a group $G$. The composition of $f$ with an automorphism $\pi$
of $G$ is denoted by $f^{\pi}$ and hence it is defined by $f^{\pi}:x\to f^{\pi}(x):=(f(x))^{\pi}$ for every $x\in X$.
The same holds if we compose $f$ with the conjugation by an element $t$ to obtain the function
$f^{t}:x\to f^{t}(x):=(f(x))^{t}=t^{-1} f(x)t$.
Finally, observe that if $h$ is a function from $X$ to $Aut(G)$, then to be coherent with $(1)$, the element
 $h^{\pi}(x)$ must denote $\pi ^{-1} h(x)\pi$.
\end{enumerate}
%%%%%%%%%%%%%%%%%%%%%%%%%%%%%%%%%
\section{Solvable and Nonsolvable groups}\label{solv+nonsolv}
\textbf{Definition.} The \emph{\bf{derived subgroup}} $G'$ of a group $G$ is the subgroup of
$G$ generated by the commutators $[x,y]=yxy^{-1}x^{-1}=yx(xy)^{-1}$ of elements $x,y\in G$.

If $f$ is a homomorphism from $G$ onto a group $S$, then an element $[f(x),f(y)]$ of $S'$
is the image of $[x,y]\in G'$ so that $f(G')=S'$.
Thus every automorphism of $G$ leaves $G'$ invariant. A \emph{\bf{perfect group}}
 is a group $G$ such that $G=G'$. If $f$ is an homomorphism from a perfect group $P$ onto
 group $S$ then $S$ is perfect since $S=f(G)=f(G')$ and $f(G')=S'$.
%maps $[x,y]$ to $[x^{\omega},y^{\omega}]\in G'$ and
%thus $G'$ is left invariant by every automorpism of $G$ (it is a characteristic subgroup).
%Moreover the same holds for a group homomorphism $f$ : the image $G'$ under $\omega$
%is the derived subgroup

Let $N$ be a normal subgroup of $G$. Observe that since $yx=[x,y]xy$,  then $G/N$ is abelian if and only if $[x,y]\in N$ for
 every $x,y\in G$. Hence, $G/N$ is abelian if and only if $G'\subseteq N$ and $G'$ is the
 smallest normal subgroup of $G$ that produces an abelian factor group.

 The \emph{\bf{derived series}} is the sequence $G^{i}$, $i\in\mathbb{N}$, of characteristic
 subgroups of $G$ such that $G^{0}=G$ and $G^{i+1}=(G^{i})'$. Each factor $G^{i}/G^{i+1}$
 of this series is thus abelian. If $G$ is finite, the derived series must reach a stationary term and there
  exists a number $n$ such that
 $G^{n}=G^{n+1}=G^{n+k}$ for every $k\in\mathbb{N}$. The stationary term $G^{n}$ is a perfect
 group denoted $G^{(\infty)}$ and called the \emph{\bf{perfect residuum}} (or solvable residual) of $G$.

A group $G$ (finite or infinite) is called \emph{\bf{solvable}} if $G^{n}=1$ $=G^{(\infty)}$ for some
positive integer $n$. If $G$ is solvable and perfect, then $G=1$ since $G=G'=G^{(\infty)}=1$.
Factor groups and subgroups of solvable groups are solvable and an extension of a solvable group
by a solvable group is solvable (see \cite{Hall} page 139).
Thus, if $f$ is an homomorphism from a perfect group $P$ onto a solvable group $S$ then
$f(S)$ is both perfect and solvable so that $S=1$.
A finite nonsolvable group may be defined as a finite group such that $G^{(\infty)}\neq 1$.

%Let us list some basic properties of  solvable groups (see \cite{Hall} page 139).
\subsubsection{The iterated cyclic extensions method}\label{iteratedcyclic}
One can show that a group $G$ is solvable if and only if it has a subinvariant series
$G=B_{1}\unrhd B_{2}\unrhd\ldots\unrhd B_{t}=1$ in which every $B_{i-1}/B_{i}$, $i=1,\ldots,t$
is abelian (see \cite{Hall} page 140).  For a finite nonsolvable group $G$, the series
$G/G^{(\infty)}\geq G'/G^{(\infty)}\geq G''/G^{(\infty)}\geq\ldots\geq G^{(\infty)}/G^{(\infty)}=1$
is a subnormal series whose factors are abelian since
$\frac{G^{i}/G^{(\infty)}}{G^{i+1}/G^{(\infty)}}\cong G^{i}/G^{i+1}$ (see Proposition \ref{isotheorems}).
Hence, every  \emph{\bf{finite nonsolvable group}} $G$ is an extension of a nontrivial perfect group $G^{(\infty)}$ by the
finite solvable group $G/G^{(\infty)}$. A catalogue of perfect group (see \cite{Holt_Plesken})
 and a catalogue of solvable group (for instance the Small Groups library of Eick and Besche)
 are thus crucial tools for classifying finite nonsolvable groups by the extension
 method.

 Let $G=B_{1}\unrhd B_{2}\unrhd\ldots\unrhd B_{t}=1$ be a subnormal series with abelian factors for a finite
solvable group $G$. Since every finite abelian group is a direct product of cyclic groups, one can show that
every finite solvable group has a subnormal series $G = C_{1} > C_{2} > \ldots > C_{n} > C_{n+1} = 1$
with cyclic factors.  Hence, for group extensions, a finite solvable group is just a group obtained by iterated
cyclic extensions : for $i\in 1\ldots n$, $C_{i}$ is an extension of $C_{i+1}$ by the cyclic group
$C_{i}/C_{i+1}$. It follows also, that a finite nonsolvable group $G$ with perfect residuum $P$
can be constructed from $P$ by a finite sequence of cyclic extensions. The cyclic factors correspond to the
factors of some subinvariant series for $G/P=G/G^{(\infty)}$.

 If $Aut(N)$ is known, there is a method to construct the extensions of a given group $N$ by
a cyclic group. This method is described in subsection \ref{cyclicext} and has been used by Betten
(\cite{Betten}) to construct finite solvable groups up to order 242. But it seems that the isomorphism problem
for this method is quite hard because a lot of isomorphic copies of the same group are produced. The
iterated Frattini extension method has been used instead by Eick and Besche (\cite{Eick_Ulrich}) and they
have provided more tools to reduce isomorphic copies. Consequently, they where able to list all finite
solvable group up to order 2000. Nevertheless, for the nonsolvable groups up to order 2000, they have
also used iterated cyclic extensions from a given perfect group that starts the sequence. All together,
they obtained a list of the finite groups up to order 2000 that is known as "The Small Groups Library" and
 is available on both GAP and MAGMA systems.
%Hence, if $N$ is a normal solvable subgroup
%of $G$ and if $G/N$ is solvable then $G$ is solvable. It also follows that $G/G^{(\infty)}$ is solvable.
%%%%and one can show that if $G/N$ is solvable then $G^{(\infty)}\leq N$ (see \cite{Rose} page 149).
%Therefore every finite \emph{\bf{nonsolvable}} group $G$ is an extension of a perfect group $G^{(\infty)}$
% by a solvable group $G/G^{(\infty)}$.
%%%%%%%%%%%%%%%%%%%%%%%%%%%%%
%%%%%%%%%%%%%%%%%%%%%%%%%%%%%%%%%

%have an efficency that can be even compared to that
%whose efficiency can be compared to ... for permutation groups.
%confluent system of rewriting rules such that
%every word in the generators can be converted into a unique normal form.
%give a summary of the article of Ako Seress (\cite{seress2001}), refer to it and give the hierarchy of
%presentations efficency (depending on the algorithm). Also cite \cite{Cannon_Havas}
%give also some comments that cannon gave in his talk in Brussel.\\
\subsubsection{pcgs and polycyclic groups}\label{polycyclic}
%%%MMMMMMMMMMM  Abstract from the gap 4 manual.\\
The efficient algorithms that have been developed to study finite solvable groups use the polycyclic presentation.
We briefly outline now what a polycyclic presentation is. Further details can be found in \cite{GAP4}.
 A group G is \emph{\bf{polycyclic}} if there exists a subnormal series
 $G = C_{1} > C_{2} > \ldots > C_{n} > C_{n+1} = {1}$ with cyclic factors (not necessarily finite).
Such a series is called pc series of $G$. Every polycyclic group is solvable and every finite solvable
 group is polycyclic. However, there are infinite solvable groups which are not polycyclic, for instance, the
 multiplicative group of real numbers (which is abelian and thus solvable).

%In Computational Group Theory
%In computer algebra systems like GAP and MAGMA, there exists a large number of methods for polycyclic
%groups which are based upon the polycyclic structure of these groups.
%These methods are usually very efficient and hence GAP tries to use them whenever possible.
 Let $G$ be a polycyclic group with a pc series as above. A polycyclic generating sequence
 (pcgs for short) of $G$ is a sequence $P : = (g_{1},\ldots, g_{n})$ of elements of $G$
 such that $C_{i} = \langle C_{i+1}, g_{i} \rangle$ for $1 \leq i \leq n$.
%Note that each polycyclic group has a pcgs, but except  for very small groups, a pcgs is not unique.
%For each index $i$ the subsequence of elements $(g_{i},\ldots, g_{n})$ forms a pcgs of the
% subgroup $C_{i}$. In particular, these tails generate the subgroups of the pc series and
%  hence we say that the pc series is determined by P.

Let $r_{i}$ be the index of $C_{i+1}$ in $C_{i}$ which is either a finite positive number
or infinity. Then $r_{i}$ is the order of $g_{i} C_{i+1}$ and we call the resulting list
of indices the relative orders of the pcgs P.

Moreover, with respect to a given pcgs $(g_{1},\ldots, g_{n})$ each element $g$ of $G$ can
 be represented in a unique way as a product
 $g = g_{1}^{e_{1}} *g_{2}^{e_{2}} \ldots g_{n}^{e_{n}}$ with exponents
 $e_{i}\in  \{0,\ldots, r_{i-1}\}$, if $r_{i}$ is finite, and $e_{i}\in \mathbb{Z}$ otherwise.
 Words of this form are called %normal words or
 words in normal form. Then the integer vector $[e_{1}, \ldots , e_{n}]$ is called the exponent
vector of the element $g$.
%Furthermore, the smallest index $k$ such that $e_{k}\neq 0$
%is called the depth of g and $e_{k}$ is the leading exponent of $g$.

%For many applications we have to assume that each of the relative orders
%ri is either a prime or infinity. This is equivalent to saying that there
% are no trivial factors in the pc series and the finite factors of the pc
% series are maximal refined. Then we obtain that $r_i$ is the order of $g C_{i+1}$
% for all elements $g in C_i \backslash C_{i+1}$ and we call $r_i$ the relative order of the
% element g.
%pcgs which has particularly nice properties, for example it always refines
% an elementary abelian series, for p-groups it even refines a central
%series. These nice properties permit particularly efficient algorithms.

%Let $G$ be a polycyclic group with pcgs $P = (g_1, \ldots, g_n)$
%and corresponding relative orders $(r_1, \ldots, r_n)$. Recall that the
%$r_i$ are positive integers or infinity and let $I$ be the set of indices
%$i$ with $r_i$ a positive integer. Then $G$ has a finite presentation
%on the generators $g_1, \ldots, g_n$ with relations of the following
%form.
%$$
%\begin{array}{l l l}
   %    $(g_i)^{r_i}$ & $=$ & $g_{i+1}^{a(i,i,i+1)} \ldots g_n ^{a(i,i,n)}$\\
     %            & & \hbox{ for } $1 \leq i \leq n$ \hbox{ and } $i \in I$\\
%$g_i^{-1} g_j g_i$ & $=$ & $g_{i+1}^{a(i,j,i+1)} \ldots g_n^{a(i,j,n)}$\\%
%& & \hbox{ for } $1 \leq i \< j \leq n$\\

%\end{array}
%$$
\index{laissez tombez power commutator presentation: pas utile car je ne detaille pas dans chap 5}
%A finite presentation of this type is called a power-conjugate presentation and a pc group is a
%polycyclic group defined by a power-conjugate presentation. Instead of conjugates we could just as
%well work with commutators and then the presentation would be called a power-commutator
%presentation. Both types of presentation are abbreviated as pc presentation.
%Note that a pc presentation  is a rewriting system.

%MMM COLLECTED WORD\\

%%%%   $e_k$ works, but not $e_kbg$
%Note from Chapter 5 :tchap5.tex
%Note that a pc-group with a pcgs of length $n$ is defined by
%$r=n(n+1)/2$ relations (the $k^{th}$ generator gives one power relation and $n-k$ conjugate
%relations for its action on the $n-k$ remaining generators : Thus $r=n+(n-1)+\ldots +1$).
%MMMMMMMMMMMMMMMMMMMM

%%%%%%%%%%%%%%%%%%%%%%%%%
\section{Extensions}
%%%%%%%%%%%%%%%%%%%%%%%%%%%%%%%%%%%%%%%%
\subsection{The action on $G$ associated with an extension}\label{couplingbasic}
Let $E$ be an extension of $G$. Since $G\lhd E$, each $e\in E$ induces by
 conjugation an automorphism $f(e)$ of $G$. This defines an action $f$ of
 $E$ on $G$, that is to say a homomorphism from $E$ into $Aut(G)$.
% Many times in our proof we need this $f$, and we say that $f$ is induced by conjugation in $E$.
In $E$, the elements of $G$ act on $G$ itself and they induce inner
 automorphisms of $G$. Therefore, $f(G)=Inn(G)$ and $f(E)$ is a subgroup of
 $Aut(G)$ containing $Inn(G)$ and we say that the \emph{\bf{extension}} $E$ is
\emph{\bf{ associated with the subgroup $f(E)$}}. In this work, we often choose
 a subgroup $L$ such that $Inn(G)<L<Aut(G)$ and we try to construct the
 extensions of $G$ associated with $L$.
 Since the elements of $Ker f$ are the elements of $E$ acting trivially
 on $G$ by conjugation, $Ker f=C_{E}(G)$, the centralizer of $G$ in $E$.
 We get that $E/C_{E}(G)\cong f(E)$ and there is a one-to-one correspondence
  between subgroups $f(E)$ and subgroups of $E$ containing $C_{E}(G)$.

%it corresponds to a subgroup $B$ of $Out(G)=Aut(G)/Inn(G)$.
%Followed by the homomorphism from $Aut(G)$ onto $Out(G)$
There is another homomorphism which plays a fundamental role in order to
describe an extension of $G$. If $eG$ is a coset of $G$ in $E$, its image
 $f(eG)=f(e)f(G)=f(t)Inn(G)$ is  a coset of $Inn(G)$ in $Aut(G)$. In this way
$f$ induces a homomorphism $\overline f$ from $E/G$ into $Out(G)=Aut(G)/Inn(G)$,
 the group of outer automorphisms, and $\overline f (E/G)=f(E)/Inn(G)$.
  We say that each \emph{\bf{extension}} of $G$ by $H=E/G$ is
\emph{\bf{associated with a homomorphism from $H$ into $Out(G)$}}. But it is not
true that any homomorphism from $H$ into $Out(G)$ is associated with an
extension of $G$ by $H$ . A counterexample for $G=SL_{2}(9)$ is described page xxiv, chapter 7
in \cite{Atlas}.

%\subsection{Conjugation on $G$ and coupling}\label{couplingbasic}
%If $E$ is an extension of a group $G$ and if for $e\in E$, $f(e)$ denotes the automorphism
%$g\rightarrow e^{-1}ge$ ($g\in G$), then $f:E\rightarrow Aut(G)$ is an homomorphism.
%Since $f(G)=Inn(G)$, another homomorphism from $E/G$ into $Out(G):=Aut(G)/Inn(G)$ is induced
%by $f$; it is defined by $eG\rightarrow f(eG)=f(e)Inn(G)$.
%ùNow, assume that $E$ is an extension of $G$ by $H$, and that a transversal
%$\{\overline{h} :\, h\in H\}$ of $G$ is given such that
%$\overline{h}G\rightarrow h$ is an isomorphism from $E/G$ onto $H$.
%We obtain an homomorphism $\Phi :H\to Out(G)$ defined by $h\to \overline{h}G\to f(\overline{h}) Inn(G)$.
%From now we say that the homomorphism \emph{$\Phi$ is associated to the extension $E$}.
%%%%%%%%%%%%%%%%%%%%%%%%%%%%%
%\subsection{Group Construction methods}
We now expose some classical group construction methods.

\subsection{Semi-direct products}
%a family of extensions, which is one of the most known.
%%$G\rtimes S$  $S\ltimes _{\phi} G$
\textbf{Definition .} Let $G$ and $S$ be groups. \emph{An action of S on G}
is a homomorphism from $S$ into $Aut(G)$. We say that $S$ acts on $G$.

\begin{proposition}Let $\phi$ be an action of a group $S$ on a group $G$.
The set of ordered pairs $(s,g)$ where $g\in G,\,s\in S$ provided with the following
multiplication is a group.
$$(s_{1},g_{1})(s_{2},g_{2})=(s_{1}s_{2},g_{1}^{\phi (s_{2})}g_{2})$$
\end{proposition}
A proof can be found in \cite{Suzuki_vol1}, page 67. \\
This group is called \emph{the semi-direct product} of $S$ and $G$ with respect to the
action $\phi$ and we use the notation $S\ltimes _{\phi}G$ for it.
Note that in   the inverse of the element $(s,g)$ is $(s^{-1},(g^{\phi (s^{-1})})^{-1}).$

Where does this strange multiplication  comes from ?

If $G$ is a normal subgroup of a group $E$ and if $S$ is a subgroup of $E$ then, $S$
 acts on $G$ by conjugation. For $s\in S$
and $g\in G$, the mapping $\phi$ that maps $s$ onto the automorphism
$\phi(s):g\to s^{-1}gs:=g^{\phi(s)}$ of $G$, is a homomorphism from $S$ into $Aut(G)$.
If $E$ is generated by $G$ and $S$, then $E=SG$ and every element of
 $E$ can be written as a product $sg$ of an element of $S$ and of an element of $G$.
The product $s_{1}g_{1}s_{2}g_{2}$ of such elements can be written as a product $sg$
by the trick $s_{1}g_{1}s_{2}g_{2}=s_{1}s_{2}s_{2}^{-1}g_{1}s_{2}g_{2}=s_{1}s_{2}g_{1}^{\phi(s_{2})}g_{2}$.
The inverse of the element $sg$ is
$g^{-1}s^{-1}=s^{-1}sg^{-1}s^{-1}=s^{-1}(g^{\phi (s^{-1})})^{-1}.$

It can be proved that if $S\cap G=1$ then $E$ is isomorphic to the semi-direct product $S\ltimes _{\phi}G$
(\cite{Suzuki_vol1} page 69 ).

\textbf{Definition .}  Let $G$ be a  normal subgroup of a group $E$. A subgroup $S$ of $E$ such that
$E=SG$ and such that $S\cap G=1$ is called a \emph{\bf{complement}} to $G$ in $E$.
Such extension $E$ of $G$ is called a \emph{\bf{split extension}}.
Note that $S\ltimes _{\phi}G$ is a split extension of its normal subgroup $(1_{S},G)$
by its complement subgroup $(S,1_{G})$.

%\begin{proposition}\label{defsemi} Let $G$, $S$ be subgroups of a group $E$
%such that: \begin{enumerate}
%\item $G\lhd E$;
%\item $G\cap S=1$;
%\item $E$ is generated by $G$ and $S$;
%\end{enumerate}
%Then $E$ is called a (inner) semi-direct of $G$ and $S$ and : \\
%$(i)$ every $e\in E$ can be written in a unique as $e=sg$ for $g\in G$ and $s\in S$;\\
%$(ii)$ the quotient $E/G$ is isomorphic to $S$.
%\end{proposition}
%For a proof see \cite{Armstrong}.
%Such extension $E$ is called a \emph{\bf{split extension}} of $G$.
%Note that $S\ltimes _{\phi}G$ is a (inner) semi-direct product of its normal subgroup $(1_{S},G)$
%by the subgroup $(S,1_{G})$.

%  Let $\overline E=\{gh:\,g\in G,\,h\in H\}$. The knowledge of the action
%  $\Phi$ is sufficient to define the product of elements in $\overline E$.
%  If $g_{k}\in G,\,h_{k}\in H(k=1,2)$ then
%  $$g_{1}h_{1}g_{2}h_{2}=g_{1}h_{1}g_{2}h_{1}^{-1}h_{1}h_{2}=g_{1}\Phi_{h_{1}}(g_{2})h_{1}h_{2}=(*).$$
%  Since $g_{1}\Phi_{h_{1}}$ belongs to $G$, the element $(*)$ belongs to $\overline E$.
%  For an element $gh\in\overline E$, its inverse is $$(gh)^{-1}=h^{-1}g^{-1}=h^{-1}g^{-1}hh^{-1}=\Phi_{h^{-1}}(g^{-1})h^{-1}.$$
%  By the same argument, $(gh)^{-1}\in\overline E$ and this proves that
%  $\overline E$ is a subgroup of $E$. The next result shows the importance
%  of the multiplication defined by $(*)$.
\subsubsection{The Schur-Zassenhaus theorem}
The \emph{Schur-Zassenhaus} theorem whose proof can be found in
\cite{Suzuki_vol1}, page 236-7, is stated as follows. %first version of the theorem=1902 schur,
% 30 years later Zassenhaus generalized it to the present version
\begin{theorem}\label{schurzass} Let $G$ be a normal subgroup of a finite group $E$. If its order $|G|$
 is relatively prime to $|E|/|G|$ then there exists a complement of $G$ in $E$ and
 any two complements are conjugate in $E$.
\end{theorem}
The proof in \cite{Suzuki_vol1} of the conjugacy of the complements uses the
Feit-Thompson Theorem on the solvability of every group of odd order. %While writing his book
%in 1982, Suzuki mentioned that there was no known proof which does not use this theorem.
According to \cite{Suzuki_vol1} and \cite{LyonsNotes} there is no known proof which does not use this theorem.
With the additional hypothesis that either $E/G$ or $G$ is solvable, it becomes possible
to have a proof that do not use the odd order Theorem.

%%\section*{Extensions and subdirectproducts.}
%%\subsection{Central products}
%%%??????????? ALLER VOIR ASCHBASCHER PG 32 SUR CENTRAL PRODUCT
\subsection{Cyclic extensions}\label{cyclicext}
If $E$ is an extension of $G$ such that $E/G$ is a cyclic group, then
$E$ is called  a \emph{cyclic extension} of $G$. If $E/G$ is a cyclic group of order $n$, then
$E$ is generated by $G$ and by an element $\theta \notin G$ such that
 $\theta^{n}=\tilde g\in G$ and $\theta^{i} \notin G$ for $i\in 0,1\ldots n-1$. The elements of $E$ can be
 written in a unique way as $\theta^i g$ for some $g\in G$ and some $i\in 1\ldots n-1$.
 The element $\theta$ induces by conjugation an automorphism  $\alpha :g\to \theta^{-1} g\theta:=g^{\alpha}$
of $G$. Hence, the pair $< \alpha, \tilde{g}>$ completely describes $E$ since in $E$ the product of two elements
$\theta^{i} g_{1}$ and $\theta^{j} g_{2}$ is equal to
$\theta^{i} g_{1}\theta^{j} g_{2}=\theta^{i+j}\theta^{-j} g_1 \theta^{j}g_2 :=\theta^{i+j} g_{1}^{\alpha^j} g_2$.

In order to construct a cyclic extension of $G$, the pair $< \alpha, \tilde g>$ can not be chosen anyhow.
As $\theta^{n}=\tilde g$,
\begin{enumerate}
\item the automorphism $\alpha^{n}$ must be the conjugation by $\tilde g$ in $G$;
\item $\alpha$ must fix $\tilde g$ since $\tilde{g}^{\alpha}:=\theta^{-1}\tilde{g}\theta=\theta^{-1}\theta^{n}\theta=\theta^{n}=\tilde{g}$.
\end{enumerate}
 One can show that conditions $(1)$ and $(2)$ are sufficient (\cite{Hall} page 224).
  To every pair $< \alpha, \tilde g>$ that fulfills these conditions
 corresponds a cyclic extension of $G$ whose elements are $\theta^{i} g\, (g\in G,\,i=0,1,..,n-1)$
 with the following product :
$$\theta^{i} g_{1}\star \theta^{j} g_{2}=M g_{1}^{\alpha^j}g_2$$
where $M=\theta^{i+j}$ if $i+j < n$ and $M=\theta^{i+j(mod\, n)} \tilde{g}$ if $i+j\geq n$.

%%%%%%%%%%%%%%%%%%%%%%%%
%Montrer des dessins une seule fois et dire que dans les demos, c'est utile de les
%avoir en tete pour y voir plus clair.\\
%L'application principale de la these est l'etude des groupes qui ont
%un perfect residuum $P\neq 1$ a partir des proprietes de P et de E/P solvable
%This covers in particular the finite non-solvable groups.
%Dessin pour subdirect product . If $K_{1}\cap K_{2}=1$ the $G$ is isomorphic to a su
%subdirect product of $(G/K_{1})\times (G/K_{2})$.
%Dessin du subdirect product en disant que c'est un diagonal generaluse.
%Other explanation : pour l'expliquer : the combination of two existing homomorphism
%gives a third one on $A\times B$. Double modulo.on a l'info de $A$ et de $B$.
% Dire que solvable are very well known since the work of Hall and others.
% See the chapter of tome 2 suzuki on all the various characterization of solvable groups.
% For p-groups, see Newman-O'brien (thus also for nilpotent group, their direct product)
% For non nilpotent solvable group, see bettina eick.
%%%%%%%%%%%%%%%%%%%%%%%%%%%%%%%%ù
\index{pas besoin des central product=cas particulier de quotient de semi direct}
\subsection{G-isomorphisms}
 For $i=1,2$ let $E_{i}$ be extensions of a normal subgroup $G$.
\begin{definition}\textbf{Definition }  An isomorphism $j$ from $E_{1}$ onto $E_{2}$ such
that $j(G)=G$ is called a $G-$isomorphism.
\end{definition}

We note that a $G$-isomorphism induces an isomorphism from $E_{1}/G$ onto $E_{2}/G$.
In many cases, isomorphisms from $E_{1}$ onto $E_{2}$ $j$ are $G$-isomorphisms
: for instance in the case where $G$ and $E_{i}/G$ do not have a common composition factor.
More generally we have the following result :
\begin{lemma}\label{allgiso} Let $E_{1}\unrhd G$ and $E_{2}\unrhd G$ be  two extensions such that no factor
 group of $G$ is isomorphic to a non trivial normal subgroup of $E_{2}/G$.
 Then, any isomorphism between $E_{1}$ and $E_{2}$ is a $G$-isomorphism.
\end{lemma}
\emph{Proof:} Let $i$ be an isomorphism from $E_{1}$ onto $E_{2}$ and let
$\Psi$ be the canonical homomorphism from $E_{2}$ onto $E_{2}/G$. The subgroup $i(G)$
is normal in $E_{2}$ (because $G$ is normal in $E_{1}$) and thus
$\Psi(i(G))\cong i(G)/G\cap i(G)$ is a normal subgroup of $E_{2}/G$.
The group $i(G)/G\cap i(G)$ is isomorphic to a factor group of $G$ (since $i(G)\cong G$),
hence, our hypothesis
implies that $i(G)/G\cap i(G)\cong 1$ so $ i(G)=G\cap i(G)$ and  $ i(G)\subseteq G$.
Using the same argument
with $i^{-1}$ , which is an isomorphism from $E_{2}$ onto $E_{1}$ , we find
$i^{-1}(G)\subseteq G$ which gives ( applying $i$ to both sides of this inequality)
 $G\subseteq i(G)$. Finally $i(G)=G$ , so $i$ is a $G$-isomorphism$\,\Box.$

 For a classification of finite nonsolvable groups as extensions of a perfect group $G$
 by a solvable group, this is also the case. Any normal subgroup of the solvable group
 $E/G$ is solvable too, but since $G$ is perfect, its only solvable quotient is $1$
  (the trivial one). Therefore we obtain the following corollary of Lemma \ref{allgiso}:
%%%%%%%%%%%%%%%%%%%%%
%a normal  is a subgroup of $H$, so it is solvable ; but it is also a quotient of $i(G)$.
% Since $i(G)$ is perfect, its only solvable quotient is $1$, (the trivial one),
%and so $\Psi(i(G))=1$, which implies that $i(G)\subseteq G$. Using the same argument
%with $i^{-1}$ , which is an isomorphism from $E_{2}$ onto $E_{1}$ , we find
%$i^{-1}(G)\subseteq G$ which gives ( applying $i$ to both sides of this inequality)
%%%%%%%%%%%%%%%%%%%
\begin{corollary} \label{giso}Let $E_{1}\unrhd G$ and $E_{2}\unrhd G$ be  two extensions
of a perfect group $G$ by a solvable group $H$. Any isomorphism between
$E_{1}$ and $E_{2}$ is a $G$-isomorphism.
\end{corollary}

There is another situation where $G$-isomorphisms exist : if in the extension $E_{2}$ of
$G$, all the normal subgroups isomorphic to $G$ such that $E_{2}/N\cong E_{2}/G$ are in
the same orbits under $Aut(E_{2})$. Then if  $j:E_{1}\rightarrow E_{2}$ is an
 isomorphism and $j(G)=G^{\omega}$ for some $\omega\in Aut(E_{2})$.
Thus $\omega^{-1}j$ is a $G$-isomorphism (right-sided notation) and in this situation,
there exists a $G$-isomorphism if and only if there exists an isomorphism.

%%%%%%%%%%%%%%%%%%%%%%%%%%%%%%%
%%  SOME LATEX LAYOUT %%%%%%
%%%%%%%%%%%%%%%%%%%%%%%%%%%%%
%  \chapter{Basic results on extensions}
%  \section{The action on $G$ associated to an extension.}
%  \emph{associated with a homomorphism from $H$ into $Out(G)$}.
%  \subsection{Extensions with inner action on $G$}
%
%  \begin{proposition}\label{inneract}
%  \end{proposition}
%  \emph{Proof }. .\quad\Box$

%  \textbf{Definition }.

%  \begin{corollary}
%  Every extension of a complete group  is a direct product.
%  \end{corollary}
%
%  begin{enumerate}\item
%  \item see \cite{Carter}).
%  \end{enumerate}
%
%  \begin{corollary}If $H$ is a group of odd order and if $G$ belongs to the list
%  below then, the only extension of $G$ by $H$ is the direct product $G\times H$.
%  \begin{itemize}
%  \item The sporadic groups $M_{12},M_{22},J_{2},Suz,HS,McL,He,Fi_{22},
%  Fi^{'}_{24},HN,O'N,J_{3}$.
%  \item The alternating groups $A_{n}$ ( $n>3$).
%  \item $L_{2}(p^{k})$ where $p$ is an odd prime number and $k$ is a power of $2$.
%  \end{itemize}
%  \end{corollary}
%%%%%%%%%%%%%%%%%%%%%%%%%%%%%%%%%

\chapter{Constructing extensions}\label{quotientsemi}
\subsubsection{Introduction}
In chapter \ref{basic} we have recalled some classical methods for building a new
group from two given groups. We will present here a new construction that generalizes both
the semi-direct product and the central product. This new construction is a semi-direct
product where some subgroups are amalgamated and we will use it for almost all extensions that
are considered in this work. In chapter \ref{findsup}, we will prove that this method can be used to
reduce the classification of extensions of a group $G$ by a group $H$ to the classification of
the extensions of a nilpotent subgroup $U$ of $G$ by $H$. %, where $U$ is a nilpotent subgroup of $G$.
For our main application,  the extensions of a perfect group $P$ by a solvable group $H$,
%We prove for instance that if $P/Z(P)$ has less than 30,720 elements, then the order of the
%subgroup $U\leq G$ is less than $2|Z(P)|$, which is a significant improvement. Moreover
the extensions of $U$ by $H$ are solvable, so that everything can be reduced to polycyclic
groups in the finite case.

%In this chapter we view extensions $E$ of a group $G$ as groups containing $G$ as a
%normal subgroup. It is possible to consider instead that $E$ contains a normal subgroup isomorphic
%to $G$, but this formulation makes our proofs much less readable.

\section{Reduction to extensions of a small subgroup $U$}
\index{ Dessin pour supplement (but montrer que le quotient se lit dans le quotient
du supplment}
\textbf{Definition}. Let $G$ and $S$ be subgroups of a group $E$. If every element $e$
of $E$ can be written as $e=sg$ for some $g\in G$ and some $s\in S$ (that is
$E=SG$), then $S$ is called a \emph{\bf{supplement to $G$ in $E$}}. Note that if $E=SG$ then
$E=GS$ since $E=E^{-1}=(SG)^{-1}=G^{-1}S^{-1}=GS$.
%%%%%% the following commentary is for theorem about G-isomorphisms between extensions
The way an element of $E$ can be written as a product $s.g$ is not necessarily unique : for every
$u\in S\cap G$ we have $su\in G$, $u^{-1}g\in G$ so that $s.g=su.u^{-1}g$. Conversely,
if for $i=1,2$, $s_{i}\in S$ and $g_{i}\in G$, the equality $s_{1}g_{1}=s_{2}g_{2}$ implies
$s_{1}^{-1}s_{2}=g_{1}g_{2}^{-1}\in S\cap G$ so that there exists $u\in S\cap G$ such that
$s_{2}=s_{1}u$ and $g_{2}=u^{-1}g_{1}$. This proves that every element of $E$ is uniquely written as a
product $sg$ where $s\in S$ and $g\in G$ if and only if $S\cap G=1$.

Note that by Proposition \ref{isotheorems}, $E/G\cong S/(S\cap G)$ for every supplement $S$ to $G$ in $E$.
If $G$ is normal in $E$ and $S\cap G=1$, then the supplement $S$ is
 a complement to $G$, and $E$ is a split extension.% (or semi-direct product).
%Let us recall that if $\phi$ is an action of a group $S$ on a group $G$,
%the semi-direct product $S\ltimes _{\phi}G$ is the set of ordered pairs $(s,g)$ where
%$g\in G,\,s\in S$ with the multiplication
%$$(s_{1},g_{1})(s_{2},g_{2})=(s_{1}s_{2},g_{1}^{\phi (s_{2})}g_{2})$$
                        
\subsection{Supplements and quotient of semi-direct product}

%%%%%%%%%%%%%%%%%%%%
\begin{theorem}\label{semi}Let $E$ be an extension of a group $G$. Let $S$ be a
supplement to $G$ in $E$  and let $\phi$ be the action of $S$ on $G$ by conjugation.
\begin{enumerate}
\item $\varphi :S\ltimes _{\phi}G\rightarrow E: (s,g)\rightarrow sg$ is an epimorphism.
%\item $E$ is isomorphic to $S\ltimes _{\phi}G/Ker \varphi$ and
\item $E\cong S\ltimes _{\phi}G/Ker \varphi$ where $Ker \varphi=\{(u,u^{-1})\,:\,u\in G\cap S\}\cong G\cap S$
%$E\cong S\ltimes _{\phi}G/(G\cap S) $.\scriptstyle
%\item In $S\ltimes _{\phi}G$, we have $Ker \varphi\cap S=1=Ker \varphi\cap G$.
% \item $K\cap (S\times 1_{G})=1_{S\ltimes _{\phi}G}=K\cap (1_{S}\times G)$
\item $K\cap (S_{\times 1_{G}})=1_{S\ltimes _{\phi}G}=K\cap (G_{\times 1_{S}})$ where $K=Ker \varphi$.
\end{enumerate} \end{theorem}

\emph{Proof}.\begin{enumerate}
%\item
%For $s\in S$, the conjugation by $s$ on $G\unlhd E$ defines the automorphism $\phi(s)$
% of $G$ defined by $g^{\phi(s)}=s^{-1}gs=g^{s}$ for $g\in G$.
%Observe that the product of two elements of $E=SG$ can also be written
% as an element of $SG$ by the following rule :
%$s_{1}g_{1}s_{2}g_{2}=s_{1}(s_{2}s_{2}^{-1})g_{1}s_{2}g_{2}=s_{1}s_{2}g_{1}^{s_{2}}g_{2}
%=s_{1}s_{2}g_{1}^{\phi (s_{2})}g_{2}$.\\% $(\star)$

\item The mapping
 $\varphi :(s,g) \to sg$, from $S\ltimes _{\phi}G$ to $E$ is  onto since $E=SG$ and it is a
 well-defined homomorphism
 because $\varphi [(s_{1},g_{1}).(s_{2},g_{2})]=\varphi [(s_{1}s_{2},g_{1}^{\phi (s_{2})}g_{2})]=
s_{1}s_{2}g_{1}^{\phi (s_{2})}g_{2}=s_{1}g_{1}s_{2}g_{2}=\varphi [(s_{1},g_{1})]\cdot \varphi [(s_{2},g_{2})].$

% Since in $E$, the group $G$ is normal, $S$ is acting on $G$ by conjugation and this
% action defines a homomorphism $\Phi$ from $S$ into  $Aut(G)$.

\item Let us compute the kernel $K$ of this homomorphism. \\
$(s,g) \in Ker \ \varphi $ means that $sg=1$ in $E$, whence $g=s^{-1}\in G \cap S$.
 Conversely, for every $u\in G\cap S $ the element $(u,u^{-1})$ is in $Ker \ \varphi $.
 Hence, \mbox{$Ker \ \varphi =\{\ (u,u^{-1})\ : \ u\in G\cap S \}$}.\\

\item Trivial by 2.
% \item Let $K:=Ker \varphi$. In $S\ltimes _{\phi}G$, we identify $S$ with $S\times 1_{G}$
% and $G$ with $1_{S}\times G$.
%$K\cap (S\times 1_{G})=1_{S\ltimes _{\phi}G}$.
% Indeed if  $x=(u,u^{-1})$ is in $1_{S}\times G$,
% we must have $u=1$ and thus $x=1_{S\ltimes _{\phi}G}$. In the same way,
% $K\cap (1_{S}\times G)=1_{S\ltimes _{\phi}G}$ since
%  $(u^{-1},u)\in S\times 1_{G}$ implies $u=1$ and $x=1_{S\ltimes _{\phi}G}$.
%Clearly, the projection $(u,u^{-1}) \to (u,1)$ is an isomorphism
% from $K$ onto $(G\cap S)\times 1_{G}$.$\qquad \Box$\\
\end{enumerate}
%%%%%%%%%%%%%%%%%%%%%%%
\subsubsection{Outline of this section}
Hence, every extension $E$ of $G$ for which we know a supplement $S$ to $G$, can be completely
described by $S$ and by its conjugation action $\phi$ on $G$. If $E$ is an extension of $G$ by a group $H$,
then $S$ is an extension of a subgroup $U< G$ by $H$. This analysis of extensions suggests a
way to construct extensions by just constructing an extension of a subgroup $U< G$ together with a
homomorphism $\phi$ from $S$ into $Aut(G)$.

For this purpose, we will first need to determine under which conditions such a pair $(S,\phi )$ corresponds
to a supplement to $G$ in an extension (see Proposition \ref{ucns}). We will show in subsection
\ref {semi-identify}, that the key property is the existence of a normal subgroups $K\unlhd S\ltimes _{\phi}G$
such that $K\cap G=K\cap S=1$. All the rest follows from it.
Then we will need to know how to construct such a pair and we will illustrate this by the example of section
\ref{extsl29}. Since one supplement is sufficient to describe each extension, we will also need a function
that associates to each extension, a unique supplement to $G$ (see subsection \ref{suppfunct}).
%%%%%%%%%%%%%%%%%%%%%%%%%%%%%%%%%%%%%%%%%%%%%%%
\subsection{Semi-direct product with identified subgroups}\label{semi-identify}
%CONVERSELY UNDER WHICH CONDITIONS IS IT A SUPPLEMENT
\subsubsection{Identification conditions}
We state for the group $S\ltimes _{\phi}G$, the conditions (necessary and sufficient) under which
we can identify a subgroup $U\leq G$ and a subgroup $i(U)\leq S$ as in Proposition \ref{semi}.
%%by factoring $S\ltimes _{\phi}G$ by a normal subgroup $K$ such that $K\cap G=K\cap S=1$.

\begin{proposition}\label{ucns} Let $S\ltimes _{\phi}G$ be a semi direct-product of the groups $G$
 and $S$. Let $i:u\rightarrow \tilde{u}$ be an isomorphism from a subgroup $U$ of $G$ onto a
 normal subgroup $i(U)$ of $S$.
 The set  $K=\{(\tilde{u},u^{-1})\textrm{ for } u\in U\}$ is a normal subgroup
 of $S\ltimes _{\phi}G$ if and only if for every $u\in U$\\
$(i)\quad \tilde{u}^{s}=\widetilde{u^{\phi (s)}}$ for every $s\in S$  and\\
$(ii)$\quad$g^{u}=g^{\phi (\tilde{u})}$ for every $g\in G$.
\end{proposition}
%%%%%%%%%%%%%%%%%%%%%%%%%%%%%%%%%%%%%%%%%%%%%%%%%%%%%%%%%%%%%%%%%
\emph{Proof }.
%Before performing computations let us remind that the multiplication in $S\ltimes _{\phi}G$
%is defined by $(s_{1},g_{1})(s_{2},g_{2})=(s_{1}s_{2},g_{1}^{\phi (s_{2})}g_{2})$
%Moreover $(1,g)^{-1}$ is $(1,g^{-1})$ and $(s,1)^{-1}=(s^{-1},1)$.

We recall that $\tilde{u}^{s}=s^{-1}\tilde{u}s\in S$ whereas $u^{\phi(s)}$ is the image of $u\in G$ under
$\phi(s) \in Aut(G)$. Clearly, $K$ is normal if and only if $K$ is normalized by the subgroups
$(S,1)$ and $(1,G)$, because the elements of $S\ltimes _{\phi}G$ are products of their elements.
 The element $(s,1)$ normalizes $K$ if and only if the element
$$x:=(s^{-1},1)(\tilde{u},u^{-1})(s,1)=(s^{-1}\tilde{u},u^{-1})(s,1)=
(s^{-1}\tilde{u}s,(u^{-1})^{\phi(s)})=(\tilde{u}^{s},(u^{\phi(s)})^{-1})$$
belongs to $K$ for every $u\in U$. Since $i(U)$ is normal in $S$, $\tilde{u}^{s}\in i(U)$
 and thus $x\in K=\{(\tilde{u},u^{-1})| u\in U\ \}$ is equivalent to condition $(i)$.
The element $(1,g)$ normalizes $K$ if and only if the element
$$(1,g^{-1})(\tilde{u},u^{-1})(1,g)=(1,g^{-1})(\tilde{u},u^{-1}g)=
(\tilde{u},(g^{-1})^{\phi (\tilde{u})}u^{-1}g)$$
belongs to $K$. This amounts to asking that $t:=(g^{-1})^{\phi (\tilde{u})}u^{-1}g=u^{-1}$,
%The inverse of this identity is
wich is equivalent to
$u=t^{-1}=g^{-1}ug^{\phi (\tilde{u})}\Leftrightarrow gu=ug^{\phi (\tilde{u})}
\Leftrightarrow u^{-1}gu=g^{\phi (\tilde{u})}$, which is condition $(ii)$.

Finally let us prove that condition $(ii)$ implies that $K$ is a subgroup.
 Let $u_{1}$ and $u_{2}$ be elements of $U$.
Let $x:=(\tilde{u_{1}},u_{1}^{-1})(\tilde{u_{2}},u_{2}^{-1})=
(\tilde{u_{1}}\tilde{u_{2}},(u_{1}^{-1})^{\phi (\tilde{u_{2}})}u_{2}^{-1})=
(\widetilde{u_{1}u_{2}},(u_{2}u_{1}^{\phi (\tilde{u_{2}})})^{-1})$ because
$\tilde{u_{1}}\tilde{u_{2}}=i(u_{1})i(u_{2})=i(u_{1}u_{2})$. As $u_{1}$ is in
$G$, condition $(ii)$ implies $u_{1}^{\phi (\tilde{u_{2}})}=u_{1}^{u_{2}}$ so that
the second component of $x$ is the inverse of $u_{2}u_{1}^{u_{2}}=
u_{2}u_{2}^{-1}u_{1}u_{2}=u_{1}u_{2}$ and this proves that $x$ is in $K$.
The same computation with $u_{2}=u_{1}^{-1}$ shows that the inverse of
$(\tilde{u_{1}},u_{1}^{-1})$ is $(\tilde{u_{1}}^{-1},u_{1})\in K$.
$\quad\Box$

The reader interested in an explicit construction for such $S$ and for such $\phi$ satisfying
the two conditions of Proposition \ref{ucns},  may jump to subsection \ref{extsl29}. %for an example.
%$SL_{2} (9)$ where such pairs $<S,\phi>$ are explicitely constructed.
%%%%%%%%%%%%%%%%%%%%%%%%%%%%%%%%
\subsubsection{The role of the hypothesis $K\cap G=1=K\cap S$}
%We proved in Proposition \ref{semi} that each extension $E$ of $G$ (where $S$ is a proper supplement to $G$) can be
%constructed as a quotient $(S\ltimes _{\phi}G)/K$. %% of a semi-direct product that amounts to identify
%%%a subgroup of $S$ with a subgroup of $G$.
It remains to show that every quotient of a semi-direct product $S\ltimes _{\phi}G$ by a normal subgroup
$K$, as described in Proposition \ref{ucns}, corresponds to an extension of $G$ with a supplement
(isomorphic to $S$) whose conjugation action on $G$ corresponds to $\phi$.

This can be done easily just by using the property $K\cap G=1=K\cap S$.  Furthermore, we will show
 (Proposition \ref{uu-1}) that it is not necessary to assume that $K$ has a particular form as in
Proposition \ref{ucns} : the assumption $K\cap G=1=K\cap S$ implies that
%quotient amounts to identify a subgroup $i(V)$ of $S$ with a subgroup $V$ of $G$
%( that is to say $K=\{\ (i(v),v^{-1})\ : \ v\in V \}$).
$K=\{\ (i(v),v^{-1}): v\in V \}$ where $V$ is a subgroup of $G$ and  $i:V\to i(V)$ is an isomorphism onto
 a subgroup of $S$.  We call $K$ \emph{\bf{the identifier subgroup}}.
%In $S\ltimes _{\phi}G$, the normal subgroup $K$ allows us to identify $u\in G$ and $u\in S$
% for every $u\in U$.

%every property described in the previous section about $E$ are consequences of the property
% $K\cap G=1=K\cap S$.

Let $\phi$ be a homomorphism from a group $S$ onto $B\leq Aut(G)$ and let $M=S\ltimes _{\phi} G$
 be the corresponding semi-direct product. If $\varphi:M\rightarrow \varphi (M)$ is a
 homomorphism such that $Ker \varphi\cap G=1=Ker \varphi\cap S$ then :

\begin{itemize}
\item the restrictions $\varphi|_{G}:G\rightarrow \varphi (G)$ and $\varphi|_{S}:S\rightarrow \varphi (S)$
 are isomorphisms since $Ker \varphi|_{G}=Ker \varphi\cap G=1$ and
$Ker \varphi|_{S} =Ker \varphi\cap S=1$.

\item $\varphi (M)$ is an extension of $\varphi (G)\cong G$ (since $G\unlhd M$) where
$\varphi (S)\cong S$ is a supplement to $\varphi (G)$ (since $\varphi (M)=\varphi (SG)=
\varphi (S)\varphi (G)$);

%\item $\varphi (E)/\varphi (G)\cong S/i(V)$  where $V$ is the subgroup of $G$
%such that $\varphi (G)\cap\varphi (S)=\varphi (V)$ and
%$i$ is an isomorphism from $V$ onto a normal subgroup of $S$.

\item The isomorphism $\varphi|_{G}:G\rightarrow \varphi (G)$ induces a natural isomorphism
from $Aut(G)$ onto $Aut(\varphi (G))$. The image of $b\in Aut(G)$ is the automorphism $\tilde b$
of $\varphi (G)$ defined by $\tilde b:\varphi (g)\mapsto \varphi (g^b)$ (for every $g\in G$).
 For $m\in M$, the conjugation action of $\varphi (m)$ on $\varphi (G)$ is $\tilde b$
 if and only if the conjugation action of $m$ on $G$ is $b$. Indeed, as
$\varphi (g)^{\varphi (m)}=\varphi (m^{-1})\varphi (g)\varphi (m)=\varphi (g^{m})$,
we have
\beq \label{b+e}  \varphi (g)^{\varphi (m)}=\varphi (g^{b})\Leftrightarrow \varphi (g^{m})=
\varphi (g^{b})\Leftrightarrow g^{m}=g^{b}
\eeq
 since $\varphi|_{G}$ is an isomorphism.
Consequently, the conjugation action of the supplement
$\varphi (S)$ on $\varphi (G)$ coincides exactly with the action of $S$ on $G$ in $S\ltimes _{\phi}G$.

\end{itemize}

Let $E$ be an extension of $G$ and let $f:E\to Aut(G)$ be its conjugation action on $G$.
We will need in subsection \ref{suppfunct} %explained in the outline of our main method (see the introduction), most of the
supplements of type $f^{-1}(B)$ where $B$ is a subgroup of $Aut(G)$.
The second part of the following proposition is stated for this purpose, in order to have
a criterion to recognize such supplements in $(S\ltimes _{\phi}G)/K$. It will also be used
to prove a result in section \ref{classifs}.

\begin{proposition} \label{uu-1}
Let $\phi$ be a homomorphism from a group $S$ onto $B\leq Aut(G)$ and let $M=S\ltimes _{\phi} G$
 be the corresponding semi-direct product. If $\varphi:M\rightarrow \varphi (M)$ is a
 homomorphism such that $Ker \varphi\cap G=1=Ker \varphi\cap S$ then
\begin{enumerate}

\item $Ker \varphi=\{(i(v),v^{-1})\textrm{ for } v\in V\}$ where $i$ is an isomorphism from
a subgroup $V\leq G$ onto a normal subgroup of $S$.

%\item $\varphi (E)/\varphi (G)\cong S/i(V)$.
\item For every $g\in G$, the isomorphism $j:g\rightarrow \varphi (g)$ induces an isomorphism
$\tilde j$ from $Aut(G)$ onto $Aut(\varphi (G))$. Let $\tilde B=\tilde j(B)$
and let $U$ be the subgroup of $G$ corresponding to $Inn(G)\cap B$.
%\leq Aut(\varphi G)$ be
If $f:\varphi (M)\to Aut(\varphi (G))$ is the natural conjugation action of $\varphi (M)$ on $\varphi (G)\cong G$  then
$$  f^{-1}(\tilde B)=\varphi (SU)\qquad and \,\,\,\qquad f^{-1}(\tilde B)=\varphi (S)
\Leftrightarrow V=U.$$
\end{enumerate}
\end{proposition}
\emph{Proof }.
\begin{enumerate}
\item Let us prove that $Ker \varphi=\{(i(v),v^{-1})\textrm{ for } v\in V\}$ where $V$
 is the subgroup of $G$ such that $\varphi (G)\cap\varphi (S)=\varphi (V)$ and
$i$ is an isomorphism from $V$ onto a normal subgroup of $S$.
 The isomorphism $\varphi |_{G}$ induces an isomorphism from $V$ onto $\varphi (V):=\varphi (G)\cap\varphi (S)$
 and $(\varphi |_{S})^{-1}$ induces an isomorphism from $\varphi (V)$ onto a subgroup of $S$
 that we note $V_{S}$. If $i$ is the composition of these isomorphisms then, $i$ maps $v\in V$
 to $v_{S}:=(\varphi |_{S})^{-1}(\varphi (v))\in V_{S}$ thus $\varphi (v_{S})=\varphi (v)$.
Conversely if for $g\in G$ and $s\in S$, $\varphi (g)= \varphi (s)$ then
$s=(\varphi |_{S})^{-1}(\varphi (g))=i(g)$ and thus $g\in V$ and $s\in V_{S}$.
 Now, $sg\in Ker \varphi\Leftrightarrow \varphi (s)=\varphi (g^{-1})\Leftrightarrow
 s=i(g^{-1})$ ( and $g^{-1}\in V$)$\Leftrightarrow Ker \varphi=\{(i(v),v^{-1})\,|\, v\in V\}$.
 Note that the subgroup $V_{S}$ is normal in $S$ since its image by the isomorphism
 $\varphi |_{S}$ is $\varphi (V_{S})=\varphi (V)=
 \varphi (G)\cap \varphi (S)$ that is normal in $\varphi (S)$ .

\item $f$ is the conjugation action of $\varphi (M)$ on $\varphi (G)$ and $\tilde B$
be the image in $Aut(\varphi (G))$ of $B\leq Aut(G)$. $U$ is the
subgroup of $G$ that corresponds to $Inn(G)\cap B$ (i.e. $G\cap f^{-1}(B)$). Let us prove that
$$  f^{-1}(\tilde B)=\varphi (SU)\qquad and \,\,\,\qquad f^{-1}(\tilde B)=\varphi (S)
\Leftrightarrow V=U$$
% this is useful for supplements of type $f^{-1}(B)$ that we study later on.
Let $h$ be the conjugation action of $M=S\ltimes _{\phi}G$ on $G$.\\
 Note that $U=G\cap h^{-1}(B)$ and that $h|_{S}=\phi$.
 For $m\in M$ we have showed that $\varphi (g)^{\varphi (m)}=\varphi (g^{b})\Leftrightarrow
 g^{m}=g^{b}$ (equalities \ref{b+e}) and therefore $f^{-1}(\tilde B)=\varphi (h^{-1}(B))$. Let $m=sg$ be
 an element of $M$ ( $s\in S$ and $g\in G$) : $m\in h^{-1}(B)\Leftrightarrow
 h(s)h(g)\in B\Leftrightarrow h(g)\in h(s)^{-1}B=B$ since $h(s)=\phi(s)\in B$. Hence
$sg\in h^{-1}(B)\Leftrightarrow g\in U=G\cap h^{-1}(B)$ so that $h^{-1}(B)=SU$ and
we have proved that $f^{-1}(\tilde B)=\varphi (SU)$. \\
It remains to  prove that $\varphi (SU)=\varphi (S)\Leftrightarrow V=U$.
In order to prove this statement we observe that $V\subset U$ because
$\varphi (V)\subseteq \varphi (h^{-1}(B))$ ( since $\varphi (V)\subset \varphi (S)
\subset\varphi (h^{-1}(B))$) so that by $(\star)$
 $V\subset h^{-1}(B)$ whence $V\subset U=h^{-1}(B)\cap G$.
Finally if $\varphi (SU)=\varphi (S)$ then $\varphi (U)\subset\varphi (S)\Rightarrow
\varphi (U)\subset\varphi (S)\cap\varphi (G)=\varphi (V)\Rightarrow U\subset V$ (since
$\varphi|_{G}$ is a bijection) whence $U=V$. Conversely if $U=V$ we get $\varphi (SU)=
\varphi (SV)=\varphi (S)$ since $\varphi (V)\subset\varphi (S)$.
\end{enumerate}
$\quad\Box$
%%%%%%%%%%%%%%%%%%%%%%%%%%%%%%%%%
% ??????????? ALLER VOIR ASCHBASCHER PG 32 SUR CENTRAL PRODUCT
%To avoid such a situation with $G\ltimes _{\phi}S$, we must ask that the set
%$K=\{\ (u^{-1},u)\ : \ u\in U\}$ is its own normal closure.
%
%Indeed, here is an example to illustrate that in a group, we cannot identify
%elements anyhow.
%The group $D_{8}=\{x,y\vert x^4, y^{2}, (xy)^{2}\}$ is
%isomorphic to the dihedral group of $8$ elements. It is a semi-direct product
%of $<x>\cong 4$ by $<y>\cong 2$ and the subgroups $<x^{2}>$ and $<y>$ are
%isomorphic. Let us see what happens if we identify them.
%Let $N$ be the new group defined by $\tilde D=\{x,y\vert x^4, y^{2}, (xy)^{2}, x^{2}y\}$.
%Since $N=\{1,x^{2}, y, x^{2}y\}$ is the normal subgroup of $D_{8}$
% generated by $x^{2}y$, the group $\tilde D=D_{8}/N$ is of order two.
%Thus, together with $x_{2}y$, other elements have been "destroyed" by this
%identification.
%To avoid such a situation with $G\ltimes _{\phi}S$, we must ask that the set
%$K=\{\ (u^{-1},u)\ : \ u\in U\}$ is its own normal closure.
%
%%%%%%%%%%%%%%%%%%%%%%%%%%%%%%%%%%%%%%%%%%%%%%%%%%%%%%%%%%%%%%%%%%

%%%%%%%%%%%%%%%%%%%%%%%%%%%%%%
%\subsubsection{Summary}
%To construct $S$ and the homomorphism $S\to B$ that fulfills condition of proposition \label{ucns}.
%to construct an admissible pair $<S,\phi>$. The reader may jump to section for an example for extensions of
%$SL_{2} (9)$ where such pairs $<S,\phi>$ are explicitely constructed.
%%%%%%%%%%%%%%%%%%%%%%%%%%%%%%%%%%%%%%%%%%%%%
\subsection{Supplement functions}\label{suppfunct}
To each extension containing a supplement $S$  of $G$, we can associate a pair
 $(S,\phi )$, but since there could be a lot of supplements, the description of
 an extension  by such a pair is not unique. Clearly, we do not want to construct all the supplements :
 one per extension is sufficient. If $\mathcal{F}$ is a family of extensions that we want to construct,
 we need a function that gives for each extension $E\in \mathcal{F}$ a supplement $S_{E}$ to $G$ in $E$.
We call such a function $E\to S_{E}$, a \emph{\bf{supplement function}} for the family $\mathcal{F}$.

In our work we encounter various family of extensions : the extensions of a group $G$ associated with a
given subgroup $L$ such that $Inn(G)\leq L\leq Aut(G)$; the extensions of $G$ by a given group $H$ ;
the extensions of $G$ by $H$ associated to a given homomorphism from $H$ into $Out(G)$ or even the whole
family of extensions of $G$.
Examples of supplement functions for these families are described in subsection \ref{funcNp} and
\ref{Gisofunctphi}.
%%%%%%%%%%%%%%%%%%%%%%%%%%%%%%%%%%%
\subsubsection{Preimage supplements}\label{preimsupp}
%In this work, we often choose a subgroup $L$ such that $In(G)<L<Aut(G)$ and we try to classify the
%family $\mathcal{F}_{L}$ of extensions of $G$ associated with $L$.
\begin{wrapfigure}[10]{l}{0.4\linewidth}
%\mbox{\input{Giso1.latex}
\includegraphics[width=5.5cm,height=4.5cm]{outline.eps}
\end{wrapfigure}
Let $G$ be a non abelian group,
let $L$ be a subgroup such that $Inn(G)\leq L\leq Aut(G)$ and let $\mathcal{F}_{L}$ be the family of
extensions of $G$ associated with $L$. We will construct a supplement function for $\mathcal{F}_{L}$.
%The picture on page \pageref{picturesupplement} shows how  the structure of $Aut(G)$ provides supplements
This picture shows how  the structure of $Aut(G)$ provides supplements
to $G$ for the extensions in $\mathcal{F}_{L}$. We detail this fact here.

%%We note first that the structure of $Aut(G)$ provides families of supplements.
Let $E$ be an extension of $G$, let $f$ be the conjugation action of $E$ on $G$ and assume that $f(E)=L$.
If $S$ is a supplement to $G$ in $E$ then $L=f(E)=f(GS)=f(G)f(S)=Inn(G)f(S)$ and thus
$f(S)$ is a supplement to $Inn(G)$ in $L$. Conversely, if $B$ is a supplement to $Inn(G)$
in $L$, then for every $e\in E$ we have $f(e)=f(g)b$, for some $f(g)\in Inn(G)$ and some
$b\in B$ ($g\in G$). Thus for $s=g^{-1}e$, $f(s)=b$ and $s\in S:=f^{-1}(B)$
whence $e=gg^{-1}e=g\in GS$. Therefore, the preimage in $E$ of a supplement to $Inn(G)$
in $L$, is a supplement to $G$ in $E$.  These preimage supplements are exactly the supplements to $G$
that contain $Ker\,f=C_{E}(G)$, the centralizer of $G$ in $E$, and they are in
one-to-one correspondence with the supplements to $Inn(G)$ in $L$.  If $G$ is abelian then
$G=Z(G)\leq C_{E}(G)\leq f^{-1}(B)$ and $f^{-1}(B)$ is not a proper supplement to $G$ in $E$.
Observe that since $f(G\cap f^{-1}(B))=Inn(G)\cap B$, the supplement $f^{-1}(B)$ is proper in $E$ if and
only if $B$ is a proper supplement to $Inn(G)$ in $L$. Hence, for every proper supplement $B$ to $Inn(G)$
in $L$ we may create a "proper" supplement function for  $\mathcal{F}_{L}$ described as $E\to f^{-1}(B)$.
In chapter \ref{findsup} we propose a method to find such supplements when $G$ is not nilpotent.
%% observe that since nilpotent groups are successive central extensions, If Inn(G) is nilpotent
%then $G$ is nilpotent because it is a central extension of Z(G) by the nilpotent G
%%%%%%%%%%%%%%%%%%%%%%%%%%%%%%%%%%%%%%%%%

\subsubsection{Structure of $f^{-1}(B)$}
We now briefly describe some properties of $f^{-1}(B)$ that will be usefull in order to construct the
extensions of $\mathcal{F}_{L}$.

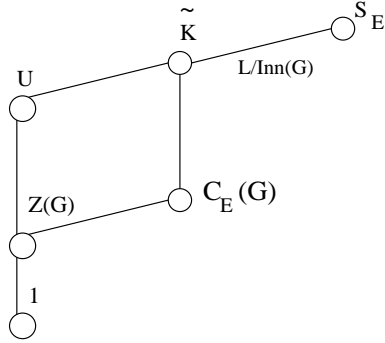
\begin{wrapfigure}[10]{l}{0.35\linewidth}
%\mbox{\input{Giso1.latex}
\includegraphics[width=5cm,height=4.5cm]{structS2.eps}
\end{wrapfigure}
%\begin{wrapfigure}[9]{l}{0.35\linewidth}
%\input{structS.latex}
%\end{wrapfigure}

%\begin{wrapfigure}[height of figure in lines]{l|r}[overhang]{width}
%  figure, caption, etc.
%\end{wrapfigure}
%
%\setlength{\intextsep}{0pt}
% \begin{wrapfigure}[6]{l}{0.2\linewidth} % l for left r for right
%  \mbox{\input{latt1.latex}} %%% IT IS THE WRONG GRAPHIC
% \end{wrapfigure}
%%\subsubsection{Structure of $f^{-1}(B)$}
First observe that since $U:=G\cap f^{-1}(B)$ is the subgroup of $G$ that by conjugation induces the inner
automorphisms of $Inn(G)\cap B$, it is constant for every $E\in \mathcal{F}_{L}$ :  every supplement
$S_E:=f ^{-1} (B)$ is an extension of the same subgroup $U<G$ %($f$ is the conjugation action of $E$ on $G$).

Let us show that $GC_{E}(G)$ can be constructed as a quotient of $G\times C_{E}(G)$.
This construction is called a \emph{\bf{central product}} (see \cite{Aschb} page 32).
Since $C_{E}(G)$ is a supplement to  $G$ in $GC_{E}(G)$, and since it acts trivially on $G$ by conjugation,
 Proposition \ref{semi} shows that the group $GC_{E}(G)$ is isomorphic to
 $(C_{E}(G)\times G)/\tilde{Z}$ where $\tilde{Z}=\{(z,z^{-1})\,|\,z\in Z(G)\}$.
 Conversely if $\mathcal{C}$ is a central extension of $Z(G)$ ( i.e $Z(G)\leq Z(\mathcal{C})$ ),
 then by Proposition \ref{ucns}, $\tilde{Z}$ is a normal subgroup of $\mathcal{C}\times G$
 and by Proposition \ref{uu-1}, $(\mathcal{C}\times G)/\tilde{Z}$ is an extension of $G$
 where the image of $\mathcal{C}$ is the centralizer of the image of $G$. This holds not only for $G$
 but for every subgroup of $G$ that contains $Z(G)$. %that is centralized by $C_{E}(G)$.
Consequently, the subgroup $\tilde{K}:=UC_{E}(G)$ can be constructed as a central product
$(\mathcal{C}\times U)/\tilde{Z}$.
%This construction is called a \emph{\bf{central product}} (see \cite{Aschb} page 32 ).

Observe that since $E=GS_E$ and since $S_E \cap GC_{E}(G)=UC_{E}(G)=\tilde{K}$ we have
$S_E /\tilde{K} \cong E/GC_{E}(G)$.
But $E/GC_{E}(G)\cong \frac{E/C_{E}(G)}{GC_{E}(G)/C_{E} (G)}\cong L/Inn(G)$. Hence, $S_E$ is an
extension of $\tilde{K}$ by the subgroup $L/Inn(G)<Out(G)$. This suggest a
\emph{\bf{two steps strategy to construct $S_E$}} :
first construct a central extension $\mathcal{C}$ of $Z(G)$ and the corresponding
central product $\tilde{K}=(\mathcal{C}\times U)/\tilde{Z}$ ; next, try to build an extension $S$ of $\tilde{K}$
by $L/Inn(G)$ such that there is an homomorphism (with $\mathcal{C}$ as kernel) that respects the two
conditions of Proposition \ref{ucns}.
%%%%%%%%%%%%%%%%%%%%%%%%%%%%%%%%%%%%%%%%%%
\subsection{Reduction to Frattini extensions}\label{frattini}
%In the previous sections we have shown how it was possible to construct any extension from a supplement $S$
%only. It reduces the construction of an extension of $G$ by $H$ to the construction of an extension of
%$S\cap G$ by $H$ and to the construction of an homomorphism $S\to Aut (G)$ that satisfies the two conditions
%of Proposition \ref{ucns}.
We show now that if $G$ is finite, it is always possible to find a supplement $S$ such
that $S\cap G$ is nilpotent. Hence, the construction of extensions of a finite group can be reduced to the extensions
of finite nilpotent subgroup.

Let us recall what a finite nilpotent group is.
Let $G$ be a finite group of order $p^{n}.m$ such that the prime number $p$ is not a
divisor of the integer $m$. A \emph{Sylow-p} subgroup of $G$ is a subgroup of order $p^{n}$.
 Such a subgroup always exists and for each prime $p$ that divides $|G|$, the Sylow-p
 subgroups are conjugate in $G$.
A finite group is nilpotent if, and only if it is the direct product of its
Sylow subgroups (see \cite{Hall} pg. 155). A finite group $N$ is nilpotent if, and only if every
Sylow subgroup of $N$ is normal in $N$.
A nilpotent group is solvable and the subgroups of a nilpotent group are nilpotent.
%%%%%%%%%%%%%%%
%% There is another way to prove this result.
 The Frattini subgroup $\Phi(E)$ is the intersection of the maximal subgroups of $E$. If
 $E$ is finite, $\Phi(E)$ is nilpotent (\cite{Hall} page 157).
 % If $E$ is infinite, the set of all maximal subgroups  could be empty and in this case $\Phi(E)=E$.
\index{$\Phi (E)$ for$E$ infinite ?}
 \index{fratini for infinite ? existence of maxiamls ?}
% MMMMMM fratini for infinite ? existence of maxiamls ? . The descending sequence for supp
% depends on "AXIOME DU CHOIX". ?\\

\textbf{Definition.} Let $E$ be an extension of a normal subgroup $G$ such that $E/G\cong H$.
$E$ is a \emph{\textbf{Frattini extension of G by H}}  if $E/\Phi (E) \cong H/\Phi (H)$ (see \cite{Eick_Ulrich}).

\begin{proposition}\label{redfrat} Let $G$ be a normal subgroup of a finite group $E$.
The following propositions are equivalent.
\begin{enumerate}
\item $G$ has no proper supplement in $E$.
\item $G\leq \Phi(E)$.
\item $E$ is a Fratini extension of $G$ by $E/G$.
\end{enumerate}
\end{proposition}\emph{Proof }.
\begin{itemize}
\item $(1)\Longleftrightarrow (2)$.
If $G$ is not a subgroup of $\Phi(E)$, then there exists a maximal subgroup $M$ of $E$
that do not contain $G$. Then $M$ is a proper subgroup of the subgroup $GM\neq M$ and
$GM=E$ since $M$ is maximal. Hence, $M$ is a proper supplement to $G$ in $E$.
Now suppose that $G\leq \Phi(E)$. Let $S$ be a proper supplement to $G$ in $E$. If $S$
 is maximal, it contains $\Phi(E)$ and thus $G$, so that it is not a proper supplement to
$G$. If $S$ is not maximal, it is contained in a maximal subgroup $M\neq E$. But since
$E=GS$, we have $E=GM$ and since $M$ is maximal it contains $G\leq \Phi(E)$ so $E=M$
 which is a contradiction.
 \item $(2)\Longleftrightarrow (3)$. Let $\mathcal{M}_{G}$ be the set of maximal subgroups of $E$ that contain
$G$. Observe that $\mathcal{M}_{G} /G$ are precisely the maximal subgroups of $E/G$ so that $\Phi (E/G)=B/G$
 where $B=\bigcap (M_{G}\in \mathcal{M}_{G})$. Obviously $\Phi (E)\leq B$ since $\mathcal{M}_{G}$ is a
 subset of the maximals of $E$. If $G\leq \Phi(E)$ then every maximal of $E$ contains $G$, $\mathcal{M}_{G}$
  contains every maximal of $E$ so that $B=\Phi(E)$ and thus $\Phi (E/G)=B/G$. Hence, $\frac{E/G}{\Phi (E/G)}
 =\frac{E/G}{B/G}$ is isomorphic, by Theorem \ref{isotheorems}, to $E/B=E/\Phi (E)$ so that $E$ is a Frattini
 extension of $G$ by $E/G$. Conversely let us show that if $G\nleqq \Phi(E)$ then $E$ is not a Frattini extension.
$\Phi (E) G\leq B$ because $\Phi(E)\leq B$ and $G\leq B$. If $G\nleqq \Phi(E)$ then $\Phi(E)$ is a proper
subset of $\Phi(E)G\leq B$. As $\frac{E/G}{\Phi (E/G)}\cong E/B$ it is also isomorphic to
$Q:=\frac{E/\Phi (E)}{B/\Phi (E)}$ but since $\Phi (E)$ is a proper subset of $B$, the group $B/\Phi (E)$ is not
trivial so that $Q\ncong E/\Phi (E)$. $\quad\Box$
% add a line for the infinite case
 \end{itemize}

Hence, if $S$ is a supplement to a finite normal subgroup $G$ of $E$ and if $S\cap G$ is not nilpotent, then
there exists a "smaller" proper supplement $\tilde{S}$ to $S\cap G$ in $S$ ( "smaller" relatively to the
order of $S\cap G$). It is also a supplement to $G$ in $E$ since $S=(S\cap G)\tilde{S}$ and $E=GS$.
%the subgroup $\tilde{S}$ is also a supplement to $G$ in $E$.
Observe that $|\tilde{S}\cap G|$ is strictly smaller than $|S\cap G|$.
We can iterate the process and if $G$ is finite, it stops on a supplement $\mathcal{S}$ to $G$ in $E$ whose intersection with
$G$ is nilpotent and that is a Frattini extension of $\mathcal{S}\cap G$ by $E/G$. \index{even if $H$ infinite ?}
If $H$ is solvable then $\mathcal{S}$ is solvable because it is an extension of the solvable group $\mathcal{S}\cap G$ by $H$.
It is interesting to note that the tool used by Eick and Besche (\cite{Eick_Ulrich}) to construct finite solvable
groups, was precisely iterated Frattini extensions. This is a first indication that shows how tools used for finite
solvable groups may be used for finite groups in general.

\emph{\bf{Consequently, in order to construct the extensions of a finite group $G$ by $H$ it suffices to construct
Frattini extensions of nilpotent subgroups of $G$ (by $H$)}}.

 Unfortunately, this algorithm, suggested by the proof of Proposition \ref{redfrat} (part 1),  produces
 minimal supplements by iterated computations of maximal subgroups. Since the determination of maximal
 subgroups is generally a hard problem, we propose another approach that requires iterated computations
 of normalizer of sylow subgroups.  Performances of this approach will be analyzed in section \ref{fastsplitest}.
 The supplements obtained by this approach are not necessarily minimal (i.e. frattini extensions).

\subsubsection{The iterated Frattini argument}
Let $E$ be an extension of $G$. For a given prime divisor $p_{1}$ of $|G|$,
let $S_{p_{1}}$ be a Sylow $p_{1}$-subgroup of $G$ and let $N_{p_{1}}=N_{E}(S_{p_{1}})$
be the normalizer of $S_{p_{1}}$ in $E$. If $x\in E$, we have
$S_{p_{1}}^{x}=S_{p_{1}}^{g}$  for some $g\in G$ because all the Sylow $p_{1}$-subgroups
 of $G$ are conjugate in $G$. Hence, $n=xg^{-1}\in N_{p_{1}}$ and since
 $x=ng\in N_{p_{1}}G$, the subgroup $N_{p_{1}}$ is a supplement to $G$ in $E$.
 This observation is called \emph{ the Frattini argument}.
\begin{proposition}\label{nilsup}% Every non nilpotent, finite normal subgroup $G$ of a group $E$ has
 %a proper supplement $S$ such that $S\cap G$ is nilpotent.
 If $G$ is a finite normal subgroup of a group $E$ and if $G$ is not nilpotent then it
has a proper supplement $S$ such that $S\cap G$ is nilpotent.
\end{proposition}\emph{Proof }.
Let $E$ be an extension of $G$ and let $S_{p_{1}}$ be a Sylow $p_{1}$-subgroup of $G$.
%(the other ones are $\{S_{p_{1}}^{g}:\,g\in G \}$).
Let $N_{p_{1}}=N_{E}(S_{p_{1}})$ be the normalizer of $S_{p_{1}}$ in $E$.

% If $x\in E$, we have $S_{p_{1}}^{x}=S_{p_{1}}^{g}$ for some $g\in G$ because all the
% Sylow $p_{1}$-subgroups of $G$ are conjugate in $G$.
% Hence, $n=xg^{-1}\in N_{p_{1}}$ and since $x=ng\in N_{p_{1}}G$, the subgroup $N_{p_{1}}$
% is a supplement to $G$ in $E$. (This is called "the Frattini argument").\\
%Clearly $E$ is a supplement but it is not interesting at all and the supplement $N_{p_{1}}$
The supplement $N_{p_{1}}$ of $G$
is different from $E$ if and only if $N_{p_{1}}\cap G=N_{G}(S_{p_{1}})\neq G$.
Hence, $N_{p_{1}}$ is a proper supplement if and only if $S_{p_{1}}$ is not normal in $G$.
Since the finite nilpotent groups are exactly those that do not have any non normal
Sylow subgroups, there exist proper supplements to every finite non nilpotent normal
subgroup.\\
Let $G_{p_{1}}=G\cap N_{p_{1}}$. If $G_{p_{1}}$ is not nilpotent yet, we can repeat our
 argument for the extension $N_{p_{1}}$ of $G_{p_{1}}$ : there exists a prime
$p_{2}$ such that a Sylow $p_{2}$-subgroup  $S_{p_{1},p_{2}}$ of $G_{p_{1}}$
 is not normal in $G_{p_{1}}$. %($p_{2}\neq p_{1}$ since $S_{p_{1}}\unlhd K_{p_{1}}$).
  Hence, $N_{p_{1},p_{2}}$,
its normalizer in $N_{p_{1}}$, is a proper supplement to $G_{p_{1}}<G$ and thus %in particular
a proper supplement to $G$ in $E$. The order of its intersection with $G$ is strictly
smaller than $|G_{p_{1}}|$.
%The normalizers of the others Sylow $p_{2}$-subgroups
%$S_{(p_{1},p_{2})}^{k}$ are $N_{(p_{1},p_{2})}^{k}$ for $k\in K_{p_{1}}<G$
% such that finally all the possible choices for $N_{(p_{1},p_{2})}$ are conjugate
% under $G$. Conversely, any conjugate $N_{(p_{1},p_{2})}^{g}$ can be reached choosing
% the sequence $S_{p_{1}}^{g},S_{(p_{1},p_{2})}^{g}$ of Sylow subgroups.
 The process can be continued to obtained a sequence of supplements $S$ of
 $G$, as long as $S\cap G$ is not nilpotent. Since $G$ is finite and since
$|S\cap G|$ is strictly decreasing, the process reaches a nilpotent group $S\cap G$ in
 a finite number of steps. $\quad\Box$
 %It defines a function from a sequence of primes
 %$(p_{1},\ldots,p_{n})$ onto the conjugacy class of supplements $N_{(p_{1},\ldots,p_{n})}$
 %as announced in (1) and (2).
%%%%%%%%%%%%%%%%%%%%%%%%%%%%%%%%%%%%%%%%%%
\subsubsection{Remark}
We have mentionned in many places that the existence of supplement functions with nilpotent intersection
reduces the problem that consists of constructing all extensions of a finite group $G$ to the construction of
extensions of a nilpotent subgroup of $G$. But this does not solve the extension problem completely, since
at this point, we do not know how to study the isomorphisms between these extensions. Experience has
showed that a "good" method to construct extensions can produce a "bad" isomorphism problem ; for instance,
 this was the case with the iterated cyclic extension method to construct solvable groups (see \cite{Betten}).
We will prove in chapter \ref{findsup} that our approach of extensions by supplements also reduces the
isomorphism problem to isomorphisms between extensions of a nilpotent subgroup of $G$. For this
purpose, we will introduce restrictions for the "good" supplement functions (see section \ref{isopreserved}).

%\mbox{\input{structS.latex}}

%The subgroup $f^{-1}(Inn(G))$ is equal to $C_{E}(G)G$ because for $e\in f^{-1}(Inn(G)$
% we have $f(e)=f(g)\in Inn(G)$ for some $g\in G$ whence $f(eg^{-1})=1$ and
% $eg^{-1}=c\in Ker f$ so that finally $e=eg^{-1}g=cg$.
%Hence, $S_E /\tilde{K}\cong E/G\tilde{K}=E/(GC_{E}(G))\cong L/Inn(G)<Out(G)$
%
%We have $Z(G)=G\cap C_{E}=U\cap C_{E} (G)$. and $\tilde{K}:=UC_{E} (G)$
%%%%%%%%%%%%%%%%%%%%%%%%%%%%%%%%%%%%%%%%%%%%%%%%%%%%%%%%%%%%
%%%%%%%%%%%%%%%%%%%%%%%%%%%%%%%%%%%%%%%%%%%%%%
\subsection{Extensions of $SL_{2}(9)$}\label{extsl29}
We will explain in subsection \ref{58320} that this is the first perfect group for wich we do not provide
a cross-check.  %%are availab(except pgl29 and central product)

The group $G=SL_{2}(9)$ can be defined on generators $[ a, b, c ]$ by the relations
$$< 1=a^{4}=b^{3}=c^{3}=(bc^{-1})^{5}, (bc)^{4}=a^{2},a=b^{-1}cbcb^{-1}cbc^{-1},
a^{2}\,is\,central>$$
 ( see \cite{Holt_Plesken} page 213). $Aut(G)\cong Aut(A_{6})$ and $Out(G)\cong 2\times 2$.
Using $GAP$ we have determined that $Aut(G)$ does not split over $Inn(G)$ and we provide
two automorphisms of $G$ that, together with $Inn(G)$, generate $Aut(G)$ :\\
$\alpha:[a,b,c]\rightarrow [b^{-2}c^{-2}b^{-2}a^{-2}c^{-2},
a^{-1}c^{-1}b^{-1}c^{-1},b^{-2}c^{-2}b^{-2}a^{-1}]$ and \\
$\beta:[a,b,c]\rightarrow [c^{-2}b^{-2}a^{-1}c^{-2}b^{-2}a^{-2}c^{-1}b^{-2}
, b^{-2}c^{-2}b^{-2}a^{-2}b^{-1}a^{-1}, c^{-2}b^{-2}a^{-1}]$.\\
The group $B:=<\alpha,\beta>$ is the direct product of the cyclic subgroups\\
$<\alpha>\cong C_{2}$ and $<\beta>\cong C_{4}$. The subgroup
$U$ of $G$ corresponding to $<\beta ^{2}>=Inn(G)\cap B$ is cyclic of order $4$ and is
generated by\\
$u:=b^{-1}c^{-2}b^{-2}a^{-1}c^{-2}b^{-2}a^{-1}c^{-1}a^{-1}c^{-2}a^{2}$ and
$Z(G)=<u^{2}>$.
 The action of $B$ on $U$ is described by $\alpha:u\rightarrow u^{3}$
and $\beta:u\rightarrow u$.
% mage $u^{\alpha}$ of $u$ under $\alpha$ is $u^{3}$ and $u^{\beta}=u$.

With this information only, % us show how this information is sufficient to classify
% all the extensions of $SL_{2}(9)$
our method reduces the classification of extensions of $SL_{2}(9)$ by a group $Q$, to
a classification of extensions of $<u>\cong C_{4}$ by $Q$. Let us show now, how to proceed.
\\
\subsubsection{An extension of $G=SL_{2}(9)$ by $D_{8}$ associated with $Out(G)$.}
\index{Sl(2,9) by D8 : ununderstandable}
The dihedral group $D_{8}$ of order $8$ can be defined on generators $[\tilde x,\tilde y]$ by the relations
$< \tilde{x}^{2}=1, \tilde{y}^{\tilde{x}}=\tilde{y}^{3}, \tilde{y}^{4}=1>$. Let $<u>=U\cong C_{4}$ be the subgroup of $G=SL_{2}(9)$
described in the previous section. We aim to construct an extension $E$ of $SL_{2}(9)$ by
$D_{8}$ associated to the following homomorphism $\Psi$ from $D_{8}\cong E/G$ onto $Out(G)$ :
\begin{displaymath}
\Psi : D_8 \to Out(G) :\left \{ \begin{array}{r}
\tilde{x}\to Inn(G)\alpha \\
\tilde{y} \to Inn(G)\alpha\beta
\end{array} \right. \end{displaymath}
%$\tilde{x}\to Inn(G)\alpha$ and  $\tilde{y}$ to $Inn(G)\alpha\beta$.
%\end{displaymath}
Suppose that such $E$ exists and if $f$ is its
conjugation action on $G$, then let $x$ and $y$ denote elements that by conjugation on $G$
 induce $\alpha$ and $\alpha\beta$ respectively. The isomorphism from $E/G$ onto $D_{8}$ is
 now described by the mapping
 $$E/G \to D_{8}:Gx\to \tilde{x}, \,Gy\to \tilde{y}$$ and $<x,y>$ is a supplement to $G$ in
 $E$. Since $B=<\alpha,\beta>$ is a supplement to $Inn(G)$ in $Aut(G)$ then $S:=f^{-1}(B)$
 supplements $G$ in $E$. But since $<x,y>\subseteq S$, then $<x,y>$ supplements $G\cap S=U$.
 The group $S=<x,y,U>$ is a succession of cyclic extensions, $U$, $<y,U>$, $<x,y,U>$. Furthermore,
 $u^{x}$ (respectively $u^{y}$) is equal to $u^{\alpha}=u^{3}$ (respectively $u^{\alpha\beta}=u^{3}$).
 Hence, $S$ can be defined by a power-conjugate presentation
$$S=<x,y,u|\,x^{2}=r_{1},\,y^{x}=y^{3}r_{2},\, \,y^{4}=r_{3},u^{x}=u^{3},\,u^{y}=u^{3},\, u^{4}=1>$$
where $r_{i}\in U$ for $i=1,2,3$. The restriction of $f$ to $S$, it is defined by
$$f|_{S}:x\to\alpha,y\to \alpha\beta, \, u\to\beta ^{2}.$$

 We have proved that in order to construct $E$ it is sufficient to
 construct $S$ (if it exists) and an homomorphism $\phi:S\rightarrow B$
 that fulfills the two conditions of Proposition \ref{ucns} ( $\phi$ corresponds to the restriction of $f$ to $S$).
 In order to lighten our notations, we consider here that $U\subset G$ is also a subset of $S$
 so that the isomorphism  $u\rightarrow \tilde{u}$ described in these conditions is just
 the identity. The conditions on $S$ and $\phi$ now become :

 $(i)\quad u^{s}=u^{\phi (s)}$ for every $s\in S$  and\\
 $(ii)$\quad$g^{u}=g^{\phi (u)}$ for every $g\in G.$

 Hence, the constructions of all possible extensions $E$ is reduced to the determination
 of the $r_{i}\in U$ such that
$$\tilde{S}=<x,y,u|\,x^{2}=r_{1},\,y^{x}=y^{3}r_{2},\, \,y^{4}=r_{3},u^{x}=u^{3},\,u^{y}=u^{3},\, u^{4}=1>$$
 defines an extension of $U$ by $D_{8}$ (this can be achieved
 by a repeated use of the cyclic extension conditions of Chapter \ref{basic}), and such that
 the mapping $$\phi :x\to\alpha,\,y\to \alpha\beta,\,u\to\beta ^{2}$$ is an homomorphism.
This is sufficient because if $\phi$ is a homomorphism then the relations $u^{x}=u^{\alpha}=u^{3}$
and $u^{y}=u^{\alpha\beta}=u^{3}$ implies that $\phi$ satisfies condition $(i)$ and since
in $G$,  conjugation by $u$ is $\beta ^{2}$, condition $(ii)$ is also satisfied.
In this case $E=(\tilde{S}\ltimes _{\phi}G)/K$, where $K=\{(v,v^{-1})| v\in U\ \}$, is a solution
 to our problem.

 Up to now, $\phi$ is only defined on
$\{x,y,u\}$ but it can be extended multiplicatively on $\tilde{S}=<x,y,u>$. The condition
for $\phi$ to be a homomorphism is that $\phi$ preserves every defining relation
$R(x,y,u)=1$ of $\tilde{S}$ (i.e $R(x,y,u)=1\Rightarrow R(\phi (x),\phi (y),\phi (u))=1$).
 Therefore the conditions are :
\begin{enumerate}
\item $\phi (r_{1})=\phi (x^{2})=\phi (x)^{2}=\alpha^{2}=1\Rightarrow r_{1}\in
C_{E}(G)\cap U=Z(G)=\{1,u^{2}\}.$
\item $\phi (y^{x})=\phi (y^{3})\phi (r_{2})\Rightarrow \beta=\beta ^{\alpha}=\beta ^{3}\phi(r_{2})
\Rightarrow \phi (r_{2})=\beta ^{2}\Rightarrow r_{2}\in \{u,u^{3}\}.$
\item $\phi (r_{3})=\phi (y^{4})=\phi (y)^{4}=(\alpha\beta)^{4}=1\Rightarrow r_{3}\in
Z(G)=\{1,u^{2}\}.$
\end{enumerate}

Now, there are $8=2^{3}$ possible presentations and for each of them it remains to determine whether
it defines an extension of $U$ by $D_{8}$ ( this extension must have order 32=4.8 ). Here, we show how to
proceed for the case $\{r_{1}=1,r_{2}=1,r_{3}=u^{3}\}$ i.e the group
$\tilde{S}=<x,y,u|\,x^{2}=1,y^{4}=1,y^{x}=y^{3}u^{3}, u^{x}=u^{y}=u^{3}, u^{4}=1>.$\\
 Firstly, it is straightforward to
check cyclic extension conditions on $U_{y}=<y,u|\,y^{4}=1,u^{y}=u^{3}, u^{4}=1>$
 which is thus a group of order 16 and a split extension of $U$ .
The conjugation action of $x$ on $U_{y}$ in $\tilde{S}$ is defined by the mapping
$\omega:u\rightarrow u^{3},\,y\rightarrow\tilde{y}:=y^{3}u^{3}$ ; it extends to an
 automorphism  of $U_{y}$ because it preserves the relations of $U_{y}$ ( for instance observe that
 $\tilde{y}^{4}=1$) and because $\{\tilde{y},u^{3}\}$ generate $U_{y}$. Similar calculations
  show that $\omega ^{2}$ is the identity on $U_{y}$. Once again, this guarantees that
$\tilde{S}$ is a cyclic extension (even a split extension of $U_{y}$) : $\tilde{S}$ is an extension of $U$, it has
$2.16=32$ elements and
  $\tilde{S}/U\cong D_{8}$. Finally $E=(\tilde{S}\ltimes _{\phi}G)/K$, where $K=\{(v,v^{-1})| v\in U\ \}$
  is an extension of $SL_{2}(9)$ by $D_{8}$ associated to the homomorphism $\Phi$ from
  $D_{8}$ onto $Out(SL_{2}(9))$. Observe also that $<x,y>$ is not a complement of $U$
since $y^{x}\notin <y>$.\\

For the isomorphism problem, observe that by Corollary \ref{liftiso}, if $\tilde{S_{1}}$ and $\tilde{S_{2}}$ are two such
 extensions of $U$ that are not isomorphic then the corresponding extensions of $SL_{2}(9)$
 are not isomorphic.
%%%%%%%%%%%%%%%%%%%%%%%%%%%%%%%%%%%%%%%%%%%%
%%%%%%%%%     %%%%%%%%%%     %%%%%%%%%%     %%%%%%%%%%      %%%%%%%
\section{Isomorphisms of extensions}
We have constructed extensions as some specific factor groups $(S\ltimes G) /K$ of a semi-direct product.
A $G$-isomorphism between such extensions preserves $G/K$ but not necessarily $S/K$. Nevertheless
we will prove in Chapter \ref{findsup} that it is possible to consider only $G$-isomorphisms that preserve $S/K$
and we prove in this section that such isomorphisms can be lifted to isomorphisms of
the corresponding semi-direct products that preserve the complement $S$.
We first describe  such isomorphisms between semi-direct products.

\begin{lemma}\label{fpi} Let $i:E_{1}\rightarrow E_{2}$ be a $G$-isomorphism between extensions
of a group $G$ and let $\pi\in Aut(G)$ be the restriction of $i$ to $G$. Then for
every $g\in G$ and every $t\in E_{1}$
$$ g^{i(t)}=((g^{\pi^{-1}})^{t})^{\pi}$$
\end{lemma}
\emph{Proof }. Since $g^{t}\in G$, its image $i(g^{t})$ is equal to
$(g^{t})^{\pi}=(((g^{\pi})^{\pi^{-1}})^{t})^{\pi}$=x.
On the other hand  $i(g^{t})=i(t^{-1}gt)=i(t)^{-1}i(g)i(t)=i(t)^{-1}g^{\pi}i(t)=(g^{\pi})^{i(t)}$=y.
Since any element of $G$ can be written $g^{\pi}$ for
 some $g\in G$, the equality x=y is equivalent to $g^{i(t)}=((g^{\pi^{-1}})^{t})^{\pi}$
$.\quad\Box$

The following Proposition has been previously proved by Eick in \cite{Eick_Ulrich}.
\begin{proposition}\label{semiso}Let $A\ltimes _{\phi}G$ and  $B\ltimes _{\theta}G$
 be two semi-direct products. There exists an isomorphism
 $j:A\ltimes _{\phi}G\to B\ltimes _{\theta}G$ such that  $j(G)=G$ and $j(A)=B$
 if and only if there exists $\pi\in Aut(G)$ and an isomorphism $i:A\to B$ such that
 $$\theta (i(a))=\pi^{-1}\phi (a)\pi\qquad\qquad \textrm{for every } a\in A.$$

%for some $\pi\in Aut(G)$ and for every $a\in A$.
\end{proposition}
\emph{Proof .} Lemma \ref{fpi} shows that this condition is necessary.
 Let us prove that it is also sufficient. If there exists an isomorphism
 $i:A\to B$ and $\pi\in Aut(G)$ which satisfy this condition then, the function
 $$j:(a,g)\to(i(a),g^{\pi})$$  is an isomorphism from $A\ltimes_{\phi}G$ onto $B\ltimes _{\theta}G$
 (for more legibility we choose right-sided notation for the second components).
It is clearly a one-to-one correspondence, $j(G)=G$ and $j(A)=B$.
 Let us show that it preserves product. For $g_{k}\in G$ and $a_{k}\in A$
 ( $k=1,2$), we have :
$$j(a_{1},g_{1}).j(a_{2},g_{2})=(i(a_{1}),g_{1}^{\pi}).(i(a_{2}),g_{2}^{\pi})$$
$$=\Big(i(a_{1})i(a_{2}),g_{1}^{\pi\theta (i(a_{2}))} g_{2}^{\pi} \Big)=(*) $$
Our condition implies that $\pi\theta (i(a_{2}))=\phi (a_{2})\pi$ in $Aut(G)$
 so that $(*)$ becomes\\
$$\Big(i(a_{1})i(a_{2}),g_{1}^{\phi (a_{2})\pi} g_{2}^{\pi} \Big)=
\Big(i(a_{1}a_{2}),(g_{1}^{\phi (a_{2})}g_{2})^{\pi} \Big)
=j[(a_{1}a_{2},g_{1}^{\phi (a_{2})}g_{2})]$$ which is equal to
$j[(a_{1},g_{1}).(a_{2},g_{2})].$ $\quad\Box$\\

%%%%%%%%%%%%%%%%%%%%%%%%%%%%%%%
%\begin{proposition}\label{semiso}Let $S^{1}\ltimes _{\phi}G$ and  $S^{2}_{\theta}G$
% be two semi-direct products. There exists an isomorphism
% $j:S^{1}_{\phi}\to S^{2}_{\theta}G$ such that  $j(G)=G$ and $j(S^{1})=S^{2}$
% if and only if there exists an isomorphism $i:S^{1}\to S^{2}$ such that
% $$\theta_{i(s)}=\pi^{-1}\phi_{s}\pi$$
%for some $\pi\in Aut(G)$ and for every $s\in S^{1}$.
%\end{proposition}
%\emph{Proof .} Lemma \ref{fpi} shows this condition is necessary.
% Let us prove that it is a sufficient one. If there exists an isomorphism $i:S^{1}\to S^{2}$,
% and  $\pi\in Aut(G)$ which satisfy this condition then, the function $$j:(g,s)\to(g^{\pi},i(s))$$
% is an isomorphism from $G_{\phi}S$ onto  $G_{\theta}S$.
%It is clearly a one-to-one correspondence, $j(G)=G$ and $j(S^{1})=S^{2}$.
% Let us show that it preserves multiplication. For $g_{k}\in G$ and $s_{k}\in S^{k}$
% ( $k=1,2$), we have :
%$$j(g_{1},s_{1})\circ j(g_{2},s_{2})=(g_{1}^{\pi},i(s_{1}))\circ(g_{2}^{\pi},i(s_{2}))$$
%$$=\Big(g_{1}^{\pi}\theta_{i(s_{1})}(g_{2}^{\pi}),i(s_{1})i(s_{2})\Big)=(*) $$
%Our condition implies $(\theta_{i(s_{1})}\pi)(g_{2})=(\pi\phi_{s_{1}})(g_{2})$
%and so $(*)$ becomes\\$$\Big(g_{1}^{\pi}\pi(\phi_{s_{1}}(g_{2})),i(s_{1})i(s_{2})\Big)=
%\Big(\pi(g_{1}\phi_{s_{1}}(g_{2})),i(s_{1}s_{2})\Big)=j[(g_{1},s_{1}).(g_{2},s_{2})]$$
% The symbols  $.$ and $\circ$ denote multiplication in $G_{\phi}S^{1}$ and  $G_{\theta}S^{2}$.
%$\quad\Box$
%%%%%%%%%%%%%%%%%%%%%%%%%%%%%%

\begin{proposition}\label{isogh} For $r=1,2$, let $S_{r}$ be a supplement to a group $G$
in an extension $E_{r}$ of $G$.
 Let $f_{r}$ be the action of $E_{r}$ on $G$ by conjugation, let $\pi\in Aut(G)$ and let
 $j$ be an isomorphism from $S_{1}$ onto $S_{2}$.
%For $g\in G$ and $s\in S_{1}$, the application
 The mapping $\alpha:E_1\to E_2:s.g\rightarrow j(s).g^{\pi}$ is an isomorphism from $E_{1}$ onto $E_{2}$ if and only if
  \begin{enumerate}\item $f_{2}(j(s))=f_{1}^{\pi}(s)$ for every $s\in S_{1}$.
 \item $j(u)=u^{\pi}$ for every $u\in S_{1}\cap G$.
 \end{enumerate}
\end{proposition}
\emph{Proof}. $\Leftarrow$  : Since the different ways to write an element $s.g$ of
$E_{1}$ are the products $su.u^{-1}g$, where $u\in S_{1}\cap G$ the application
$\alpha:s.g\rightarrow j(s).g^{\pi}$ is a well-defined function if and only if
$j(s).g^{\pi}=j(su).(u^{-1}g)^{\pi}$ $\Leftrightarrow$
$j(s).g^{\pi}=j(s)j(u).(u^{\pi})^{-1}g^{\pi}\Leftrightarrow j(u).(u^{\pi})^{-1}=1$ which
is equivalent to condition $(2)$. If condition $(2)$ is satisfied, $\alpha$ maps $G$ onto
$G$, $S_{1}$ onto $S_{2}$ and thus $U_{1}:=S_{1}\cap G$ onto $U_{2}:=S_{2}\cap G=j(U_{1})=U_{1}^{\pi}$.

We know that $E_{r}\cong (S_{r} \ltimes_{\tilde{f_r}}G)/K_{r}$ where
$K_{r}=\{(u,u^{-1}):\,u\in U_{r}\}$ and where $\tilde{f_{r}}$ is just $f_{r}$ restricted to $S_r$.
Note that since $U_{2}=U_{1}^{\pi}$ we can write $K_{2}$ as $\{((u^{\pi},(u^{\pi})^{-1}):\,u\in U_{1}\}$.
 By Proposition \ref{semiso}, equation $(1)$ is the condition for the function
$\tilde{\alpha}:(s,g)\rightarrow (j(s),g^{\pi})$ to be an isomorphism between the
semi-direct products $S_{1}\ltimes_{\tilde{f_1}}G$ and $S_{2}\ltimes_{\tilde{f_2}}G$.
% that preserves $1_{S}\times G$ and $S\times 1_{G}$.
 Let us show that equation $(2)$ is the necessary and sufficient condition for $\tilde{\alpha}$ to map $K_{1}$ to
 $K_{2}$. Indeed for $u\in U_{1}$, $\tilde{\alpha}$ maps $k_{1}=(u,u^{-1})\in K_{1}$ to
$(j(u),(u^{\pi})^{-1})$ that belongs to $K_{2}$ if and only if $j(u)=u^{\pi}$.
Consequently the isomorphism $\tilde{\alpha}$ induces an isomorphism from
$E_{1}\cong (S_{1}\ltimes_{\tilde{f_{1}}}G)/K_{1}$ onto
$\alpha(S_{1}\ltimes_{\tilde{f_{1}}}G)/\alpha (K_{1})=(S_{2}\ltimes_{\tilde{f_{2}}}G)/K_{2}\cong E_{2}$.

$\Rightarrow$ : If the function $\alpha:s.g\rightarrow j(s).g^{\pi}$ is an isomorphism
then $\alpha (s)=j(s)$ for every $s\in S_{1}$ and $g^{\alpha}=g^{\pi}$ for every $g\in G$.
 Hence, $\alpha (u)=j(u)=u^{\pi}$ for every $u\in S\cap G$.
 Condition $(1)$ follows from Lemma \ref{fpi} ($\alpha$ is a $G$-isomorphism)$.\quad\Box$

\begin{corollary}\label{liftiso} For $r=1,2$, let $S_{r}$ be a supplement to a group $G$
in an extension $E_{r}$ of $G$.
 Let $\phi_{r}$ be the conjugation action of $S_{r}$ on $G$.% by conjugation.
 Every $G$-isomorphism from $E_{1}$ onto $E_{2}$ that maps $S_{1}$ onto $S_{2}$ can be
 lifted to a $G$-isomorphism from $(S_{1})\ltimes _{\phi_{1}}G$ onto $(S_{2})\ltimes _{\phi_{2}}G$ that
 maps $S_{1}$ onto $S_{2}$.
\end{corollary}

%%%%%%%%%%%%%%%%%%%%%%%%%%%%%%%%%%%%%%%%%%%%%%%%%%%%%%%%%%
\subsubsection{Correspondence between extensions of non isomorphic groups}

We have reduced the determination of the extensions of a group $G$ to the determination
of the extensions of a smaller subgroup $U<G$. All we need, in order to construct such smaller
extensions, is to have some supplement $B$ of $Inn(G)$ in a subgroup of $Aut(G)$ and to
 know how it acts on $U$. But such a small amount of information ($B$ and its action on
 $U$) do not characterizes $G$ completely : non isomorphic groups could share the same
 information about $B$ and $U$. Therefore, the extension problem for a group $G_{1}$
 could be equivalent to the extension problem for a group $G_{2}$ non isomorphic to
 $G_{1}$. This is the meaning of the Theorem \ref{bij3840}.
\index{Bijection Theorem (Proof)}
\begin{theorem}\label{bij3840} Let $G_1$ and $G_2$ be groups.
For $n=1,2$ let $B_n$ be a supplement to $Inn(G_n)$ in $Aut(G_n)$ and $U_n$ be the subgroup of
$G_n$ corresponding to $Inn(G_n)\cap B_n$.
 Assume that  for $b_2\in B_{2}$, the function
$q:b_2\rightarrow b_1$ is an isomorphism from $B_2$ onto $B_1$ and that
for $u_2\in U_2$, the function $k:u_2\rightarrow u_1$ is an isomorphism
from $U_{2}$ onto $U_{1}$ such that\\
$(i)\quad k:u_{2}^{b_2}\rightarrow u_{1}^{b_1} $ and\\
$(ii)$\quad$ q: \overline{u_2}\rightarrow \overline{u_1}\qquad \forall u_2 \in U_2 ,\, \forall b_2 \in B_2$\\
 where $\overline{u_n}$ denotes the inner automorphism of $G_n$ induced by $u_n \in U_n$.
% Assume that  for $s_{1}\in S_{1}$, the function
%$j:s_{1}\rightarrow \tilde{s_{1}}$ is an isomorphism from $S_{1}$ onto $S_{2}$ and that
%for $u_{1}\in U_{1}$, the function $k:u_{1}\rightarrow \tilde{u_{1}}$ is an isomorphism
%from $U_{1}$ onto $U_{2}$ such that
%$$k:u_{1}^{s_{1}}\rightarrow \tilde{u_{1}}^{\tilde{s_{1}}}. $$

Let $\mathbb{E}_n$ be the set of $G_n$-isomorphism classes of extensions of $G_n$. %%($n=1,2$).
Then there is a one-to-one correspondence between $\mathbb{E}_{1}$ and $\mathbb{E}_{2}$.
\end{theorem}
\emph{Proof }.
%For each $G$-isomorphism class $x\in \mathbb{E}_{1}$ let us choose a representative $E_x$.
Let $Ext_1$ (respectively $Ext_2$), be the set of all extensions of $G_1$ (respectively $G_2$).
In order to prove the theorem, we construct a bijection $\beta$ between $Ext_1$ and $Ext_2$ such that
for $x,y\in Ext_1$, the extensions $\beta (x)$ and $\beta (y)$ of $G_2$ are $G_2$-isomorphic if and only if
$x$ and $y$ are $G_1$-isomorphic.

For $n=1,2$, let $E_n$ be an extension of $G_n$, let $f_n$ be its conjugation action on $G_n$, let
$\widetilde{B_n}:=B_1 \cap f_n(E_n)$, then $\widetilde{S_n}=f_n ^{-1}(\widetilde{B_n})$ is a supplement to $G_n$ in
$E_n$ and  $\widetilde{S_n}\cap G_n$ is $U_n$. If $\widetilde{\phi_n}:\widetilde{S_n}\to \widetilde{B_n}$ is the
restriction of $f_n$ to $\widetilde{S_n}$ then by Theorem \ref{semi}, $E_n$ is isomorphic to
$(\widetilde{S_n}\ltimes_{\widetilde{\phi_n}} G_n)/K_n$ where $K_n=\{(u_n,u_n ^{-1})\, :\, u_n\in U_n\}$.

Now, let $U$ be a group and let $i_1:u_1\to u$ ( for $u_1\in U_1$) be an isomorphism from $U_1$ onto $U$.
Let $i_2:=i_1 k:u_2 \to u$ be the corresponding isomorphism from $U_2$ onto $U$.
By Proposition \ref{ucns} and \ref{uu-1} we may construct $E_n$ as
$(S_n\ltimes_{\phi_n} G_n)/K_n$ where $S_n$ is an extension of $U$, where $\phi_n$ is an homomorphism
from $S_n$ onto $\widetilde{B_n}$ and where $K_n=\{(u,u_n ^{-1})\, :\, u_n\in U_n\}$ if and only if the
following two equalities hold :
\begin{eqnarray}
 & u^s=i_n ( u_n ^{\phi_n (s)} ) &  \label{ucns1}\\
& \overline{u_n}=\phi_n (u) &  \textrm{for every} s\in S_n \textrm{ and for every } u_n\in U_n. \label{ucns2}
\end{eqnarray}

For every $E_1:=(S_1\ltimes_{\phi_1} G_1)/K_1$ in $Ext_1$, we construct the
extension $E_2:=\beta (E_{1})$ as follows. We take $S_2:=S_1$ and
$\phi_2:=q^{-1}.\phi_1:s\to q^{-1}(\phi_1 (s))$ for every $s\in S_1$. Then we
construct the semi-direct product $S_1\ltimes_{\phi_2} G_2$ and we now show that to obtain $E_2$, we may
 identify $U\unlhd S_1$ and $U_2\leq G_2$ through the isomorphism $i_2$, because  equalities (\ref{ucns1})
 and (\ref{ucns2}) correspond respectively to conditions $(i)$ and $(ii)$ of the present theorem.
Equality (\ref{ucns1}) holds for $n=2$ since by the first condition of the present theorem, $k(u_2 ^{\phi_2 (s)})$
  is equal to $u_1 ^{q\phi_2 (s)}=u_1 ^{\phi_1 (s)}$ and thus applying $i_1$ to both sides we have
$i_2(u_2 ^{\phi_2 (s)})=i_1 (u_1 ^{\phi_1 (s)})$ which is equal to $u^{s}$ since (\ref{ucns1}) holds for $n=1$.
Equality (\ref{ucns2}) holds for $n=2$ since by the second condition of the present theorem
$\overline{u_2}=q^{-1}(\overline{u_1})$ which is equal to $q^{-1}\phi_1 (u)=\phi_2 (u)$ since (\ref{ucns2}) holds
for $n=1$. Finally we may define $E_2$ as $(S_1\ltimes_{\phi_2} G_2)/K_2$ where
$K_2=\{(u,u_2 ^{-1})\, :\, u_2\in U_2\}$.

Let us show that every extension of $G_2$ is the image under $\beta$ of an extension of $G_1$.
Since $k^{-1}$ and $q^{-1}$ also satisfy the hypothesis of the present Theorem when $G_1$ and $G_2$ are
permuted, our arguments are symmetrical with respect to $G_1$ and $G_2$ if we permute respectively $1, k, q$ with $2, k^{-1}, q^{-1}$.
Therefore, a correspondence $\beta ':Ext_2 \to Ext_1$ can be constructed in the same way :
if $E_2:=(S_2\ltimes_{\phi_2} G_2)/K_2$ as before then $\beta '(E_2)$ is defined as
$(S_2\ltimes_{\phi_1} G_1)/K_1$ for $\phi_1:=q \phi_2$ and $K_1:=\{(u,u_1 ^{-1})\, :\, u_1\in U_1\}$).
Now, observe that since $q^{-1}\phi_1=\phi_2$, then $\beta(\beta ' (E_2))$ is precisely
$(S_2\ltimes_{\phi_2} G_2)/K_2 =E_2$ so that $\beta$ is surjective on $Ext_2$ and $\beta^{-1}=\beta'$.

We end with the $G_n$-isomorphism problem. Assume that there is a $G_1$-isomorphism $i'$ between two extensions
$E_1$ and $E_1 '$ of $G_1$. Since $Aut(G_1)=B_1 Inn(G_1)$, there is another $G_1$-isomorphism $i$
whose restriction to $G_1$ is $b_1\in B_1$  (it suffices to apply after $i'$, some inner automorphism of
$E_1 '$). Moreover $i$ maps the preimage of a subgroup $\widetilde{B_1}\subseteq B_1$ in $E_1$ to the
corresponding preimage in $E_1 '$ (full details will be given in the proof Proposition \ref{Allisopres}).
Hence if as before, $E_1$ is described as $(S_1\ltimes_{\phi_1} G_1)/K_1$ and
$E_1 '=(S_1 '\ltimes_{\phi_1 '} G_1)/K_1 '$ then by Corollary \ref{liftiso} and Proposition \ref{isogh}, $i$ is
completely described by $b_1$ and its restriction $j:S_1 \to S_1 '$ to $S_1$. Moreover, for every $s\in S_1$
and $u\in U$ we have by the same argument as in Proposition \ref{isogh}
\begin{eqnarray}
&\phi_1 '(j(s))=b_1 ^{-1}\phi_1 (s)b_1 & \label{isoghrel1} \\
&  j(u)=i_1 (u_1 ^{b_1}) & \label{isoghrel2}.
\end{eqnarray}
Now, we show that $j$ and $b_2=q^{-1} (b_1)$ define a $G_2$-isomorphism from
$\beta (E_1)=(S_1\ltimes_{\phi_2} G_2)/K_2$ onto  $\beta (E_1 ')=(S_1 '\ltimes_{\phi_2 '} G_2)/K_2 '$ where
$\phi_2=q^{-1}\phi_1$ and  $\phi_2=q^{-1}\phi_1$. By the arguments of Proposition \ref{isogh}, we must only
verify that equalities (\ref{isoghrel1}) and (\ref{isoghrel2}) hold for $n=2$. For (\ref{isoghrel1}) it suffices to
apply $q^{-1}$ to both sides and for (\ref{isoghrel2}), since $u_1 ^{b_1}=k(u_2 ^{b_2})$ (hypothesis $(i)$) and
since $i_1 k=i_2$ we obtain (\ref{isoghrel2}) for $n=2$.
Once again all these arguments are symmetrical with respect to $G_1$ and $G_2$ so that if two extensions
$\beta (E_1)$ and $\beta (E_1 ')$ of $G_2$ are $G_2$-isomorphic, then $E_1$ and $E_1 '$ are
$G_1$-isomorphic $.\quad\Box$

There are 3 interesting cases of a Perfect group $P$ such that $Aut(P)$ do not split
over $Inn(G)$, for which such situation occurs.

 It concerns the case where the isomorphism type of
$P/Z(P)$ is either $PSL(2,9)$, the perfect group $[1920,4]$ or $PSL(2,25)$. We determined
(using $GAP$), that for these 3 categories, $Inn(P)$ has a supplement $S$ of order $8$ in
$Aut(P)$ such that the subgroups $U$ of $P$ corresponding to $Inn(P)\cap S$ are isomorphic
and such that the action of $S$ on the subgroup $U$ are "isomorphic" in the sense of
Theorem \ref{bij3840}.

%%%%%%%%%%%%%%%%%%%%%%%%%%%%%%%%%%
%and thus there are three central products of $3A_{6}$ whose order is less than
%3840.% Since $Out(3A_{6})\cong2^{2}$, for an extension whose action is not inner,
% only the extensions of $3A_{6}$ by 2 need to be considered.

%%%%%%%%%%%%%%%%%%%%%%%%%%%%%%%%%
%%%\section{Supplement function preserved by a $G$-isomorphism}
\subsection{Supplements preserved by a $G$-isomorphism}\label{isopreserved}
%\subsection{finding solvable supplement for non-solvable group}
%Let us remind \emph{Schur-Zassenhaus theorem}. If a finite group $E$ has a normal
%subgroup $G$ such that the greatest common divisor of $|G|$ and $|E|/|G|$ is $1$, then
%there exist a subgroup $C$ of $E$ that is a complement of $G$ ($E=GC$ and $G\cap C=1$)
%(see \cite{}).
Our method reduces the construction of extensions of $G$ to the construction of
supplements of $G$, which are extensions of smaller subgroups of $G$.
It is very helpful in order to construct extensions,  %determine the existence of a given extension
but if we aim to classify extensions we would like moreover that isomorphisms between extensions could
  be reduced to isomorphism between supplements. In this situation, we can lift the isomorphisms
  to semi-direct product (Corollary \ref{liftiso}). If $G$-isomorphic extensions $E_{i}$
 (for $i=1,2$) have been constructed from supplements $S_{i}$  we would like that there
 exists a $G$-isomorphism $j$ mapping $S_{1}$ onto $S_{2}$ because in that case $j$
 is completely described by  $j|_{S_{1}}$ and by $j|_{G}\in Aut(G)$. %the automorphism $j|_{G}$ of $G$.
This is the motivation for the following definition.

 \begin{definition}\textbf{Definition.} Let $\mathcal{F}$ be a family of extensions of $G$ and let $E\to S_{E}$
 be a supplement function for the extensions in  $\mathcal{F}$.  We say that
\emph{\bf{the function $S_{E}$ is preserved by a $G$-isomorphism}} if for every pair $(E_1, E_2)$ of $G$-isomorphic
extensions in $\mathcal{F}$, there exists a $G$-isomorphism $j$ from $E_1$ onto $E_2$ that maps
$S_{E_1}$ onto $S_{E_2}$.
\end{definition}

It it equivalent to require that there exists a $G$-isomorphism $j:E_1\to E_2$ mapping $S_{1}=S_{E_1}$ onto a
conjugate of $S_{2}=S_{E_{2}}$ in $E_2$. Indeed, if $j(S_{1})=S_{2}^{e}$ for some $e\in E_{2}$ then we
can compose $j$ with the inner automorphism induced by $e^{-1}$ to get a new
$G$-isomorphism $\tilde j$ that maps $S_{1}$ onto $S_{2}$. Hence, $S_{E}$
is preserved by a $G$-isomorphism if and only its
conjugacy class is preserved by a $G$-isomorphism. Note also that since every $e\in E$ can
 be written $sg$ for some $s\in S_{1}$ and some $g\in G$, the conjugacy class of the
 supplement $S_{1}$ in $E$ is the set $\{S_{1}^{g}:\,g\in G \}$. We would like to insist on the difference
 between the notion "being preserved by a $G$-isomorphism" and the notion
 "being preserved by $G$-isomorphism" ; the latter is much more restrictive since it requires that every
$G$-isomorphism from $E_1$ onto $E_2$ maps $S_1$ onto $S_2$.
Generally, our supplement functions are not preserved by every $G$-isomorphism : to classify extensions
up to $G$-isomorphism, it is sufficient to find a single $G$-isomorphism between a pair of extensions.

An immediate consequence of our definition is the following corollary that can be used to prove non isomorphism :
\begin{corollary} Let $\mathcal{F}$ be a family of extensions of $G$
and let $E\to S_E$ be a supplement function for $\mathcal{F}$ that is preserved by a $G$-isomorphism.
If $S_{E_1}\ncong S_{E_2}$, then $E_1$ and $E_2\in \mathcal{F}$ are not $G$-isomorphic.
\end{corollary}
%We note first that the structure of $Aut(G)$ provides families of supplements.
%Let $E$ be an extension of $G$, let $f$ be the conjugation action of $E$ on $G$ and
%let $L\geq Inn(G)$ be $f(E)$.
%If $S$ is a supplement to $G$ in $E$ then $L=f(E)=f(GS)=f(G)f(S)=Inn(G)f(S)$ and thus
%$f(S)$ is a supplement to $Inn(G)$ in $L$. Conversely, if $B$ is a supplement to $Inn(G)$
%in $L$, then for every $e\in E$ we have $f(e)=f(g)b$, for some $f(g)\in Inn(G)$ and some
%$b\in B$ ($g\in G$). Thus for $s=g^{-1}e$, $f(s)=b$ and $s\in S:=f^{-1}(B)$
%whence $e=gg^{-1}e=g\in GS$. Therefore, the preimage in $E$ of a supplement to $Inn(G)$
%in $L$, is a supplement to $G$ in $E$.\\

Let $G$ be a non abelian group,
let $L$ be a subgroup such that $Inn(G)\leq L\leq Aut(G)$ and let $\mathcal{F}_{L}$ be the family of
extensions of $G$ associated with $L$.  In subsection \ref{preimsupp} we have constructed a
supplement function for $\mathcal{F}_{L}$. We will now determine sufficient conditions for such preimage
supplement to be preserved by a $G$-isomorphism. This will help in the search for any kind of supplement
function preserved by a $G$-isomorphism. Indeed, every $G$-isomorphism $i:E\rightarrow E_{2}$ between
extensions of $G$, maps the centralizer $C_{E}(G)$ of $G$ in $E$ onto $C_{E_{2}}(G)$.
Hence, if a supplement function $E\to S_E$ is preserved by a $G$-isomorphism, then
 the supplement function $E\to S_E C_{E}(G)$ has the same property (and it is a preimage supplement
 function as described in subsection \ref{preimsupp}. Therefore the determination of the preimage supplements
 preserved by a $G$-isomorphism is the first step to solve the isomorphism problem for extensions of $G$.

\begin{wrapfigure}[8]{l}{0.6\linewidth}
%\mbox{\input{Giso1.latex}
\includegraphics[width=5.5cm,height=4cm]{normL.eps}
\end{wrapfigure}

\begin{proposition}\label{Lisopres}
Let $G$ be a group and let $L$ be a subgroup such that $Inn(G)\leq L\leq Aut(G)$.
 Let $\mathcal{F}_{L}$ be the set of extensions of $G$ associated to $L$, let $B$ be a supplement to
 $Inn(G)$ in $L$ and for every $E\in \mathcal{F}_{L}$ let $S_E$ be the preimage of $B$ in $E$. If
 $N_{Aut(G)}(L)=N_{Aut(G)}(B)L$ then
\begin{enumerate}
\item the function $E\to S_E$ on $\mathcal{F}_{L}$ is preserved by a $G$-isomorphism.
\item For $E_1$ and $E_2$ in $\mathcal{F}_{L}$, every $G$-isomorphism from $E_{1}$ onto
$E_{2}$ maps the conjugacy class of $S_{E_1}$ onto the conjugacy class of $S_{E_2}$.
%%Every $G$-isomorphism $E_1 \to E_2$ preserves the conjugacy classes of the supplements $S_{E_k}$
\end{enumerate}
\end{proposition}
\emph{Proof }.
For $k=1,2$, let $E_{k}$ be $G$-isomorphic extensions in $\mathcal{F}_{L}$.
 Let $f_{k}$ be the conjugation action of $E_{k}$ on $G$ ; by hypothesis  $L=f_{k}(E_{k})$.
Let $S_k:=S_{E_k}=f_{k}^{-1}(B)$. If $j:E_{1}\rightarrow E_{2}$ is a $G$-isomorphism
whose restriction to $G$ is $\gamma\in Aut(G)$ then  by Lemma \ref{fpi}, we have $L^{\gamma}=L$ so
that $\gamma\in N_{Aut(G)}(L)$. Using the same Lemma, $f_{2}(j(S_{1}))=B^{\gamma}$.
By hypothesis there is $l\in L$ such that $\gamma\in N_{Aut(G)}(B).l$, whence
$f_{2}(j(S_{1}))=B^{l}$. Note that $j(S_1)$ contains $C_{E_2}(G)$ (because $C_{E_1}(G)<S_1$ and a
$G$-isomorphism maps $C_{E_1}(G)$ onto $C_{E_2}(G)$) ; thus  $j(S_1)$ is the preimage under $f_2$
of a subgroup of $Aut(G)$ and since $f_{2}(j(S_{1}))=B^{l}$ we have $j(S_{1})=f_2 ^{-1} (B^{l})$.
But $f_2 (E_2)=L$ implies that there is a $x\in E_2$  such that $f_2 (x)=l$. Therefore, $S_{2} ^{x}=S_{2} ^{l}=
(f_2 ^{-1} (B))^{l}=f_2 ^{-1} (B^{l})$ which is equal to $j(S_{1})$ (the equality in the middle is obtained applying
$f_2$ to both terms). Finally we have proved that every $G$-isomorphism from $E_1$ onto $E_2$ maps
$S_1$ onto a conjugate $S_2 ^{x}$ of $S_2$ in $E_2$, thus $(2)$ is proved and according to what we have
explained at the beginning of this subsection, this is sufficient for $E\to S_E$ to be preserved by a
$G$-isomorphism.$\quad\Box$

Remark that since $L=BInn(G)$ and $B\leq N_{Aut(G)}(B)$, the condition of Proposition \ref{Lisopres}
amounts for $N_{Aut(G)}(B)$ to be a supplement to $Inn(G)$ in $N_{Aut(G)}(L)$.
\subsubsection{Summary}
Proposition \ref{Lisopres} gives a strategy to classify the extensions of $G$ associated with
a subgroup $L\geq Inn(G)$. First, we look for a supplement $B$ to $Inn(G)$ in $L$ such that its
 normalizer $N$ is a supplement to $Inn(G)$ (or to $L$) in $N_{Aut(G)}(L)$. The preimage $S_E$
 of $B$ in an extension $E$ associated to $L$ is a supplement to $G$ in $E$
 and the construction of $E$ is reduced to the construction of $S_E$.  Moreover isomorphism testing
 for such extensions $E$ is reduced to the consideration of isomorphisms $j$ between
 the supplements $S_E$  ( and such that $j|_{G}\in N$).
% In this case, the restriction $\gamma$ of a $G$-isomorphism $i:E_{1}\rightarrow E_{2}$
%belongs to $N_{L}=N_{Aut(G)}(B)Inn(G)=$. Hence, there exists $g\in G$ such that
%$\gamma$ maps $B$ onto $B^{f_{2}(g)}$ and consequently $i$
%maps $S_{1}$ onto $S_{2}^{g}$. This is a stronger result than just being preserved by
%a $G$-isomorphism since it implies that every $G$-isomorphism preserves
%the conjugacy classes of the supplements $S_{E_k}$.
%   For $k=1,2$, let $E_{k}$,  be extensions of $G$, let $f_{k}$ be the conjugation
%  action of $E_{k}$ on $G$ and let $L_{k}:=f_{k}(E_{k})$.
%
%\begin{proposition}\label{supL} Let $L\geq Inn(G)$ be a subgroup of $Aut(G)$.
%Assume that for each extension $E$ associated with a conjugate of $L$, a function gives
%a supplement----------
%Let $E$ be an extension of $G$ and let $f$ be the
%conjugation action of $E$ on $G$. Let $L=f(E)$ and let $B$ be a supplement to $Inn(G)$
%in $L$. If $N_{Aut(G)}(B)$ is a supplement to $L$ in $N_{Aut(G)}(L)$, then every
%$G$-isomorphism maps the conjugacy class of $S_{1}$ onto the conjugacy class of $S_{2}$.
%Let $S_{k}:=f_{k}^{-1}(B_{k})$.
%\end{proposition}

If one wants a supplement function preserved by a $G$-isomorphism for a larger family than $\mathcal{F}_{L}$
then the following result can be used.
\begin{proposition}\label{Allisopres}
Let $G$ be a group and let $\mathcal{F}$ be the set of extensions of all extensions of $G$.
Suppose that there is a supplement $B$ to $Inn(G)$ in $Aut(G)$. For every $E\in \mathcal{F}$,
 let $S_E$ be the preimage in $E$ of $B\cap f(E)$ where $f$ is the conjugation action of $E$ on $G$.
\begin{enumerate}
\item the function $E\to S_E$ on $\mathcal{F}$ is preserved by a $G$-isomorphism.
\item For $E_1$ and $E_2$ in $\mathcal{F}$, every $G$-isomorphism from $E_{1}$ onto
$E_{2}$ maps the conjugacy class of $S_{E_1}$ onto the conjugacy class of $S_{E_2}$.
%%Every $G$-isomorphism $E_1 \to E_2$ preserves the conjugacy classes of the supplements $S_{E_k}$
\end{enumerate}
\end{proposition}
\emph{Proof }.
%%The proof is similar to the to the proof of Proposition \ref{Lisopres}.
For $k=1,2$, let $E_{k}$ be extensions of $G$.
 Let $f_{k}$ be the conjugation action of $E_{k}$ on $G$ and $L_{k}=f_{k}(E_{k})$.
Let $B_k=B\cap L_k$. By hypothesis $S_k:=S_{E_k}=f_{k}^{-1}(B_{k})$.
If there is a $G$-isomorphism $j:E_{1}\rightarrow E_{2}$ that induces $\pi\in Aut(G)$ on $G$  then
 by Lemma \ref{fpi}, we have $L_2=L_1 ^{\pi}$
 %and hence any such $G$-isomorphism must induce on $G$ an element of $N_{Aut(G)}(L_{1})\pi$.
Since $Aut(G)=B Inn(G)$, we can write $\pi=b\tilde{g}$ where $b\in B$ and where $\tilde{g}$ is the inner
automorphism of $G$ induced by some $g\in G$.  Then $L_2=L_1 ^{\pi}=L_1 ^{b\tilde{g}}=L_1 ^{b}$ (since
$Inn(G)$ normalizes every subgroup that contains it). Consequently $B_1 ^{b}=(B\cap L_1)^{b}=
B^{b}\cap L_1 ^{b}=B\cap L_2=B_2$.

At this point, the proof is identical to the proof of Proposition \ref{Lisopres} :
$f_{2}(j(S_{1}))=B_1 ^{\pi}=B_1 ^{b\tilde{g}}=B_2 ^{\tilde{g}}$ implies
$j(S_{1})=(f_2 ^{-1} (B_2))^{g}=S_2 ^{g}$ which proves $(2)$ and $(1)$ is a consequence of $(2)$.
$\quad\Box$

%%%%%%%%%%%%%%%%%%%%%%%%%%%%%%%%%%%%%%%%%%%%%%%
%\section{Finding supplements}
\chapter{Finding supplements} \label{findsup}
\subsubsection*{Introduction}
In this chapter  we show that for the extensions $E$ of a finite group $G$ it is always possible to find
supplement functions $E\to S(E)$ preserved by a $G$-isomorphism and such that $G\cap S(E)$ is nilpotent
for every $E$.
We describe explicitely an algorithm to obtain such functions efficiently (section \ref{funcNp}).
Moreover, we will prove that any supplement function $E\to S(E)$
preserved by a $G$-isomorphism, can be refined into a function $E\to \widetilde{S}(E)\subseteq S(E)$
that is also preserved by a $G$-isomorphism and such that $G\cap S(E)$ is nilpotent
for every extension $E$ of $G$.
%\subsection{Various supplement function for various families}\label{variousfam}
%\begin{enumerate}
%\item If $F$:=all extensions of $G$, then take  , where $\vec{p}$ is an ordering of the prime
%factors of $|G|$.
%\end{enumerate}
%A last family will be described in chapter \ref{General} (\ref{Gisofunctphi})

% observe that since nilpotent groups are successive central extensions, If Inn(G) is nilpotent
%then $G$ is nilpotent because it is a central extension of Z(G) by the nilpotent G
%%%%%%%%%%%%%%%%%%%%%%%%%%%%%%%%
%%%%%%%%%%%%%%%%%%%%%%%%%%%%%%%%

%%%            %%%%%%        %%%%   %%%
%%%%%%%%%%%%%%%%%%%%%%%%%%%%%
%%%%%%%%%%%%%%%%%%%%%%%%%%%%%

%\subsection{The functions $N_{\vec{p}}$ and $N_{\vec{p}+}$}
\section{The functions $N_{\vec{p}}$}\label{funcNp}
Let $E$ be an extension of a group $G$ and let $p_1$ be a prime divisor of $|G|$.
%For a given prime divisor $p_{1}$ of $|G|$, let $S_{p_{1}}$ be a Sylow $p_{1}$-subgroup  of $G$.
Let us define $N_{(p_{1})}(E,G)$ as the set of all the
 normalizers in $E$ of  Sylow $p_{1}$-subgroups of $G$.
Let $S_{p_{1}}$ be a Sylow $p_{1}$-subgroup  of $G$ an let  $N_{p_{1}}:=N_E (S_{p_1})$.
 For $g\in G$, since the normalizer of $S_{p_{1}}^{g}$ in $E$ is the image $N_{p_{1}}^{g}$ of
 $N_{p_{1}}$ under $g$, the set $N_{(p_{1})}(E,G)$ is the conjugacy class $N_{p_{1}}^{G}$,
 whose elements are supplements of $G$ in $E$.
To avoid repeating each time that $p_{1}$ is a divisor of $|G|$, we choose
$N_{(p_{1})}(E,G):=\{E\}$ as set of supplements  in the case where $p_{1}$ is not a divisor
 of $|G|$.

 Let $\vec{p}=(p_{1},\ldots ,p_{k})$ be a sequence of prime divisors of $|G|$.
For the sequence $(\vec{p},p_{i})=(p_{1},\ldots ,p_{k},p_{i})$, we define inductively
the set  of subgroups $N_{(\vec{p},p_{i})}(E,G)$  as the union of the sets
$N_{(p_{i})}(N_{\vec{p}},G_{\vec{p}})$ where $N_{\vec{p}}$ belongs to
$N_{\vec{p}}(E,G)$ and $G_{\vec{p}}=G\cap N_{\vec{p}}$
% To avoid repeating each time that $p_{i}$ is a divisor of $|G_{\vec{p}}|$, we choose
% $N_{(\vec{p},p_{i})}(E):=N_{\vec{p}}(E)$ in the case where $p_{i}$ is not a divisor
% of $|G_{\vec{p}}|$ ( and $N_{(p)}(E):=\{E\}$ if $p$ does not divide $|G|$).
 ( note that $N_{(\vec{p},p_{i})}(E,G)=N_{\vec{p}}(E,G)$ if $p_{i}$ does not divide $|G_{\vec{p}}|$).
For an element $N$ of $N_{(p_{i})}(N_{\vec{p}},G_{\vec{p}})$ we have $N_{\vec{p}}=G_{\vec{p}}N$, by the
Frattini argument. If $N_{\vec{p}}$ is a supplement to $G$ in $E$, then
$E=GN_{\vec{p}}=GG_{\vec{p}}N=GN$ so that $N$ is also a supplement to $G$ in $E$.
%Since $E=GN_{(p_{1})}(E)$, by induction $N_{(\vec{p},p_{i})}(E)$ is a set of supplements
Since $N_{(p_{1})}(E,G)$ is a set of supplements
of $G$ in $E$, by induction the set $N_{(\vec{p},p_{i})}(E,G)$ has the same property.
 As shown in the previous paragraph, $N_{(p_{i})}(N_{\vec{p}},G_{\vec{p}})$ is
equal to $\{N^{x}:x\in G_{\vec{p}}\}$.\\

Whenever no confusion is possible for the normal subgroup $G$ of $E$ that
is considered, we abbreviate $N_{\vec{p}}(E,G)$ as $N_{\vec{p}}(E)$. Otherwise $G$ must
be mentioned explicitly. For instance, if $G$ is not nilpotent and has a non trivial
maximal normal nilpotent subgroup $Fit(G)$, then $N_{\vec{p}}(E,Fit(G))=E$ whereas
$N_{\vec{p}}(E,G)\neq E$.
We now investigate further the properties of these supplements.\\

%ffffffffffffffffffffff
%Let $E$ be an extension of $G$ and let $S_{p_{1}}$ be a Sylow $p_{1}$-subgroup of $G$
%Let $N_{p_{1}}=N_{E}(S_{p_{1}})$ be the normalizer of $S_{p_{1}}$ in $E$ and let
%$N_{(p_{1})}$ be the set of all such normalizers for a given prime $p_{1}$.

% $(\vec{p},p_{i})$ be the sequence $(p_{1},\ldots ,p_{k},p_{i})$. We define the set
% of subgroups $N_{(\vec{p},p_{i})}$ inductively for as follows ; Let $N$ be a subgroup
% in the set of $N_{\vec{p}}$, let $K=N\cap G$ then the subgroups of
% $N_{(\vec{p},p_{i})}$ are the

\begin{proposition} Let $E$ be an extension of a finite group $G$.
\begin{enumerate}
\item $N_{\vec{p}}(E)$ is a conjugacy class of supplements of $G$ in $E$.

\item If $E_{i}\,$, $i=1,2$, are extensions of $G$, every $G$-isomorphism from $E_{1}$ onto
$E_{2}$ maps $N_{\vec{p}}(E_{1})$ onto $N_{\vec{p}}(E_{2})$.
\item If $p\in \vec{q}$ then $N_{(\vec{q},p)}=N_{\vec{q}}$.  %%% EQUALITY OF FUNCTION
\item There exists a sequence of distinct primes $\vec{p}=(p_{1},\ldots,p_{k})$
such that $N\cap G$ is nilpotent for any $N$ in $N_{\vec{p}}(E)$ and
any ordering of the distinct prime factors of $|G|$ is such a sequence.
%If $\vec{p}=\{p_{a},\ldots,p_{b}\}$ are the distinct prime factors of $|G|$ then
%$N_{\vec{p}}(E)$ has the required property
\end{enumerate}
\end{proposition}
\emph{Proof }.\begin{enumerate}
\item This has already been proved in the case where $\vec{p}=(p_{1})$ is of length $1$ and
we already know that the elements of $N_{\vec{p}}(E)$ are supplements of $G$ in $E$.
Let us use an induction argument. Assume that, $N_{\vec{p}}(E)$ is a
conjugacy class $\{N^{e}:e\in E\}=\{N^{g}:g\in G\}$ for some supplement $N$.
Then by definition, \begin{displaymath}
N_{(\vec{p},p_{i})}(E)=\bigcup_{g\in G}N_{(p_{i})}(N^{g},N^{g}\cap G).
\end{displaymath}
Let $M\in N_{(p_{i})}(N,N\cap G)=\{M^{x}:x\in N\cap G\}$.
Since a conjugation maps the normalizer of a subgroup onto the normalizer of its image,
 for $g\in G$, the subgroup $M^{g}$ belongs to $N_{(p_{i})}(N^{g},N^{g}\cap G)$ and
 therefore $N_{(\vec{p},p_{i})}(E)$ contains the conjugacy class $M^{G}$.
This also implies that $N_{(p_{i})}(N^{g},N^{g}\cap G)$ is $\{M^{gx}:x\in N^{g}\cap G\}\subset M^{G}$
 so that $N_{(\vec{p},p_{i})}(E)=M^{G}$.

\item A classical induction argument can easily (but tediously) prove this statement.
Nevertheless, we prefer to link it with the general situation described in
subsection \ref{isopreserved}. %% ( remark $n^{\circ}3$).
 For $k=1,2$, let $E_{k}$,  be extensions of $G$, let $f_{k}$ be the conjugation
  action of $E_{k}$ on $G$ and let $L_{k}:=f_{k}(E_{k})$. In Theorem \ref{preNp} we
will  prove that $S_{k}:=N_{\vec{p}}(E_{k})$ is  the preimage of
  $B_{k}:=N_{\vec{p}}(Aut(G))\cap L_{k}$ under $f_{k}$.
Since $N_{\vec{p}}(Aut(G))$ is a supplement to $Inn(G)$ in $Aut(G)$, we have proved
in Proposition \ref{Allisopres}  %%section \ref{isopreserved} %%( remarks $n^{\circ}2$ and $3$)
that every $G$-isomorphism $i:E_{1}\rightarrow E_{2}$ maps $S_{1}$ onto $S_{2}$.
%% proposition}\label{Lisopres}
\index{use of L-preimage as been reformulated in chapter 1 : so reformulate also here}

\item If $p\in \vec{q}$ then $\vec{q}=(p_{1},\ldots,p_{n})$ has a subvector
$\vec{p}=(p_{1},\ldots,p)$. If $N\in N_{\vec{p}}(E)$ then by definition, the Sylow
$p$-subgroup $S_{p}$ of $K=N\cap G$ is normal in $K$. %% and in $N$
 If $|S_{p}|=p^{v}$, the order of $K$ is $p^{v}k$ and $k$ is prime to $p$.
$N_{\vec{q}}$ is defined recursively from $N_{\vec{p}}$. From each representative $N$ of
$N_{\vec{p}}(E)$ there is a decreasing sequence of supplements $N\supseteq\ldots\supseteq N'$,
 where $N'\in N_{\vec{q}}(E)$. %Since $N\supseteq N'$, we get $N'\cap K:=K'<K$. %% Since $N=KN'$

Let $K':=K\cap N' <K$. The subgroup $\tilde S=S_{p}\cap K'$ is normal in $K'$ as intersection of $K'$ with a
normal   subgroup of $K$. It is a Sylow $p$-subgroup of $K'$ because its index
  $l=|K'|/|\tilde S|$ is prime to $p$.
Indeed, since $S_{p}\lhd K$, the set $S_{p}K'$ is a subgroup of $K$
and its order $p^{v}|K'|/|S_{p}\cap K'|=p^{v}l$ must divide $|K|=p^{v}k$.
Thus $l=|K'|/|\tilde S|$ divides $k$, it is prime to $p$ and so $K'$ has a normal Sylow
$p$-subgroup $\tilde S$. Hence, the normalizer $N''$ of $\tilde S$ in $N'$ is equal to $N'$.
By definition $N''$ is a representative of $N_{(\vec{q},p)}(E)$ but since it also
a representative of $N_{\vec{q}}(E)=(N')^{G}$, we conclude that these conjugacy classes
are equal.

\item The existence of such a finite sequence for a finite group $G$ is just
 a rephrasing of Proposition \ref{nilsup} with the $N_{\vec{p}}$ notation.
By (3) we could restrict to distinct primes. Let $N$ be a representative of $N_{\vec{p}}(E)$
 ($p$ is an arbitrary sequence )
and let $p$ be a prime factor of $|N\cap G|$. In the proof of (3), we showed that if $p$
 belongs to $\vec{p}$ then the Sylow-$p$ subgroup of $N\cap G$ is normal. If
 $\vec{p}=\{p_{a},\ldots,p_{b}\}$ are the distinct prime factors of $|G|$ then
$p\in\vec{p}$ since the order of the subgroup $N\cap G$ of $G$ divides the order of $G$.
Hence, every Sylow subgroup of $N\cap G$ is normal and $N\cap G$ is nilpotent.$\quad\Box$
% Note however that a proper subsetsince the prime divisors of $N\cap G$ can be a proper
% subset of $\vec{p}$.$\quad\Box$. NOO, this is a bad argument, an absent divisor could
% have been used during the process to reduce the order of the supplements.
\end{enumerate}

Note however, that
different orderings  for the set $\{p_{1},\ldots,p_{n}\}$ could give non conjugate supplements.
%%%%%%%%%%%%%%%%%%%%%%%%%%%%%%%%%
\subsubsection{$N_{\vec{p}}(E)$ is a preimage}
The following lemma is needed for Theorem \ref{preNp} and Proposition \ref{refinep}.
\begin{lemma}\label{prei} Let $G$ be a normal subgroup of a group $E$ and $S_{p}$ a Sylow
$p$-subgroup of $G$. If $f:E\rightarrow f(E)$ is a homomorphism such that
$Ker f\leq N_{E}(S_{p})$ then $N_{E}(S_{p})$ is the preimage under $f$ of $N_{f(E)}(f(S_{p}))$.
\end{lemma}
\emph{Proof }.
Let $M:=N_{f(E)}(f(S_{p}))$. We must show that
$N_{E}(S_{p})$ is equal to $f^{-1}(M):=\{e\in E\,|\,f(e)\,\, normalizes\,\,f(S_{p}) \}$.
Trivially, if $e\in N_{E}(S_{p})$, then
$S_{p}^{e}=S_{p}$ and $f(S_{p})=f(S_{p}^{e})=f(e^{-1}S_{p}e)=f(S_{p})^{f(e)}$ so that
$f(e)$ normalizes $f(S_{p})$. Hence, $N_{E}(S_{p})\leq f^{-1}(M)$.
Conversely let us show that under our assumptions, if $f(e)$ normalizes $f(S_{p})$ then
$e$ normalizes $S_{p}$.
Let $f(e)\in M$. Let $K:=Ker f$. We have $f(S_{p})=f(S_{p})^{f(e)}=f(S_{p}^{e})$
$\Rightarrow KS_{p}^{e}=KS_{p}$. But if $e$ normalizes $KS_{p}$ then it
normalizes $N:=G\cap KS_{p}$. Since $S_{p}\subseteq N\leq G$, and since $S_{p}$ is a Sylow
$p$-subgroup of $G$, it is also a Sylow $p$-subgroup of $N$. But we have assumed that
$K$ normalizes $S_{p}$ so that $N\leq KS_{p}$ normalizes also $S_{p}$ and thus
$S_{p}$ is normal in $N$ whence it is the unique Sylow $p$-subgroup of $N$.  As $e$ normalizes $N$, the
subgroup $S_{p}^{e}\subseteq N$ is a Sylow $p$-subgroup of $N$ and hence $S_{p}^{e}=S_{p}$
$.\quad\Box$

Let $S\leq L\leq A$ be an increasing sequence of groups. The normalizer $N_{L}(S)$ is
equal to $N_{A}(S)\cap L$. Since the function $N_{\vec{p}}$ is defined by iterated
normalizers, we obtain the following consequence : if $E_{1}\leq E_{2}$ are extensions of
a group $G$ then $N_{\vec{p}}(E_{1})=E_{1}\cap N_{\vec{p}}(E_{2})$.

\begin{theorem}\label{preNp}Let $E$ be an extension of $G$ and let $f$ be the action
of $E$ on $G$ by conjugation.  Then $N_{\vec{p}}(E)$ is
 the preimage of $N_{\vec{p}}(Aut(G),Inn(G))\cap f(E)$ under $f$.
% If we consider $Aut(G)$ as an extension of $Inn(G)$, then $N_{\vec{p}}(E)$ is
% the preimage of $N_{\vec{p}}(Aut(G))\cap f(E)$ under $f$.
\end{theorem}
Let us make three preliminary observations.\\

$(i)$ If a subgroup $S$ of $E$ is the preimage of $f(S)$ then $Ker f\subseteq S$. Due to the
one-to-one correspondence between such subgroups and the subgroups of $f(E)$,
each conjugate of $f(S)$ in $f(E)$ is the image of exactly one conjugate of $S$ in $E$.
Hence, the conjugacy class $\mathcal{C}_{1}$ of $S$ is the preimage of the conjugacy class
 $\mathcal{C}_{2}$ of $f(S)$ and
we just have to prove that one element of $\mathcal{C}_{1}$ is the preimage of an
element of $\mathcal{C}_{2}$.\\

$(ii)$ If $S_{p}$ is a Sylow $p$-subgroup of $G$ and $f$ a homomorphism with kernel $K$,
then the index $i:=[f(G):f(S_{p})]$ is
$|KG|/|KS_{p}|=\frac{|K|.|G|/|K\cap G|}{|K|.|S_{p}|/|K\cap S_{p}|}=\frac{|G|.|K\cap S_{p}|}
{|S_{p}|.|K\cap G|}=[G:S_{p}]/l$ where $l$ is the index of $K\cap S_{p}$ in $K\cap G$.
Since $[G:S_{p}]$ is coprime to $p$, the index $i$ is also coprime to $p$ and since
$d:=|f(S_{p})|$ divides $|S_{p}|$, either $d=1$ and $p$ does not divide $i=|f(G)|$
 or $d$ is a $p$-power and $f(S_{p})$ is a $p$-Sylow of $f(G)$.\\

$(iii)$ $f(G\cap S)\subseteq f(G)\cap f(S)$ always holds. But if $Ker f\subseteq S$ then
the reverse inclusion is also true so that $f(G\cap S)=f(G)\cap f(S)$. Indeed, if
$x\in f(G)\cap f(S)$ then $x=f(g)=f(s)$ for some $g\in G$ and some $s\in S$. Thus
$\exists k_{s}\in Ker f\subseteq S$ such that $g=k_{s}s$ whence $g\in G\cap S$ and
$x=f(g)\in f(G\cap S)$.\\

\emph{Proof }.
Let $C$ be the kernel of $f$ so that $C$ is the centralizer of $G$ in $E$. Since $C$ normalizes
every subgroup of $G$, Lemma \ref{prei} can be used iteratively with $f$.
We use an induction argument. Assume that $S$  is a supplement to $G$ in $E$ and that $S$
is the preimage under $f$ of $f(S)$ (for instance we can start with $S=E$).
% Note that $f(S)$ is then a supplement to $In(G)$ in $f(E)$
Let $G_{s}:=G\cap S$ and $I_{s}:=Inn(G)\cap f(S)$. Since $Ker f\subseteq S$, by $(iii)$
$f(G_{s})=f(G)\cap f(S)=Inn(G)\cap f(S)=I_{s}$.\\
If $p$ is a prime number, let us prove that an element of $N_{(p)}(S,G_{s})$ is the
preimage under $f$ of some element of $N_{(p)}(f(S),I_{s})$.
 If $p$ does not divide $|G_{s}|$, it does not divide
$|I_{s}|=|f(G_{s})|$ and trivially $N_{(p)}(S,G_{s})=\{S\}$ is the preimage of $N_{(p)}(f(S),I_{s})=\{f(S)\}$.\\
If $p$ divides $|G_{s}|$, let $S_{p}$ be a $p$-Sylow subgroup of $G_{s}$. Then by Lemma
\ref{prei}, $A:=N_{S}(S_{p})\in N_{(p)}(S,G_{s})$ is the preimage of
$B:=N_{f(S)}(f(S_{p}))$ and we must prove that $B\in N_{(p)}(f(S),I_{s})$. %% ($\star$).
 Now by $(ii)$ there are two possible cases : either
$f(S_{p})$ is a $p$-Sylow of $f(G_{s})=I_{s}$ whence $B\in N_{(p)}(f(S),I_{s})$ ;
or $f(S_{p})=1$, $p$ does not divide $|f(G_{s})|=|I_{s}|$ (see $(ii)$) whence
$N_{(p)}(f(S),I_{s}):=\{ f(S)\}$  and since $B:=N_{f(S)}(1)=f(S)$ ($\star$),
 we have that $B\in N_{(p)}(f(S),I_{s})$%%holds.

 Let $L:=f(E)$. By induction we have proved that for a sequence $\vec{p}$ of prime
 numbers, $N_{\vec{p}}(E)$ is the preimage of $N_{\vec{p}}(L)$ under $f$ but as we
 mentioned before this theorem, since $Inn(G)\leq L\leq Aut(G)$, then
 $N_{\vec{p}}(Aut(G))\cap L=N_{\vec{p}}(L)$  (we consider $Aut(G)$ as an extension of $Inn(G)$).$\quad\Box$
%
%  This proves that the preimage of
%$N_{\vec{p}}(L)\cap Inn(G)$ is nilpotent. If we look at the upper central series, it is
%not difficult to prove that a central extension of a nilpotent group by another
%nilpotent group is still a nilpotent group.
%
%\begin{enumerate}
%\item $f(S_{p})$ is a $p$-Sylow of $f(G_{s})=I_{s}$ whence $B\in N_{(p)}(f(S),I_{s})$
%and its preimage is $A\in N_{(p)}(S,G_{s})$ ( by $(i)$, that is all we need ) ;
%\item $f(S_{p})=1$ whence $B=f(S)\in N_{(p)}(f(S),I_{s})$ and $S_{p}\leq Ker f$.
%But $Ker f\leq S\Rightarrow Ker f \cap G=Z(G)\leq S\cap G=G_{s}$ and
%$S_{p}\leq Ker f$ implies $S_{p}\leq G_{s}\cap Ker f$
%$\Rightarrow $ .
%\end{enumerate}
%
% argument. Let $\vec{q}=(p_{1},\ldots,p_{i-1})$ and let
% $\vec{p}=(p_{1},\ldots,p_{i-1},p_{i})$. First, if

The crucial consequence of Theorem \ref{preNp} is the possibility to calculate the
$N_{\vec{p}}$ supplements once, for all the extensions of $G$, in a unique computation on $Aut(G)$.
We just mentioned for the reader that if we apply this theorem iteratively,
we obtain that $N_{\vec{p}}(E)$ contains in turn $Z(G)$,  $Z(G/Z(G)),\ldots$ and
eventually $Z_{\infty}$, the stationary term of the upper central series of $G$.
% !! $Z_{\infinity}$ is a nilpotent normal subgroup of $G$ but it can be strictly smaller
% than $Fitting(G)$. Think to the numerous examples of centerless groups that have a non
% trivial  Fitting subgroup
% ex : Symm(3) (Fit=3), Symm(4) (Fit=2*2), PerefctGroups(960) (Fit =2^{4}
\subsubsection{Nilpotent refinement of supplement functions}
Finally, let us show that if $G$ is finite, it is always possible to consider supplement functions ($E\to S(E)$)
preserved by a $G$-isomorphism and such that $G\cap S(E)$ is nilpotent for every extension $E$ of $G$.
\begin{proposition}\label{refinep} Let $G$ be a finite group and let $\mathcal{F}$ be a family of extensions
of $G$. Let $\Sigma$ be the set of all supplement functions for $\mathcal{F}$ that are preserved by a
$G$-isomorphism.  Then for every $S\in\Sigma$ there exists $\widetilde{S}\in \Sigma$ such that
 for every $E\in \mathcal{F}$, the group $G\cap \widetilde{S}(E)$ is nilpotent and such that
 $\widetilde{S}(E)\subseteq S(E)$.
\end{proposition} \emph{Proof }. Let $S\in\Sigma$ and for $E\in\mathcal{F}$ let $\vec{p}(E)$ be the prime factors of
$G\cap S(E)$ in increasing order. For every $E\in\mathcal{F}$ let $V_E$ be the set $N_{\vec{p}(E)}(S(E),S(E)\cap G)$.
 In order to define a supplement function $E\to \widetilde{S}(E)\subset S(E)$, let us show that it is possible
 for each $E$, to choose unequivocally one element in $V_E$ without assuming  the axiom of choice
 (note that $\mathcal{F}$ may be infinite if $E/G$ is infinite).
 
 For this purpose since $G$ is finite, it has a finite number of subgroups and we may assign to each subgroup
 of $G$ a different positive integer. The function $N_{\vec{p}}$ is defined inductively  as a repeated union
 of normalizers of some subgroup of $G$. Hence, at each step  we may assign to each normalizer, the number
 of the subgroup of $G$ that is used to define it. We associate in this way to each element $N$ of $V_E$ at
 least one finite sequence of positive integers (the length of the sequence is the length of $\vec{p}(E)$).
 By construction, it is possible from each sequence, to find back the corresponding subgroups of $G$ that define
 $N$ unequivocally. Different sequences could produce the same element $N\in V_E$, but since $G$ has a finite number
 of subgroups and since $\vec{p}(E)$ is of finite length, the number of such sequences for $N$ is finite. We can
  assign a different integer to each sequence $\sigma$ (for instance the integer $n_{\sigma}$ whose sequence of decimal
  digits is  $\sigma$) and we assign to each element of $V_E$ the smallest such positive integer $n_{\sigma}$.
  Finally define $\widetilde{S}(E)$  as the element of $V_E$ whose number is the smallest and we have
 constructed unequivocally a supplement function $E\to \widetilde{S}(E)\subset S(E)$ (even if $\mathcal{F}$
 contains an infinite number of extensions).

Now, we show that   this function is preserved by a $G$-isomorphism. Let $E_1$ and $E_2$ be two
$G$-isomorphic extensions of $\mathcal{F}$. Since $S\in\Sigma$, there is a $G$-isomorphism
$j:E_1 \to E_2$ that maps $S(E_1)$ onto $S(E_2)$. Since $Ker\, j=1$ normalizes every subgroup of $G$,
Lemma \ref{prei} may be used iteratively. Moreover remark $(iii)$ of Theorem \ref{preNp} holds since
$Ker\,j=1$ is contained in every subgroup of $G$ and the proof is identical to the proof of Theorem \ref{preNp} ;
$j$ maps %%$N_{\vec{p}(E_1)}(S(E_1),S(E_1)\cap G)$ onto $N_{\vec{p}(E_2)}(S(E_2),S(E_2)\cap G)$.
 $V_{E_1}$ onto $V_{E_2}$ .
Since these sets are conjugacy classes of supplements, we have explained in subsection \ref{isopreserved}
that if $j$ is composed with the conjugation in $E_2$ by a well-chosen element of $G$, then there is a
$G$-isomorphism that maps $\widetilde{S}(E_1)\in V_{E_1}$ onto $\widetilde{S}(E_2)\in V_{E_2}$.
$\quad\Box$
%%%%%%%%%%%%%%%%%%%%%%%%%%%%%%
\subsection{Solvable techniques for nonsolvable groups}
Let $E$ be an extension of a finite group $G$ by a (non necessarily finite) solvable
group $H$. We have proved the existence of a supplement function $E\to S_{E}$ preserved by a
$G$-isomorphism and such that $S_{E}\cap G$
is nilpotent (hence solvable). Therefore $S_{E}$ is solvable because it is an extension
of the solvable group $S_{E}\cap G$ by the solvable group $H$.
 %As a particular case we have the following result :
%\begin{corollary} If $E$ is a nonsolvable group with a finite perfect residuum
%$P\unlhd E$ then $P$ has a solvable supplement that is preserved by a $P$-isomorphism.
%\end{corollary}

This result implies that the construction of nonsolvable groups
having $P$ as perfect residuum can be reduced to the construction
of solvable extensions of a fixed nilpotent subgroup of $P$. From
a computational point of view, this is a major improvement. The
polycyclic presentation of finite solvable groups $S$ is one of the
most efficient group presentations. Most of the computational
problems on $S$ can be solved very quickly if such a presentation
is provided (see \cite{seress2001}). On the other hand, most of the algorithms
are much slower if the group is given as an arbitrary finitely
presented group (Fp-group for short), and unfortunately this is the general way to
handle arbitrary nonsolvable groups. For instance an extension $E$
of a perfect group $P$ has to be constructed as a Fp-group, in the
general case. The first way to study such an extension $E$ is to use the coset
enumeration algorithm of Todd-Coxeter  (see \cite{Suzuki_vol1} page 174). We have tried
this approach but it has turned out that, even for relatively
small nonsolvable groups (of order around 3000),
computations in $E$ were very slow : Todd-Coxeter's method spent a
long time to enumerate cosets and the amount of memory needed for
this purpose was quickly increasing. If the perfect
 group $P$ is given as a transitive permutation group acting on $d$ points an alternative method is to
% for a non transitive , the same holds if the kernels $K_{i}$ of the actions on each
% orbits stay invariant (normal) in the extensions
 extend this presentation to the extensions of $P$ by a group of order $n$ in order to
 construct a permutation group of degree $d.n$. Despite the fact that permutation
 presentations are computationnaly efficient for nonsolvable groups, a large degree
 is still a strong limitation. For instance $P=SL(2,5)$ has a presentation of degree $24$,
 but an extension of $P$ of order $3600=120*30$ has degree $30*24=720$.

But the main problem that makes these approaches extremely unpractical is isomorphism testing.
This is usually the most time-consuming part of groups classifications, especially for
groups of large order. Hence, the elaboration another method seemed essential to us for classification
of large collections of groups and this leaded us to our solvable supplement method (i.e. supplement
functions that are preserved by some isomorphisms).
 For the isomorphism problem, the advantage of our method is due to the fact that
isomorphism testing has to be done only on solvable subgroups of much smaller order.

%%%%%%%%%%%%%%%%%%%%%%%%%%%%%%%%%%%%
\section{A fast splitting test}\label{fastsplitest}

Assume that we are in the general situation where we want to find supplements of a normal
subgroup $G$ of a group $E$ and assume that we are looking for supplements of minimal
order. Essentially, the whole subgroup lattice has to be computed first, then we need to
 select the subgroups that are supplements of $G$ and finally we need to compare their
 order. A slight variation of this is a recursive determination of the maximal subgroups
  $M$ of $E$ which are supplements of $G$ ;  the maximals of the maximals which are still
  supplements are computed next and so on, until we reach
  a supplement $S$ of $G$ such that $S$ does not contain any smaller supplement to $G$.
The last step is the order comparison of the "final" supplements
obtained in this way. In some special cases, better methods are
available : for instance if $G$ (but not necessarily
$E/G$) is solvable, then there is an algorithm that starts computing an
elementary abelian series through $G$ (see \cite{GAP4}).

We propose to speed up the process by starting the search into the
supplements provided by the function $N_{\vec{p}}$. The process
is the following. An order $u$ is fixed and we are looking for
supplements $S$ of $G$ such that $|S\cap G|\leq u$. First choose
an arrangement $\vec{p}=(p_{1},\ldots,p_{n})$ of the prime factors
of $|G|$. Then, compute the supplements $S\in N_{\vec{p}}(E)$ (we
assume that
 $G$ is not nilpotent). If $E/G$ is solvable (for instance if $G$ is the perfect residuum
 of $E$), we are sure that $S$ is solvable too.
Finally, the last step is the general approach described in the previous paragraph where
maximal subgroups need to be computed but it is now applied on a subgroup $S$,
 generally much smaller than $E$. Moreover, the maximal subgroups of a solvable subgroup
 $S$ are much cheaper to compute than the maximal subgroups of $E$.
%% ( with algorithms that uses a special polycyclic presentation).
If no supplement such that $|S\cap G|\leq u$ is found, we start once again the process with
another arrangement of the prime factors of $|G|$.

At the end of the process, we are sure to obtain proper
supplements, but, except in the case where a complement is found,
there is no guarantee to obtain the minimal ones.

Let us show on some examples how this method allows to investigate
large nonsolvable groups for which we cannot compute the maximal
subgroups. However, our computing times are not optimal at all here. We have not even
 transform the solvable subgroups into pc-groups. Hence, one could expect further improvements
 of our timings by restricting maximal subgroups computations into polycyclic subgroups.% : once it will be done, maximal subgroups computation will be restricted to
%pc-groups : this will improve the timing a lot).
The identification number for perfect groups refers to those indicated
in the library of Holt and Plesken (\cite{Holt_Plesken}).

\begin{itemize}
\item
For $P:=PerfectGroup(7500,1)$, $|Out(P)|=40$ and thus $Aut(P)$ is
a group of order $300,000$. Our method has found a complement of
$Inn(P)$ in $Aut(P)$ in 34 seconds (containing 16 seconds to determine
$Aut(P)$), while the computation of the maximal subgroups of
$Aut(P)$ could not be achieved with 800 mega ram on a $800$ Mhz
Duron processor. Next, we found that the smallest transitive
permutation representation  for $Aut(P)$ has degree 49. The
computation of the maximal subgroups with the 49 points
representation for $Aut(P)$ took 110 minutes with $400$ megaram  %(?could be less ?)
dedicated to Gap 4.2 on the same machine.

\item For $P=PerfectGroup(29160,4)$, $|Out(P)|=72$ and we get in $80$ seconds, (containing
$49$ seconds for $Aut(P)$ ), a supplement $S$ of $Inn(P)$ such
that $u=|S\cap Inn(P)|=2$. Testing the minimality of such
supplement would require to compute the full subgroup lattice of
$Aut(P)$, which is a group of order $29,160*72=2,099,520$
elements. Such subgroups computation is out of the capacity of our
machine. However, even if there is a hidden complement, $u=2$ is
small enough to classify the extensions of $P$ efficiently. For
perfect groups
 up to order $30,000$, $2$ is the maximal value for $u$.

\item Even more impressive is the case of $P=PerfectGroup(15360,4)$ whose
outer automorphism group is the nonsolvable group $A\Sigma L(2,4)$ of order 5760.
Thus $Aut(P)$ is a group of order $15,360*5,760=88,473,600$ which is of course completely
 beyond the capacity of our machine (note moreover that $Aut(P)$ is not given as a permutation group
but as the automorphism group of a permutation group). First $GAP$
took 66 minutes to determine $Aut(P)$. Then using
$N_{\vec{p}}(Aut(P))$ a supplement $S$ such that $|S\cap
Inn(P)|=2$ was found in 30 seconds. Finally, $GAP$ took 90 minutes
to determine the maximal subgroups of $S$ and among them, a
complement has been found.

\item A last example is given by $P=PerfectGroup(32256,1)$ whose
outer automorphism group has order 168 whence $|Aut(P)|=5,419,008$. We found
a complement of $Inn(P)$ after 217 seconds, including the computation of $Aut(P)$.
\end{itemize}

%%%%%%%%%%%%%%%%%%%%%%%%%%%%%%%%%%%%%%%%%%%%%%%%%%%%%%%%%%%%
%\subsection{Tables of Perfect groups of order $< 15360$ with centre=1}\label{tableout}
\section{Tables of Perfect groups}\label{tableout}
\pagestyle{empty}
\index{explain this table :Mult,split,smallest sup}
% for multiplier : see the file /testbig
\begin{tabular}{ c|c|p{2.3 cm}|c|p{2.5 cm}|c }
Perf. Id & Structure & $Out(P)$ & Split & Smallest sup. in & Mult \\% for the Schur multiplier
\hline\hline
$[60,1]$ & $A_{5}$ & $[2,1]$ &yes & $[3,2]$ & $2$ \\
$[168,1]$ & $L_{3}(2)\cong L_{2}(7)$ & $[2,1]$ &yes & $[3,2]$ & $2$\\
$[360,1]$ & $A_{6}$ & $[4,2]\cong 2^{2}$ &$no:\,u=2$ & $[5,2],[ 2 ]$ & $2\times 3$ \\
$[504,1]$ & $L_{2}(8)$ & $[3,1]$ & yes& $[2,7]$ & $1$\\
$[660,1]$  & $L_{2}(11)$ & $[2,1]$ & yes & $[11,5]$ & $2$\\
$[960,1]$ & $2^{4}.A_{5}$ & $[24,12]\cong S_{4}$ & yes  & $[3,2]$ & $4^{2}\times 2$\\
$[960,2]$ & $2^{4}.A_{5}$ & $[2,1]$ & yes & $[2,3]$ & $2^{2}$\\
$[1092,1]$ & $L_{2}(13)$ & $[2,1]$ & yes & $[7,2]$ & $2$\\
$[1344,1]$ & $L_{3}(2) 2^{3}$ & $[2,1]$ & yes & $[7,3]$ & $2^{2}$\\
$[1344,2]$ & $L_{3}(2)N 2^{3}$ & $[2,1]$ & yes & $[7,3]$ & $2$\\
$[1920,4]$ & $A_{5} 2^{1} E 2^{4}$ & $[4,2]\cong 2^{2}$ &$no:\,u=2$ & $[3,2],[2,3]$ & $2$\\% $2^{5}.A_{5}$
$[2448,1]$  & $L_{2}(17)$ & $[2,1]$ &yes & $[3,2]$ & $2$\\
$[2520,1]$ & $A_{7}$ & $[2,1]$ &yes & $[7,3]$ & $2\times 3$\\
$[3000,1]$ & $A5 2^{1} 5^{2}$ & $[4,1]\cong C_{4}$ &yes & $[5,2]$ & $5$ \\  %$5-local=5$\\
$[3420,1]$ & $L_{2}(19)$ & $[2,1]$ &yes & $[3,2]$ & $2$\\
$[3600,1]$ & $A_{5}\times A_{5}$ & $[8,3]$ &yes & $[2,3]$ & $2^{2}$\\
$[4080,1]$ & $L_{2}(16)$ & $[4,1]$ &yes & $[3, 2]$ & $1$\\
%Mult=1 because 4080=5*3*17*2^4 and holt and plesken do not contain any
%central extension of L2(16) by 2,3,5 or 17
$[4860,1]$ & $A_{5} 3^{4'}$ & $[12,4]$ &yes & $[2,3]$ & $2\times 3$  \\ % $3-local=3$\\
$[4860,2]$ & $A_{5} N 3^{4'}$ & $[6,2]\cong C_{6}$ & yes & $[2,3]$ & $2$ \\ % $3-local=1$\\
$[5616,1]$ & $L_{3}(3)$ & $[2,1]$ & yes & $[13,3]$ & $1$\\
%Mult=1 because 5616=3^3*2^4*13 and holt and plesken do not contain any
%central extension of L3(3) by 2,3 or 13
$[5760,1]$ & $A_{6} 2^{4}$ & $[4,2]$ & yes& $[3,2]$ & $4\times 2$ \\ % $2-local=4\times 2$\\% page 203 of Holt_Plesken
$[6048,1]$ & $U_{3}(3)$ & $[2,1]$ & yes & $[7,3]$ & $1$\\
$[6072,1]$ & $L_{2}(23)$ & $[2,1]$ &yes & $[23,11]$ & $2$\\
$[7500,1]$ & $A_{5} 5^{3}$ & $[40,12]\cong$ & yes & $[3,2]$ & $2\times 5$ \\ % $5-local=5$\\
 & & $5\times 2:Aut_{5}$ & & &\\
$[7500,2]$ & $A_{5} N 5^{3}$ & $[10,1]\cong D_{10}$ &yes & $[3,2]$ & $2$ \\ % $5-local=1$\\
$[7800,1]$ & $L_{2}(25)$ & $[4,2]\cong 2^{2}$ & $no:\,u=2$ & $[13,2],[3,2],[2]$ & $2$ \\ % 7800=2^3*5^2*3*13
$[7920,1]$ & $M_{11}$ & $[1,1]$ &yes & $--$ & $1$\\ %7920=16*5*9*11
$[9720,3]$ & $A_{5} 2^{1} 3^{4}$ & $[8,3]\cong D_{8}$ &$no:\,u=2$ & $[2,3],[3,2]$ & $3^{2}$ \\ % $3-local=3^{2}$\\
% yes u=2 is the smallest : see /root/gap/gap4r2/pkg/extperf/lib.g
$[10080,1]$ & $A_{5} \times L_{3}(2)$ & $[4,2]$ & yes & $[7,2,3]$ & $2^{2}$?\\
%2*2= mult =mult a5*mult l3(2) ???
$[10752,1]$ & $L_{3}(2)2^{3}.2^{3} $ & $[24,12]$ & yes & $[7,3]$ & $2^{3}$\\ % pg 175->181 Holt_Plesken
$[10752,2]$ & $L_{3}(2) 2^{3} \times N 2^{3}$ & $[8,3]$ &yes & $[7,3]$ & $2^{2}$\\
$[10752,3]$ & $L_{3}(2) 2^{3} A 2^{3}$ & $[4,2]$ & yes & $[7,3]$ & $2^{2}$\\
$[10752,4]$ & $L_{3}(2) N 2^{3} A 2^{3}$ & $[4,2]$ &yes & $[7,3]$ & $2$\\
$[10752,5]$ & $L_{3}(2) 2^{3} \times 2^{3'}$ & $[8,3]$ &yes & $[7,3]$ & $2^{4}$\\
$[10752,6]$ & $L_{3}(2) 2^{3} \times N 2^{3'}$ & $[4,2]$ &yes & $[7,3]$ & $2^{2}$\\
$[10752,7]$ & $L_{3}(2) N 2^{3} \times N 2^{3'}$ & $[8,3]$ &yes & $[7,3]$ & $4$\\
$[10752,8]$ & $L_{3}(2) 2^{3} E 2^{3'}$ & $[2,1]$ &yes & $[7,3]$ & $2^{3}$\\
$[10752,9]$ & $L_{3}(2) N 2^{3} E 2^{3'}$ & $[2,1]$ &yes & $[7,3]$ & $2$\\
$[11520,4]$ & $A_{6} 2^{1} E 2^{4}$ & $[2,1]$ & yes & $[3,2]$ & $2-local=1$\\ % pg 205
$[12180,1]$ & $L_{2}(29)$ & $[2,1]$ &yes & $[5,2]$ & $2$ \\
$[14520,1]$ & $A_{5} 2^{1} 11^{2}$ & $[5,1]$ &yes & $[2,3]$ & $11-local=11$\\ % pg 166
$[14880,1]$ & $L_{2}(31)$ & $[2,1]$ &yes & $[3,2]$ & $2$\\
\end{tabular}\label{perf15360}
%%%%%%%%%%%%%%%%%%%%%%%%%%%%%%%%%%%%%%%%%%%%%%%%%%%%%%%

\begin{minipage}{10cm}\vspace{-2.5cm}

%\subsubsection{Tables 2}
%\pagestyle{empty}
\begin{tabular}{ c|c|c|c|c|c }
Perf. Id & Structure & $Out(P)$ & Split & Smallest sup. in & Mult \\% for the Schur multiplier
\hline\hline

%gap> List([2,3,4,10,11,22,23,24,25,36,37],x->outz([30720,x]));time;
%[ [ [ 64, 138 ], "z=", [ 1, 1 ] ], [ [ 64, 138 ], "z=", [ 1, 1 ] ],
%  [ [ 48, 48 ], "z=", [ 1, 1 ] ], [ [ 192, 1493 ], "z=", [ 1, 1 ] ],
%  [ [ 16, 13 ], "z=", [ 1, 1 ] ], [ [ 4, 2 ], "z=", [ 1, 1 ] ],
%  [ [ 4, 2 ], "z=", [ 1, 1 ] ], [ [ 4, 2 ], "z=", [ 1, 1 ] ],
%  [ [ 4, 2 ], "z=", [ 1, 1 ] ], [ [ 2, 1 ], "z=", [ 1, 1 ] ],
%  [ [ 2, 1 ], "z=", [ 1, 1 ] ] ]

$[15360,3]$ & $A_{5} 2^{4} \times 2^{4}$ & $A\Sigma L(2,4)$* &yes & $<3,2>$ & $$ \\ % * nonsolvable group of order 5760
$[15360,4]$ & $$ & $[24,12]$ & $u=2$  & $<2,3>$ & $$\\  %  yes  % $<,>^{\pi}$
$[15360,5]$ & $$ & $[24,12]$ & yes & $<2,3>$ & $$\\ % \begin{verbatim{ text \end{verbatim}
$[15360,6]$ & $$ & $[12,4]$ & yes & $<2,3>$ & $$\\
$[15360,7]$ & $$ & $[4,2]$ & yes  & $<2,3>$ & $$\\
$[16464,1]$ & $ L_{3}(2) 2^{1} 7^{2}$ & $[6,2]$ & yes  & $<7,2,3>$ & $$\\  %   L3(2) 2^1 7^2
$[20160,4]$ & $A_{8}$ & $[2,1]$ &yes  & $<7,3>$ & $$\\
$[20160,5]$ & $L_{3}(4)$ & $[12,4]$ &yes   & $<2,3>$ & $$\\ % also $<7,3>$
$[21600,1]$ & $A_{5}\times A_{6}$ & $[8,5]\cong 2^{3}$ &no $u=2$  & $<2,3>$ & $$ \\
$[25308,1]$ & $ L_{2}(37)$ & $[2,1]$ & yes  & $<2,3>$ & $$\\
$[25920,1]$ & $U_{4}(2)$ & $[2,1]$ &yes   & $<3,2>$ & $$\\
$[28224,1]$ & $L_{3}(2)\times L_{3}(2)$ & $[8,3]\cong D_{8}$ & yes & $<7,3>$ & $$\\
$[29120,1]$ & $ Sz(8)$ & $[3,1]$ &yes   & $<2,7>$ & $$\\
$[29160,4]$ & $A_{6} 3^{4'}$ & $[72,40]$ &  u=2 & $<5,2>$ & $2\times 3^{3}$\\
$[30240,1]$ & $A_{5}\times L_{2}(8)$ & $[6,2]$ & yes & $<2,3,7>$ & $$\\ % [6,2]\cong C_{6}
$[30720,2]$ & $ A_{5} ( 2^{4} E 2^{1} E 2^{4} ) A$ & $[64,138]$ &$u=2$ & $<2,3>$ & $$\\
$[30720,3]$ & $$ & $[64,138]$ & $u=2$ & $<2,3>$ & $$ \\
$[30720,4]$ & $$ & $[48,48]$ & $u=2$  & $<2,3>$ & $$\\
$[30720,10]$ & $$ & $[192,1493]$ & $u=2$ & $<2,3>$ & $$\\ % is a subdirect product
$[30720,11]$ & $$ & $[16,13]$ &$u=2$  & $<2,3>$ & $$ \\ % is a subdirect product
$[30720,22]$ & $$ & $[4,2]$ &$u=4$   & $<5,2>$ & $$\\
$[30720,23]$ & $$ & $[4,2]$ & $u=2$  & $<2,3>$ & $$\\  % is a subdirect product
$[30720,24]$ & $$ & $[4,2]$ & yes  & $<2,3>$ & $$\\
$[30720,25]$ & $$ & $[4,2]$ & yes  & $<2,3>$ & $$\\
$[30720,36]$ & $$ & $[2,1]$ & yes  & $<2,3>$ & $$\\
$[30720,37]$ & $$ & $[2,1]$ & yes & $<2,3>$ & $$ \\
$[32256,1]$ & $L_{2}(8) 2^{6}$ & $[168,43]$ & yes  & $<3,2>$ & $$\\
$[32256,2]$ & $L_{2}(8) N 2^{6}$ & $[24,13]$ & yes  & $<2,7>$ & $$\\
$[32736,1]$ & $ L_{2}(32)$ & $[5,1]$ &yes   & $<2,31>$ & $$\\
$[34440,1]$ & $L_{2}(41)$ & $[2,1]$ & yes  & $<3,2>$ & $$\\
$[39600,1]$ & $A_{5}\times L_{2}(11)$ & $[4,2]$ & $yes$  & $<11,2,3,5>$ & $$\\
%$[39732,1]$ & $ L_{2}(43)$ & $[2,1]$ &   & $<,>$ & $$\\
%$[40320,2]$ & $A_{7} 2^{4}$ & $[1,1]$ &  & $<,>$ & $$ \\
%$[43008,1]$ & $L_{3}(2) 2^{1}(2^{3} E 2^{1} E 2^{3})$ & $[1,1]$ &   & $<,>$ & $$\\
%$[43008,2]$ & $$ & $[2,1]$ &   & $<,>$ & $$\\
%$[43008,3]$ & $$ & $[2,1]$ &   & $<,>$ & $$\\
%$[43008,25]$ & $$ & $[2,1]$ &   & $<,>$ & $$\\
%$[43320,1]$ & $A_{5} 2^{1} 19^{2}$ & $[9,1]$ &yes   & $<3,2>+$ & $$\\  %%%  A5 2^1 19^2
%$[43740,1]$ & $A_{5} 3^{6}$ & $[16,8]$ &yes   & $<2,3>+$ & $$\\
%$[48000,1]$ & $A_{5}\#2^{5} 5^{2}[1]$ & $[48,30]$ &  & $<,>$ & $$ \\
%$[48000,2]$ & $$ & $[16,6]$ &   & $<,>$ & $$\\
%$[48000,3]$ & $$ & $[4,1]$ &   & $<,>$ & $$\\
%$[51888,1]$ & $L_{2}(47)$ & $[2,1]$ &   & $<,>$ & $$\\
%$[57600,1]$ & $A_{5}\times A_{5}\#2^{4}[1]$ & $[48,48]$ &   & $<,>$ & $$\\
%$[57600,2]$ & $$ & $[4,2]$ &   & $<,>$ & $$\\
%$[57624,1]$ & $ L_{3}(2) 7^{3}$ & $[12,5]$ &   & $<,>$ & $$\\
%$[57624,2]$ & $L_{3}(2) N 7^{3}$ & $[2,1]$ &  & $<,>$ & $$ \\
%$[58320,2]$ & $ A_{6} 2^{1} 3^{4}$ & $[16,8]$ &   & $<,>$ & $3^{3}$\\
%$[58800,1]$ & $ L_{2}(49)$ & $[4,2]$ & u=2  & $<5,2>$ & $$\\
%$[60480,3]$ & $L_{3}(2)\times A_{6}$ & $[8,5]$ & no   & $<,>$ & $$\\

\end{tabular}
\end{minipage}

%%%%%%%%%%%%%%%%%%%%%%%%%%%%%%%%%%%%%%%%%%%%%%
\subsubsection{How to read these tables}
\pagestyle{plain}
Each line of the table corresponds to a perfect group $P$ whose centre is $1$ and provides information
about $Aut(P)$. If $\tilde{P}$ is a perfect group with a non trivial centre then $\tilde{P}/Z(\tilde{P})$ has a
trivial centre and $Aut(\tilde{P})$ is a subgroup of $Aut(\tilde{P}/Z(\tilde{P}))$ (see the Appendix).
Hence, the informations contained in this table about perfect groups with trivial centre holds for every
perfect group.

% that is a central extension of $P$
For a given perfect group $P$, "Perf. Id" and "Strucure" refer to its Id and its description in the Library of
perfect groups (\cite{Holt_Plesken}) whereas the Id in column $Out(P)$ refers to the Id of $Out(P)$ in
the Small Groups Library (if $|Out(P)|\leq 2000$). Column "Split" tells whether $Aut(P)$ splits over $Inn(P)$.
 If it is not the case then let $S$ be a supplement to $Inn(P)$ in $Aut(P)$ such that $u=|S\cap Inn(P)|$ is
 minimal. The value of $u$ is given and next column ("Smallest sup. in") gives a sequence $\vec{p}$ of prime
 numbers for which such a minimal supplement $S$ can be found in an element of $N_{\vec{p}}$.

The column "Mult" reflects the informations that can be obtained from (\cite{Holt_Plesken}) about the
\emph{Schur multiplier} of a $P$. For every subgroup $\tilde{Z}$ of the abelian group $Mult(P)$,
there is a perfect group $\tilde{P}$ such that $\tilde{P}/Z(\tilde{P})\cong P$ and by the results of the Appendix
(chapter \ref{appendix} ), the informations about every such $\tilde{P}$ can be deduced from the informations given
about $P$.
%%%%%%%%%%%%%%%%%%%%%%
%%%%%%%%%%%%%%%%%%%%%%%%%%%%%%%%%%%%%%%%%%%%%%%%%

%%%%%%%%%%%%%%%%%%%%%%%%%%%%%%%%%%%%%%%%%%%%%%%%%%%%%%%%%%%ù
\chapter{$Z\phi$-classes and classification theorems}\label{classiftheo}
%%MMMM FOR MORE READIBILITY USE THE "EQUATION" ENVIRONEMENT +LABEL (NUMBER)
\index{In this chapter, use \BEGIN{equation}}
\subsubsection{Introduction}
%\begin{figure}[ht]
%  %\begin{center}
%\includegraphics[width=4cm,height=7cm]{uz1.eps}
%\caption{}
%\label{fig:ccircuit}
%  %\end{center}
% \end{figure}
%%%%%%%%%%%%%%%%%%%%%%%%%%%%%%%%%
%macros/latex/contrib/other/misc/wrapfig.sty has syntax:
%\begin{wrapfigure}[height of figure in lines]{l|r}[overhang]{width}
%  figure, caption, etc.
%\end{wrapfigure}
%
%\setlength{\intextsep}{0pt}
% \begin{wrapfigure}[6]{l}{0.2\linewidth} % l for left r for right
%  \mbox{\input{latt1.latex}} %%% IT IS THE WRONG GRAPHIC
% \end{wrapfigure}
%%%%%%%%%%%%%%%%%%%%%%%%%%%%%%%%%%
%%MMMM see figure \ref{fig:uz1}\\

From the second half of the 19th century up to nowadays, the history of group theory
is full of incorrect group classifications. Some historical examples can be found in
\cite{O_Brien_Short}. In large catalogues of
groups where many groups share many common properties, it is very frequent either to forget some
 case or to count it twice. The best mathematicians have faced such errors. This was largely due to
hand computations and nowadays, computer implementations of these classification techniques help
avoiding such problems. Nevertheless programming mistakes are still possible as well as a mathematical
error in the method itself. Therefore it is extremely important to compare new groups catalogues with
previously existing partial catalogues obtained by other methods. Every classification method
should provide such cross-checks.
In this chapter we expose for some cases, an alternative to our general classification method.
It allows us to enumerate group extensions in a different way and provides the required cross-checks.
Even if this method does not cover every group extension, it covers the overwhelming majority
of the groups involved to classify  the nonsolvable groups of order less than 58,320.

\section{$Z\phi$-classes}
%MMMM (- Note that if Z=1, the definition gives the cns of isomorphism for extensions of
%centerless groups. This generalizes the case z=1. )
\index{cite-Laue z=1 to show the CNS for z=1 }
%\\
%\\
%\\
%\\
%macros/latex/contrib/other/misc/wrapfig.sty has syntax:
%\begin{wrapfigure}[height of figure in lines]{l|r}[overhang]{width}
%  figure, caption, etc.
%\end{wrapfigure}
%
%\setlength{\intextsep}{0pt}
% \begin{wrapfigure}[6]{l}{0.2\linewidth} % l for left r for right
%  \mbox{\input{latt1.latex}} %%% IT IS THE WRONG GRAPHIC
% \end{wrapfigure}
%
%
%
Let $G$ be a group. Its center $Z(G)$ is a characteristic subgroup and every
automorphism $\omega$ of $G$ induces an automorphism of $Z(G)$. Since conjugations
in $G$ fix every element of $Z(G)$, for $\omega\in Aut(G)$, all the elements of the coset
$Inn(G)\omega$ induce the
 same automorphism on $Z(G)$ and this defines a canonical action of $Out(G)$ on $Z(G)$.
%\textbf{Definition }Let $Z=Z(G)$, let $S$ be an extension of $Z$ and let
%$\phi$ be a homomorphism from $S/Z$ into $Out(G)$ . The pair $(S,\phi)$ is called a
%\emph{$Z\phi$-extension} iff $z^{s}=z^{\tilde\phi(s)}$ for every $s\in S$
%(where $\tilde \phi:S\rightarrow Out(G)$ is the homomorphism $\phi$ extended to $S$ by the
%rule $\tilde\phi(s)=\phi(sZ)$). The action of $S$ on $Z(G)$ by
%conjugation must thus coincide with the action of $\tilde\phi(S)\subseteq Out(G)$ on
%$Z(G)$. \\
%For $i=1,2$, let $(S_{i},\phi_{i})$ be  $Z\phi$-extensions.
%We say that pair $(S_{1},\phi_{1})$ and pair $(S_{2},\phi_{2})$ are \emph{equivalent}
% if there is an isomorphism $j:S_{1}\rightarrow S_{2}$ and an element $\pi\in Out(G)$
% such that
%\begin{enumerate}
%\item for $z\in Z(G)$, $j$ maps $z$ to $z^{\pi}$
%\item for $s\in S_{1}$ we have $\tilde\phi_{2}(j(s))=\tilde\phi_{1}^{\pi}(s)$
%\end{enumerate}

\textbf{Definition.} Write $Z=Z(G)$, let $S$ be an extension of $Z$ and let
$\phi$ be a homomorphism from $S$ into $Out(G)$ with $Z\subset Ker\phi$
 (or equivalently a homomorphism from $S/Z$ into $Out(G)$).
 We say that the pair $(S,\phi)$ is a \emph{\textbf{$Z\phi$-extension}} if and only if
 $z^{s}=z^{\phi(s)}$ for every $s\in S$ and every $z\in Z$. The conjugation action of every $s\in S$
 on $Z(G)$ must thus coincide with the action of $\phi(s)\subseteq Out(G)$ on $Z(G)$.
If $\phi(S)=L/Inn(G)$ (where  $Inn(G)\leq L\leq Aut(G)$),
we say that \emph{$(S,\phi)$ is associated with $L/Inn(G)$}. % (or more simply, with $L$).
Finally let us precise that since $\phi(s)$ is a coset of $Inn(G)$ in
$Aut(G)$, for $\pi\in Aut(G)$ we use the notation $\pi^{-1}\phi(s)\pi:=\phi^{\pi}(s)$.\

For $i=1,2$, let $(S_{i},\phi_{i})$ be  $Z\phi$-extensions.
We say that the pair $(S_{1},\phi_{1})$ and $(S_{2},\phi_{2})$ are \emph{\textbf{equivalent}}
 if there is an isomorphism $j:S_{1}\rightarrow S_{2}$ and an element $\pi\in Aut(G)$
 such that
\begin{enumerate}
\item for every $z\in Z(G)$, $j$ maps $z$ to $j(z)=z^{\pi}$
\item for every $s\in S_{1}$ we have $\phi_{2}(j(s))=\pi^{-1}\phi_{1}(s)\pi:=\phi_{1}^{\pi}(s)$.
\end{enumerate}

If ${\pi _{2}}$ is another representative of $Inn(G)\pi$, then $z^{\pi}=z^{\pi _{2}}$
and $\pi^{-1}\phi_{1}(s)\pi=\pi _{2}^{-1}\phi_{1}(s)\pi _{2}$ ; therefore
pairs that are equivalent through $j$ and $\pi$ are also equivalent through $j$ and
$\pi _{2}$. Hence, we can say that pairs are equivalent if there exists an isomorphism
$j$ and $\tilde{\pi}\in Out(G)$ that satisfies conditions 1 and 2.
It is not difficult to check that we have defined an equivalence relation. Indeed, suppose that
$(S_{1},\phi_{1})$ is equivalent to $(S_{2},\phi_{2})$ through $j$
 and $\pi$ and that $(S_{3},\phi_{3})$ is equivalent to $(S_{2},\phi_{2})$ through $j'$
 and $\pi '$ then $(S_{3},\phi_{3})$ is also equivalent to $(S_{1},\phi_{1})$ through $j'j$
 and $\pi '\pi$. The relation is symmetrical and a pair is equivalent to itself.
This notion must not be confused with the notion of equivalent extensions in Schreier's
general theory (see paragraph \ref{equivext})% \cite{Zassenhaus} §49 or\cite{Hall} page 221).
We call $Z\phi$-classes, the equivalence classes of this equivalence relation.

It follows from this definition that :
\begin{enumerate}
\item if for $i=1,2$, $H_{i}\cong S_{i}/Z(G)$ and if the $(S_{i},\phi _{i})$ are equivalent $Z\phi$-extensions
of $Z(G)$, then the groups $H_{1}$ and $H_{2}$ are isomorphic ;
% the factor groups $S_{1}/Z(G)$ and $S_{2}/Z(G)$ are isomorphic
\item if for $i=1,2$, the $(S_{i},\phi _{i})$ are equivalent $Z\phi$-extensions of $Z(G)$,
then $\phi_{1}(S)$ and $\phi_{2}(S)$ are conjugate subgroups of $Out(G)$ ;
\item conversely, if $(S_{1},\phi _{1})$ is a $Z\phi$-extension associated with a subgroup
 $V$ of $Out(G)$, then for every $\pi\in Out(G)$, $(S_{1},\phi ^{\pi} _{1})$ is a $Z\phi$-extension
 associated with $V^{\pi}$. Hence, in order to list representatives of $Z\phi$-classes, it
 is sufficient to construct only $Z\phi$-extensions associated with $V_{k}\leq Out(G)$ where
 $\{V_{k}\}$ is a set of representatives for the conjugacy classes of subgroups in $Out(G)$.
\end{enumerate}

% We call the equivalence classes of this relation, the $Z\phi$-classes.

Let us remark that for every homomorphism $\widetilde{\phi}:H\rightarrow Out(G)$, there exists at
least one $Z\phi$-extension $(S,\phi)$ such that $S/Z(G)\cong H$ and such that $\phi$ is the
homomorphism from $Z(G)s_{h}\to \widetilde{\phi (h)}$ corresponding to a given isomomorphism $s_h\to h$ from
$S/Z$ onto $H$. It suffices to consider the semi-direct product $S=H\ltimes_{\phi '}Z$ where
$\widetilde{\phi '}:H\to Aut(Z(G))$ is the action induced by $\phi$ on $Z(G)$.

\subsection{Applications of $Z\phi$-classes}
%for the next theorem G and H could be infinite but since the set of extensions could
%be then infinite, we must speak of bijection between sets instead of same number of elements in both sets
%The notions of $Z\phi$-classes of are fundamental to formulate classification theorem like Theorem \ref{uzsplit}.
%Cite Mac-Lane's work.
An example of use of these concepts is the following result which will be proved later as a
consequence of Proposition \ref{uzprime}.
\begin{theorem}\label{autsplit}Assume that $Aut(G)$ is a split extension of $Inn(G)$.
There is a bijection between the $G$-isomorphism classes of extensions of $G$ by $F$
and the $Z\phi$-classes of $Z\phi$-extensions of $Z(G)$ by $F$.
\end{theorem}
%%%This Theorem will be proved later with Proposition \ref{uzprime}
%%%%% BE CAREFUL THE TERM EQUIVALENCE CLASSES COULD NOT BE USED . IT COULD BE CONFUSED
%% WITH THE NOTION OF "EQUIVALENT EXTENSIONS" OF SCHREIER.
%.NOOOO: I HAVE ANOTHER PROOF, MUCH SIMPLER WHERE WE DO NOT
% HAVE TO FIX $\Phi$. SEE THEOREM $(Z(G),Out(G))$ are coprime (=$U=U\times Z$).
% DONC RENVOYER A LA DEMO DE $U=U\times Z$.

%Hence, under the hypothesis of Theorem \ref{autsplit}, the extensions of $G$ by a group
%$F$ are in one-to-one correspondence with the $Z\phi$-classes of $Z\phi$-extensions
%$(S,\phi)$ such that $S/Z\cong F$.
  %\subsubsection*{an algorithm to determine $(Z,\phi)$ classes}
\index{label autsplit multiply defined}

Since equivalent $Z\phi$-extensions correspond to isomorphic extensions of $Z(G)$
(but not conversely), a natural step toward an enumeration of $Z\phi$-classes
is to enumerate the classes associated with a fixed isomorphism type. % In some cases
%Hence, rather than aiming to classify the isomorphism types of the extensions of $Z(G)$
%it is some times easier to use available groups catalogues.
Groups catalogues can be used for such tasks.  For instance, if we aim at
classifying extensions of a group $G$ by a solvable group $H$, the $Z\phi$-extensions of $Z(G)$
by $H$ are solvable too and we need a catalogue of solvable groups. Such a catalogue
has been developed by Bettina Eick and Hans Ulrich Besche for groups up to order $2000$,
 and is available on the computer algebra systems $GAP$ and $Magma$.
%ù (in $GAP$, excepting $1024$ %since 1997 and recently
%and more recently in Magma up to order 1000, excepting 512 and 768).
 We have used it to compute tables of $Z\phi$-classes.

Let $S$ be a group picked up in a catalogue ; to identify $S$ with an extension of
$Z(G)$, we need to find a normal subgroup $N\unlhd S$ and an isomorphism $i$ from $Z(G)$ onto $N$.
The problem is that there are several
ways to identify $N$ with $Z(G)$ which can lead to non-equivalent $Z\phi$-extensions.
There is a unique identification if and only if $Aut(Z(G))=1$ thus when $|Z(G)|=1$ or $2$.
 We discuss this case in the next paragraph.%%( and thus $N$ is central in $S$).

\subsubsection{$|Z(G)|\leq 2$ and $|\phi(S)|\leq 2$}\label{tableout=2}
Numerous examples of groups $G$ such that $Z(G)=1$ or $2$ can be found in the table of
section \ref{tableout}. The group $Out(G)$ has a subgroup $V$ of order $2$ if and
only if $|Out(G)|$ is even. Infinite families of such groups $G$ with either $|Z(G)|=2$ or $|Z(G)|=1$ exist.
The alternating groups $A_{n}$ and their twice bigger perfect central extensions $2.A_{n}$, the groups
$SL_{n}(q)$ and $PSL_{n}(q)$ for $q$ an odd prime power and $gcd(n,q-1)=2$ are
examples of such families.

Given such a group $G$, a group $S$, and an order $2$ subgroup $V\leq Out(G)$, we want
to enumerate the $Z\phi$-classes of $Z\phi$-extensions $(S,\phi)$ where $\phi (S)=V$ or $\phi (S)=1$.
 We show how to proceed and we explain the meaning of the numbers $\mathcal{O}_{Z}$,
$\mathcal{O}_{K}$ ,$\mathcal{O}_{Z,K}$ and "grps" in our tables.

\begin{enumerate}
\item If $Z(G)=1$ and $\phi$ is the trivial mapping ($\phi (S)=1$) then for each $S$, there is
just one such $Z\phi$-extension. The number of  such $Z\phi$-classes for a fixed order $n$
is just the number of groups of order n ("grps" in the table).

\item If $Z(G)=2$ and $\phi$ is the trivial mapping ($\phi (S)=1$) then each $Z\phi$-extension
is fully defined by the order-2 normal subgroup of $S$  % of order $2$ ($N that
that we identify with $Z(G)$. Observe that $N\cong Z(G)$ is central in $S$.
Let $\mathcal{P}_{Z}$ be the set of central subgroups
of order $2$ in $S$. Two groups $N_{1}$ and $N_{2}$ correspond to
 equivalent $Z\phi$-extensions if and only if there exists an automorphism of $S$ that
 maps $N_{1}$ to  $N_{2}$ and hence the number of such $Z\phi$-classes is the number
 $\mathcal{O}_{Z}$ of orbits in $\mathcal{P}_{Z}$ under the action of $Aut(S)$.

\item If $Z(G)=1$ and $\phi (S)=V$, the homomorphism $\phi$ is completely described by
$K=Ker \phi$, an index 2 subgroup of $S$. If $\mathcal{P}_{K}$ is the set of index-2 subgroups
of $S$, then the number of such $Z\phi$-classes is the number $\mathcal{O}_{K}$ of orbits
in $\mathcal{P}_{K}$ under the action of $Aut(S)$.

\item If $Z(G)=2$ and $\phi (S)=V$, the $Z\phi$-extension is fully described by
a pair $[N,K]$ where
$K=Ker \phi$  (so $[S:K]=2$) and where $N\cong Z(G)$ is an order 2 normal subgroup of $S$
 such that $N\subseteq K$. %%%viewed as $Z(G)$.
According to the definition of equivalence for $Z\phi$ extensions, such pairs  $[N_{1},K_{1}]$
and $[N_{2},K_{2}]$ %%% ( $Z(G)\cong N_{i}\unlhd S$ and $[S:K_{i}]=2$)
define equivalent $Z\phi$-extensions if and only if there exists an automorphism of $S$ that
 maps $N_{1}$ to  $N_{2}$ and $K_{1}$ onto $K_{2}$. Hence, the number of such
 $Z\phi$-classes is just the number of orbits $\mathcal{O}_{Z,K}$ of such pairs under
 the action of $Aut(S)$.

\end{enumerate}
\subsection{There are more than $8.10^{6}$ nonsolvable groups of order at most $15360$}\label{15360}
Each nonsolvable group of order up to $n$ is an extension of a perfect group $G$ by a
solvable group $H$ of order $|H|\leq n/|G|$.
It appears in this work that the majority of nonsolvable groups up to some fixed order $n$,
 are extensions of perfect groups $G$ such that $Z(G)=1$ or $2$, associated with a subgroup
 of order $2$ of $Out(G)$. This is due to the fact that up to a fixed order, %$n$,
 the extensions of small perfect groups seem to predominate because much more groups $H$
 are considered. For a lot of small perfect groups $G$,
$|Z(G)|=1$ or $2$ and $|Out(G)|=2$ (see Table \ref{perf15360}).
 Hence, the number of such extensions (given by the orbits described in the previous subsection)
 is a very good estimate of the number of nonsolvable groups up to some given order.
  We emphasize here, the number of such orbits for orders for which a lot of groups exist.

\begin{flushleft}
{\footnotesize
\begin{tabular}{||c||c|c|c|c||c||c|c|c|c|}
\hline
$|S|$ &grps& $\mathcal{O}_{K}$ & $\mathcal{O}_{Z}$ & $\mathcal{O}_{Z,K}$ &
$|S|$ &grps& $\mathcal{O}_{K}$ & $\mathcal{O}_{Z}$ & $\mathcal{O}_{Z,K}$   \\
\hline\hline
4 & 2 &2 &2&2 & 6 & 2 &2 &1&0\\
8 & 5 &7 &6&8 & 12 & 5 &5 &4&4\\
16 & 14 &28 &21&42 & 24 & 15 &24 &16&24\\
32 & 51 &144 &95&286 & 48 & 52 &118 &72&166\\
64 & 267 &1,120 &649&3,026&  96 & 231 &797 &422&1,574\\
128 &2,328  &16,996 &8,308& 72,010& 192 & 1,543 &8,551 &3,914&23,798\\
256$^{*}$ & 56,092 &1,027,380 &337,956& 6,856,498& 384  & 20169 & ?& ?& ?\\
\hline
\end{tabular} \label{256sl2p}

\medskip
{\scriptsize $^{*}$ Update 2026: we would like to thank Max Horn for his updates and recalculations of some 2002 values in line 256 and 128. For a 2026 reader, his website \url{https://groups.quendi.de/} offers 25 years of countings for group extensions.}
}
\end{flushleft}
%%%%%%%%%%%%%%%%%%%%%%%%%%%
%\begin{flushleft}
%{\footnotesize
%\begin{tabular}{||c||c|c|c|c|c||c||c|c|c|c|c|}
%\hline
%$|S|$ &grps& $\mathcal{O}_{K}$ & $\mathcal{O}_{Z}$ & $\mathcal{O}_{Z,K}$ & $\Sigma$ &
%$|S|$ &grps& $\mathcal{O}_{K}$ & $\mathcal{O}_{Z}$ & $\mathcal{O}_{Z,K}$ & $\Sigma$ \\
%\hline\hline
%4 & 2 &2 &2&2&8 & 6 & 2 &2 &1&0&5\\
%8 & 5 &7 &6&8&26 & 12 & 5 &5 &4&4&18\\
%16 & 14 &28 &21&42&105 & 24 & 15 &24 &16&24&79\\
%32 & 51 &144 &95&286&576 & 48 & 52 &118 &72&166&408\\
%64 & 267 &1,120 &649&3,026& & 96 & 231 &797 &422&1,574&\\
%128 &2,328  &16,996 &8,308& 72,010& & 192 & 1,543 &8,551 &3,914&23,798&\\
%256$^{*}$ & 56,092 &1,027,380 &337,956& 6,856,498& &384  & 20169 & ?& ?& ?& ?\\
%\hline
%\end{tabular}
%}
%\end{flushleft}
%%%%%%%%%%%%%%%%%%%%%%%
It turns out that almost all nonsolvable groups of order at most $58,320$ are extensions that can be
classified using $Z\phi$-classes. Here is an illustration of how accurate is the estimation of the number of
nonsolvable groups if only $Z\phi$-classes are taken it account.
  Assume we want to count the number of nonsolvable groups
of order less than $23,040=60*384$, since $384$ is the first order for which we do not
know the number of orbits ($60$ is the smallest order of a perfect group).

The first perfect group with nontrivial center, for which the number of extensions cannot
be given by $Z\phi$-classes is $SL_2(9)$ (see table \ref{perf15360}) .
% that do not belong to the familly $Sl_2 (p) ($p$ prime), is $G=L_2 (8)$
 Since $|SL_{2}(9)|=720$,  for  extensions $E$ of $SL_2(9)$ by groups $H$ such that $|E|<23,040$ we have
to consider the groups $H$ of order less than $32=23,040/720$. % have to be considered .
Table \ref{tab2a6} will show that there are $1535$
 such extensions, which is negligible compared  to the number ($>7.10^{6}$) of extensions of
 $SL_{2}(5)$  of order less than $23,040$. Nevertheless, the isomorphism problem is much
 harder in the case of $SL_2(9)$ and the small number of extensions is only due to the small number
 of groups $H$ to consider. Indeed there exist only $93$ solvable groups $H$ of order less
than $32$  but $4435$ solvable groups $H$ of order less than $192=23,040/120$ that have to be considered
for the extensions of $SL_{2}(5)$ ($|SL_{2}(5)|=120$).
%%%%%%%% 4435=4441-6

Therefore, $1535$ must rather be compared to the number $k$ of extensions
of $SL_2(p)$ by a group $H$ of order less than 32, and it turns out that $k=154$ and is thus
much smaller than $1535$.

\subsubsection{Computing times}
For computing the automorphism group of the $p$-groups in this table, we have used
 the $GAP$ (\cite{GAP4}) share package "autpgrp" implemented by Bettina Eick. For other groups we have
 used the standard $GAP$ method which is much slower. For instance for the $20,169$ groups of
 order $384$, it should at least compute during more than 2 weeks on our machine and we have not achieved such a
 computation. The enumeration for groups of order 128 took 75 minutes and the enumeration
 for the order 256 took 8 and a half days (on a 800 MHz AMD Duron).
 We have not tried to use the most efficient algorithm to compute these orbits under $Aut(S)$.
Of course  picking up $S$ in a catalogue is much faster than  constructing explicitly all supplements $S$.
%% (see Chapter \ref{quotientsemi}
 %%%%since here we pick $S$ in a catalogue and we do not have to construct it.

%\begin{tabular}{||c||c|c||c||c|c|}
%\hline
%$|S|$ &grps& $\mathcal{O}_{K}+ \mathcal{O}_{Z} + \mathcal{O}_{Z,K}$  &
%$|S|$ &grps& $\mathcal{O}_{K}+ \mathcal{O}_{Z} + \mathcal{O}_{Z,K}$  \\
% $|S|$ &grps& $\mathcal{O}_{K}$ & $\mathcal{O}_{Z}$ & $\mathcal{O}_{Z,K}$ \\
%\hline\hline
% 4 & 2 &2 &2&2&8 & 6 & 2 &2 &1&0&5\\
% 8 & 5 &7 &6&8&26 & 12 & 5 &5 &4&4&18\\
% 16 & 14 &28 &21&42&105 & 24 & 15 &24 &16&24&79\\
% 32 & 51 &144 &95&286&576 & 48 & 52 &118 &72&166&408\\
% 64 & 267 &1,120 &649&3,026& & 96 & 231 &797 &422&1,574&\\
% 128 &2,328  &17,012 &8,307& 72,019& & 192 & 1,543 &8,551 &3,914&23,798&\\
% 256 & 56,092 &1,027,380 &337,943& 6,856,371& &384  & 20169 & ?& ?& ?& ?\\
%\hline
%\end{tabular}

\subsubsection{How to read this table ?}
Let $G$ be a group that satisfies Theorem \ref{autsplit} or its generalizations which will be given
 in the next section (\ref{classifs}) so that  isomorphisms between extensions are reduced to isomorphisms between
extensions of $Z(G)$. In section \ref{classifs} we reconstruct an extension of $G$ by $H$ associated
to a subgroup $L<Out(G)$, from a $Z\phi$-extension of $Z(G)$ by $H$ associated with the same $L$.
Assume now that $Z(G)\leq 2$ and $Out(G)=2$
(or for any order 2 subgroup of $Out(G)$)  as for the groups $SL(2,p)$ or
$PSL(2,p)$. How many extensions of $G$ by a solvable group $F$ of order $128$ do exist ?

%% problem and the isomorphism is completely reduced
If $Z(G)=1$ we have to look in row $128$ of table \ref{256sl2p} which indicates that there are
$2,328$ direct products of $G$ and a group of order $128$ and $17,012$ other extensions.
If $Z(G)=2$ we look in row $256$ since an extension of $Z(G)$ by $F$ has order $256$.
There are $337,956$ central products and $6,856,498$ other extensions
( the classes in $\mathcal{O}_{Z} $ correspond to central products of $G$ and $S$).\\

 \emph{Hence, for p prime ($p\geq 5$)  there are 7,194,454 extensions of $SL(2,p)$  by a group of order
128 and there are 1,083,472 extensions of $PSL(2,p)$ by a group of order 256.}

% We make such tables of $Z\phi$-classes available to anyone that
% asks for it at our address \begin{verbatim} carcher@ulb.ac.be \end{verbatim}

For the groups $G$ that will be described by Theorem \ref{usplitz} $(iii)$, these numbers give an upper bound for the
number of extensions.\\
%%- It gives a cross check for classifications (that construct the whole nonsolvable group $E$
%Moreover a lot of these examples satisfy Theorem \ref{autsplit}.

%%%%%%%%%%%%%%%%%%%%%%%%%%%%%%%%%%%%%%%%%%%%%%%
%\subsection{Extensions of $3A6$}
\section{A classification theorem}\label{classifs}
\index{remove "MMMM" everywhere}
%\chapter{A classification theorem}\label{classifs}
\subsubsection{Introduction}
%\begin{wrapfigure}[12]{l}{0.3\linewidth}
%%\mbox{\input{uz1.eps}}
%\end{wrapfigure}
% The next section shows cases where such a homomorphism is sufficient
%to describe an extension. + UPPER BOUNDS ...
We now describe a general situation where one can associate an extension of $G$ to every
$Z\phi-$ extension of $Z(G)$ and where $G-$isomorphic extensions are in one-to-one
correspondence with the equivalence classes of $Z\phi-$extensions. The results of this
section generalize Theorem \ref{autsplit}.

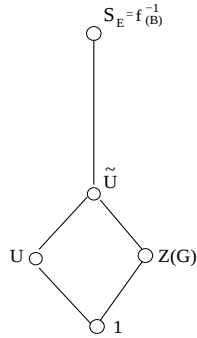
\begin{center}
\includegraphics[width=2.5cm,height=4.375cm]{uz1.eps}
\label{fig:uz1}
\end{center}

Suppose we want to construct an extension $E$ of a group $G$ associated with a subgroup
$L/Inn(G)$ of $Out(G)$. Let $f$ be the action of $E$ on $G$ by conjugation. As seen
in previous sections, we choose a supplement $B$ of $Inn(G)$ in $L$ and so
$S_{E}=\{e\in E\,|\,f(e)\in B  \}$ is a supplement to $G$ in $E$ that contains
$C_{E}(G)$, the centralizer of $G$ in $E$.
 Its intersection $\tilde{U}=S_{E}\cap G$ with $G$ is the the subgroup of $G$ corresponding
 to $Inn(G)\cap B$. Assume that %there exists a supplement $B$ such that $\tilde{U}\subset G$ splits over $Z(G)$ and is equal to
 $\tilde U=U\times Z(G)$ where $U$ is a subgroup of $G$ invariant under the action of $B$.
In this case $U$ is a normal subgroup of $S_{E}$. Let us show
that modulo some identifications, $f$ and $\tilde S=S_{E}/U$  define a $Z\phi$-extension of $Z(G)$.

Let $Z=Z(G)$. Since $(U\times Z)$ is invariant in $S_{E}$, the subgroup $\tilde Z=(U\times Z)/U$
is a normal subgroup of $S_{E}/U$ and for $z\in Z$, the projection $Uz\rightarrow z$
is an isomorphism from $\tilde Z$ onto $Z$. Modulo this isomorphism, $\tilde S$ is an
extension of $Z(G)$ and let us show that $f$ induces a homomorphism $\phi$ from
$\tilde S/\tilde Z$ onto $L/Inn(G)=f(E)/Inn(G)$ such that $(\tilde S,\phi)$ is a $Z\phi$-extension.

The restriction of $f$ to $S_{E}$ is an homomorphism onto $B=f(S_{E})$ and therefore it
induces an homomorphism $\phi$ from $\tilde S=S_{E}/U$ onto $f(S_{E})/f(U)=B/(B\cap Inn(G))$
which can be identified with the subgroup $Inn(G)B/Inn(G)=L/Inn(G)$ of $Out(G)$(by the isomorphism theorems).
Since $f(\tilde Z)\subset Inn(G)$, the kernel of $\phi$ contains $\tilde Z$ and therefore $\phi$ can
also be defined as an homomorphism from  $\tilde S/\tilde Z$ onto $L/Inn(G)$.
 For an $x=Us\in \tilde S$ we have $\phi(x)=Inn(G).f(s)$ (it is independent of the
representative $s$ of $x$ because $f(U)\subset Inn(G)$). By conjugation, $Inn(G)$ fixes
every element $z\in Z(G)$ so that the action of $x$ on $Z$ is
$z^{x}=z^{s}=z^{f(s)}=z^{Inn(G).f(s)}=z^{\phi(x)}$ and this is the condition for
$(\tilde S,\phi)$ to be a $Z\phi$-extension of $Z(G)$ (modulo the identification with
 $\tilde Z$). Now, we prove that :
\begin{enumerate}
\item \emph{It is possible to reconstruct $E$ from $(\tilde S,\phi)$ and $B$ only.}
\item \emph{For every $Z\phi$-extension $(\tilde S,\phi)$ of $Z(G)$ associated with $L$, there exists an
extension $E$ that produces $(\tilde S,\phi)$ in the way described previously}
\end{enumerate}
%\begin{enumerate}
%\item $S_{E}$
%and its action on $G$ are completely described by the triple $(\tilde S,\phi,B)$.
%This fact is crucial because it as been proved in Theorem \ref{semi} that $E$ can be
%builded as a quotient of the semi-direct product $(S_{E})\ltimes_{f}G$, where the quotient
%amounts to identify $U\times Z$ in $G$ and in $S_{E}$. Finally $E$ is completely
%described by $(\tilde S,\phi,B,U)$ and we write $E=E(\tilde S,\phi,B,U)$.
%
%\item For every $Z\phi$-extension $(\tilde S,\phi)$ of $Z(G)$
%(such that $\phi(\tilde S)=L/Inn(G)$), there exists an extension $E=E(\tilde S,\phi,B,U)$
%associated with $L$.
%\end{enumerate}

$(1)$: If we combine the homomorphism $S_{E}\rightarrow S_{E}/U=\tilde S$ with
the homomorphism $f_{S_{E}}:S_{E}\rightarrow B$, we create a third homomorphism
$\gamma:s\rightarrow (Us,f(s))$ from $S_{E}$ into $\tilde S\times B$. Its kernel $Ker\gamma$
is $U\cap C_{E}(G)$, the intersection of the two previous kernels.
But $U\cap C_{E}(G)\subset G\cap C_{E}(G)=Z(G)$ and by our hypothesis $U\cap Z(G)=1$
so that $Ker\gamma=1$. Therefore $\gamma$ is injective
and $S_{E}$ is isomorphic to the subgroup $D=\{(Us,f(s))|s\in S_{E}\}$
of $\tilde S\times B$. Actually, let us show that we could define $D$ as the subgroup
 $\{(x,b) |\,b\in\phi(x)\}$ of $\tilde S\times B$.
Obviously,  for an element $(Us,f(s))$ of $D$, $f(s)\in\phi(Us)=Inn(G)f(s)$.
 Conversely, suppose that for an element $d=(Us_{1},f(s_{2}))$ of $\tilde S\times B$,
 the element $f(s_{2})$ belongs to $\phi(Us_{1})=Inn(G)f(s_{1})$. Then as $f(s_{i})\in B$,
 $f(s_{2})$ belongs to $B\cap Inn(G)f(s_{1})=(B\cap Inn(G))f(s_{1})=f(U)f(s_{1})=f(Us_{1}).$
Thus there is an element $u\in U$ such that $f(s_{2})=f(us_{1})$ and for
$s_{3}=us_{1}$, the element $d$ is equal to $(Us_{3},f(s_{3}))$ and belongs to $D$.
As announced, $S_{E}$ and its action on $G$ are completely defined by $(\tilde S,\phi,B)$.

By Theorem \ref{semi}, $E$ is isomorphic to a quotient of the semi-direct product
$D\ltimes G$. The action of $D$ on $G$ is given by the canonical projection of $D$ on $B$ and
 the quotient amounts to identify $Z\times U$ in $G$ and $\tilde Z\times f(U)$ in $D$.
 Let us describe this identification with more details. For $\tilde z=Uz\in \tilde Z$,
 the inverse projection $i_{Z}:z\rightarrow \tilde z$ is an isomorphism and the restriction
  of $f$ to $U$ induce an isomorphism $i_{U}$ from $U$ onto $f(U)$ because $U\cap Ker f=1$.
  Finally the identification is defined by the isomorphism
  $i=(i_{Z},i_{U}):z.u\rightarrow (\tilde z,f(u))$. Note that this identification depends
  on the decomposition of $\tilde U$ into $Z\times U$ ; the choice of another complement to
  $Z$ than $U$ defines another extension.
   Since the extension $E$ is completely described by the $Z\phi$-extension $\tilde S$ and by $B$,
\emph{\bf{we introduce the notation $\mathbf{E=E(\tilde S,\phi,B,U)}$}}.

%By the isomorphism theorem we can identify $\tilde S/\tilde Z=\frac{S_{E}/U}{(U\times Z)/U}$ and $S_{E}/(U\times Z)$.
%Let $x=(U\times Z)s$ be an element of $\tilde S/\tilde Z=S_{E}/(U\times Z)$ and let us
%define $\phi(x)=Inn(G).f(s)$ ( since $f(U\times Z)\subset Inn(G)$ the element $\phi(x)$
% is the same for every $s\in x$).%The action on $Z(G)$ of an element $x=(U\times Z)s$ of $\tilde S\tilde Z=S_{E}/(U\times Z)$
%Its action on $Z(G)$ is $z^{x}=z^{s}=z^{f(s)}=z^{Inn(G).f(s)}=z^{\phi(x)}$ for $z\in Z(G)$ and
%this is the definition of a $Z\phi$-extension of $Z(G)$.
%In order to prove that $\phi$ is an homomorphism let us remark that the restriction of
%$f$ to $S_{E}$ is an homomorphism onto $B=f(S_{E})$ and therefore it induces an homomorphism
%from $S_{E}/(U\times Z)$ onto $f(S_{E})/f(U\times Z)=B/(B\cap Inn(G))$ which can be
%identified with the subgroup $Inn(G)B/Inn(G)=L/Inn(G)$ of $Out(G)$.\\

$(2)$: Conversely, let us show that every $Z\phi$-extension associated with $L$ produces an
extension of $G$ associated with $L$. % this from an extension give rise to an extension of $G$.
Let $\tilde Z$ be a normal subgroup of a group $S$ and let $i_{Z}:z\rightarrow \tilde z$ be an
isomorphism from $Z(G)$ onto $\tilde Z$. Assume that modulo this isomorphism, $(S,\phi)$ is a
$Z\phi$-extension of $Z(G)$ associated with $L$. This means that
$\tilde z^{s}=\widetilde{ z^{\phi(s)}}$ holds in $S$ and that $\phi(S)=L/Inn(G)$.
 Let $D$ be the subgroup $\{(s,b)|\,b\in\phi(s)\}$ of $S\times B$.
The projection $h$ of $D$ onto $B\subset Aut(G)$ induces an action of $D$ on $G$
that enables us to construct the semi-direct product $D\ltimes_{h}G$.\\
Let $\overline{u}$ denote the conjugation by $u\in U\subset G$. Since $U\cap Z(G)=1$ we
have seen that $i_{U}:u\rightarrow\overline{u}$ is an isomorphism from $U$ onto
the subgroup $\overline{U}=Inn(G)\cap B$.
Clearly, $i=(i_{Z},i_{U}):z.u\rightarrow (\tilde z,\overline{u})$ is an isomorphism and
its image $\tilde Z\times\overline{U}$ is a subgroup of $D$ since
$\overline{U}\subset Inn(G)=\phi(\tilde Z)$.

We will now prove  that in every such semi-direct product $D\ltimes_{h}G$, we may identify the
subgroups $(Z\times U)\subset G$ and $\tilde Z\times\overline{U}$ of $D$ % can be identified
 to give a quotient isomorphic to one of the extensions $E$ described in $(1)$.
For this task we use the two conditions of Theorem \ref{ucns} to prove that the subgroup
$K=\{[i(x),x^{-1}] |\,x\in Z\times U\}$ is normal in $D\ltimes_{h}G$. In order to avoid confusion, we use
 brackets for the elements of the semi-direct product and parentheses for the elements of
 the direct product $\tilde Z\times\overline{U}$.

 %%!!!! AND RECALL THAT IN THE QUOTIENT D IS $F^{-1}(B)$

\begin{itemize}
\item condition $(1)$: the action by conjugation of $D$ on its subgroup $\tilde Z\times\overline{U}$
must coincide with its action on $Z\times U$ through $h$. Let $d=(s,h(d))$ be an element of
 $D$, let $\tilde x=(\tilde z,\overline{u})\in \tilde Z\times\overline{U}$ and let $x=z.u\in Z\times U$.
We must prove the equality $\tilde x^{d}=i(x^{h(d)})$ $(\star)$. Conjugation
 in $D$ is conjugation on each component so that $\tilde x^{d}=(\tilde z^{s},\overline{u}^{h(d)})$.
The fact that we have a $Z\phi$-extension means that $\tilde z^{s}=\widetilde{ z^{\phi(s)}}$ which is also
equal to $\widetilde{z^{h(d)}}$ because $d\in D$ implies $h(d)\in\phi(s)$.
 We also claim that $\overline{u}^{h(d)}=\overline{u^{h(d)}}$ $(\star\star)$, which could
 seem obvious, but is actually more subtle. We prove it with Lemma \ref{fpi} : since $b\in Aut(G)$
 is an isomorphism from $G$ onto $G$ and if $\overline{g}$ denotes the conjugation automorphism in $G$,
  we get $\overline{g^{b}}=b^{-1}\overline{g}b=\overline{g}^{b}$ (which proves
  $(\star\star)$). Finally $\tilde x^{d}=(\widetilde{z^{h(d)}},\overline{u^{h(d)}})$ coincides
   with the image of $(z.u)^{h(d)}=x^{h(d)}$ under $i$ and $(\star)$ is established.

\item condition $(2)$ : $g^{x}=g^{h(\tilde x)}$ for every $g\in G$. This is immediate because
$h(\tilde x)=\overline{u}$ since $\tilde x=(\tilde z,\overline{u})$ and therefore
$g^{x}=g^{z.u}=g^{u}:=g^{\overline{u}}=g^{h(\tilde x)}$.
\end{itemize}

Note also that since $Z\times U$ is the subgroup of $G$ corresponding to $B\cap Inn(G)$,
 Proposition \ref{uu-1} guarantees that the image of $D$ in $E=D\ltimes_{h}G/K$ is the preimage
 $f^{-1}(B)$ of $B$ in $E$ ($f$ is the conjugation action of $E$ on $G$).

% At first sight, the way this construction depends on $U$ it is not clear. If the same
%construction is carried through another decomposition $\tilde{U}=Z\times U_{2}$, the
%identification of $\tilde{U}$ and $Z\times \overline{U}$ will change and we get a different
%identifier subgroup $K_{2}$. It could seem paradoxal that the same group $E$ is obtained
%by factoring out a semi-direct product either by $K$ or by $K_{2}$.
% To understand what is going on, we have to remind us that the extension $S$ of $Z(G)$
% is isomorphic to $S_{E}/U$. If $U_{2}$ is
%chosen instead of $U$, wet get a group $S_{2}=S_{E}/U_{2}$ that is not necessarily
%isomorphic to $S$ so that the involved semi-direct product are different.
%???? NO SINCE $S_{E}$ REMAINS THE SAME BUT WE HAVE 2 DIFFERENT EXPRESSIONS OF $D=S_{E}$
%AS SUBDIRECT PRODUCT. So, it only gives two different expressions of $S_{E}=D$ as a
%subdirect product, once as a subgroup of $S\times B$, the other time as a subgroup
%of $S_{2}\times B$. Finally the
%same extension $E$ is obtained either as $S::$
%???? NON CAR S(E) RESTE LE MEME ?
%RIEN NE DIT QUE LE 2EME GROUPE IDENTIFIANT EST NORMAL  ? \\
%Si on part de E, on voit que $f^{-1}(B)/U$ n'est pas forcement isomorphe a $f^{-1}(B)/U_{2}$
%To see more clearly how this construction depends on $U$, let us note that for another
%decomposition $\tilde{U}=Z\times U_{2}$, the extension $S_{E}/U_{1}$ could be non
%isomorphic to $S_{E}/U$.

\begin{proposition}\label{uppers6} If $(S_{1},\phi_{1})$ and $(S_{2},\phi_{2})$ are equivalent %\\
 through an isomorphism $\tilde j:S_{1}\rightarrow S_{2}$ and $\pi\in Aut(G)$ then %\\
$E(S_{1},\phi_{1},B,U)$ and $E(S_{2},\phi_{2},B^{\pi},U^{\pi})$ are $G$-isomorphic.
\end{proposition}
\emph{Proof }.
Let $E_{1}=E(S_{1},\phi_{1},B,U)$ and $E_{2}=E(S_{2},\phi_{2},B^{\pi},U^{\pi})$.
To shorten the notations, we consider here that $Z=Z(G)$ is included
in both $S_{1}$ and $S_{2}$. Let $D_{1}=\{(s_{1},b)\in S_{1}\times B\,|\,b\in\phi_{1}(s_{1})\}$
 and $D_{2}=\{(s_{2},b^{\pi})\in S_{2}\times B^{\pi}\,|\,b^{\pi}\in\phi_{2}(s_{2})\}$. Let $h_{1}$
 (respectively $h_{2}$) be the canonical projection from $D_{1}$ onto $B$ (respectively
 from $D_{2}$ onto $B^{\pi}$).
 The equivalence of the $(S_{i},\phi_{i})$ means that $\tilde j$ maps $z\in Z(G)$ to
$z^{\pi}$ and that for $s\in S_{1}$ we have $\phi_{2}(\tilde j(s))=\phi_{1}^{\pi}(s)$.
 Let us show that these conditions are sufficient to extend $\tilde j$ and $\pi$ into a
 $G$-isomorphism from $D_{1}\ltimes_{h_{1}}G$ onto $D_{2}\ltimes_{h_{2}}G$. The isomorphism $\tilde j$ and
 $\pi$ define componentwise an  isomorphism  from $S_{1}\times B$ onto $S_{2}\times B^{\pi}$
 whose restriction to $D_{1}$ is an isomorphism $j_{\pi}:(s,b)\rightarrow(\tilde j(s),b^{\pi})$
 onto $D_{2}$. Indeed, $(s,b)\in D_{1}\Leftrightarrow b\in\phi_{1}(s)\Leftrightarrow
 b^{\pi}\in\phi_{1}^{\pi}(s)=\phi_{2}(\tilde j(s))\Leftrightarrow (\tilde j(s),b^{\pi})\in D_{2}$.\\
The function $t:[d,g]\rightarrow [j_{\pi}(d),g^{\pi}]$ is an isomorphism of semi-direct products since the
condition of Proposition \ref{semiso}, namely $h_{2}(j_{\pi}(d))=h_{1}^{\pi}(d)$, is satisfied here
because $d=(s,b)\in D_{1}\Rightarrow h_{2}(j_{\pi}(d))=h_{2}((\tilde j(s),b^{\pi}))=b^{\pi}=h_{1}^{\pi}(d)$.

By definition $E_{l}=(D_{l}\ltimes_{h_{l}}G)/K_{l}$ for $l=1,2$, and we now prove that the
$G$-isomorphism $t$ maps $K_{1}$ onto $K_{2}$ so that finally the $E_{l}$ are $G$-isomorphic.
We recall that the identifier subgroups are $K_{1}=\{[i_{1}(x),x^{-1}]\,|x\in Z\times U\}$
and $K_{2}=\{[i_{2}(x),x^{-1}]\,|x\in Z\times U^{\pi}\}$ where
$i_{1}:z.u\rightarrow (z,\overline{u})$ and $i_{2}:z.u^{\pi}\rightarrow (z,\overline{u^{\pi}})$
 are isomorphisms from $Z\times U$ (respectively $Z\times U^{\pi}$) onto
 $Z\times\overline{U}<D_{1}$ (respectively $Z\times\overline{U^{\pi}}<D_{2}$ ).
Let $x=z.u\in Z\times U$. The $G$-isomorphism $t$ maps $k_{1}=[i_{1}(x),x^{-1}]$ to
$[j_{\pi}(i_{1}(x)),(x^{-1})^{\pi}]=[(\tilde j(z),\overline{u}^{\pi}),(z^{\pi}.u^{\pi})^{-1}]$.
By the equivalence hypothesis, $\tilde j(z)=z^{\pi}$ and we have showed previously that
$\overline{u}^{\pi}=\overline{u^{\pi}}$ so that finally $k_{1}$ is mapped to
$[(z^{\pi},\overline{u^{\pi}}),(z^{\pi}.u^{\pi})^{-1}]=[i_{2}(x_{2}),x_{2}^{-1}]\in K_{2}$
for $x_{2}=z^{\pi}.u^{\pi}\in Z\times U^{\pi}$
$.\quad\Box$

Does the converse of this proposition hold ? Not necessarily, but
if $E_{1}$ is an extension of $G$ associated with $L$ and if $i$ is a $G$-isomorphism
from $E_{1}$ onto an extension $E_{2}$ of $G$, then by Lemma \ref{fpi} $E_{2}$ is
associated with $L^{\pi}$ where $\pi\in Aut(G)$ is the restriction of $i$ to $G$.
Since $B$ is a supplement to $Inn(G)$ in $L$ that preserves $U$, its image $B^{\pi}$, is a
supplement to $Inn(G)$ in $L^{\pi}$ and the subgroup of $G$ corresponding to $Inn(G)\cap B^{\pi}$
is $Z(G)\times U^{\pi}$ and $B^{\pi}$ preserves $U^{\pi}$.
 Hence, $E_{2}=E(S_{2},\phi_{2},B^{\pi},U^{\pi})$ for some $Z\phi$-extension $(S_{2},\phi_{2})$
 associated with $L^{\pi}/Inn(G)$.
We will now prove that under some extra assumptions, $S_{1},\phi_{1})$ is equivalent to $(S_{2},\phi_{2})$
 (the converse of Proposition \ref{uppers6})
holds and that there is a one-to-one correspondence between the $G$-isomorphism
classes of a family of $G$-extensions and
the equivalence classes of the corresponding family of $Z\phi$ extensions of $Z(G)$.\\
In the next proposition, the hypothesis on $L,B$ and $U$ are still the same as in the
beginning of this section.
%$E(S_{1},\phi_{1},B,U)$ and $E(S_{2},\phi_{2},B,U)$ are $G$-isomorphic if and only if
%$(S_{1},\phi_{1})$ and $(S_{2},\phi_{2})$ are equivalent. Moreover $E(S,\phi,B)/G\cong S/\tilde Z$.

\begin{proposition}\label{onetone} Let $\pi\in Aut(G)$ and let $N_{U,B}=N_{Aut(G)}(B)\cap N_{Aut(G)}(U)$.
Assume that $N_{Aut(G)}(L)=LN_{U,B}$. Then $E(S_{1},\phi_{1},B,U)$ and $E(S_{2},\phi_{2},B^{\pi},U^{\pi})$
are $G$-isomorphic if and only if $(S_{1},\phi_{1})$ and $(S_{2},\phi_{2})$ are equivalent.
\end{proposition}
\emph{Proof }.$\Rightarrow$ : Let $E_{1}=E(S_{1},\phi_{1},B,U)$ and $E_{2}=E(S_{2},\phi_{2},B^{\pi},U^{\pi})$.
For $l=1,2$ let $f_{l}$ be the action of $E_{l}$ on $G$ by conjugation. By their definition,
$f_{1}(E_{1})=Inn(G)B=L$ and $f_{2}(E_{2})=Inn(G)B^{\pi}=(Inn(G)B)^{\pi}=L^{\pi}$.
Assume that there exists a $G$-isomorphism $\tilde i:E_{1}\rightarrow E_{2}$ and
let $\gamma\in Aut(G)$ be its restriction to $G$.
By Lemma \ref{fpi} $f_{2}(E_{2})=f_{1}^{\gamma}(E_{1})=L^{\gamma}=L^{\pi}$. Consequently
 $\gamma\pi^{-1}\in N_{Aut(G)}(L)$ so that $\gamma\in N_{Aut(G)}(L).\pi=LN_{U,B}.\pi$
 and thus $\gamma=l.n.\pi$ for some $l\in L$ and some $n\in N_{U,B}$. There exists
 $e\in E_{1}$ such that the conjugation $f_{1}(e)=l^{-1}$ (because $f_{1}(E_{1})=L$)
 and if we compose  this conjugation with $\tilde i$, we get a new
 $G$-isomorphism $i=f_{1}(e).\tilde i$ whose restriction to $G$ is $l^{-1}ln\pi=n\pi$.

Let $S_{E_{1}}=\{e\in E_{1}\,|\,f_{l}(e)\in B  \}=f_{1}^{-1}(B)$ and
$S_{E_{2}}=\{e\in E_{2}\,|\,f_{2}(e)\in B^{\pi}  \}=f_{2}^{-1}(B^{\pi})$.
For an element $e_{1}$ of $E_{1}$, we get $f_{2}(i(e_{1}))=f_{1}^{n\pi}(e_{1})$ (by Lemma
\ref{fpi}) which implies that $f_{2}(i(e_{1}))\in B^{\pi}\Leftrightarrow f_{1}(e_{1})\in B$
 (since $B^{n}=B$). Consequently, $i$ maps $S_{E_{1}}$ onto $S_{E_{2}}$ and maps $U$ onto
 $U^{\pi}$ (since $U^{n}=U$ if $n\in N_{U,B}$).

Let $U_{1}=U$ and $U_{2}=U^{\pi}$. The $Z\phi$-extension of $Z(G)$ associated with $E_{l}$
 has been defined as $(S_{E_{l}}/U_{l},\phi_{l})$,
where $\phi_{l}(U_{l}s_{l})=Inn(G)f_{l}(s_{l})$ for $s_{l}\in S_{E_{l}}$. Since $i$ maps
$U_{1}$ onto $U_{2}$ it induces an isomorphism from $S_{E_{1}}/U_{1}$ onto $S_{E_{2}}/U_{2}$
defined by $j:U_{1}s_{1}\rightarrow i(U_{1}s_{1})=U_{2}i(s_{1})$.
Let us show that the $Z\phi$-extensions  associated with the $E_{l}$ are equivalent
through $j$ and $n\pi$. First, $z\in Z(G)$ correspond to the coset $U_{1}z$ which is mapped
 by $j$ to $U_{2}i(z)=U_{2}z^{n\pi}$ that corresponds to $z^{n\pi}\in Z(G)$ (because $i|_{G}=n\pi$).
Next, since by Lemma \ref{fpi} $f_{2}(i(s_{1}))=f_{1}^{n\pi}(s_{1})$ we get
$\phi_{2}(j(Us_{1}))=\phi_{2}(U_{2}i(s_{1})):=Inn(G)f_{2}(i(s_{1}))$ which is equal to
$Inn(G)f_{1}^{n\pi}(s_{1})=(Inn(G)f_{1}(s_{1}))^{n\pi}:=\phi_{1}^{n\pi}(Us_{1})$.
Hence, the two conditions for equivalence are fulfilled.

$\Leftarrow$ : already proved in Proposition \ref{uppers6}$.\quad\Box$

We integrate the previous results in the following theorem :

\begin{theorem}\label{usplitz} Let $L/Inn(G)$ be a subgroup  of $Out(G)$.
 If there exists a supplement $B$ of $Inn(G)$ in $L$
such that the subgroup of $G$ corresponding to $Inn(G)\cap B$ is equal to
$U\times Z(G)$ for some $U<G$ such that $U^{B}=U$ then

 \begin{itemize}%\begin{enumerate}
 \item $(i):\,$ there exists an extension of $G$ by a group $H$ associated with $L$ if and only if there
 exists an epimorphism $H\rightarrow L/Inn(G)$.
\end{itemize}
Let $N_{U,B}=N_{Aut(G)}(B)\cap N_{Aut(G)}(U)$.
Let $Z\phi_{L,H}$ be the set of equivalence classes of $Z\phi$ extensions of $Z(G)$ by
$H$ associated with $L$.

\begin{itemize}
\item $(ii):\,$ If $N_{Aut(G)}(L)=LN_{U,B}$
then there is a one-to-one correspondence between $Z\phi_{L,H}$ and the $G$-isomorphism
classes of extensions of $G$ by $H$ associated with $L$.

\item $(iii):\,$ If for every $\pi\in N_{Aut(G)}(L)$ there exists
$n\in N_{U,B}$ such that \begin{enumerate}
\item $b^{n}\in Inn(G)b^{\pi}$  $\quad\forall\, b\in B$
\item $z^{n}=z^{\pi}$  $\qquad\qquad\forall\, z\in Z(G)$.\end{enumerate}
then, if $|Z\phi_{L,H}|<\infty$  there are, up to $G$-isomorphism, at most
$|Z\phi_{L,H}|$ extensions of $G$ by $H$ associated with $L$.
 \end{itemize}%\end{enumerate}
\end{theorem}

\emph{Proof }.\begin{itemize}
\item $(i)$ Let $E$ be one such extension. If $f$ is the conjugation action of $E$ on $G$, then
it induces an epimorphism from $H\cong E/G$ onto $f(E)/f(G)=L/Inn(G)$.
 Conversely let $\phi$ be an epimorphism $H\rightarrow L/Inn(G)$. Through the action
 of $Out(G)$ on $Z=Z(G)$, it defines an action  $\tilde\phi:H\rightarrow Aut(Z)$ and
 the semi-direct product $S=H\ltimes_{\tilde\phi}Z$ is a $Z\phi$-extension.
 We now that $E=E(S,\phi,B,U)$ is an extension of $G$ associated with $L$
and let us prove that $E/G\cong H$ so that
$E$ is an extension of $G$ by $H$.

Let $D$ be the subgroup $\{(x,b) |\,b\in\phi(x)\}$ of $\tilde S\times B$.
As proved in Proposition \ref{uu-1}, the image $D_{k}\cong D$ of $D$ in $D\ltimes_{h}G/K=E(S,\phi,B,U)$
 is a supplement to $G$ and $D_{k}\cap G=Z\times U$. Hence, $E/G$ is isomorphic
 to $D_{k}/(Z\times U)\cong D/(\tilde Z\times\overline{U})$.
Let $\pi$ be the projection of $S\times B$ on $S$, whose kernel is $1\times B$.
For every $s\in S$ there exists $b\in\phi(s)\in L/Inn(G)$ because $B$ is a supplement to
$Inn(G)$ in $L$. Hence, the image of $D$ under $\pi$ is the full group $S$ and
$S\cong D/(1\times\overline{U})$ since $Ker\pi\cap D=1\times\overline{U}$.
Finally, $E/G\cong D/(\tilde Z\times\overline{U})\cong \frac{D/(1\times\overline{U})}{\tilde Z\times\overline{U}/(1\times\overline{U})}\cong S/\tilde Z$,
by the isomorphism theorem and we have proved that $E/G\cong H$.
% We have proved that $E(S,\phi,B)\in\mathcal{E}_{L}$ is and it is an extension of $G$ by $H$.

\item $(ii)$ We have proved that every extension of $G$ associated with $L$ can be realized
as $E(S_{1},\phi_{1},B,U)$ where $(S_{1},\phi_{1})$ is a $Z\phi$-extension of $Z(G)$
associated with $L$. If  $N_{Aut(G)}(L)=LN_{U,B}$, the one-to-one correspondence is
stated by Proposition \ref{onetone} ; $E(S_{1},\phi_{1},B,U)$ and $E(S_{2},\phi_{2},B,U)$
 are $G$-isomorphic if and only if $(S_{1},\phi_{1})$ and $(S_{2},\phi_{2})$ are
 equivalent. If we just want to classify extensions of $G$ by $H$ then, since
$H\cong E_{i}/G\cong S_{i}/Z(G)$, only $Z\phi$ extensions of $Z(G)$ by $H$ are
to be taken into account.\\

\item $(iii)$
Let $Z\phi_{L,H}$ be the set of equivalence classes of $Z\phi$-extensions of $Z(G)$ by
$H$ associated with $L$ and assume that $|Z\phi_{L,H}|=k<\infty$. We know that every
extension of $G$ by $H$ associated with $L$ is of type $E(S_{1},\phi_{1},B,U)$
for a $Z\phi$-extension of $Z(G)$ by $H$ associated with $L$. Let us show that hypothesis $(iii)$
implies that each equivalence
class in $Z\phi_{L,H}$ corresponds to only one $G$-isomorphism class of extensions of $G$
by $H$ associated with $L$, so that there are at most $k$ non $G$-isomorphic such extensions.

As proved above, if $(S_{1},\phi_{1})$   and $(S_{2},\phi_{2})$, associated with $L$,
  are equivalent through an  isomorphism $j:S_{1}\rightarrow S_{2}$ and $\pi\in Aut(G)$
  then $\pi\in N_{Aut(G)}(L)$. For every $s\in S_{1}$, $\phi_{1}(s)\in L/Inn(G)$
 and since $L=Inn(G)B$, we have $\phi_{1}(s)=Inn(G)b$ for some $b\in B$. Hence, our hypothesis
 implies that for some $n\in N_{U,B}$, $\phi_{2}(j(s))=\phi_{1}^{\pi}(s)=\phi_{1}^{n}(s)$ and $j(z)=z^{\pi}=z^{n}$
and it means that $(S_{1},\phi_{1})$   and $(S_{2},\phi_{2})$
are equivalent through $j$ and $n$. By Proposition \ref{uppers6},
 $E(S_{1},\phi_{1},B,U)$ and $E(S_{2},\phi_{2},B^{n},U^{n})=E(S_{2},\phi_{2},B,U)$ are $G$-isomorphic
 ( since $n$ normalizes both $B$ and $U$).
Thus all the element of an equivalence class in $Z\phi_{L,H}$ correspond to the same
$G$-isomorphism class.
This result gives an upper bound for the number of $G$-isomorphism classes in a case
 that is more general than $(ii)$ since here, $N_{U,B}$ does not need to be a supplement
 of $L$ in $N_{Aut(G)}(L)$. $\quad\Box$
 \end{itemize}

%%Even if as for $S6$, the normalizer of $B$ is not a supplement to $L$.\\

\begin{proposition}\label{uzprime} If $B$ is a supplement to $Inn(G)$ in $Aut(G)$ such that
\mbox{$|B\cap Inn(G)|$} is coprime to $|Z(G)|$,  then there is a one-to-one correspondence
 between the equivalence classes of $Z\phi$-extensions of $Z(G)$ and the $G$-isomorphism
classes of extensions of $G$.
\end{proposition}
\emph{Proof }. Let $\tilde{U}$ be the subgroup of $G$ corresponding to $U_{1}=B\cap Inn(G)$.
It is an extension of $Z(G)$ by $U_{1}$. By Schur-Zassenhaus Theorem (\ref{schurzass}),
$\tilde{U}$ splits over $Z(G)$ but, since $Z(G)$ commutes with every subgroup of $G$,
  $\tilde{U}=Z(G)\times U$ where $U$ is a complement to $Z(G)$ isomorphic to $U_{1}$.
Since all complements are conjugate and $U\unlhd \tilde{U}$, $U$ is the unique complement
of $Z(G)$ in $\tilde{U}$.
 Every $b\in B$ normalizes $U_{1}=B\cap Inn(G)$ and thus normalizes $\tilde{U}$.
 Since $b$  normalizes $Z(G)$, it also normalizes $U$ $(\star)$, because $U$ is the
  unique complement in $\tilde{U}$ .

 Since $Aut(G)=Inn(G)B$, for every subgroup $L/Inn(G)<Out(G)$, the subgroup
$B_{L}=B\cap L$ is a supplement to $Inn(G)$ (thus $L=Inn(G)B_{L}$). Hence, every extension of
$G$ is realized as $E(S_{1},\phi_{1},B_{L},U)$ where $(S_{1},\phi_{1})$ is a $Z\phi$-extension
of $Z(G)$ such that $\phi_{1}(S_{1})=L$. Moreover, by $(\star)$, $B$ is a subgroup of
 $N_{U,B}=N_{Aut(G)}(B)\cap N_{Aut(G)}(U)$ and since $Aut(G)=Inn(G)B$, we get
 $N_{Aut(G)}(L)=Inn(G)B_{N,L}$ where $B_{N,L}:=B\cap N_{Aut(G)}(L)$. Since $Inn(G)\leq L$
 and $B\leq N_{U,B}$ we have $N_{Aut(G)}(L)=LN_{U,B}$ ($N_{U,B}$ normalizes $L=Inn(G)B$).
 Therefore the hypotheses of Proposition \ref{onetone} are realized : there is a one-to-one correspondence
 between  the $Z\phi$-classes and the $G$-isomorphism classes of all extensions of $G$
 ( and not only for the extensions associated with a specific $L/Inn(G)$)$.\quad\Box$

Examples of groups $G$ satifying the hypotheses of Proposition \ref{uzprime} can be found, using the tables
of subsection \ref{tableout}. The following Theorem and the corollary that follows it, show that the perfect
$3$-covers of $G=A_6$ and $G=PerfectGroup(9720,3)$ are such examples (because $Z(G)=3$ and
$Out(G)$ is a $2$-group).
% and $3.[29160,1] ??? no it is amistake ????

% in the next theorem, the hypothesis G finite is necessary for the existence of
%a supplement by sylow normalizers, and for out(G) to be finite
%but H could be infinite. But then since the
% number of extensions could be infinite, we should use the words bijection between sets
%instead of same number of extensions as
%\begin{theorem}\label{uzout} Let $G$ be a finite group, let $H$ be a group and let
%$\phi:H\rightarrow Out(G)$ be an homomorphism. Assume that the order $\phi(H)$ is
%relatively prime to the order of $Z(G)$.
%There is a bijection between the $G$-isomorphism classes of extensions of $G$ by $H$
%associated with $\phi$ and the $Z\phi$-classes of $Z\phi$-extensions of $Z(G)$
%by $H$ associated with $\phi$.
%\end{theorem}
%\emph{Proof }.
%$.\quad\Box$

\begin{theorem}\label{ZL=1} Let $L/Inn(G)$ be a subgroup of $Out(G)$ whose order
is relatively prime to $|Z(G)|$.
There is a bijection between the $G$-isomorphism classes of extensions of $G$ associated
with $L$  and the $Z\phi$-classes of $Z\phi$-extensions of $Z(G)$ associated
with $L$.
\end{theorem}

\emph{Proof }. Let us show that every hypothesis of Theorem \ref{usplitz} $(ii)$ is satisfied.
$Aut(G)$ is an extension of $Inn(G)$. By Proposition \ref{nilsup},
$Inn(G)$ has a supplement $S$ such that $U:=S\cap Inn(G)$ is nilpotent. The group
 $S_{L}:=S\cap L$ is a supplement to $Inn(G)$ in $L$.
Let $\pi$ be the set of  primes $p$ that do divide $|L/Inn(G)|=|S_{L}/U|$ and
 let $\pi'$ be the set of primes $p'$ that do not.
Nilpotency of $U$ implies that $U$ is the direct product $U_{\pi}\times U_{\pi'}$
where $U_{\pi}$ (respectively $U_{\pi'}$) is the direct product of the $p$-Sylow
 subgroups of $U$  where $p\in\pi$ (respectively $p'$-Sylow where $p'\in\pi'$).
  $U_{\pi'}$ is characteristic in $U$, so it is normal
 in $S_{L}$ and since the primes $p'\in\pi'$ do not divide $|S_{L}/U_{\pi'}|=|S_{L}/U|.|U_{\pi}|$,
by \emph{Schur-Zassenhaus} theorem, there is a complement $C_{\pi}$ of $U_{\pi'}$ in
$S_{L}$ and all such complements are conjugate in $S_{L}$ (we do not need to use the Feit-Thompson
theorem since $U_{\pi'}$ is solvable). $C_{\pi}\cap U$ is a complement of $U_{\pi'}$ in $U$ ;
  by the same theorem all such complements are conjugate to the complement $U_{\pi}$, but
  since $U_{\pi}\unlhd U$, we get $C_{\pi}\cap U=U_{\pi}$.
  %  its order is $|U/U_{\pi'}|=|U_{\pi}|$, which is the product of the orders of the
%$Sylow-p$ subgroups of $U$ such that $p\in\pi$, $C_{\pi}\cap U$ must contain a $Sylow-p$
% subgroup of $U$ for every $p\in\pi$. Hence
As $C_{\pi}$ is a supplement to $U$ in $S_{L}$, it is also a supplement to $Inn(G)$ in $L$
 and its intersection $C_{\pi}\cap Inn(G)=C_{\pi}\cap U=U_{\pi}$.

Since $Aut(G)=Inn(G)S$, an element $\gamma\in N_{Aut(G)}(L)$ is equal to $gs$ for $g\in Inn(G)$
% and $s\in S$ that normalizes $L$ (because $L^{gs}=(L^{g})^{s}=L^{s}$). But as $L=Inn(G)S_{L}$
 and $s\in S$ that normalizes $L$ (because $g$ normalizes $L$). But as $L=Inn(G)S_{L}$
and $L^{s}=L$, it implies that $S_{L}^{s}\subseteq L$ and $S_{L}^{s}\subseteq S$ whence
$S_{L}^{s}\subseteq S\cap L=S_{L}$ thus $s$ normalizes $S_{L}$. Next, $s$ must normalize
$U=S_{L}\cap Inn(G)$ and also $U_{\pi'}$ which is a characteristic subgroup of $U$.
As  $C_{\pi}$ is a complement of $U_{\pi'}$ in $S_{L}$, its image $C_{\pi}^{s}$ under $s$
 is also a complement of $U_{\pi'}$ in $S_{L}$ : it is thus a conjugate $C_{\pi}^{s_{l}}$
 %for some $s_{l}\in S_{L}$ so that $s_{l}^{-1}s$ normalizes $C_{\pi}$. Finally
 for some $s_{l}\in S_{L}$ so that $ss_{l}^{-1}$ normalizes $C_{\pi}$. The element
 $\tilde{s_{l}}=(s_{l}^{-1})^{s^{-1}}$ belongs to $S_{L}$ since $s$ normalizes $S_{L}$.
 But since $\tilde{s_{l}}s=ss_{l}^{-1}s^{-1}s=ss_{l}^{-1}\in N_{Aut(G)}(C_{\pi})$, then
 finally $s\in S_{L}N_{Aut(G)}(C_{\pi})$ and $\gamma=gs\in Inn(G)S_{L}N_{Aut(G)}(C_{\pi})=LN_{Aut(G)}(C_{\pi})$ :
$N_{Aut(G)}(C_{\pi})$ is a supplement to $L$ in $N_{Aut(G)}(L)$.

Now, we go back into $G$.
 Let $\tilde{U}$ be the subgroup of $G$ corresponding to $U_{\pi}=C_{\pi}\cap Inn(G)$.
It is an extension of $Z(G)$ by $U_{\pi}$. By hypothesis, the prime divisors of $|Z(G)|$
are not in $\pi$ and by Schur-Zassenhaus Theorem (\ref{schurzass}), $Z(G)$ has a
complement $U_{1}$ isomorphic to $U_{\pi}$. $Z(G)$ commutes with every subgroup of $G$
thus  $\tilde{U}=Z(G)\times U_{1}$. Since all complements are conjugate and
$U_{1}\unlhd \tilde{U}$, $U_{1}$ is the unique complement of $Z(G)$ in $\tilde{U}$.
 Every $n\in N_{Aut(G)}(C_{\pi})$ normalizes $U_{\pi}=C_{\pi}\cap Inn(G)$ and thus normalizes
 $\tilde{U}$. Since $n$  normalizes $Z(G)$, it also normalizes $U_{1}$, its unique
 complement in $\tilde{U}$ so that $N_{Aut(G)}(C_{\pi})\cap N_{Aut(G)}(U_{1})=N_{Aut(G)}(C_{\pi})$.
  Hence, we have proven that there exists a supplement $C_{\pi}$
 of $Inn(G)$ in $L$ that satisfies the hypothesis of Theorem \ref{usplitz}(ii).
 $\quad\Box$
% or of Proposition \ref{onetone}.
%Every extension of $G$ associated with $L$  is realized as $E(S,\phi,C_{\pi},U_{1})$ where
% $(S,\phi)$ is $Z\phi$-extensions of $Z(G)$ associated
% to $L$ and  each equivalence class of such $Z\phi$-extensions, corresponds to exactly one
% $G$-isomorphism class of extensions of $G$ associated with $L$. $.\quad\Box$

In the case $L/Inn(G)=Out(G)$ (hence $|Z(G)|$ and $|Out(G)|$ are coprime), the subgroup
$C_{\pi}$ constructed in the proof of Theorem \ref{ZL=1} is a supplement to $Inn(G)$ in
$Aut(G)$ that satisfies the hypotheses of Proposition \ref{uzprime} and this gives the
following result.
\begin{corollary}\label{uout=1}
If $|Z(G)|$ and $|Out(G)|$ are coprime then there is a one-to-one correspondence
 between the equivalence classes of $Z\phi$ extensions of $Z(G)$ and the $G$-isomorphism
classes of extensions of $G$.
\end{corollary}

%%%%%%%%%%%%%%%%%%%%%%%%%%%%%%%%%%%%%%%%%%%%%%%%%%%%%
\subsection{The nonsolvable groups of order less than $58,320$}\label{58320}

Let $G$ be a perfect group and $\mathcal{E}$ a family of nonsolvable groups whose perfect
residuum is $G$. The isomorphism classes in $\mathcal{E}$ are exactly the $G$-isomorphism
classes in $\mathcal{E}$. If $Aut(G)$ splits over $Inn(G)$, we have completely solved
the isomorphism problem for the extensions of $G$ by a group $H$ : it has been reduced to the
 classification of $Z\phi$-extensions of $Z(G)$ by $H$ (see Theorem \ref{autsplit}).

It is noteworthy that this is the case for all but $3$ categories of perfect groups such
that $|G/Z(G)|< 29,160$. These categories consist of the cases where the isomorphism type of
$G/Z(G)$ is either $PSL(2,9)$, the perfect group $[1920,4]$,  $PSL(2,25)$ or the perfect group
$[9720,3]$.

The possible center for the corresponding covering group $G$ are

\begin{tabular}{|c|c|c|}
\hline
$G/Z(G)$ &  $Out(G)$ & $Z(G)$\\
\hline\hline
$PSL(2,9)$ & $2\times 2$ & $1,2,3,6$ \\
$[1920,4]$ & $2\times 2$& $2$\\
$PSL(2,25)$ & $2\times 2$ & $1,2$  \\
$[9720,3]$ & $D_{8}$& $1,3$\\
\hline
\end{tabular}

Since in each case, $Out(G)$ is a $2$-group, the cases $Z(G)=3$ or $1$ are completely classified by
Corollary \ref{uout=1}. For the two remaining situations where $Z(G)=2$ and $Z(G)=2\times 3$,
our $z\phi$-class method does not provide a classification of the extensions.  Instead, we provide tables for
these cases, that where obtained by our "supplement" method (combined with our results in chapter \ref{General})
%The case $Z(G)=2\times 3$ will be reduced to the solution of
%case $Z(G)=3$ and $Z(G)=2$ in Chapter 2.

In the section about Theorem \ref{bij3840}, we have proved that the case $Z(G)=2$ and
$G/Z(G)$ isomorphic to either PerfectGroup(1920,4) or $PSL(2,25)$ are equivalent to the case
$G=SL(2,9)$. Hence, all cases that are not yet classified by the previous results can be reduced
to the extensions of $SL(2,9)$. At the end of chapter \ref{General}, we will provide a table for the extensions
of $SL(2,9)$.
%The $Aut(SL(2,9)\cong A_{6}$ and the subgroups of $Out(G)$ correspond to $A_{6}$,
%$PGL(2,9)$, $S_{6}$, $M_{10}$ and $P\Gamma L(2,9)$. Let us show that $Z\phi$-classes solve
%completely the 2 first case, that the number of $Z\phi$-classes is an upper bound in the
%third case and that for the 2 last case, some $Z\phi$-classes do not correspond to any
%extensions of $SL(2,9)$.
%...to be continued

Hence, we have a way to classify all the finite nonsolvable groups with perfect residuum $G$
such that $|G/Z(G)|<29,160$ and therefore all nonsolvable groups of order
$2*29,160=58,320$. This is done by computing $z\phi$-classes in almost every case except
for some central extensions of $PSL(2,9)$ for which tables obtained by the supplement method are provided.
%%%%%%%%%%%%%%%%%%%%%%%%%%%%%%%%%%%%%%%%

%%%%%%%%%%%%%%%%%%%%%%%%%%%%%%%
%%  SOME LATEX LAYOUT %%%%%%
%%%%%%%%%%%%%%%%%%%%%%%%%%%%%
%  \chapter{Basic results on extensions}
%  \section{The action on $G$ associated with an extension.}
%  \emph{associated with a homomorphism from $H$ into $Out(G)$}.
%  \subsection{Extensions with inner action on $G$}
%
%  \begin{proposition}\label{inneract}
%  \end{proposition}
%  \emph{Proof }. .\quad\Box$

%  \textbf{Definition }.

%  \begin{corollary}
%  Every extension of a complete group  is a direct product.
%  \end{corollary}
%
%  begin{enumerate}\item
%  \item see \cite{Carter}).
%  \end{enumerate}
%
%  \begin{corollary}If $H$ is a group of odd order and if $G$ belongs to the list
%  below then, the only extension of $G$ by $H$ is the direct product $G\times H$.
%  \begin{itemize}
%  \item The sporadic groups $M_{12},M_{22},J_{2},Suz,HS,McL,He,Fi_{22},
%  Fi^{'}_{24},HN,O'N,J_{3}$.
%  \item The alternating groups $A_{n}$ ( $n>3$).
%  \item $L_{2}(p^{k})$ where $p$ is an odd prime number and $k$ is a power of $2$.
%  \end{itemize}
%  \end{corollary}
%%%%%%%%%%%%%%%%%%%%%%%%%%%%%%%%%

% % macro duble overline
% \newcommand{\overdub}[1]{\overline{\overline{#1 }}}
%%%%%%%%%%%%%%%%%%%%%%%

\chapter{A general classification algorithm}\label{General}

\section{Describing extensions}
Let us first expose a classical way of describing extensions. % We follow the notation of \cite{Hall}.\\
 Let $E$ be an extension of $G$ by a given factor group $H\cong E/G$.
%Let us choose an isomorphism  $H\rightarrow E/G$ and for each $u\in H$,
%a representative $\overline{u}$ in the coset of $G$ corresponding to $u$.
Let us choose an isomorphism  $j:H\rightarrow E/G$ and for each $u\in H$,
a representative $\overline{u}$ in the coset $j(u)$ of $G$.
  Every element of $E$ is %can now be
 written in a single way as $\overline{u}g$ for some $u\in H$ and some $g\in G$.
Conjugation   by $\overline{u}$ in $E$
induces \mbox{an automorphism}  %defined as
%%$g\to g^{u}=\overline{u}^{-1}g\overline{u}$ of $G$ %for $g\in G$
$g\to \overline{u}^{-1}g\overline{u}$ of $G$. We denote it by $\xi (u)$
and thus the transversal $\{\overline{u}:\,u\in H\}$ to $G$ in $E$ induces
 the function $\xi:H\to Aut(G):u\to \xi (u)$. An equivalent way to define $g^{\xi (u)}$ is
 the equation
 \beq \label{rule1}
 g.\overline{u}=\overline{u}.g^{\xi (u)}
 \eeq
  For any two elements $u, v$ of $H$, the product
 $\overline{u}.\overline{v}$ belongs to the coset $j(uv)$ and
\beq \label{rule2}
 \overline{u}.\overline{v}=\overline{uv} [u,v]
 \eeq
  for some element of $G$ denoted by $[u,v]$ and equal to $\overline{uv}^{-1}\overline{u}.\overline{v}$.
Hence, each transversal to $G$ in $E$ defines a pair of functions $(\xi,\varphi)$ where
$\varphi:H\times H \to G$ maps $(u,v)$ to $[u,v]$. This pair completely defines the multiplication
%% The automorphisms $\xi (u)$ and the elements $[u,v]$ completely define multiplication
 between elements $\overline{u}g$ of $E$ since the rules (\ref{rule1}) and (\ref{rule2})
  imply
$\overline{u} a.\overline{v} b=\overline{u}.\overline{v}a^{\xi (v)}b=
\overline{uv}[u,v]a^{\xi (v)}b,\,\textrm{for\, any\, a,b\,}\in G$.
The extension $E$ is then isomorphic to the group defined by the multiplication
\beq \label{multext}
(u,a).(v,b)=(uv,[u,v]a^{\xi (v)}b)
\eeq
on the cartesian product  of the sets $H$ and $G$. %$\{ (u,a):u\in H\,a\in G\}$
%Conversely, to build an extension of $G$ by  $H$, we need to choose,
%for every $u,v$ in $H$ an element $[u,v]$ in $G$ and automorphisms $g\to g^{u}$.
Conversely, suppose we want to build an extension of $G$ by  $H$.
Which are the conditions  that a function $\xi:H\to Aut(G)$ and a function
$\varphi:H\times H \to G:(u,v)\to [u,v]$
% automorphisms $\xi (u)$ and elements $[u,v]\in G$
must satisfy for the multiplication defined by (\ref{multext}) on set $E=\{ (u,a):u\in H,\,a\in G\}$
to become a group ?\\
%Schreier has proved (see \cite{}) that  $E$ is a group if and only if
%$$i)\qquad(g^{u})^{v}=[u,v]^{-1}(g^{uv})[u,v]$$
%$$ii)\qquad[uv,w][u,v]^{w}=[u,vw][v,w]$$
%(provided that $1_{G}$ is chosen as representative of $G$).\\
%Condition $(i)$ amounts to ask that conjugation in $E$ defines a homomorphism from
%$E$ into $Aut(G)$ that is to say $(g^{\overline{u}a})^{\overline{v}b}=g^{\overline{u}a.\overline{v}b}$
% and condition $(ii)$ is equivalent to associativity.\\
Schreier has proved (see \cite{Hall} page 221) that  $E$ is a group if and only if
\begin{eqnarray}
 & \qquad(g^{\xi (u)})^{\xi (v)}=[u,v]^{-1}(g^{\xi (uv)})[u,v] & \label{schreier1}\\
 & \qquad[uv,w][u,v]^{\xi (w)}=[u,vw][v,w] & \label{schreier2}
\end{eqnarray}
for every $g\in G$ and every $u,v,w \in H$
(provided that $1_{G}$ is chosen as representative of $G$).\\
Condition \ref{schreier1} amounts to ask that conjugation in $E$ defines a homomorphism from
$E$ into $Aut(G)$ %%that is to say $(g^{\overline{u}a})^{\overline{v}b}=g^{\overline{u}a.\overline{v}b}$
 and condition \ref{schreier2} is equivalent to associativity in $E$.

\textbf{Definition }. If for some given function $\xi:H\to Aut(G)$, a function
$\varphi :\,(u,v)\to[u,v]$ from
$H\times H\to G$ satisfies \ref{schreier1} and \ref{schreier2},
 it is called a \emph{\bf{factor set}} for $\xi$. If so, the extension
 denoted by $\mathbf{E(\xi,\varphi)}$ %%corresponds to  the pair $(\xi,\varphi)$  and
 is defined by
 multiplication (\ref{multext}) on the cartesian product of the sets $H$ and $G$. If in
 $E(\xi,\varphi)$ we choose the transversal $\tau=\{\overline{u}:=(u,1)\,|\, u\in H\}$,
  it is  easy to verify from (\ref{multext}) that $(\xi,\varphi)$ is
 the pair induced by $\tau$. The transversal $\tau$ is called
 \emph{\textbf{the canonical transversal}} of $E(\xi,\varphi)$.
%$$ f:\,(u,v)\to[u,v]=\overline{uv}^{-1}.\overline{u}.\overline{v}.$$

\subsubsection*{Equivalent extensions}\label{equivext}
Let $E(\xi,\varphi)$ be an extension of $G$ by $H$ and let $\tau:=\{\overline{h} :
\,h\in H\}$ be its canonical transversal. A change of coset representatives for $G$ in $E(\xi,\varphi)$
is described by $\overline{h}\to\tilde{h}:=\overline{h}g_h$ for some function $h\to g_h$ from
$H$ into $G$. The new transversal $\tau':=\{\tilde{h}:h\in H\}$ gives rise to a new
pair of associated functions $(\xi ',\varphi ')$ and the extension $E(\xi ',\varphi ')$
is clearly isomorphic to $E(\xi,\varphi)$.

\textbf{Definition }. Let $E_1:=E(\xi,\varphi)$ and $E_1':=E(\xi ',\varphi ')$ be extensions
of $G$ by $H$ with canonical transversal  $\tau:=\{\overline{h} :\,h\in H\}$
and $\tau':=\{\tilde{h}:h\in H\}$ respectively. The extension $E_1'$ is \emph{\bf{equivalent}} to $E_1$ if there
 exists a function $h\to g_h$ from $H$ into $G$ such that $(\xi',\varphi')$ is the pair of
 associated functions produced in $E_1$ by the transversal $\{\overline{h}g_h:h\in H\}$.

Hence, the mapping from $E_1$ onto $E_1'$ defined by $\overline{h}g_h\to \tilde{h}$ 
(for $h\in H$) and by $g\to g$ ($g\in G$) is an isomorphism.
Further details can be found in \cite{Suzuki_vol1} pages 195-199.

% Another equivalent definition is that $E(\xi,\varphi)$ and $E(\xi ',\varphi ')$ are
% equivalent if and only if there exists a function $h\to g_h$ from $H$ into $G$ such that
%the mapping defined by $g\to g$ and $\overline{h}g_h\to\tilde{h}$ ($g\in G,\,h\in H$) is
%an isomorphism from $E(\xi,\varphi)$ onto $E(\xi ',\varphi ')$.
%Further details can be found in \cite{Suzuki_vol1} page 195-199.
% Two factor sets are equivalent if

% by a function $g\to g_h$ from $G$ to $H$ such that the new transversal is
% $\tau'$ is described by $:=\{\tilde{h}:\}$ give function give rise to a pair $(\xi ',\mu ')$
%of associated functions which we proceed to compute...
%...If we change the coset representatives of $G$ in $E$.

\subsection{Cohomology and extensions of abelian groups}
%%Define $Z^2$, $B^2$ and $H^2$ : not as differential form but has a function corresponds
%%to an extension. Hall pg 237 and Suzuki 1 page ?;Rothman is the best\\
Let $A$ be an abelian group, let $H$ be a group and let $E(\xi,\varphi)$ be an extension
of $A$ by $H$. Since conjugation is trivial in an abelian group, the automorphism $\xi (h)$ of $A$
does not depend on a particular choice of coset representative and the first Schreier
condition means that $\xi:H\to Aut(A)$ is an homomorphism. Therefore, if an homomorphism
$\xi:H\to Aut(A)$ is given, a pair of functions $(\xi,\varphi)$ defines an extension if
and only if $\varphi:H\times H\to A$ is a factor set function (defined by relation
(\ref{schreier2}), the second Shreier condition).
There is thus a correspondence between the extensions of $A$ by $H$ associated with
$\xi$ and the set $Z_{\xi} ^2 (H,A)$ of all factor sets for $\xi$.
This notation come from Cohomology theory where equation
\ref{schreier2} for factor sets corresponds exactly to the definition of a $2$-cocycle.
% The set $Z_{\xi} ^2 (H,A)$ is not empty since the trivial function that
% not usefull : it works also for non-abelian group if there exist a phi:H-->Aut(G) that
% allows to do a semi-direct product, but then not every coupling gives a phi if non abelian
We now briefly outline the link between Cohomology theory and extensions of abelian groups.
The proofs can be found in \cite{Rothman} (chapter 5, page 141 to 144).

An important property is that $Z_{\xi} ^2 (H,A)$ is an abelian group for the usual sum of
functions. Equivalence of extensions of abelian groups, can be reformulated in terms of cohomological
 notions as follows.
A function $b:H\times H\to A$ is a \emph{2-coboundary} for $\xi$ if and only if there exists
a function $h\to a_h$ from $H$ to $A$ such that
$$b(h_1,h_2)=a_{h_2} ^{\xi (h_1)} - a_{h_1 h_2} + a_{h_1}.$$ The set $B_{\xi} ^2 (H,A)$ of
all coboundaries is a subgroup of $Z_{\xi} ^2 (H,A)$. A key result states that two
extensions $E(\xi,\varphi_1)$ and $E(\xi,\varphi_2)$ of $A$ by $H$ are equivalent if and only if
$\varphi_1 - \varphi_2 \in B_{\xi} ^2 (H,A)$. Therefore there is a one-to-one correspondence
between the equivalence classes of extensions of $A$ by $H$ associated with $\xi$ and
the factor group $H_{\xi} ^2 (H,A):=Z_{\xi} ^2 (H,A)/B_{\xi} ^2 (H,A)$ called
\emph{\textbf{the second cohomology group}}. The construction of extensions of an abelian group
is often referred as the determination of the second cohomology group.

%%%%%%%%%%%%%%%%%%%%%%%%%%%%%%%%%%%%%%%%%%%%%%%%%
\section{Reducing the action on $G$}

\subsection{Couplings and $G$-isomorphism}
Let $E$ be an extension of a group $G$ by $H$.
In subsection \ref{couplingbasic}, we have described how each isomorphism from $E/G$ onto
$H$ induces a well-defined homomorphism $\Phi$ from $H$ into $Out(G)$.
If $E$ is described as $E(\xi,\varphi)$, then the canonical transversal $\{\overline{h}:h\in H\}$
provides us with a canonical isomorphism $\overline{h}G\to h$ from $E/G$ onto $H$.
The corresponding homomorphism $\Phi:H\to Out(G)$ defined for every
$h\in H$ by $h\to \Phi(h)=\xi (h)Inn(G)$ is called \emph{\textbf{the coupling of $E(\xi,\varphi)$}}(see \cite{Robinson}).
More generally we call coupling, any homomorphism from $H$ into $Out(G)$.

%\textbf{Definition} Let $E(\xi,\varphi)$ be an extension of $G$ by $H$. The homomorphism
%$\Phi:H\to Out(G)$ defined for every $h\in H$ by $h\to \xi (h)Inn(G)$ is called
%\emph{\textbf{the coupling of $E(\xi,\varphi)$}}.
If as in the previous section, we want to construct an extension of $G$ by $H$ from
given automorphisms $\{\xi (h): h\in H\}$
 of $G$ and from a factor set associated with them, we cannot choose the
automorphisms $\xi (h)$ anyhow. A homomorphism $\Phi:H\rightarrow Out(G)$ must be chosen
first and then, each automorphism $\xi (h)$ must be chosen in $\Phi(h)$.\\

%%%MMMMMMMMMMMMMMMMMMMM\\
%This requirement together with the ask for factor sets to belong to the good element
%%of $Inn(G)$ is equivalent to first schreier's condition.(see ?suz,Hall?).
%%%%%%%%%%%%%%%%%%%%%%%%%%%%%%
%\textbf{Definition}: Let $E_{1}\unrhd G\,$ and $E_{2}\unrhd G\,$ be  extensions
%of $G$. An isomorphism $i:E_{1}\rightarrow E_{2}$ such that $i(G)=G$, is
%called a \emph{$G-isomorphism$}.

%%%%%MMMM try $\alpha|_{E/G}$ rather than $\alpha|_{H}$ MMMMM\\
%%%%%\index{$\alpha|_{E/G}$ rather than $\alpha|_{H}$}
For $i=1,2$, let $E_i:=E(\xi_i,\varphi_i)$ be extensions of $G$ by $H$ and let
 $\{\overline{h_{i}} :\, h\in H\}$ be the canonical transversal to $G$ in $E_i$.
Let $\alpha$ be a $G$-isomorphism from $E_{1}$ onto $E_{2}$. It induces an
automorphism $\pi =\alpha|_{G}$ of $G$ and an automorphism $\omega:=\alpha|_{E/G}$ of $H$
described for every $h\in H$ by the sequence
$$h\to \overline{h_{1}}G\to \alpha(\overline{h_{1}}G)=\overline{h^{\omega}_{2}}G\to h^{\omega} \quad for\, h\in H.$$
%%\index{change the word action by coupling}
\begin{proposition}\label{pin1}For $i=1,2$ let $E_{i}$ be an extension of $G$ by $H$
with coupling $\Phi_{i}:H\rightarrow Out(G)$. Let $\alpha:E_{1}\to E_{2}$
be a $G$-isomorphism that induces $\omega \in Aut(H)$ and such that
 $\alpha|_{G}=\pi\in Aut(G)$. Then \\
 $$\Phi_{2}(h^{\omega})=\Phi_{1}^{\pi}(h),\forall\, h\in H.$$
\end{proposition}
\emph{Proof}:
Let $f_{t}\in Aut(G)$(respectively $f'_{t}$) be the conjugation by
 $t \in E_{1}$(respectively $E_{2}$) on $G$.
%let us chose a transversal $\{\overline{h_{i}} :\, h\in H\}$ of $G$ in $E_{i}$.
% such that $G\overline{h_{i}}\rightarrow h$ is an isomorphism from $E_{i}/G$ onto $H$.
By Lemma \ref{fpi}, we have
 $f'_{\alpha(t)}=\pi^{-1} f_{t}\pi =f_{t}^{\pi}\,$(*).\\
  For $h\in H$, if $t\in \overline{h_{1}}G $, then $\alpha (t)\in \overline{h^{\omega}_{2}}G$. Since
$f_{\overline{h_{1}}G}=f_{\overline{h_{1}}Inn(G)}$, it is by definition the element
  $\Phi_{1}(h)$ of $Out(G)$ and similarly $f'_{\alpha(t)}=f'_{\overline{h^{\omega}_{2}}G}=\Phi_{2}(h^{\omega})$. %%%\\
Now, equation (*) becomes $\Phi_{2}(h^{\omega})=\Phi_{1}^{\pi}(h)\qquad\Box$. \\
% % which is equivalent to $\Phi_{2}(h)=\Phi_{1}^{\pi}(\omega(h)),\, h\in H$ if
%%  $\omega=\gamma^{-1}\qquad\Box$. \\

Note that equation $\Phi_{2}(h^{\omega})=\Phi_{1}^{\pi}(h)$ implies that
$\Phi_{2}(H)=\Phi_{1}^{\pi}(H)$ and so, the subgroups $\Phi_{i}(H)$ are conjugate  in $Out(G)$ ($i=1,2$).
Moreover, $(Ker\Phi_{1})^{\omega}=Ker\Phi_{2}$, since $k\in Ker\Phi_{1}\Leftrightarrow\Phi_{1}(k)=1
\Leftrightarrow\Phi_{1}^{\pi}(k)=1\Leftrightarrow\Phi_{2}(k^{\omega})=1
\Leftrightarrow k^{\omega}\in Ker\Phi_{2}$.
We obtain the following corollary that can be used to prove quickly that two extensions
are not G-isomorphic.
% Note that the equation $\Phi_{2}(h)=\Phi_{1}^{\pi}(\omega(h))$ implies that
% $\Phi_{2}(H)=\Phi_{1}^{\pi}(H)$ and $\Phi_{i}(H)$ are conjugate subgroup of $Out(G)$ ($i=1,2$).
% Moreover, $\omega(Ker\Phi_{1})=Ker\Phi_{2}$, since $k\in Ker\Phi_{2}\Leftrightarrow\Phi_{2}(k)=1
% \Leftrightarrow\Phi_{1}^{\pi}(\omega(k))=1\Leftrightarrow\Phi_{1}(\omega(k))=1
%\Leftrightarrow \omega (k)\in Ker\Phi_{1}$.
% We get the following corollary that can be used to prove quickly that two extensions
% are not G-isomorphic.
\begin{corollary}\label{corpin1} For $i=1,2$ let $E_{i}$ be an extension of $G$ by $H$ with coupling
 $\Phi_{i}:H\rightarrow Out(G)$. If the $E_{i}$'s are $G-isomorphic$  then
 \begin{enumerate}
 \item the $\Phi_{i}(H)$ are conjugate subgroups of $Out(G)$
 \item the subgroups $ker\Phi_{i}\unlhd H$ are in the same orbit under the action of $Aut(H)$.
 \end{enumerate}
 \end{corollary}
                      %%%%%%%%%%%%       %%%%%%%%%%%%      %%%%%%%%%%
\subsection{The action of $Aut(G)\times Aut(H)$ on the couplings.}\label{orbcoupling}
In order to classify all the extensions of a group $G$ by a group $H$, we have to list
all the possible homomorphisms $\Phi:H\rightarrow Out(G)$. We would like to partition
such list into equivalence classes such that homomorphisms from distinct
classes %give extensions that are not G-isomorphic.
correspond to extensions that are not $G$-isomorphic.

For this purpose we define a right action $\mathcal{A}$ of $Aut(G)\times Aut(H)$ on the set $F$
of all homomorphisms from $H$ into $Out(G)$ as follows :

% Let $F$ be the set of all homomorphisms from $H$ into $Out(G)$.
% Let $A=Aut(G)\times Aut(H)$. We define an action of $A$ on $F$ as follows :
\begin{equation}  \label{eqaction1}
\mathcal{A}(\pi,\omega):\Phi_{1}\rightarrow \Phi_{2}\quad \textrm{such that}\,\,\Phi_{2}(h^{\omega})=\Phi_{1} ^{\pi}(h)
  \end{equation}
%$$(\pi,\omega):\Phi_{1}\rightarrow \Phi_{2}\quad \textrm{such that}\,\,\Phi_{2}(\omega (h))=\Phi_{1} ^{\pi}(h)$$
for $(\pi,\omega)\in Aut(G)\times Aut(H)$,  for $\Phi_{1}\in F$ and for every $h\in H$.
Note that in order to check that this defines an action, it is equivalent  to define
$\Phi_{2}(h)=\Phi_{1} ^{\pi}(h^{\omega^{-1}})$. Obviously, since $\Phi^{\tilde{\pi}}=\Phi^{\pi}$
for every $\tilde{\pi}\in \pi Inn(G)$, this also defines and action of $Out(G)\times Aut(H)$ on $F$.
Computing the orbits of such an action is a quite classical problem that has arisen
since a long time, in different group classification problems such as \begin{itemize}
\item Semi-direct products. Let $\Phi:H\to Aut(G)$ be an homomorphism. Proposition \ref{semiso}
shows that the action defined by \ref{eqaction1} classifies semi-direct products $H\ltimes_{\Phi}G$
up to $G$-isomorphism that preserve $H$.
\item Subdirect products. Let $G$ be a group such that $Z(G)=1$.  We have proved in \cite{Archermem}
that there is a one-to-one correspondence between the orbits of $\mathcal{A}$ and the $G$-isomorphism
 classes of extensions of $G$ by $H$. %% equation \ref{eqaction1} classifies up to $G$-isomorphism, the
\end{itemize}
%%%%%MMMMM et autre classif de subdirect product for $6A_{6}$.
%%%%%\index{citer classif de $6A_{6}$}.

The determination of the orbits under the action $\mathcal{A}$ leads to computation of double cosets
into $Aut(\Phi (H))$ (see subsection \ref{countorbits}).  A very interesting historical overview on the applications
of this type of action to various classification problems can be found in \cite{Laue}. The same reference
describes various algorithms to compute efficiently double cosets.

Now,  an equivalent formulation of Proposition \ref{pin1} is the following one :
 \begin{center}\emph{Two extensions associated with different $\mathcal{A}$-orbits are not G-isomorphic.}\end{center}
Next, we will prove in Proposition \ref{newext} that
\begin{center}\emph{if $\Phi_{1}$ and $\Phi_{2}$ belong to the same  $\mathcal{A}$-orbit, then
 every extension with coupling $\Phi_{2}$ is G-isomorphic to an extension with coupling
 $\Phi_{1}$}.\end{center}
Therefore, if we aim to classify extensions up to G-isomorphism it is
 sufficient to choose for each orbit, a representative $\Phi$ and to construct all the
 extensions with coupling $\Phi$ (if they exist). Then, we only need to consider
 $G$-isomorphisms between extensions that have the same coupling $\Phi$.

%%THIS IS THE SAME AS ROBINSON ACTION OF PAGE 68
%%\subsubsection{citation of Robinson pg 68}
%%%\index{compare my method to robinson one, explain the differences}
%%MMMM I HAVE TO EXPLAIN THE DIFFERENCE BETWEEN MY METHOD AND THIS OF BETTINA OR ROBINSON
%%%OTHERWISE BETTINA THINKS THAT I COPIED IT FROM SOME OTHER AUTORS. BETTINA DID IT
%%%IN ITS ARTICLES IN ONLY 4 SENTENCES THE DIFFERENCE BETWEEN ITS METHOD AND THIS OF LAUE
%AND OF ROBINSON\\
%MMMMMMMMMMMMMMMMM\\
%"Define a left action of $Aut(H)$ and a right action of $Aut(G)$ on the set of functions
% from $H\times H\to G$ by the rules. Let $\omega\in Aut(H)$ and $\pi\in Aut(G)$.
% Let $\tilde{\pi}$ be the conjugation by $\pi$ in $Aut(G)$. If we
% apply schreier's conditions we see that if $(\xi,\varphi)$ is an associated pair of
% functions for an extension $E$, then $(\xi\tilde{pi},\varphi\pi)$ is also an associated pair
% of functions for an extension and we find that $(\omega\xi,\omega\varphi)$ is also an
%associated pair of functions for an extension."\\

%%%%%%%%%%%%%%%%%%%%%%%%%%%%%%%%%%%%%%%
\subsubsection*{Generating extensions with another coupling}
Let $\omega\in Aut(H)$ and $\pi\in Aut(G)$. For every extension $E_1$ of $G$ by $H$,
with coupling $\Phi$, we generate an extension $E_2$ with coupling $h\to\Phi^{\pi} (h^{\omega^{-1}} )$
 ($h\in\, H$) that is $G$-isomorphic to $E_1$.
Similar ideas can be found in the literature. Robinson describes in \cite{Robinson} an
 action $\Gamma$ of $Aut(H)\times Aut(G)$ on the set $S$ of pairs of functions $(\xi,\varphi)$
 associated with an extension of $G$ by $H$ and we show here that this action generates
 $G$-isomorphic extensions. For the case where $G$ is elementary abelian, such a method has been
 used by Eick and Besche in \cite{Eick_Ulrich} to solve the $G$-isomorphism problem with the extra
 assumption that $H$ is finite.  %%and where $G$ is elementary abelian.
They proved, in that case, that each $G$-isomorphism class is the
union of the equivalence classes of extensions $\mathcal{E}(x)$, where $x$ ranges over an
orbit of $\Gamma$. %%%We generalize here this process to any groups $H$ and $G$.\\
%\begin{theorem}\label{newext} Let $\xi:H\to Aut(G)$ and $\varphi:H\times H\to G$ be a
%pair of associated functions for an extension $E_1$ of $G$ by $H$.
%For every $\omega \in Aut(H)$ and $\pi \in Aut(G)$ , the functions $\xi':h\to \xi (h^{\omega^{-1}})^{\pi}$ and
%$\varphi':(h_1,h_2)\to \varphi (h_1^{\omega^{-1}},h_2^{\omega^{-1}})^{\pi}$,
% describe an extension $E_{2}$ that is $G$-isomorphic to $E_1$.
%\end{theorem}

\begin{theorem}\label{newext} Let $E(\xi,\varphi)$ be an extension of $G$ by $H$.
For every $\omega \in Aut(H)$ and every $\pi \in Aut(G)$ ,% the functions $\xi':h\to \xi (h^{\omega^{-1}})^{\pi}$ and
% $\varphi':(h_1,h_2)\to \varphi (h_1^{\omega^{-1}},h_2^{\omega^{-1}})^{\pi}$,
the functions $\xi' (h):=\xi (h^{\omega^{-1}})^{\pi}$ and
$\varphi'(h_1,h_2):=\varphi (h_1^{\omega^{-1}},h_2^{\omega^{-1}})^{\pi}$,
 describe an extension $E(\xi',\varphi')$ that is $G$-isomorphic to $E(\xi,\varphi)$.

For $h\in H$, let $\{\overline{h}\}$ (respectively $\{\tilde{h}\}$) be the canonical
transversal of $E(\xi,\varphi)$ (respectively $E(\xi',\varphi')$). Then the mapping defined
by $\overline{h}\to\widetilde{h^{\omega}}$ and $g\to g^{\pi}$ ($g\in G$) is a $G$-isomorphism from $E(\xi,\varphi)$
onto $E(\xi',\varphi')$.
\end{theorem}

\emph{Proof }. As before, for $u,v$ in $H$ we denote the value of the factor
 set $\varphi$ on $(u,v)$ by $[u,v]$.% and since $\xi (u)$ is the conjugation
 %by $\overline{u}$ on $G$, we write $g^{\overline{u}}$ for $g^{\xi (u)}$ ($g\in G$).
We can view $E$ as  the group defined by the multiplication
$(u,a).(v,b)=(uv,[u,v]a^{\xi (v)}b) \quad u,v \in H \,and\,\,a,b\in G $ on the cartesian product  of set $H$ and
set $G$. Let $j:u\rightarrow u_{j}$ be an isomorphism from $H$ onto a group $j(H)$ and
$i:a\rightarrow a_{i}$ be an isomorphism from $G$ onto a group $i(G)$. Let $E_{2}$
be the cartesian product of set $j(H)$ and set $i(G)$. The canonical bijection
$\phi: (u,a)\rightarrow (u_{j},a_{i})$ from $E$ to $E_{2}$ can be used to define a
product between two elements $\phi (x)$ and $\phi (y)$ of $E_{2}$ as the image
of the product in $E$: hence,
$$\phi (x).\phi (y)=\phi(x.y) \iff (u_{j},a_{i}).(v_{j},b_{i})=((uv)_{j},([u,v]a^{\xi (v)}b)_{i}).$$
This product is completely defined on $E_2$ because $\phi$ is surjective
 on $E_2$.
Does this make $E_{2}$ a group ? It does because $\phi$ preserves product
 so that any equation in $E$ gives an equivalent equation in $E_{2}$. For instance
 associativity (Schreier's second condition), in $E_{2}$ is obtained as the image
 of equation  $(x.y).z=x.(y.z)$ in $E$. The image $\phi (1_{H},1_{G})$ of the unit will be
 the unit $1_{E_{2}}$of $E_{2}$ and the inverse of $\phi (x)$ will be $\phi(x^{-1})$
 since $x.x^{-1}=1_{E}\Rightarrow \phi (x).\phi (x^{-1})=1_{E_{2}}$. Thus $\phi$ is
 an isomorphism from the extension $E$ of $G$ onto the extension $E_{2}$ of $i(G)$.

For $u\in H$, let us choose in $E_2$ the transversal $\tau'=\{\widetilde{u_{j}}:=(u_{j},1)\}$
which is the image under $\phi$ of the canonical transversal $\tau=\{\overline{u}:=(u,1)\}$ in $E$.
What are the automorphisms $\xi' (u_j):a_{i}\rightarrow a_{i}^{\xi' (u_j)}$ and the factor set
 $\varphi':(u_j,v_j)\to [u_{j},v_{j}]_{2}$ associated with $\tau'$ in $E_{2}$ ?

 First remind that these are defined by equations \ref{rule1} and \ref{rule2}.
Applying $\phi$ on equation $a.\overline{u}=\overline{u}a^{\xi (u)}$ in $E$ implies
$a_{i}.\widetilde{u_{j}}=\widetilde{u_{j}}(a^{\xi (u)})_{i}$ in $E_2$ and since by
 definition $a_{i}.\widetilde{u_{j}}=\widetilde{u_{j}}a_{i}^{\xi' (u_j)}$, we obtain 
 \beq \label{newxit}
a_{i}^{\xi' (u_j)}=(a^{\xi (u)})_{i}\quad a\in G,\,\,u\in H.
\eeq
Similarly, applying $\phi$ on equation $\overline{u}.\overline{v}=\overline{uv} [u,v]$
 implies $\widetilde{u_{j}}.\widetilde{v_{j}}=\widetilde{(uv)_{j}} [u,v]_{i}$
 that must be equal to $\widetilde{u_{j}v_{j}}.[u_{j},v_{j}]_{2}$ we obtain
\beq \label{newphit}
 [u_{j},v_{j}]_{2}=[u,v]_{i}\quad u,v\in H.
\eeq

Finally let us choose $j=\omega\in Aut(H)$ and $i=\pi\in Aut(G)$ (thus $u_j=u^{\omega}$ and
$g_i=g^{\pi}$). We obtain an extension $E_{2}=E(\xi',\varphi')$ of $G=\pi (G)$ by
$H=\gamma (H)$ that is $G$-isomorphic to $E=E(\xi,\phi)$. By \ref{newphit}, we have
$\varphi'(u^{\omega},v^{\omega})=\varphi (u,v)^{\pi}$ and so
$$\varphi' (u,v)=\varphi'((u^{\omega ^{-1}})^{\omega},(v^{\omega ^{-1}})^{\omega})=
\varphi (u^{\omega ^{-1}},v^{\omega ^{-1}})^{\pi}.$$

similarly, for $g\in G$ and $u\in H$, by \ref{newxit} we obtain
$(g^{\pi})^{\xi ' (u^{\omega})}=(g^{\xi (u)})^{\pi}$ so that
$\xi '(u^{\omega})=\pi ^{-1}\xi (u)\pi$, whence
$$\xi '(u)=\xi '((u^{\omega^{-1}})^{\omega})=\pi^{-1}\xi (u^{\omega^{-1}})\pi .$$

We recall that for $u\in H$ and $g\in G$, the $G$-isomorphism $\phi$ from $E$ onto $E_2$ is
defined by $\overline{u}\to \widetilde{u_j}$ and $g\to g_i$.
Hence, there is a $G$-isomorphism from $E(\xi ,\varphi)$ onto $E(\xi',\varphi')$ that is
defined by $\overline{u}\to \widetilde{u^{\omega}}$ and $g\to g^{\pi}$.
$\Box$

%%%%%%%%%%%%%%%%%%%%%%%%%%%%%%%%%%%%%%%

\subsection{Determination of the orbits}\label{countorbits}
%%The formulation must be not only for the case $Z(G)=1$. J'en ai besoin aussi pour d' autre classif de
%%%subdirect product for $6A_{6}$. Dire que ca classe les nonsolvable groups such that $Z(G)=1$.
%% oui mais pas le temps de faire 6A6-->le faire en gap avec reduced pcgs
%%%Rem: je n'ai pas besoin de citer mon memoire car Proposition \ref{newext} a pour corollaire de de prouver
%%%( de manière détournée) le cas $Z=1$. Le dire en citant suzuki.
As explained in the previous section, the first step in order to classify up to $G$-isomorphism
the extensions of a group $G$, is to determine representatives of the orbits under the action
$\mathcal{A}$ of $Aut(H)\times Aut(G)$ on the couplings.

Note that In the case where $Z(G)=1$, this step is the
only one to be performed. Indeed, if $Z(G)=1$, for each coupling $\Phi:H\to Out(G)$ there is a unique
extension of $G$ by $H$ associated with $\Phi$ and it is isomorphic to the subgroup $T_{\Phi}$ of $H\times Aut(G)$
$$T_{\Phi}:=\{(h,\pi): h\in H,\,\pi\in Aut(G)\,|\, \pi\in \Phi(h)\}.$$
A proof can be found in \cite{Suzuki_vol1} page 196 or in subsection \ref{extmodZ}.
Since $T_{\Phi}$ is unique, by the results of the previous section,  if $Z(G)=1$, there is a
one-to-one correspondence between the orbits of $\mathcal{A}$ and the $G$-isomorphism classes
of extensions of $G$.

In \cite{Archermem} we described, for any given groups $G$ and $H$, the following process to obtain
exactly one coupling $\Phi$ in each orbit.
\begin{enumerate}
\item For each conjugacy class $\mathcal{C}$ of subgroups in $Out(G)$, choose a representative $V_c$.
\item For each $V_c$,  let $\mathcal{K}_{\mathcal{C}}$ be the set of all normal subgroups $K$ of $H$ such that there
exists an isomorphism  from $H/K$ onto $V_c$.  If $\omega$ is an automorphism of $H$ and if
$K\in \mathcal{K}_{\mathcal{C}}$, then $K^{\omega}$ also belongs to $\mathcal{K}_{\mathcal{C}}$.
Hence, $Aut(H)$ acts on $\mathcal{K}_{\mathcal{C}}$. For each orbit $O$ under the latter action,
 choose a representative  $K_{c,O}$ and an isomorphism $i_{c,O}$ from
 $H/K_{c,O}$  onto $V_c$.
  \item Let $\mathcal{P}$ be the set of pairs $(V_c,K_{c,O})$ and  consider the couplings $\Phi:H\to Out(G)$
  such that $\Phi(H)=V_c$ and $Ker \Phi=K_{c,O}$.  By Corollary \ref{corpin1}, two couplings that
  correspond to two different pairs in  $\mathcal{P}$ are not in the same $\mathcal{A}$-orbit.
\item Let $(V_c,K_{c,O})$ be a pair in $\mathcal{P}$. Every automorphism $\omega$ of $H$ that
preserves $K_{c,O}$ induces an automorphism
$i_{c,O}^{-1} \omega i_{c,O}$ of $V_c$ and let $\Omega_{c,O}$ be the subgroup of those induced automorphisms.
 The normalizer of $V_c$ in $Out(G)$ induces by conjugation on $V_c$ a subgroup $N_c$ of automorphisms
 of $V_c$.
\item For the pair $p=(V_c,K_{c,O})$,  we already have a coupling described by
$i_{c,O}:H/K_{c,O} \to V_c$. Then, the other couplings attached to this pair are $\Phi_{x}:=i_{c,O}x$
where $x$ is any automorphism of $V_c$. Two couplings $\Phi_{x}$ and $\Phi_{y}$ are in the same
$\mathcal{A}$-orbit if and only if $y$ belongs to the double coset $\Omega_{c,O}xN_c$
(\cite{Archermem} page 36 or \cite{Laue} page 6). If $\mathcal{D}(p)$ denotes the set of such double cosets
in $Aut(V_c)$ then the $\mathcal{A}$-orbits are in one-to-one correspondence with the union set
$$\bigcup_{p\in\mathcal{P}}\mathcal{D}(p).$$
 \end{enumerate}
\subsubsection{Tables}
We have implemented in GAP 4.3 this algorithm to compute $\mathcal{A}$-orbits for a given group $H$ and
a given group $Out(G)$. We consider every group $H$ of order less than $64$ in the Small Groups Library and
 every group $Out(G)$ (of order less or equal to $24$) that
corresponds to a centerless perfect group of order less than $61,440$. We recall that $61,440$ is the smallest
order for which the perfect groups are not completely known (\cite{Holt_Plesken} page 261).
The groups $Out(G)$ are denoted by their Id in the Small Groups Library. Up to order $29,120$, the only
centerless perfect groups whose outer automorphism group has order above 24 are
$G=PerfectGroup(7500,1)$ and $G=PerfectGroup(15360,3)$. Excepted for these 2 groups, our table enumerate
all the extensions of such centerless perfect groups by a group of order less than 64
(for $|Out(G)|=2$ see subsection \ref{tableout=2}).  For instance, one can
 deduce from \cite{Holt_Plesken} that for $G=PerfectGroup(30720,11)$, the Id of $Out(G)$ is $[16,13]$.
 A look in our table at the intersection of column $[16,13]$ and line $32$ shows that there are, up to
 isomorphism, $2010$ extensions of $G$ by a group of order $32$.
We have also used this table, as the first step for classifying extensions of an arbitrary perfect group $P$
such that $|P/Z(P)|\leq 29,120$.

On a 800 MHz AMD Duron,  with  200 Mega Ram dedicated to GAP 4.3, it took 620 minutes to compute
the whole table. In particular, the case $|H|=32$ took 235 minutes, while the case  $|H|=48$ took 167 minutes.

%We give In partcular it gives the number of extensions
\index{false : i have forgotten Out(Perfect(32256,2))=24,13}
%Dire que out=2 est déja fait dans zphi-classes.

%\input{zis1tabtex}
%\documentclass[11pt]{article}
%	\usepackage{umlaut}
%	\usepackage{color}
	%\oddsidemargin -0.54cm
%\evensidemargin -0.1cm
%\oddsidemargin -0.1cm
	\textwidth 17cm
	\parskip 1em
%\begin{document}

\begin{flushleft}

\begin{minipage}{\linewidth} % pour que la syntaxe soit correcte il veut un argument
% apres minipage  ({10cm} ou {0.5\linewidth} par exemple ) mais il n'en tient pas compte :
\vspace{-2.2cm}

%\begin{tabular}{|l|l|l|l|l|l|l|l|l|l|l|l|l|l|l|l|}\hline
%% l'environement \tabular* éconne complètement
%% pour donner la largeur d'une colonne remplacer {|l|l|} ou {|c|c|} par {|l|p{0.55cm}|
% astuce trouvée dans file:/usr/share/texmf/doc/latex/general/essential.dvi page 12
%%% pour donner la hauteur d'un ligne apres le \\ de fin de ligne faire
%%%%%%%%  SAMUEL M'A DONNE LE TRUC : A LA FIN D'UN LIGNE (APRES LE \\) PLACER
%%%%%%%% [-0.5ex] [-1ex] POUR RETRECIR L'ESPACE ENTRE LES LIGNES [0.5ex] [2ex] POUR ELARGIR
%\settoheight{0.6\baselineskip}  % ne change rien
% \setlength{\baselineskip}{0.5\baselineskip}  % ne change rien non plus
{\footnotesize
\begin{tabular}{|p{0.85cm}||p{0.45cm}|p{0.45cm}|p{0.55cm}|p{0.45cm}|p{0.45cm}|p{0.55cm}|p{0.70cm}|p{0.45cm}|p{0.62cm}|p{0.55cm}|p{0.55cm}|p{0.55cm}|p{0.55cm}|p{0.75cm}|p{0.75cm}|}\hline
  \textsf{Out(G)}
	&[3,1] % &\textsf{[3,1]}
	&[4,1] % &\textsf{[4,1]}
	&[4,2] %&\textsf{[4,2]}
	&[5,1] %&\textsf{[5,1]}
	&[6,2] %&\textsf{[6,2]}
	&[8,3] %&\textsf{[8,3]}
	&[8,5] %&\textsf{[8,5]}
	&[9,1] %&\textsf{[9,1]}
	&[10,1] %&\textsf{[10,1]}
	&[12,4] %&\textsf{[12,4]}
	&[12,5] %&\textsf{[12,5]}
	&[16,6] %&\textsf{[16,6]}
	&[16,8] %&\textsf{[16,8]}
	&[16,13] %&\textsf{[16,13]}
	&[24,12] %&\textsf{[24,12]}
\\[-0.6ex]\hline
	 \textsf{order of H}
	&
	&
	&
	&
	&
	&
	&
	&
	&
	&
	&
	&
	&
	&
	&
\\[-0.6ex]\hline
	 \textsf{1}
	&\textsf{1}
	&\textsf{1}
	&\textsf{1}
	&\textsf{1}
	&\textsf{1}
	&\textsf{1}
	&\textsf{1}
	&\textsf{1}
	&\textsf{1}
	&\textsf{1}
	&\textsf{1}
	&\textsf{1}
	&\textsf{1}
	&\textsf{1}
	&\textsf{1}
\\[-0.6ex]\hline
	 \textsf{2}
	&\textsf{1}
	&\textsf{2}
	&\textsf{4}
	&\textsf{1}
	&\textsf{2}
	&\textsf{4}
	&\textsf{8}
	&\textsf{1}
	&\textsf{2}
	&\textsf{4}
	&\textsf{4}
	&\textsf{3}
	&\textsf{3}
	&\textsf{5}
	&\textsf{3}
\\[-0.6ex]\hline
	 \textsf{3}
	&\textsf{2}
	&\textsf{1}
	&\textsf{1}
	&\textsf{1}
	&\textsf{2}
	&\textsf{1}
	&\textsf{1}
	&\textsf{2}
	&\textsf{1}
	&\textsf{2}
	&\textsf{2}
	&\textsf{1}
	&\textsf{1}
	&\textsf{1}
	&\textsf{2}
\\[-0.6ex]\hline
	 \textsf{4}
	&\textsf{2}
	&\textsf{5}
	&\textsf{9}
	&\textsf{2}
	&\textsf{4}
	&\textsf{11}
	&\textsf{23}
	&\textsf{2}
	&\textsf{4}
	&\textsf{9}
	&\textsf{9}
	&\textsf{9}
	&\textsf{9}
	&\textsf{17}
	&\textsf{9}
\\[-0.6ex]\hline
	 \textsf{5}
	&\textsf{1}
	&\textsf{1}
	&\textsf{1}
	&\textsf{2}
	&\textsf{1}
	&\textsf{1}
	&\textsf{1}
	&\textsf{1}
	&\textsf{2}
	&\textsf{1}
	&\textsf{1}
	&\textsf{1}
	&\textsf{1}
	&\textsf{1}
	&\textsf{1}
\\[-0.6ex]\hline
	 \textsf{6}
	&\textsf{3}
	&\textsf{4}
	&\textsf{8}
	&\textsf{2}
	&\textsf{6}
	&\textsf{8}
	&\textsf{16}
	&\textsf{3}
	&\textsf{4}
	&\textsf{12}
	&\textsf{12}
	&\textsf{6}
	&\textsf{6}
	&\textsf{10}
	&\textsf{8}
\\[-0.6ex]\hline
	 \textsf{7}
	&\textsf{1}
	&\textsf{1}
	&\textsf{1}
	&\textsf{1}
	&\textsf{1}
	&\textsf{1}
	&\textsf{1}
	&\textsf{1}
	&\textsf{1}
	&\textsf{1}
	&\textsf{1}
	&\textsf{1}
	&\textsf{1}
	&\textsf{1}
	&\textsf{1}
\\[-0.6ex]\hline
	 \textsf{8}
	&\textsf{5}
	&\textsf{14}
	&\textsf{34}
	&\textsf{5}
	&\textsf{12}
	&\textsf{41}
	&\textsf{111}
	&\textsf{5}
	&\textsf{12}
	&\textsf{34}
	&\textsf{34}
	&\textsf{32}
	&\textsf{32}
	&\textsf{66}
	&\textsf{32}
\\[-0.6ex]\hline
	 \textsf{9}
	&\textsf{4}
	&\textsf{2}
	&\textsf{2}
	&\textsf{2}
	&\textsf{4}
	&\textsf{2}
	&\textsf{2}
	&\textsf{5}
	&\textsf{2}
	&\textsf{4}
	&\textsf{4}
	&\textsf{2}
	&\textsf{2}
	&\textsf{2}
	&\textsf{4}
\\[-0.6ex]\hline
	 \textsf{10}
	&\textsf{2}
	&\textsf{4}
	&\textsf{8}
	&\textsf{3}
	&\textsf{4}
	&\textsf{8}
	&\textsf{16}
	&\textsf{2}
	&\textsf{6}
	&\textsf{8}
	&\textsf{8}
	&\textsf{6}
	&\textsf{6}
	&\textsf{10}
	&\textsf{6}
\\[-0.6ex]\hline
	 \textsf{11}
	&\textsf{1}
	&\textsf{1}
	&\textsf{1}
	&\textsf{1}
	&\textsf{1}
	&\textsf{1}
	&\textsf{1}
	&\textsf{1}
	&\textsf{1}
	&\textsf{1}
	&\textsf{1}
	&\textsf{1}
	&\textsf{1}
	&\textsf{1}
	&\textsf{1}
\\[-0.6ex]\hline
	 \textsf{12}
	&\textsf{8}
	&\textsf{12}
	&\textsf{24}
	&\textsf{5}
	&\textsf{15}
	&\textsf{28}
	&\textsf{68}
	&\textsf{8}
	&\textsf{10}
	&\textsf{34}
	&\textsf{34}
	&\textsf{22}
	&\textsf{22}
	&\textsf{42}
	&\textsf{28}
\\[-0.6ex]\hline
	 \textsf{13}
	&\textsf{1}
	&\textsf{1}
	&\textsf{1}
	&\textsf{1}
	&\textsf{1}
	&\textsf{1}
	&\textsf{1}
	&\textsf{1}
	&\textsf{1}
	&\textsf{1}
	&\textsf{1}
	&\textsf{1}
	&\textsf{1}
	&\textsf{1}
	&\textsf{1}
\\[-0.6ex]\hline
	 \textsf{14}
	&\textsf{2}
	&\textsf{4}
	&\textsf{8}
	&\textsf{2}
	&\textsf{4}
	&\textsf{8}
	&\textsf{16}
	&\textsf{2}
	&\textsf{4}
	&\textsf{8}
	&\textsf{8}
	&\textsf{6}
	&\textsf{6}
	&\textsf{10}
	&\textsf{6}
\\[-0.6ex]\hline
	 \textsf{15}
	&\textsf{2}
	&\textsf{1}
	&\textsf{1}
	&\textsf{2}
	&\textsf{2}
	&\textsf{1}
	&\textsf{1}
	&\textsf{2}
	&\textsf{2}
	&\textsf{2}
	&\textsf{2}
	&\textsf{1}
	&\textsf{1}
	&\textsf{1}
	&\textsf{2}
\\[-0.6ex]\hline
	 \textsf{16}
	&\textsf{14}
	&\textsf{51}
	&\textsf{151}
	&\textsf{14}
	&\textsf{42}
	&\textsf{185}
	&\textsf{645}
	&\textsf{14}
	&\textsf{42}
	&\textsf{151}
	&\textsf{151}
	&\textsf{134}
	&\textsf{135}
	&\textsf{314}
	&\textsf{141}
\\[-0.6ex]\hline
	 \textsf{17}
	&\textsf{1}
	&\textsf{1}
	&\textsf{1}
	&\textsf{1}
	&\textsf{1}
	&\textsf{1}
	&\textsf{1}
	&\textsf{1}
	&\textsf{1}
	&\textsf{1}
	&\textsf{1}
	&\textsf{1}
	&\textsf{1}
	&\textsf{1}
	&\textsf{1}
\\[-0.6ex]\hline
	 \textsf{18}
	&\textsf{8}
	&\textsf{10}
	&\textsf{20}
	&\textsf{5}
	&\textsf{16}
	&\textsf{20}
	&\textsf{40}
	&\textsf{9}
	&\textsf{10}
	&\textsf{32}
	&\textsf{32}
	&\textsf{15}
	&\textsf{15}
	&\textsf{25}
	&\textsf{21}
\\[-0.6ex]\hline
	 \textsf{19}
	&\textsf{1}
	&\textsf{1}
	&\textsf{1}
	&\textsf{1}
	&\textsf{1}
	&\textsf{1}
	&\textsf{1}
	&\textsf{1}
	&\textsf{1}
	&\textsf{1}
	&\textsf{1}
	&\textsf{1}
	&\textsf{1}
	&\textsf{1}
	&\textsf{1}
\\[-0.6ex]\hline
	 \textsf{20}
	&\textsf{5}
	&\textsf{15}
	&\textsf{27}
	&\textsf{7}
	&\textsf{11}
	&\textsf{32}
	&\textsf{75}
	&\textsf{5}
	&\textsf{15}
	&\textsf{27}
	&\textsf{27}
	&\textsf{27}
	&\textsf{26}
	&\textsf{51}
	&\textsf{25}
\\[-0.6ex]\hline
	 \textsf{21}
	&\textsf{5}
	&\textsf{2}
	&\textsf{2}
	&\textsf{2}
	&\textsf{5}
	&\textsf{2}
	&\textsf{2}
	&\textsf{5}
	&\textsf{2}
	&\textsf{4}
	&\textsf{5}
	&\textsf{2}
	&\textsf{2}
	&\textsf{2}
	&\textsf{4}
\\[-0.6ex]\hline
	 \textsf{22}
	&\textsf{2}
	&\textsf{4}
	&\textsf{8}
	&\textsf{2}
	&\textsf{4}
	&\textsf{8}
	&\textsf{16}
	&\textsf{2}
	&\textsf{4}
	&\textsf{8}
	&\textsf{8}
	&\textsf{6}
	&\textsf{6}
	&\textsf{10}
	&\textsf{6}
\\[-0.6ex]\hline
	 \textsf{23}
	&\textsf{1}
	&\textsf{1}
	&\textsf{1}
	&\textsf{1}
	&\textsf{1}
	&\textsf{1}
	&\textsf{1}
	&\textsf{1}
	&\textsf{1}
	&\textsf{1}
	&\textsf{1}
	&\textsf{1}
	&\textsf{1}
	&\textsf{1}
	&\textsf{1}
\\[-0.6ex]\hline
	 \textsf{24}
	&\textsf{22}
	&\textsf{44}
	&\textsf{120}
	&\textsf{15}
	&\textsf{54}
	&\textsf{138}
	&\textsf{422}
	&\textsf{22}
	&\textsf{39}
	&\textsf{159}
	&\textsf{159}
	&\textsf{102}
	&\textsf{102}
	&\textsf{222}
	&\textsf{122}
\\[-0.6ex]\hline
	 \textsf{25}
	&\textsf{2}
	&\textsf{2}
	&\textsf{2}
	&\textsf{4}
	&\textsf{2}
	&\textsf{2}
	&\textsf{2}
	&\textsf{2}
	&\textsf{4}
	&\textsf{2}
	&\textsf{2}
	&\textsf{2}
	&\textsf{2}
	&\textsf{2}
	&\textsf{2}
\\[-0.6ex]\hline
	 \textsf{26}
	&\textsf{2}
	&\textsf{4}
	&\textsf{8}
	&\textsf{2}
	&\textsf{4}
	&\textsf{8}
	&\textsf{16}
	&\textsf{2}
	&\textsf{4}
	&\textsf{8}
	&\textsf{8}
	&\textsf{6}
	&\textsf{6}
	&\textsf{10}
	&\textsf{6}
\\[-0.6ex]\hline
	 \textsf{27}
	&\textsf{13}
	&\textsf{5}
	&\textsf{5}
	&\textsf{5}
	&\textsf{13}
	&\textsf{5}
	&\textsf{5}
	&\textsf{15}
	&\textsf{5}
	&\textsf{12}
	&\textsf{13}
	&\textsf{5}
	&\textsf{5}
	&\textsf{5}
	&\textsf{12}
\\[-0.6ex]\hline
	 \textsf{28}
	&\textsf{4}
	&\textsf{11}
	&\textsf{23}
	&\textsf{4}
	&\textsf{9}
	&\textsf{27}
	&\textsf{67}
	&\textsf{4}
	&\textsf{9}
	&\textsf{23}
	&\textsf{23}
	&\textsf{21}
	&\textsf{21}
	&\textsf{41}
	&\textsf{21}
\\[-0.6ex]\hline
	 \textsf{29}
	&\textsf{1}
	&\textsf{1}
	&\textsf{1}
	&\textsf{1}
	&\textsf{1}
	&\textsf{1}
	&\textsf{1}
	&\textsf{1}
	&\textsf{1}
	&\textsf{1}
	&\textsf{1}
	&\textsf{1}
	&\textsf{1}
	&\textsf{1}
	&\textsf{1}
\\[-0.6ex]\hline
	 \textsf{30}
	&\textsf{6}
	&\textsf{8}
	&\textsf{16}
	&\textsf{6}
	&\textsf{12}
	&\textsf{16}
	&\textsf{32}
	&\textsf{6}
	&\textsf{12}
	&\textsf{24}
	&\textsf{24}
	&\textsf{12}
	&\textsf{12}
	&\textsf{20}
	&\textsf{16}
\\[-0.6ex]\hline
	 \textsf{31}
	&\textsf{1}
	&\textsf{1}
	&\textsf{1}
	&\textsf{1}
	&\textsf{1}
	&\textsf{1}
	&\textsf{1}
	&\textsf{1}
	&\textsf{1}
	&\textsf{1}
	&\textsf{1}
	&\textsf{1}
	&\textsf{1}
	&\textsf{1}
	&\textsf{1}
\\[-0.6ex]\hline
	 \textsf{32}
	&\textsf{51}
	&\textsf{238}
	&\textsf{936}
	&\textsf{51}
	&\textsf{195}
	&\textsf{1136}
	&\textsf{5534}
	&\textsf{51}
	&\textsf{195}
	&\textsf{939}
	&\textsf{939}
	&\textsf{764}
	&\textsf{767}
	&\textsf{2010}
	&\textsf{846}
\\[-0.6ex]\hline
	 \textsf{33}
	&\textsf{2}
	&\textsf{1}
	&\textsf{1}
	&\textsf{1}
	&\textsf{2}
	&\textsf{1}
	&\textsf{1}
	&\textsf{2}
	&\textsf{1}
	&\textsf{2}
	&\textsf{2}
	&\textsf{1}
	&\textsf{1}
	&\textsf{1}
	&\textsf{2}
\\[-0.6ex]\hline
	 \textsf{34}
	&\textsf{2}
	&\textsf{4}
	&\textsf{8}
	&\textsf{2}
	&\textsf{4}
	&\textsf{8}
	&\textsf{16}
	&\textsf{2}
	&\textsf{4}
	&\textsf{8}
	&\textsf{8}
	&\textsf{6}
	&\textsf{6}
	&\textsf{10}
	&\textsf{6}
\\[-0.6ex]\hline
	 \textsf{35}
	&\textsf{1}
	&\textsf{1}
	&\textsf{1}
	&\textsf{2}
	&\textsf{1}
	&\textsf{1}
	&\textsf{1}
	&\textsf{1}
	&\textsf{2}
	&\textsf{1}
	&\textsf{1}
	&\textsf{1}
	&\textsf{1}
	&\textsf{1}
	&\textsf{1}
\\[-0.6ex]\hline
	 \textsf{36}
	&\textsf{23}
	&\textsf{36}
	&\textsf{76}
	&\textsf{14}
	&\textsf{46}
	&\textsf{88}
	&\textsf{224}
	&\textsf{26}
	&\textsf{30}
	&\textsf{111}
	&\textsf{111}
	&\textsf{68}
	&\textsf{68}
	&\textsf{132}
	&\textsf{86}
\\[-0.6ex]\hline
	 \textsf{37}
	&\textsf{1}
	&\textsf{1}
	&\textsf{1}
	&\textsf{1}
	&\textsf{1}
	&\textsf{1}
	&\textsf{1}
	&\textsf{1}
	&\textsf{1}
	&\textsf{1}
	&\textsf{1}
	&\textsf{1}
	&\textsf{1}
	&\textsf{1}
	&\textsf{1}
\\[-0.6ex]\hline
	 \textsf{38}
	&\textsf{2}
	&\textsf{4}
	&\textsf{8}
	&\textsf{2}
	&\textsf{4}
	&\textsf{8}
	&\textsf{16}
	&\textsf{2}
	&\textsf{4}
	&\textsf{8}
	&\textsf{8}
	&\textsf{6}
	&\textsf{6}
	&\textsf{10}
	&\textsf{6}
\\[-0.6ex]\hline
	 \textsf{39}
	&\textsf{5}
	&\textsf{2}
	&\textsf{2}
	&\textsf{2}
	&\textsf{5}
	&\textsf{2}
	&\textsf{2}
	&\textsf{5}
	&\textsf{2}
	&\textsf{4}
	&\textsf{5}
	&\textsf{2}
	&\textsf{2}
	&\textsf{2}
	&\textsf{4}
\\[-0.6ex]\hline
	 \textsf{40}
	&\textsf{14}
	&\textsf{49}
	&\textsf{125}
	&\textsf{19}
	&\textsf{39}
	&\textsf{147}
	&\textsf{449}
	&\textsf{14}
	&\textsf{51}
	&\textsf{125}
	&\textsf{125}
	&\textsf{119}
	&\textsf{112}
	&\textsf{251}
	&\textsf{111}
\\[-0.6ex]\hline
	 \textsf{41}
	&\textsf{1}
	&\textsf{1}
	&\textsf{1}
	&\textsf{1}
	&\textsf{1}
	&\textsf{1}
	&\textsf{1}
	&\textsf{1}
	&\textsf{1}
	&\textsf{1}
	&\textsf{1}
	&\textsf{1}
	&\textsf{1}
	&\textsf{1}
	&\textsf{1}
\\[-0.6ex]\hline
	 \textsf{42}
	&\textsf{12}
	&\textsf{12}
	&\textsf{24}
	&\textsf{6}
	&\textsf{24}
	&\textsf{24}
	&\textsf{48}
	&\textsf{12}
	&\textsf{12}
	&\textsf{36}
	&\textsf{48}
	&\textsf{18}
	&\textsf{18}
	&\textsf{30}
	&\textsf{24}
\\[-0.6ex]\hline
	 \textsf{43}
	&\textsf{1}
	&\textsf{1}
	&\textsf{1}
	&\textsf{1}
	&\textsf{1}
	&\textsf{1}
	&\textsf{1}
	&\textsf{1}
	&\textsf{1}
	&\textsf{1}
	&\textsf{1}
	&\textsf{1}
	&\textsf{1}
	&\textsf{1}
	&\textsf{1}
\\[-0.6ex]\hline
	 \textsf{44}
	&\textsf{4}
	&\textsf{11}
	&\textsf{23}
	&\textsf{4}
	&\textsf{9}
	&\textsf{27}
	&\textsf{67}
	&\textsf{4}
	&\textsf{9}
	&\textsf{23}
	&\textsf{23}
	&\textsf{21}
	&\textsf{21}
	&\textsf{41}
	&\textsf{21}
\\[-0.6ex]\hline
	 \textsf{45}
	&\textsf{4}
	&\textsf{2}
	&\textsf{2}
	&\textsf{4}
	&\textsf{4}
	&\textsf{2}
	&\textsf{2}
	&\textsf{5}
	&\textsf{4}
	&\textsf{4}
	&\textsf{4}
	&\textsf{2}
	&\textsf{2}
	&\textsf{2}
	&\textsf{4}
\\[-0.6ex]\hline
	 \textsf{46}
	&\textsf{2}
	&\textsf{4}
	&\textsf{8}
	&\textsf{2}
	&\textsf{4}
	&\textsf{8}
	&\textsf{16}
	&\textsf{2}
	&\textsf{4}
	&\textsf{8}
	&\textsf{8}
	&\textsf{6}
	&\textsf{6}
	&\textsf{10}
	&\textsf{6}
\\[-0.6ex]\hline
	 \textsf{47}
	&\textsf{1}
	&\textsf{1}
	&\textsf{1}
	&\textsf{1}
	&\textsf{1}
	&\textsf{1}
	&\textsf{1}
	&\textsf{1}
	&\textsf{1}
	&\textsf{1}
	&\textsf{1}
	&\textsf{1}
	&\textsf{1}
	&\textsf{1}
	&\textsf{1}
\\[-0.6ex]\hline
	 \textsf{48}
	&\textsf{72}
	&\textsf{198}
	&\textsf{686}
	&\textsf{52}
	&\textsf{222}
	&\textsf{799}
	&\textsf{3351}
	&\textsf{72}
	&\textsf{170}
	&\textsf{856}
	&\textsf{856}
	&\textsf{549}
	&\textsf{551}
	&\textsf{1333}
	&\textsf{653}
\\[-0.6ex]\hline
	 \textsf{49}
	&\textsf{2}
	&\textsf{2}
	&\textsf{2}
	&\textsf{2}
	&\textsf{2}
	&\textsf{2}
	&\textsf{2}
	&\textsf{2}
	&\textsf{2}
	&\textsf{2}
	&\textsf{2}
	&\textsf{2}
	&\textsf{2}
	&\textsf{2}
	&\textsf{2}
\\[-0.6ex]\hline
	 \textsf{50}
	&\textsf{5}
	&\textsf{10}
	&\textsf{20}
	&\textsf{8}
	&\textsf{10}
	&\textsf{20}
	&\textsf{40}
	&\textsf{5}
	&\textsf{16}
	&\textsf{20}
	&\textsf{20}
	&\textsf{15}
	&\textsf{15}
	&\textsf{25}
	&\textsf{15}
\\[-0.6ex]\hline
	 \textsf{51}
	&\textsf{2}
	&\textsf{1}
	&\textsf{1}
	&\textsf{1}
	&\textsf{2}
	&\textsf{1}
	&\textsf{1}
	&\textsf{2}
	&\textsf{1}
	&\textsf{2}
	&\textsf{2}
	&\textsf{1}
	&\textsf{1}
	&\textsf{1}
	&\textsf{2}
\\[-0.6ex]\hline
	 \textsf{52}
	&\textsf{5}
	&\textsf{15}
	&\textsf{27}
	&\textsf{5}
	&\textsf{11}
	&\textsf{32}
	&\textsf{75}
	&\textsf{5}
	&\textsf{11}
	&\textsf{27}
	&\textsf{27}
	&\textsf{27}
	&\textsf{26}
	&\textsf{51}
	&\textsf{25}
\\[-0.6ex]\hline
	 \textsf{53}
	&\textsf{1}
	&\textsf{1}
	&\textsf{1}
	&\textsf{1}
	&\textsf{1}
	&\textsf{1}
	&\textsf{1}
	&\textsf{1}
	&\textsf{1}
	&\textsf{1}
	&\textsf{1}
	&\textsf{1}
	&\textsf{1}
	&\textsf{1}
	&\textsf{1}
\\[-0.6ex]\hline
	 \textsf{54}
	&\textsf{30}
	&\textsf{30}
	&\textsf{60}
	&\textsf{15}
	&\textsf{60}
	&\textsf{60}
	&\textsf{120}
	&\textsf{33}
	&\textsf{30}
	&\textsf{108}
	&\textsf{120}
	&\textsf{45}
	&\textsf{45}
	&\textsf{75}
	&\textsf{69}
\\[-0.6ex]\hline
	 \textsf{55}
	&\textsf{2}
	&\textsf{2}
	&\textsf{2}
	&\textsf{7}
	&\textsf{2}
	&\textsf{2}
	&\textsf{2}
	&\textsf{2}
	&\textsf{5}
	&\textsf{2}
	&\textsf{2}
	&\textsf{2}
	&\textsf{2}
	&\textsf{2}
	&\textsf{2}
\\[-0.6ex]\hline
	 \textsf{56}
	&\textsf{13}
	&\textsf{40}
	&\textsf{112}
	&\textsf{13}
	&\textsf{35}
	&\textsf{130}
	&\textsf{406}
	&\textsf{13}
	&\textsf{35}
	&\textsf{112}
	&\textsf{112}
	&\textsf{96}
	&\textsf{96}
	&\textsf{212}
	&\textsf{98}
\\[-0.6ex]\hline
	 \textsf{57}
	&\textsf{5}
	&\textsf{2}
	&\textsf{2}
	&\textsf{2}
	&\textsf{5}
	&\textsf{2}
	&\textsf{2}
	&\textsf{5}
	&\textsf{2}
	&\textsf{4}
	&\textsf{5}
	&\textsf{2}
	&\textsf{2}
	&\textsf{2}
	&\textsf{4}
\\[-0.6ex]\hline
	 \textsf{58}
	&\textsf{2}
	&\textsf{4}
	&\textsf{8}
	&\textsf{2}
	&\textsf{4}
	&\textsf{8}
	&\textsf{16}
	&\textsf{2}
	&\textsf{4}
	&\textsf{8}
	&\textsf{8}
	&\textsf{6}
	&\textsf{6}
	&\textsf{10}
	&\textsf{6}
\\[-0.6ex]\hline
	 \textsf{59}
	&\textsf{1}
	&\textsf{1}
	&\textsf{1}
	&\textsf{1}
	&\textsf{1}
	&\textsf{1}
	&\textsf{1}
	&\textsf{1}
	&\textsf{1}
	&\textsf{1}
	&\textsf{1}
	&\textsf{1}
	&\textsf{1}
	&\textsf{1}
	&\textsf{1}
\\[-0.6ex]\hline
	 \textsf{60}
	&\textsf{19}
	&\textsf{37}
	&\textsf{77}
	&\textsf{18}
	&\textsf{41}
	&\textsf{87}
	&\textsf{237}
	&\textsf{19}
	&\textsf{39}
	&\textsf{105}
	&\textsf{105}
	&\textsf{69}
	&\textsf{67}
	&\textsf{133}
	&\textsf{79}
\\[-0.6ex]\hline
	 \textsf{61}
	&\textsf{1}
	&\textsf{1}
	&\textsf{1}
	&\textsf{1}
	&\textsf{1}
	&\textsf{1}
	&\textsf{1}
	&\textsf{1}
	&\textsf{1}
	&\textsf{1}
	&\textsf{1}
	&\textsf{1}
	&\textsf{1}
	&\textsf{1}
	&\textsf{1}
\\[-0.6ex]\hline
	 \textsf{62}
	&\textsf{2}
	&\textsf{4}
	&\textsf{8}
	&\textsf{2}
	&\textsf{4}
	&\textsf{8}
	&\textsf{16}
	&\textsf{2}
	&\textsf{4}
	&\textsf{8}
	&\textsf{8}
	&\textsf{6}
	&\textsf{6}
	&\textsf{10}
	&\textsf{6}
\\[-0.6ex]\hline
	 \textsf{63}
	&\textsf{11}
	&\textsf{4}
	&\textsf{4}
	&\textsf{4}
	&\textsf{11}
	&\textsf{4}
	&\textsf{4}
	&\textsf{14}
	&\textsf{4}
	&\textsf{9}
	&\textsf{11}
	&\textsf{4}
	&\textsf{4}
	&\textsf{4}
	&\textsf{9}
\\[-0.6ex]\hline

	&
	&
	&
	&
	&
	&
	&
	&
	&
	&
	&
	&
	&
	&
	&
\\[-0.6ex]\hline
	 \textsf{total}
	&\textsf{423}
	&\textsf{935}
	&\textsf{2721}
	&\textsf{349}
	&\textsf{990}
	&\textsf{3179}
	&\textsf{12249}
	&\textsf{437}
	&\textsf{852}
	&\textsf{3116}
	&\textsf{3146}
	&\textsf{2275}
	&\textsf{2270}
	&\textsf{5243}
	&\textsf{2580}
\\[-0.6ex]\hline
\end{tabular}
}
\end{minipage}
\end{flushleft}
%\end{document}

%%%%%%%%%%%%%%%%%%%%%%%%%%%%%%%%%%%%%%%
\section{Generating isomorphic extensions.}
%\subsection{The structure modulo $Z(G)$, of an extension of $G$.}
From now on, we will assume that
 we have chosen  in each $\mathcal{A}$-orbit one coupling as representative.
For each representative $\Phi:H\to Out(G)$, the problem is now to classify, up to $G$-isomorphism,
 all extensions with coupling $\Phi$.
 For every $h\in H$ let us choose in $\Phi (h)$ a representative $\xi (h)\in Aut(G)$.
Hence, since $\Phi (h)=\xi (h) Inn(G)$, in every extension with coupling ${\Phi}$ there exists a transversal
$\{\overline{h} : h\in H\}$ of $G$ such that the conjugation by $\overline{h}$ on $G$
is the automorphism $\xi (h)$. Thus, we obtain a description of extensions with coupling $\Phi$
as extensions $E(\xi,\varphi)$ that differ only  by their factor set $\varphi$.
Let $\varphi_1$ and $\varphi_2$ be two such factor sets for $\xi$ and let $<g>\in Aut(G)$ denote the
conjugation by an element $g\in G$.  Since for every $h_1,h_2\in H$,
$\varphi_1 (h_1,h_2)=(\overline{h_1 h_2})^{-1}\overline{h_1}\overline{h_2}$  %($i=1,2$),
we have $<\varphi_1 (h_1,h_2)>=\xi (h_1 h_2)^{-1} \xi (h_1) \xi (h_2)$ in $E(\xi,\varphi_1)$.
But in $E(\xi,\varphi_2)$, $\xi (h_1 h_2)^{-1} \xi (h_1) \xi (h_2)$ is also equal to $<\varphi_2 (h_1,h_2)>$.
Hence, the elements $\varphi_1 (h_1,h_2)$ and $\varphi_2 (h_1,h_2)$ differ by an element $z (h_1,h_2)\in Z(G)$
because they induce the same inner automorphism of $G$. MacLane-Eilenberg theorem (\ref{MacEinlen})
states that this function from $H\times H$ to $Z(G)$ is the factor set (a 2-cocycle)  of an extension
of $Z(G)$ by $H$ for the restriction of $\xi (h)$ to $Z(G)$.
Since each class of $Out(G)$ induces a unique automorphism of $Z(G)$, the restriction of $\Phi$ to $Z(G)$
is a homomorphism from $H$ into $Aut(Z(G))$ and we denote the second cohomology group associated with
this restriction as  $H_{\Phi} ^{2}(H,Z(G))=Z_{\Phi} ^{2}(H,Z(G))/B_{\Phi} ^{2}(H,Z(G))$.

\subsubsection*{Mac-Lane-Eilenberg theorem}
We follow here, the formulation of \cite{Robinson}. For a proof see \cite{kurosh} page 139-145.
\begin{theorem}\label{MacEinlen}(\cite{maclane-eilenberg})
Let $\Phi:H\to Out(G)$ be a homomorphism. Let $\xi:H\to Aut(G)$ be a function
such that $\xi (h)\in\Phi(h)$ for any $h\in H$. Then the following hold
\begin{enumerate}
\item Every extension of $G$ by $H$ with coupling $\Phi$ has an associated pair of functions
of the form $(\xi,\varphi)$.
\item If $E_1$ and $E_2$ are two such extensions with associated pairs of functions
$(\xi,\varphi_1)$ and $(\xi,\varphi_2)$ then $\varphi_2 -\varphi_1\,\in Z_{\Phi} ^{2}(H,Z(G))$.
\item $E_1$ and $E_2$ are equivalent if and only if $\varphi_2\equiv\varphi_1 \, mod\, B_{\Phi} ^{2}(H,Z(G))$.

\item Let $E$ be fixed. Then as $\zeta +B_{\Phi} ^{2}(H,Z(G))$ varies over $H_{\Phi} ^{2}(H,Z(G))$, the
extensions $E(f,\varphi_1 +\zeta)$ form a complete set of inequivalent extensions of
$G$ by $H$ with coupling $\Phi$.
\end{enumerate}
\end{theorem}

\subsection{Extensions of $G$ are known modulo $Z(G)$.}
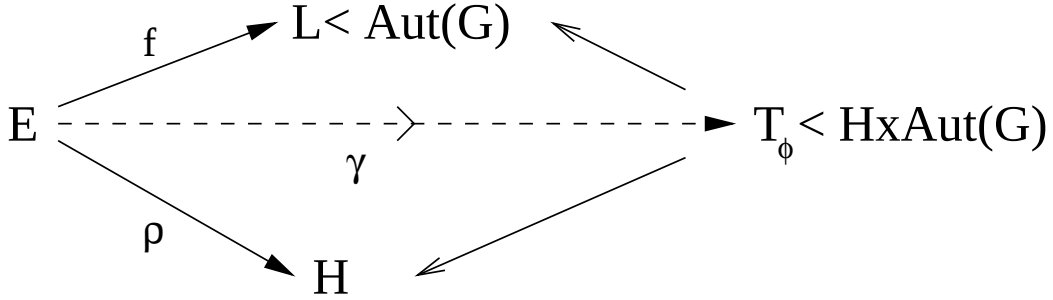
\label{extmodZ}
%\begin{wrapfigure}[9]{l}{0.35\linewidth}
%\mbox{\input{Giso1.latex}
\includegraphics[width=15cm,height=4cm]{Tphigam.eps}
%\end{wrapfigure}

Let $E$ be an extension of $G$ by $H$ described by a transversal $\tau =\{\overline{h}:h\in H\}$ of $G$.
If we combine the homomorphism $\rho :E\to H$ induced by $\tau$  with
the conjugation homomorphism $f:E\to Aut(G)$, we create a third homomorphism
$\gamma :e\rightarrow (\rho (e),f(e))$ from $E$ into $H \times Aut(G)$. The kernel
of $\gamma$ is $Z(G)=G\cap C_{E}(G)$, the intersection of the two previous kernels.
Hence, $E/Z(G)$ is isomorphic to the subgroup $\gamma (E)=\{(h,f(\overline{h}g))|h\in H,\,g\in G\}$
of $H \times Aut(G)$.
Let us show that  if $\Phi:H\to Out(G)$ is the coupling associated with $E$, then
 $\gamma (E)$ is uniquely determined by $\Phi$ and
 $$\gamma (E)=T_{\Phi}:=\{(h,\pi) |\,h\in H\,,\pi\in Aut(G)\,\textrm{such that}\,\pi\in\Phi(h)\}\subset H\times Aut(G).$$
Obviously, since $\Phi (h):=f(\overline{h})Inn(G)$ we have $(h,f(\overline{h}g)) \in T_{\Phi}$.
 Conversely if an element $\pi$ of $Aut(G)$ belongs to $\Phi (h)=f(\overline{h})Inn(G)$
 then there exists $g\in G$ such that $\pi=f(\overline{h}g)$, so that the element
 $(h,\pi)$ of $T_{\Phi}$ is equal to $(h,f(\overline{h}g))\in \gamma (E)$.
%%%%%%%%%%%%%%%%%%%%%%%%%%%%%%%%%%%%%%%%%%%%%%%%%%%
\subsubsection{Compatible pairs relatively to $\Phi$}
%Let $E_{1}$ and $E_{2}$ be extensions of $G$ by $H$ associated with $\Phi:H\to Out(G)$.
%{proposition}\label{pi}
Let $E_1$ and $E_2$ be two extensions of $G$ with coupling $\Phi:H\to Out(G)$. Let
$\alpha:E_{1}\to E_2$ be a $G$-isomorphism that induces $\omega \in Aut(H)$ and such that
 $\alpha|_{G}=\pi\in Aut(G)$. By Proposition \ref{pin1}
 $\Phi^{\pi}(h)=\Phi (h^{\omega})$ for every $h\in H$. A \emph{\bf{compatible pair for $\Phi$}} is an
element $(\omega,\pi)$ of $Aut(H)\times Aut(G)$ for which such equality holds. Therefore
every $G$-isomorphism between extensions having the same coupling $\Phi$ induces a
compatible pair for $\Phi$.
Remark that the set $Comp(\Phi)$ of compatible pairs is just the Stabilizer of $\Phi$ for the
action of $Aut(H)\times Aut(G)$ on the couplings described in subsection \ref{orbcoupling}.
Hence, $Comp(\Phi)$ is a subgroup of $Aut(H)\times Aut(G)$.
%%%%%%%%%%%%%%%%%%%%%%%%%%%%%%%%%%%%%%%%%%%%%%%%
\subsubsection{$G$-isomorphism modulo $Z(G)$}
%  \setlength{\intextsep}{0pt}
%%MMMM HOW TO WRITE "figure 10.3.2" AS A TITLE FOR THE FIGURE
\index{TITLE FOR THE FIGURE}
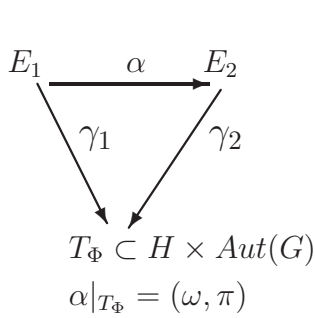
\begin{wrapfigure}[9]{l}{0.3\linewidth}\label{GisoTphi}
%\mbox{\input{Giso1.latex}}
\setlength{\unitlength}{1947sp}%
\begingroup\makeatletter\ifx\SetFigFont\undefined%
\gdef\SetFigFont#1#2#3#4#5{%
  \reset@font\fontsize{#1}{#2pt}%
  \fontfamily{#3}\fontseries{#4}\fontshape{#5}%
  \selectfont}%
\fi\endgroup%
\begin{picture}(300,380)(2001,-550)
\thicklines
{\put(2176,-661){\vector( 1,-2){885}}
}%
{ \put(4501,-736){\vector(-2,-3){1165.385}}
}%
{\put(2326,-661){\vector( 1, 0){2025}}
}%

\put(1801,-511){\makebox(0,0)[lb]{\smash{\SetFigFont{12}{14.4}{\rmdefault}{\mddefault}{\updefault}$E_1$%
}}}
\put(4276,-511){\makebox(0,0)[lb]{\smash{\SetFigFont{12}{14.4}{\rmdefault}{\mddefault}{\updefault}$E_{2}$%
}}}
%\put(3226,-136){\makebox(0,0)[lb]{\smash{\SetFigFont{12}{14.4}{\rmdefault}{\mddefault}{\updefault}$\alpha$%
%}}}
\put(2701,-1411){\makebox(0,0)[lb]{\smash{\SetFigFont{14}{14.4}{\rmdefault}{\mddefault}{\updefault}$\gamma_1$%
}}}
\put(4351,-1411){\makebox(0,0)[lb]{\smash{\SetFigFont{14}{14.4}{\rmdefault}{\mddefault}{\updefault}$\gamma_2$%
}}}
\put(2576,-2861){\makebox(0,0)[lb]{\smash{\SetFigFont{12}{14.4}{\rmdefault}{\mddefault}{\updefault}$T_{\Phi}\subset H\times Aut(G)$%
}}}
\put(2576,-3461){\makebox(0,0)[lb]{\smash{\SetFigFont{12}{14.4}{\rmdefault}{\mddefault}{\updefault}$\alpha |_{T_{\Phi}}=(\omega,\pi)$%
}}}
\put(3301,-511){\makebox(0,0)[lb]{\smash{\SetFigFont{12}{14.4}{\rmdefault}{\mddefault}{\updefault}$\alpha$%
}}}
\end{picture}
\label{fig:GisoTphi}
%\caption{OOOO}
%  \footnotesize\vspace{-10pt} %% -10000 pt =idem
%\centerline{\textbf{fig.\ref{GisoTphi}}}
%  \footnotemark%{fig.\ref{GisoTphi}}
%\leftline{\textbf{fig.\ref{GisoTphi}}}
\end{wrapfigure}
%Let $(\xi,\varphi_1)$ (respectively $(\xi,\varphi_2)$) be a pair of associated functions
% for an extension $E_1:=E(\xi,\varphi_1)$ (repectively ...). %H\to Aut(G)$ be a function such that
%Let $E_{1}:=E(\xi,\varphi_1)$ and $E_{2}:=E(\xi,\varphi_2)$ be two extensions of $G$
Let $E_{1}$ and $E_2$ be two extensions of $G$ by $H$ with coupling $\Phi$.
Figure \ref{GisoTphi} illustrates a $G$-isomorphism $\alpha:E_1 \to E_2$.
 If $\tau_1:=\{\overline{h} : h\in H\}$ is the canonical transversal to $G$ in $E_{1}$
(respectively $\tau_2:=\{\tilde{h} : h\in H\}$ in $E_{2}$), then for $k=1,2$,
let $\rho_k :E_k\to H$  be the canonical homomorphism induced by $\tau_k$.
Let $f_{k}$ ($k=1,2$), be the conjugation action of $E_k$ on $G$.
Finally let $\gamma_k :e\rightarrow (f_k(e),\rho_k (e))$ be the homomorphism from $E_k$
onto $T_{\Phi}\subset Aut(G)\times H$.

Now, since $\alpha$ maps $Ker \gamma_1=Z(G)$ onto $Z(G)=Ker \gamma_2$, it induces an isomorphism
from $E_1 /Z(G)$ onto $E_2 /Z(G)$. Since $\gamma_k$ is an isomorphism from $E_k /Ker \gamma_k =
E_k /Z(G)$ onto $\gamma_k (E_k)=T_{\Phi}$, figure \ref{GisoTphi} shows that
$$\tilde{\alpha}:=\gamma_2\alpha\gamma_1^{-1}:\gamma_1(e)\to\gamma_2 (\alpha (e))\qquad\,e\in E_1$$
 is an automorphism of $T_{\Phi}$. % Now let us describe $\tilde{\alpha}$ on each component.
 similarly, $\alpha$ induces the mapping $\omega:=\alpha |_{E_1/G}:\rho_1 (e)\to\rho_2 (\alpha (e))$
which is an automorphism of $H$ ; $\alpha$ induces on $f_1 (E_1)\leq Aut(G)$ the mapping
$f_1 (e)\to f_2 (\alpha (e))$. If $\alpha |_{G}=\pi\in Aut(G)$, then Lemma \ref{fpi} implies that
$f_2 (\alpha (e))=f_1 (e)^{\pi}$. Finally since $\gamma_1(e)=(\rho_1 (e),f_1 (e))$
is mapped by $\tilde{\alpha}$ on $\gamma_2 (\alpha (e))=(\rho_2 (\alpha (e)),f_2 (\alpha (e)))$,
we can describe the action of $\alpha$ on $T_{\Phi}$ as
%%%% $\gamma_1(e)=(\rho_1 (e),f_1 (e))\to (\rho_2 (\alpha (e)),f_2 (\alpha (e)))=\gamma_2 (\alpha (e))$
\begin{eqnarray}
\tilde{\alpha}:&(h,l)\to (h,l)^{(\omega,\pi)}:=(h^{\omega},l^{\pi}) & \qquad\, h\in H,\,l\in Aut(G) \label{Gisomod1}\\
\tilde{\alpha}:& \gamma_1 (e)\to \gamma_2 (\alpha (e))=\gamma_1 (e)^{(\omega,\pi)} & \qquad\,e\in E_1 \label{Gisomod2}
\end{eqnarray}
where $(\omega,\pi)\in Aut(H)\times Aut(G)$ is a compatible pair for $\Phi$.

%%%%%%%%%%%%%%%%%%%%%%%%%%%%%%%%%%%%%%%%%%%%%%%%%%
% If $\omega\in Aut(H)$ and if $\tilde{\pi}$ denotes conjugation by $\pi$ in $Aut(G)$ then the componentwise action of
Moreover, let us show that for $(\omega,\pi)\in Aut(H)\times Aut(G)$, the function
$\beta:(h,l)\to(h,l)^{(\omega,\pi)}:=(h^{\omega},l^{\pi})$ defines an automorphism
of $T_{\Phi}$ if and only if $(\omega,\pi)$ is a compatible pair.

%%$\tilde{\alpha}$ preserves $T_{\Phi}$ if and only if $(\omega,\pi)$ is a compatible pair.
% For $(\omega,\pi)\in Aut(H)\times Aut(G)$, the function
% $(h,l)\to(h,l)^{(\omega,\pi)}:=(h^{\omega},l^{\pi})$ defines an automorphism
% of $H\times Aut(G)$ ;%% by  $(h,l)\to(h,l)^{(\omega,\pi)}:=(h^{\omega},l^{\pi})$ ;
 %It maps $(h,l)\in H\times Aut(G)$ to $(h,l)^{(\omega,\pi)}:=(h^{\omega},l^{\pi})$
Clearly $\beta$ is an automorphism of $H\times Aut(G)$
 thus it defines an automorphism of $T_{\Phi}$ if and only if it preserves $T_{\Phi}$.
 If $t=(h,l)\in T_{\Phi}$ (i.e. $l\in \Phi (h)$), its image $t^{\tilde{\alpha}}=(h^{\omega},l^{\pi})$ under $\tilde{\alpha}$
belongs to $T_{\Phi}$ if and only if $l^{\pi}\in \Phi (h^{\omega})$. But since $l\in \Phi (h)$ % ($t\in T_{\Phi}$)
 we have $l^{\pi}\in \Phi^{\pi} (h)$ so that $t^{\tilde{\alpha}}$
belongs to $T_{\Phi}$ if and only if $\Phi^{\pi} (h)=\Phi (h^{\omega})$ (two cosets
with a nonempty intersection are equal), what was to be proved.
Hence, every compatible pair for $\Phi$ induces an automorphism of $T_{\Phi}$
Hence, a pair is compatible for $\Phi$ if and only if it induces an automorphism of $T_{\Phi}$.
For this reason, $Comp(\Phi)$ can be considered as a sort of automorphism group of the group $T_{\Phi}$
and we denote the factor group $Comp(\Phi)/Inn (T_{\Phi})$ as $ \mathbf{Out(\Phi)}$.
%macros/latex/contrib/other/misc/wrapfig.sty has syntax:
%\begin{wrapfigure}[height of figure in lines]{l|r}[overhang]{width}
%  figure, caption, etc.
%\end{wrapfigure}

%%%%%%%%%%%%%%%%%%%%%%%%%%%%%%%%%
\subsection{The action of $Comp(\Phi)$ on $H^{2}(H,Z(G))$.}
 Since $E_1$ and $E_2$
have the same coupling, we can associate the same automorphism of $G$ to
$\overline{h}$ and $\tilde{h}$.
Observe that since $f_1(\overline{h})=f_2(\tilde{h})$, we have $\gamma_1(\overline{h}g)=\gamma_2(\tilde{h}g)$
for every $g\in G$ and every $h\in H$.

%\begin{wrapfigure}[9]{l}{0.6\linewidth}
%\mbox{\input{Giso1.latex}
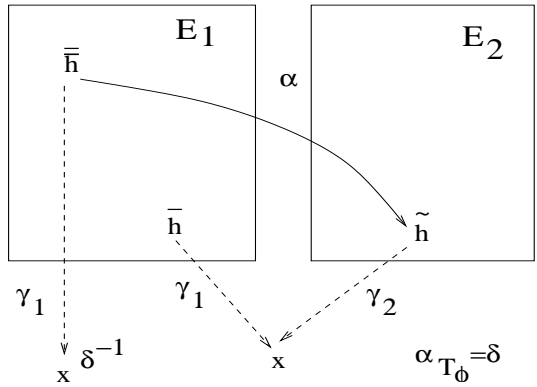
\includegraphics[width=7cm,height=5cm]{pregam.eps}
%\end{wrapfigure}

Let $E_1:=E(\xi,\varphi + \zeta_1)$ and $E_2:=E(\xi,\varphi + \zeta_2)$
 be two extensions of $G$ with coupling $\Phi$ ($\zeta_1 ,\zeta_2 \in Z^{2}(H,Z(G)) $).
We will now state necessary and sufficient conditions for $E_1$ and $E_2$ to be
$G$-isomorphic. The results of this section generalize to any group $G$ a theorem proved by Eick and Besche
 (\cite{Eick_Ulrich} page 8) for the case where $G$ is an elementary abelian group.\\
%MMM INSERT PICTURE + $\overline{\overline{h}}\to\tilde{h}$ under $\alpha ^{-1}$
% \index{picture $\overline{\overline{h}}\to\tilde{h}$}\\

As in figure \ref{GisoTphi}, let $\alpha:E_1 \to E_2$ be a $G$-isomorphism and let
$(\omega,\pi)$ be the compatible pair associated with $\alpha$.
Let $\tau_1:=\{\overline{h} : h\in H\}$ be the canonical transversal to $G$ in $E_{1}$
(respectively $\tau_2:=\{\tilde{h} : h\in H\}$ in $E_{2}$) that corresponds to the
automorphisms $\xi (h)$ of $G$ and to the factor set $\varphi_1=\varphi +\zeta_1$
(respectively $\varphi_2=\varphi +\zeta_2$). As before, for $k=1,2$, consider the
homomorphisms $\gamma_k :e\rightarrow (f_k(e),\rho_k (e))$ from $E_k$ onto $T_{\Phi}\cong E_k/Z(G)$.
Now, observe that \begin{equation}  \label{gam12}
\gamma_1 (\overline{h})=(h,\xi (h))=\gamma_2 (\tilde{h})\qquad\forall\,h\in H
\end{equation}

For any $h\in H$, let $\overdub{h}\in E_1$ be the preimage $\alpha^{-1}(\tilde{h})$ of $\tilde{h}$.
Each $\overdub{h}$ is known modulo $Z(G)$ since by equation \ref{Gisomod2},
$\gamma_1 (\overdub{h})^{(\omega,\pi)}=\gamma_2 (\alpha (\overdub{h}))$, which is equal to
$\gamma_2 (\tilde{h})$ (as $\alpha (\overdub{h})=\tilde{h}$) and hence to
$\gamma_1 (\overline{h})$ (equation \ref{gam12}). Finally, we have
\begin{equation}  \label{overduub}
\gamma_1 (\overdub{h})=\gamma_1 (\overline{h})^{(\omega^{-1},\pi^{-1})}
\end{equation}

%since for $h\in H$, the elements $\overline{h}$ and $\tilde{h}$ are
%mapped on the same element $(h,\xi (h))$ of $T_{\Phi}$ and hence

% Figure \ref{GisoTphi} describes a $G$-isomorphism $\alpha:E_1 \to E_2$ that induces
% $\omega\in Aut(H)$ and $\pi\in Aut(G)$. We have showed in the previous section that the
% compatible pair $(\omega,\pi)$ defines on $T_{\Phi}$ an automorphism
% Let $\tau_1:=\{\overline{h} : h\in H\}$ be the canonical transversal to $G$ in $E_{1}$
% (respectively $\tau_2:=\{\tilde{h} : h\in H\}$ in $E_{2}$) that corresponds to the
% automorphisms $\xi (h)$ of $G$ and to the factor set $\varphi_1=\varphi +\zeta_1$
% (respectively $\varphi_2=\varphi +\zeta_2$).

% Let $f_{k}$ ($k=1,2$), be the conjugation action of $E_k$ on $G$.
%  Let $\rho_k :E_k\to H$  be the canonical homomorphism induced by $\tau_k$.
% Finally let $\gamma_k :e\rightarrow (f_k(e),\rho_k (e))$ from $E_k$ onto
% $T_{\Phi}\subset Aut(G)\times H$. Observe that since $f_1(\overline{h})=f_2(\tilde{h})=f_h$, we have $\gamma_1(\overline{h}g)=\gamma_2(\tilde{h}g)$
% for every $g\in G$ and every $h\in H$.\\

Let $H=<X|R>$ be a presentation for $H$ where $X$ are the generators of a free group and $R$ is a set
of words (the relations) on the elements of $X$.
Let $X'=\{h_1,h_2,\ldots\}$ be the generators
of $H$ that correspond to the elements of $X$. For every relation $r\in R$,
the homomorphism $\rho_2$ from $E_2$ onto $H$ maps $g_r:=r(\tilde{h_1},\tilde{h_2},\ldots)$
 to $r(h_1,h_2,\ldots)=1\in H$ and thus  $g_r\in G$. For similar reasons
$g'_r:=r(\overdub{h_1},\overdub{h_2},\ldots)\in G$. But $g_r$ is the image of $g'_r$ under
$\alpha$ and since $\alpha |_{G}=\pi$, %we obtain
\begin{equation} \label{relpi1}
r(\tilde{h_1},\tilde{h_2},\ldots)=(r(\overdub{h_1},\overdub{h_2},\ldots))^{\pi}
\qquad\textrm{for every}\,\, r\in R
\end{equation}
As a particular case, factor set functions correspond to relation $h_3 ^{-1}h_1 h_2=1$
($h_1 ,h_2 \in H$ and $h_3=h_1 h_2$) where the set $X'$ of generators is the full set $H$.
%as generators set $X'$.
Therefore, since the factor set $\varphi_2$ of $E_2$ is defined by
$\varphi_2 (h_1,h_2)=\widetilde{h_1 h_2}^{-1}\tilde{h_1}\tilde{h_2}$

\begin{equation} \label{relpi2}
\varphi_2 (h_1,h_2)=(\varphi ' (h_1,h_2))^{\pi}\quad where\quad \varphi '(h_1,h_2)=
(\overdub{h_1 h_2})^{-1}\overdub{h_1}\,\,\overdub{h_2}
\end{equation}
Hence,
\begin{proposition}\label{newfactset}
If there is a $G$-isomorphism from $E(\xi,\varphi_1)$ to $E(\xi,\varphi_2)$
that induces the compatible pair $(\omega,\pi)$, then there
exist elements $\{\overdub{h}:h\in H\}$ of $E(\xi,\varphi_1)$
such that equality \ref{overduub} and equality \ref{relpi2} are satisfied.
%(equation \ref{relpi1} if a presentation $H=<X|R>$ is used to describe the extensions).}
%($\star\star$)
\end{proposition}
Note that if a presentation $H=<X|R>$ is used to describe the extensions,  it is
equality \ref{relpi1} that has to be considered.

\textbf{Definition }.
Let $E(\xi,\varphi_1)$ be an extension with coupling $\Phi$ and let  $(\omega,\pi)$
be a given compatible pair for $\Phi$.
%The set $\varphi_1 ^{\Gamma (\omega,\pi)}$ is defined as the set of all factor set functions
The set \mbox{\boldmath$\varphi_1 ^{\Gamma (\omega,\pi)}$} is defined as the set of all functions
 $\varphi _2$ from $H\times H$ to $G$ generated through equation \ref{relpi2} by the
solutions $\{\overdub{h}:h\in H\}$ of equation \ref{overduub}. Let us show that
this set is not empty. The pair $(\omega^{-1},\pi^{-1})$
 is also compatible (since $Comp(\Phi)$ is a group) and so, it preserves $T_{\Phi}$
 as showed at the end of the previous section. Thus for $h\in H$, the element
 $\gamma_1 (\overline{h})^{(\omega^{-1},\pi^{-1})}$ belongs to
 $T_{\Phi}=\gamma_1 (E(\xi,\varphi_1))$ and it is possible to choose
an element $\overdub{h}$ in its preimage under $\gamma_1$ that satisfies
equation \ref{overduub}.%%% holds on $\overdub{h}$.
%Condition ($\star\star$)

Proposition \ref{newfactset} may now be reformulated as follows :
\emph{if there is a $G$-isomorphism from $E(\xi,\varphi_1)$ to $E(\xi,\varphi_2)$
that induces the compatible pair $(\omega,\pi)$, then $\varphi_2\in \varphi_1 ^{\Gamma (\omega,\pi)}$.}
We prove now that the converse holds and that every element of
$\varphi_1 ^{\Gamma (\omega,\pi)}$ is a factor set for $\xi$.

%condition ($\star\star$) is also sufficient.
%We now prove that every element $\varphi_2 \in \varphi_1 ^{\Gamma (\omega,\pi)}$ is
%the factor set of an extension $E(\xi,\varphi_2)$ that is $G$-isomorphic to $E(\xi,\varphi_1)$.
\begin{theorem}\label{cnsisoGam}
 Let $E(\xi,\varphi_1)$ and $E(\xi,\varphi_2)$ be extensions of $G$ by $H$ with coupling
 $\Phi$ and $(\omega,\pi)$ be a compatible pair for $\Phi$.
 There exists a $G$-isomorphism $E(\xi,\varphi_1)\to E(\xi,\varphi_2)$
  that induces $(\omega,\pi)$ if and only if $\varphi_2 \in \varphi_1 ^{\Gamma (\omega,\pi)}$.
\end{theorem}
\emph{Proof }. It remains to prove that $\varphi_2 \in \varphi_1 ^{\Gamma (\omega,\pi)}$ is
a sufficient condition for the existence of a $G$-isomorphism that induces $(\omega,\pi)$.
Let $E_1:=E(\xi,\varphi_1)$, let $\{\overline{h}:h\in H\}$
be the canonical transversal associated with $E_1$ and let
$\gamma_1 =(\rho_1,f_1):E_1\to T_{\Phi}$ be as before. Let $(\omega,\pi)$ be a
compatible pair for $\Phi$. 
Assume that $\varphi_2 \in \varphi_1 ^{\Gamma (\omega,\pi)}$
is generated by $\{\overdub{h}:h\in H\}$ (see equation \ref{relpi2}).
By definition, $\gamma_1 (\overdub{h})=\gamma_1 (\overline{h})^{(\omega^{-1},\pi^{-1})}$
and this implies
\beq \label{imply1}
\rho_1 (\overdub{h})=h^{\omega^{-1}}\quad and \quad f_1 (\overdub{h})=
 f_1 (\overline{h})^{\pi^{-1}}=\xi(h)^{\pi^{-1}}.
\eeq
For every $h\in H$, the element of $\{\overdub{h}\}$ that is mapped on $h$ by
$\rho_1$ is $\overdub{h^{\omega}}$
(since $\rho_1 ( \overdub{h^{\omega}} )=(h^{\omega }) ^{\omega^{-1}}=h$).
	Hence, $\tau ':=\{\overdub{h^{\omega}}:h\in H\}$ is a new transversal and it describes an
 extension $E_1 ':=E(\xi_1',\varphi_1')$ equivalent to $E_1$.
Let us compute the pair $(\xi_1',\varphi_1')$ of functions associated with $\tau'$.
Using \ref{imply1}, the automorphism $f_1 (\overdub{h^{\omega}})$ of $G$ that corresponds
to the coset representative $\tau'(h)=\overdub{h^{\omega}}$ is
\begin{equation}\label{newxi1}
  \xi_1 '(h)=\xi(h^{\omega})^{\pi^{-1}}\qquad h\in H,
  \end{equation}
  and the factor set $\varphi_1 '$ corresponding to $\tau '$ is defined by
 \begin{equation}\label{newfact1}
 \varphi_1' (h_1,h_2)=(\overdub{(h_1 h_2)^{\omega}})^{-1} \overdub{h_1 ^{\omega}}\,\,\overdub{h_2 ^{\omega}}
  \qquad h_1,h_2\in H.
\end{equation}
Let $\{\hat{h}:\, h\in H\}$ be the canonical transversal to $G$ in $E_1 '$ that produces
the pair $(\xi_1',\varphi_1')$. Since this pair is equal to the pair produced by
$\{\overdub{h^{\omega}}:h\in H\}$ in $E_1$, the mapping $\beta$ defined by
$\overdub{h^{\omega}}\to \hat{h}$ and $g\to g$ for $h\in H,\, g\in G$
%\beq \label{betah}
%\beta : \overdub{h^{\omega}}\to \hat{h}\quad and \quad g\to g\quad for\,h\in H,\, g\in G
%\eeq
is a $G$-isomorphism from $E_1$ onto $E_1 '$ that induces identity on $H$
since $\rho_1 ( \overdub{h^{\omega}} )=h$, and induces identity on $G$. Consequently, $\beta$ induces the trivial
pair $(1,1)\in Comp(\Phi)$.

   Now, with Theorem \ref{newext}, we use $(\omega,\pi)$ to create a new extension
 $E(\xi_3,\varphi_3)$ that is $G$-isomorphic to $E_1 '$ and hence to $E_1$. Using Theorem
 \ref{newext} and equation \ref{newxi1}, we obtain
\begin{equation}\label{newxi2}
 \xi_3 (h)=\xi_1'(h^{\omega^{-1}})^{\pi}=(\xi (h^{\omega^{-1}\omega})^{\pi})^{\pi^{-1}}
 =\xi (h)\qquad h\in H.
\end{equation}
similarly, $\varphi_3 (h_1,h_2)=\varphi_1' (h_1 ^{\omega^{-1}},h_2 ^{\omega^{-1}})^{\pi}$
and since by equation \ref{newfact1}\\
$\varphi_1' (h_1 ^{\omega^{-1}},h_2 ^{\omega^{-1}})=
(\overdub{(h_1 h_2)^{\omega^{-1}\omega}})^{-1} \overdub{h_1 ^{\omega^{-1}\omega}}\,\,\overdub{h_2 ^{\omega^{-1}\omega}}=
(\overdub{h_1 h_2})^{-1}\overdub{h_1}\overdub{h_2}$, \\
for $h_1,h_2\in H$, we obtain
%\begin{equation} %\label{relpi2}
$$\varphi_3 (h_1,h_2)=(\varphi ' (h_1,h_2))^{\pi}\quad for\quad \varphi '(h_1,h_2)=
(\overdub{h_1 h_2})^{-1}\overdub{h_1}\,\,\overdub{h_2}.$$
which compared to equation \ref{relpi2}, is precisely the definition of $\varphi_2$, the
function of $\varphi_1 ^{\Gamma (\omega,\pi)}$ generated by $\{\overdub{h}\}$.
Hence, $\varphi_2=\varphi_3$ is a factor set for $\xi$ and
$E_2:=E(\xi_3,\varphi_3)=E(\xi,\varphi_2)$ is $G$-isomorphic to $E_1$.

If we denote the canonical transversal of $E_2$ by $\{\tilde{h}:\,h\in H\}$, the
$G$-isomorphism $\alpha$ from $E_1 '$ onto $E_2$ produced in Theorem \ref{newext}
is defined by
\beq \label{alphah}
\alpha : \hat{h}\to \widetilde{h^{\omega}}\quad and \quad g\to g^{\pi}\,for\,\,h\in H,\, g\in G
\eeq
 so that $(\omega,\pi)$ is the compatible pair
induced by $\alpha$. Observe that $\beta$ maps $\overdub{h}$ to $\widehat{h^{\omega^{-1}}}$
which is mapped to $\tilde{h}$ by $\alpha$. The composition of $\beta$ and $\alpha$
\begin{eqnarray}
\beta \alpha : & E(\xi,\varphi_1)\stackrel{\beta}{\longrightarrow}  E(\xi_1 ',\varphi_1 ')
\stackrel{\alpha}{\longrightarrow} E(\xi,\varphi_2)& \\
 : & \overdub{h}\stackrel{\beta}{\longrightarrow} \widehat{h^{\omega^{-1}}}
 \stackrel{\alpha}{\longrightarrow}\tilde{h}  & \textrm{for}\,\, h\in H
\end{eqnarray}
is a $G$-isomorphism that induces the pair $(1,1).(\omega,\pi)=(\omega,\pi)$.
$\quad\Box$
%Finally, through the intermediate extension $E_1 '$, we have constructed a $G$-isomorphism
%$\alpha$ from $E_1=E(\xi,\varphi_1)$ onto $E_2=E(\xi,\varphi_2)$. From the proof of Theorem \ref{newext},
%we see that it induces $\pi$ on $G$ and $\omega$ on $H$ since it maps $\overdub{h}$
%to $\tilde{h}$ (where $\{\tilde{h}: h\in H \}$ denotes the canonical transversal in $E_2$).
% denote the canonical transversal in $E_2$ by % is the canonical transversal in $E_2$).
%MMMMMMMMM ADD THE TRANSLATION FOR FPGROUP\\
\index{formulate the sufficient condition for FpGroup}

\textbf{Definition }. Let $E:=E(\xi,\varphi)$ be an extension with coupling $\Phi$. An automorphism
of $E$ that preserves $G$ is a $G$-isomorphism from $G$ onto itself, (a $G$-automorphism)
and hence it induces a compatible pair $(\omega,\pi)$ (i.e. an automorphism of $T_{\Phi}$).
There is a homomorphism $\theta$ from the group of $G$-automorphisms of $E$ into $Comp (\Phi)$.
A pair $(\omega,\pi)$, compatible for
$\Phi$ is called \emph{\bf{an inducible pair for $E(\xi,\varphi)$}} if there exists a
$G$-automorphism  that induces $(\omega,\pi)$, that is if  $(\omega,\pi)\in Im \theta$ (\cite{Robinson} page 66).
 We denote the subgroup $Im \theta$ of inducible pairs as $Ind (E)$. The inner automorphisms of $E$
 induce the inner automorphisms of $T_{\Phi}$. Hence, $Inn (T_{\Phi})=\theta (Inn(E))$ is a normal subgroup
 of $Ind (E)$.

Here is a straightforward corollary of Theorem \ref{cnsisoGam}.
\begin{corollary}\label{Gaminduce}
Let $E(\xi,\varphi)$ be an extension of $G$ by $H$ with coupling
 $\Phi$. A compatible pair $(\omega,\pi)$ is inducible for $E(\xi,\varphi)$
if and only if $\varphi \in \varphi ^{\Gamma (\omega,\pi)}$.
\end{corollary}

%%%%%%%%%%%%%%%%%%%%%%%%%%%%%%%
For each compatible pair $(\omega,\pi)$, the determination of the set $\varphi_1 ^{\Gamma (\omega,\pi)}$
 requires considering all solutions $\{\overdub{h}:h\in H\}$ of equation \ref{overduub}. If $H$ and
 $Z(G)$ are finite, there are $|Z(G)|^{|H|}$ such solutions. We prove now that it is possible 
 to avoid this and to determine $\varphi_1 ^{\Gamma (\omega,\pi)}$ from a single solution 
$\{\overdub{h}:h\in H\}$.
% \begin{center}
%\emph{for a given $E_1$, the class $\varphi_2 + B^{2}(H,Z(G))$ only depends on $(\omega,\pi)$.}
% \end{center}
\begin{proposition}\label{Gamisfunction} Let $E(\xi,\varphi_1)$ be an extension with coupling $\Phi$ and let
 $(\omega,\pi)$ be a compatible pair for $\Phi$. Let $B=B_{\Phi} ^{2}(H,Z(G))$. Then the following hold

 \begin{enumerate}
%\item $\varphi_1 ^{\Gamma (\omega,\pi)}=\varphi_2 + B_{\Phi} ^{2}(H,Z(G))$ for any element
% $\varphi_2$  of $\varphi_1 ^{\Gamma (\omega,\pi)}$.
\item $\varphi_1 ^{\Gamma (\omega,\pi)}=\varphi_2 + B$ for any element
 $\varphi_2$  of $\varphi_1 ^{\Gamma (\omega,\pi)}$.
%\item ($\varphi_1 + B_{\Phi} ^{2}(H,Z(G)))^{\Gamma (\omega,\pi)}=\varphi_1 ^{\Gamma (\omega,\pi)}$.
\item ($\varphi_1 + B)^{\Gamma (\omega,\pi)}=\varphi_1 ^{\Gamma (\omega,\pi)}$.
%\item $\varphi_1 ^{\Gamma (1,1)}=\varphi_1 + B_{\Phi} ^{2}(H,Z(G)).$
\item $\varphi_1 ^{\Gamma (1,1)}=\varphi_1 + B.$
\end{enumerate}
\end{proposition}
\emph{Proof }.
\begin{enumerate}
\item
\begin{itemize}
\item proof of the inclusion $\subseteq$. As before, $\{\overdub{h}\}$ and $\{\tilde{h} \}$
denote respectively the canonical transversal in $E(\xi,\varphi_1)$ and $E(\xi,\varphi_2)$.
Let $\{\overdub{h}:h\in H\}$ be a solution of equation \ref{overduub}. Then, since $Ker \gamma_1=Z(G)$
any other solution is $\{\overdub{h}\,':=\overdub{h}z_h | \,h\in H\}$ for some function
$h\to z_h$ from $H$ into $Z(G)$.
Let $\varphi_2$ be the factor set generated by $\{\overdub{h}:h\in H\}$.
We have showed that there exists a $G$-isomorphism
$\alpha:E(\xi,\varphi_1)\to E(\xi,\varphi_2)$ that induces $(\omega,\pi)$ on $T_{\Phi}$
, that maps $\overdub{h}$ to $\tilde{h}$ and thus
$\overdub{h}\,'$ to $\tilde{h}z_h ^{\pi}$.\\
%$\alpha:\overdub{h}\,'\to\tilde{h}\,'$ for $\tilde{h}':=\tilde{h}z_h ^{\pi}$.\\

The image set $\tau_2 ':=\{\tilde{h}\,':=\tilde{h}z_h ^{\pi}:h\in H\}$ is a new transversal to $G$ in $E(\xi,\varphi_2)$
 whence it defines an extension $E(\xi',\varphi_2 ')$ equivalent to $E(\xi,\varphi_2)$.
But $\tilde{h}\,'$ and $\tilde{h}$ differ by an element of $Z(G)$
so that their conjugation action on $G$ is the same (i.e. $\xi' (h)=\xi (h)$).
Hence, $E(\xi,\varphi_2 ')$ is equivalent to $E(\xi,\varphi_2)$ and by Theorem \ref{MacEinlen}
 (part 3), $\varphi_2 '\in \varphi_2 + B_{\Phi} ^{2}(H,Z(G))$.

By definition, $\varphi_2 '(h_1,h_2)=(\widetilde{h_1 h_2})\,'^{-1}\tilde{h_1}\,'\,\tilde{h_2}\,'$
for $h_1,h_2 \in H$. Since $\tilde{h}\,'$ is the image of $\overdub{h}\,'$ under $\alpha$,
% maps  to $\tilde{h}\,'$ we obtain
it implies that $\varphi_2 '(h_1,h_2)$ is the image of $\varphi '' (h_1,h_2):=
(\overdub{h_1 h_2})\,'^{-1}\overdub{h_1}\,'\,\overdub{h_2}\,'$ under $\alpha$ and since
$\alpha |_{G}=\pi$ we obtain $\varphi_2 '(h_1,h_2)=(\varphi '' (h_1,h_2))^{\pi}$.
But the factor set $\varphi_3$ generated by the solution $\{\overdub{h}\,':\,h\in H\}$
of equation \ref{overduub} is precisely $(\varphi '' (h_1,h_2))^{\pi}$ (see equation
\ref{relpi2}). Hence, $\varphi_3=\varphi_2 '$ and every $\varphi_3$ in
$\varphi_1 ^{\Gamma (\omega,\pi)}$ belongs to $\varphi_2 + B_{\Phi} ^{2}(H,Z(G))$.\\

% The factor set generated by this new solution is
%\begin{equation}\label{newhz1} \varphi_3 (h_1,h_2)=(\varphi '' (h_1,h_2))^{\pi}\quad for\quad \varphi ''(h_1,h_2)=
%(\overdub{h_1 h_2})\,'^{-1}\overdub{h_1}\,'\,\overdub{h_2}\,'.
%\end{equation}
%they induce the
%the automorphism $\xi' (h)\xi (h)$ associated with $\tilde{h}\,'$ is equal
% and conjugation by $\tilde{h}\,'$ on $G$ is also $\xi (h)$ because
%$\tilde{h}\,'=\tilde{h}z_h ^{\pi}$ and $\tilde{h}$ differ by an element of $Z(G)$.
%Let $\varphi_2 '$ be the factor set associated with $\tau_2 '$ :
%for $h_1,h_2 \in H$, we have
%$\varphi_2 '(h_1,h_2)=(\widetilde{h_1 h_2})\,'^{-1}\tilde{h_1}\,'\,\tilde{h_2}\,'$.
% But since $\alpha:\overdub{h}\,'\to \tilde{h}\,'$ %is the image under $\alpha$ of $\tilde{h}\,'$
% , by equation \ref{newhz1} we obtain that $\alpha:\varphi ''(h_1,h_2)\to\varphi_2 '(h_1,h_2)$ ;
% thus $\varphi_2 '=(\varphi '')^{\pi}$ (because $\alpha |_{G}=\pi$) which compared to
%equation \ref{newhz1} implies $\varphi_3=\varphi_2 '$. Finally, since the extension
%$E(\xi,\varphi_2 ')$ defined by $\tau_2 '$ is equivalent to $E(\xi,\varphi_2)$, Theorem
%\ref{MacEinlen} (part 3) shows that $\varphi_3=\varphi_2 '\in \varphi_2 + B^{2}(H,Z(G))
% .\quad\Box$

% For $h_1,h_2 \in H$, we have
% $\varphi_2 '(h_1,h_2)=(\widetilde{h_1 h_2})\,'^{-1}\tilde{h_1}\,'\,\tilde{h_2}\,'$
% that is the image under $\alpha$ of $(\overdub{h_1 h_2})\,'^{-1}\overdub{h_1}\,'\,\overdub{h_2}\,'$.
% The conjugation by $\tilde{h}z_h ^{\pi}$ on $G$ is also $\xi (h)$ because
% $\tilde{h}z_h ^{\pi}$ and $\tilde{h}$ differ by an element of $Z(G)$.
\item proof of the inclusion $\supseteq$.
 Suppose as before that $\varphi_2$ is generated by $\{\overdub{h}\}$ and that
$\alpha:E(\xi,\varphi_1)\to E(\xi,\varphi_2)$ such that
$\alpha:\overdub{h} \to \tilde{h}$, is the corresponding $G$-isomorphism.
 Let $\varphi_3$ be an element of $\varphi_2 + B_{\Phi} ^{2}(H,Z(G))$.
Let $\{\tilde{h}:h\in H\}$ and $\{\hat{h}:h\in H\}$ be the canonical transversals
 to $G$, respectively in $E(\xi,\varphi_2)$ and in $E(\xi,\varphi_3)$.
 % ; let $\gamma_2$ and $\gamma_3$ be the corresponding projections of
By Theorem \ref{MacEinlen} (part 3), $E(\xi,\varphi_3)$ is equivalent to $E(\xi,\varphi_2)$.
% Let us use the canonical projections of these extensions on $T_{\Phi}$ (see figure \ref{GisoTphi}).
For $g\in G$, let $<g>$ denote the conjugation by $g$ in $G$.
Equivalence (see \ref{equivext}) implies that for each $h\in H$, there exists $g_h\in G$ such that
the pair of functions associated with the transversal $\{\tilde{h}\,':=\tilde{h}g_h\}$ is
$(\xi,\varphi_3)$. That is to say $\xi (h)<g_h>=\xi (h)$, (hence $<g_h>=<1>$) and
%(the conjugation by $\tilde{h}g_h$) is equal to $\xi(\hat{h})$ (the conjugation by $\hat{h}$)
$\varphi_2 '(h_1,h_2):=(\widetilde{h_1 h_2})\,'^{-1}\tilde{h_1}\,'\,\tilde{h_2}\,'
=\varphi_3 (h_1,h_2)$, for $h_1,h_2 \in H$.
Since $<g_h>=<1>$ implies $g_h\in Z(G)$, let us write $z_h:=g_h$.
The set $\{\overdub{h}\,':=\overdub{h}z_h ^{\pi^{-1}} \}$ is the preimage of
$\{\tilde{h}\,'=\tilde{h}z_h\}$ under $\alpha$, and since it differs from $\{\overdub{h}\}$
by elements of $Z(G)$, it is also a solution of equation \ref{overduub}. As in the previous
proof, taking the preimage of $\varphi_2 '(h_1,h_2)$ under $\alpha$ shows that the
solution $\{\overdub{h}\,'\}$ generates the factor set $\varphi_2 '$
whence $\varphi_2 '=\varphi_3$ belongs to $\varphi_1 ^{\Gamma (\omega,\pi)}$.

%Seeing that there is a sequence of isomorphisms
%\begin{tabular}{cc}
%$E(\xi,\varphi_1)\to E(\xi,\varphi_2) \to E(\xi,\varphi_3)$ & \textrm{such that}\\
%$\overdub{h}\,' \to \tilde{h}\,' \to \hat{h}$ & $h\in H$\\
%\end{tabular}
\end{itemize}
\item Let us denote $B_{\Phi} ^{2}(H,Z(G))$ as $B$. It suffices to prove ($\varphi_1 + B)^{\Gamma (\omega,\pi)}
\subseteq \varphi_1 ^{\Gamma (\omega,\pi)}$ because the other inclusion is trivial.
Let $b\in B$. By Theorem \ref{MacEinlen} (part 3), $E(\xi,\varphi_1)$ is equivalent to
$E(\xi,\varphi_1 +b)$. Consequently, there exists a $G$-isomorphism
$\alpha:E(\xi,\varphi_1)\to E(\xi,\varphi_1 +b)$ that induces the compatible pair $(1,1)$.
 For each $\varphi_2$ in $(\varphi_1 + b)^{\Gamma (\omega,\pi)}$, there is a $G$-isomorphism
 $\beta:E(\xi,\varphi_1 +b)\to E(\xi,\varphi_2)$ that induces the compatible pair $(\omega,\pi)$.
 Hence, the composition $\alpha \beta:E(\xi,\varphi_1)\to E(\xi,\varphi_2)$ is a $G$-isomorphism
 that induces $(\omega,\pi)$ and therefore, by Theorem \ref{cnsisoGam},
 $\varphi_2\in \varphi_1 ^{\Gamma (\omega,\pi)}$ and the inclusion is proved.

\item By $(1)$, it suffices to show that $\varphi_1 \in \varphi_1 ^{\Gamma (1,1)}$. This is
trivial since for $\omega=1$ and $\pi=1$, the transversal $\{\overline{h}:\,h\in H\}$ itself
is a solution of \ref{overduub} and it generates the factor set $\varphi_1$.
\end{enumerate}
$\quad\Box$
%%%%%%%%%%%%%%%%%%%%%%%%%%%%%%%%%%%%%%%%%%%%%%%%%%%
\subsection{Orbits are $G$-isomorphism classes.}
%MMMMMMM DIRE QUE "OUR RESULT"
\index{DIRE QUE "OUR RESULT"}

\begin{theorem}\label{Gamaction}
Let $E(\xi,\varphi_1)$ be an extension with coupling $\Phi$ and for $i=1,2$, let
 $(\omega_i,\pi_i)$ be compatible pairs for $\Phi$. Then \\
$$(\varphi_1 ^{\Gamma (\omega_1,\pi_1)})^{\Gamma (\omega_2,\pi_2)}=
\varphi_1 ^{\Gamma (\omega_1 \omega_2,\pi_1\pi_2 )}.$$
\end{theorem}
\emph{Proof }. %Inclusion $\subseteq$. If $\varphi_3 \in
If $\varphi_3 \in (\varphi_1 ^{\Gamma (\omega_1,\pi_1)})^{\Gamma (\omega_2,\pi_2)}$ then there
exists $\varphi_2$ in $\varphi_1 ^{\Gamma (\omega_1,\pi_1)}$ such that $\varphi_3 \in
\varphi_2 ^{\Gamma (\omega_2,\pi_2)}$. Hence, by Theorem \ref{cnsisoGam}, we can construct
a sequence of $G$-isomorphisms
$$E(\xi,\varphi_1)\stackrel{(\omega_1,\pi_1)}{\longrightarrow}  E(\xi,\varphi_2)
\stackrel{(\omega_2,\pi_2)}{\longrightarrow} E(\xi,\varphi_3)$$
where the induced compatible pairs are indicated on the arrows. The composition of these
$G$-isomorphisms is also a $G$-isomorphism $E(\xi,\varphi_1)\to E(\xi,\varphi_3)$ that induces
$(\omega_1 \omega_2,\pi_1\pi_2 )$ so that by the same theorem we obtain
$\varphi_3\in \varphi_1 ^{\Gamma (\omega_1 \omega_2,\pi_1\pi_2 )}$ ($\star$) which must be equal to
$\varphi_3 + B$ (for $B:=B_{\Phi} ^{2}(H,Z(G))$ and using Proposition \ref{Gamisfunction}).
But, by Proposition, $\varphi_3 + B=\varphi_2 ^{\Gamma (\omega_2,\pi_2)}$.
Since  ($\star$) implies $(\varphi_1 ^{\Gamma (\omega_1,\pi_1)})^{\Gamma (\omega_2,\pi_2)}
\subseteq \varphi_1 ^{\Gamma (\omega_1 \omega_2,\pi_1\pi_2 )}$, we obtain a sequence
$$\varphi_3 + B=\varphi_2 ^{\Gamma (\omega_2,\pi_2)}\subseteq
(\varphi_1 ^{\Gamma (\omega_1,\pi_1)})^{\Gamma (\omega_2,\pi_2)}\subseteq
\varphi_1 ^{\Gamma (\omega_1 \omega_2,\pi_1\pi_2 )}=\varphi_3 + B$$ that proves the Theorem.
$\quad\Box$

Let $E(\xi,\varphi)$ be an extension of $G$ by $H$ with coupling $\Phi$.
Remind that every extension with coupling $\Phi$ is equal to $E(\xi,\varphi + \zeta)$
for some $\zeta\in Z_{\Phi} ^{2}(H,Z(G))$. The set of equivalence classes of extensions is in
one-to-one correspondence with the set $\Omega:=\varphi + H_{\Phi} ^{2}(H,Z(G))$.

For a compatible pair $x$, Proposition \ref{Gamisfunction} %\emph{Proof }.$\quad\Box$
states that $\Gamma (x)$ defines a function from the set $\Omega$ into itself.
Let us show that $\Gamma (x)$ is a permutation of $\Omega$. Injectivity and surjectivity can
be proved quickly with the help of Theorem \ref{Gamaction}.
For $B:=B_{\Phi} ^{2}(H,Z(G))$, let $\varphi_1 +B$ and $\varphi_2 +B$ be
elements of $\Omega$. For injectivity, if $(\varphi_1 +B) ^{\Gamma (x)}=(\varphi_2 +B)^{\Gamma (x)}$,
 we multiply both sides by $\Gamma (x^{-1})$ and we conclude with Proposition \ref{Gamisfunction}
 that $\varphi_1 +B=(\varphi_1 +B) ^{\Gamma (1)}=(\varphi_2 +B)^{\Gamma (1)}=\varphi_2 +B$.
For surjectivity, observe that $\varphi_1 +B$ is the image of $(\varphi_1+B) ^{\Gamma (x^{-1})}$
under $\Gamma (x)$.

Together with Theorem \ref{Gamaction} and Proposition \ref{Gamisfunction}, this completes the proof of the
fact that the group $Comp(\Phi)$ acts on the right, through $\Gamma$, on the set
$\varphi + H_{\Phi} ^{2}(H,Z(G))$. After a shift $\varphi_1 +B \to (\varphi_1 -\varphi) +B$,
 it becomes an action of $Comp(\Phi)$ on the set $H_{\Phi} ^{2}(H,Z(G))$.
 Observe also that since $Inn(T_{\Phi})$ is constant for every extension associated with $\Phi$,
 it is contained in every stabilizer.   Hence, $\Gamma$ defines an action of the group $Comp(\Phi)/Inn(T_{\Phi})$
 denoted by $Out(\Phi)$ on the set $H^{2}$. Hence, we have proved the following theorem :
 %%%%%%%%%%%
 \begin{theorem}\label{mainAction} Let $\Phi:H\to Out(G)$ be a homomorphism. %%% and let $\xi:H\to Aut(G)$ be
Let $E=E(\xi,\varphi)$ be an extension of $G$ with coupling $\Phi$. Let
$\zeta_1,\zeta_2 \in Z_{\Phi} ^{2}(H,Z(G))$ and let $B:=B_{\Phi} ^{2}(H,Z(G))$.
%and with associated pair $(\xi,\varphi)$.\begin{enumerate}
%\item The function $\Gamma_{E}:Comp(\Phi)\times H^{2}(H,Z(G))\to H^{2}(H,Z(G))$
% described above is a group action on the set $H^{2}(H,Z(G))$.
\begin{enumerate}
\item The function $\Gamma$ defined in this section is a right action of the group $Comp(\Phi)$ on the set
$H_{\Phi} ^{2}(H,Z(G))$.
 \item The extension $E(\xi,\varphi +\zeta_1)$ is $G$-isomorphic to the extension
$E(\xi,\varphi +\zeta_2)$ if and only if $\zeta_1 + B$ and $\zeta_2 +B$ are in the same orbit
 under $\Gamma_{E}$.
\item The stabilizer of $\zeta_1 +B$ is the subgroup of inducible pairs for
$E(\xi,\varphi +\zeta_1)$.
\item $Inn(T_{\Phi})$ fixes every element of $H_{\Phi} ^{2}(H,Z(G))$.
 \end{enumerate}
 \end{theorem}     %%%% two extensions with coupling
%MMMM once the extension exists, prove it for the Fp case ?prove that rel implies factor sets.
%%MMMMMMMMMMMMMMMMMMM SAY THAT IT IS A RIGHT-ACTION
%%%%%%%%%%%%%%%%%%%%%

%%%%%%%%%%%%%%%%%%%%%%%%%%%%%%
\subsection{Computation of action and orbits}\label{computeorb}
We recall here the points that are essential to compute the orbits of the action $\Gamma$.
We assume that we have already an extension $E(\xi,\varphi )$ with coupling $\Phi$ (note that such extension do not
always exist) and that we have $H_{\Phi} ^{2}(H,Z(G))$.
\subsubsection{Computation of the action}
As usual let $B:=B_{\Phi} ^{2}(H,Z(G))$ and let $\zeta_1$ be an element of $Z_{\Phi} ^{2}(H,Z(G))$.
For every compatible pair $x:=(\omega,\pi)$, we need the following elements in order to compute
$\varphi_1  ^{\Gamma (x)}$ :
\begin{enumerate}
\item To construct $E_1:=E(\xi,\varphi +\zeta_1)$.
\item To construct the homomorphism $\gamma$ from $E_1$ onto $T_{\Phi}$
\item To compute in $E_1$, one preimage $\overdub{h}$ under $\gamma$ of
$(\gamma (\overline{h}))^{x^{-1}}$ for each $h\in H$.
\item To compute $\varphi_2 (h_1,h_2):=((\overdub{h_1 h_2})^{-1}\overdub{h_1}\,\,\overdub{h_2}\,)^{\pi}$  (for $h_1, h_2 \in H$)
in $E_1$. Finally $(\varphi_1 +B) ^{\Gamma (x)}= \varphi_2 + B$.
\end{enumerate}
Observe that $B$ can be generated once for all as the set $\varphi ^{\Gamma (1,1)}-\varphi$ (see Proposition
\ref{Gamisfunction}).
%\subsubsection{Computation of the orbit of $\mathbf{\varphi_1}$}
\subsubsection{Computation of an orbit}
This can be done by a classical orbit-stabilizer algorithm that computes progressively  the right cosets of a
point stabilizer.\\

Assume that we want to determine all the $\zeta_2\in Z_{\Phi} ^{2}(H,Z(G))$
such that the extension $E(\xi,\varphi +\zeta_2)$ is $G$-isomorphic to $E_1:=E(\xi,\varphi +\zeta_1)$.
We have shown that it is equivalent to determine the orbit $\varphi_1 ^{\Gamma}$ (for $\varphi_1:=\varphi +\zeta_1$).
Since $Inn (T_{\Phi})$ fixes every point of $\Omega:=\varphi + H_{\Phi} ^{2}(H,Z(G))$, we can consider
$\Gamma$ as an action of $Out(\Phi):=Comp(\Phi)/Inn(T_{\Phi})$ on $\Omega$. Hence,
if we choose for each right coset $C_i$ of $Inn(T_{\Phi})$ in $Comp(\Phi)$, a representative $x_i$ of $C_i$,  %($i\in 1\ldots k$).
then the full orbit $\varphi_1 ^{\Gamma}$ is contained in the set  $\{\varphi_1 ^{\Gamma (x_i)}\,:i\in I\}$
 ($I$ is a family of indices) but this set could contain some element several times.

 Let us show how to avoid repetitions. More precisely, each time we observe a repetition, that one can be used
 to reduce (at least by a factor 2) the number of steps needed to compute the full orbit.
Since $\Gamma$ is a group action, two compatible pairs $x$ and $y$ map $\varphi_1$ on the same element
$\varphi '$ if and only if $xy^{-1}$ belongs to $Ind(E_1)$, the stabilizer of $\varphi_1$. Now, assume that we can
index the cosets $C_i$ by a sequence $I$ of positive integers. Step by step, we compute a first part of the
orbit, $\mathcal{O}:=\{\varphi_1 ^{\Gamma (x_1)}, \varphi_1 ^{\Gamma (x_2)}\ldots\varphi_1 ^{\Gamma (x_{j-1})}\}$
until we observe a repetition $\varphi_1 ^{\Gamma (x_j)}=\varphi_1 ^{\Gamma (x_i)}$ for some $i<j$. Then,
we know that $x_i(x_j)^{-1}\in Ind(E_1)$. We now have a subgroup of $Stab(\varphi_1)$ that is strictly larger
than $Inn(T_{\Phi})$, namely the group $A$ generated by the union of $x_i(x_j)^{-1}$ and $Inn(T_{\Phi})$.
Hence, it is now possible to reduce the number of potential repetitions by considering the image of
$\varphi_1$ under a set of right coset representatives for $A$ in $Comp(\Phi)$. If this is iterated, it is possible to
enlarge $\mathcal{O}$ and $A$ gradually in order to obtain the full orbit $\varphi_1 ^{\Gamma}$ together with
$Ind(E_1)$.

%we obtain the following process to compute gradually the full orbit of $\varphi_1$ together with $Ind(E_1)$.
The process is the following.
%For $l> j$, either $x_l \in \bigcup_{k<l}Ax_k$ %for some $k<l$
For $l> j$, either $x_l \in Ax_k$ for some $k<l$
so that $\varphi_1 ^{\Gamma (x_l)}$ must not be calculated because we know that it already belongs  to
$\mathcal{O}$ ; either  $x_l \notin \bigcup_{k<l}Ax_k$, then we calculate  $\varphi_1 ^{\Gamma (x_l)}$ and
one of the following situations occurs :
\begin{enumerate}\item either $\varphi_1 ^{\Gamma (x_l)}\notin \mathcal{O}$ so that we add it to $\mathcal{O}$.
\item  Either there is a new repetition ($\varphi_1 ^{\Gamma (x_l)}\in \mathcal{O}$) so that we can enlarge $A$
by adding $x_k(x_l)^{-1}$ to its generators.
\end{enumerate}
Then, we do the same with $x_{l+1}$, and the process will stop when we have the full orbit and the full group
 $Ind=(E_1)$. An important thing to note for next section is that if we have a supplement $N$ to $Ind(E_1)$ in
  $Comp(\Phi)$ (i.e. $Comp(\Phi)=Ind(E_1)N$) then it is possible to restrict our calculations to $N$ because so,
  we may choose coset  representatives for $Ind(E_1)$ within $N$.

%It amounts to compute gradually a set of right coset representatives for $Ind(E_1)$ in $Comp(\Phi)$.
 %It allows us to remove all the representatives
%$x_k$ ($k\geq j$) that belong to $A x_l$ for some $x_l \in \{x_1\ldots x_{j-1}\}$ because all these representatives
%will produce repetitons corresponding to elements of $A$. We obtain a new list $\{y_{i'} : i' \in I'\}$ whose
%union with $\{x_1\ldots x_{j-1}\}$ is just an irredundant list of right-cosets representatives for $A_1$ in
%$Comp(\Phi)$. Nothing has to be removed in $\{x_1\ldots x_{j-1}\}$ : we have already checked that they produce different
%element and therefore they correspond to distinct right-cosets of $Stab(\varphi_1)$.
%The process can be iterated to produce still larger subgroups of $Stab(\varphi_1)$ ;  the process stops when
%there is no more repetitions which amounts to have the full stabilizer of $\varphi_1$ or the full orbit.
% Therefore the same algorithm can determines simultaneously, the group $Ind(E_1)$ of inducible pairs which
% is fundamental in order to determine the automorphism group of the extension $E_1$.
%%%%%%%%%%%%%%%%%%%%%%%%%%%%%
%%%%%%%%%%%%%%%%%%%%%%%%%%%%%
\section{Determine orbits with supplements only}
This section build a bridge between our results on the action $\Gamma$ and our results on reducing the extension
problem to supplements (see chapter \ref{quotientsemi}).
The general idea is the following : \emph{"if we have a supplement function preserved by a $G$-isomorphism,
then the calculation of the image
of a point $\varphi$ under $\Gamma (\omega,\pi)$ can be performed exclusively within a supplement. Moreover
the right coset calculations needed to determine the orbit of $\varphi$ under $\Gamma$, can be performed in
a subgroup of $Comp(\Phi)$"}.
%exists a supplement $N_S$ to $Inn(T_{\Phi})$ in $Comp(\Phi)$ that preserves $S$. Hence, preimages are computed
%exclusively in $S_E$ and the coset calculation of the "Orbit-stabilizer" algorithm do not need to be performed
%in the full group $Comp(\Phi)$ but can be performed in the smaller group $N_S$."}

%%%%%%%%%%%%%%%%%%%%%%%%
\subsubsection{Supplement functions for extensions with coupling $\Phi$} \label{Gisofunctphi}
Let $\Phi$ be a homomorphism from $H$ into $Out(G)$. As before for every $h\in H$ let us choose
$\xi(h)\in \Phi(h)$ and let us assume that there exists an extension $E(\xi,\varphi)$ with coupling $\Phi$.
Let $\mathcal{F}_{\Phi}$ be the family of extensions with coupling $\Phi$. We will attempt to construct
supplement functions for $\mathcal{F}_{\Phi}$ that are preserved by a $G$-isomorphism.

We first observe that we have already constructed examples of such functions.
If for two families of extensions such that $\mathcal{F}_1\subset \mathcal{F}_2$,
there is a supplement  function $S$ for $\mathcal{F}_2$ that is preserved by a $G$-isomorphism, then
trivially the restriction of $S$ to $\mathcal{F}_1$ is a supplement  function $S$ for $\mathcal{F}_2$
that is preserved by a $G$-isomorphism. Let $\Phi (H)=L/Inn(G)$ for some subgroup $L$ such that
$Inn(G)\leq L\leq Aut(G)$ and let $\mathcal{F}_{L}$ be the family of extensions associated with $L$.
 Since the family $\mathcal{F}_{\Phi}$ of extensions
with coupling $\Phi$ is a subfamily of $\mathcal{F}_{L}$, the supplement functions described in
subsection \ref{isopreserved} are also supplement functions for $\mathcal{F}_{\Phi}$ that are preserved by
a $G$-isomorphism. %They are obtained by taking the preimage of $B$ under the natural projection $T_{\Phi}$ in
%$L\leq Aut(G)$, where $B$ is a supplement to $Inn(G)$ in $L$

As before, let $\gamma$ be the projection onto $T_{\Phi}$ of an extension $E\in\mathcal{F}_{\Phi}$.
%%The following arguments are similar to those explained in subsection \ref{isopreserved} for $\mathcal{F}_{L}$.
If $S$ is a supplement to $G$ in $E$, it is mapped by $\gamma$ onto a supplement to $Inn(G)_{\times 1_{H}}$
 in $T_{\Phi}$ and conversely, there is a one-to-one  correspondence between the preimages in $E$ of the
supplements to $Inn(G)_{\times 1_{H}}$  in $T_{\Phi}$ and the supplements to $G$ in $E$ that contains
 $Z(G)=Ker\,\gamma$.
%taking preimage under $\gamma$, supplements to $Inn(G)_{\times 1_{H}}$  in $T_{\Phi}$ are in one-to-one
% correspondence with the supplements to $G$ in $E$ that contains  $Z(G)=Ker\,\gamma$.
Determining conditions for such supplements to be preserved by a $G$-isomorphism will help to find any
type of supplement function for $\mathcal{F}_{\Phi}$. Indeed, if $E\to S_{E}$ is a supplement function for
$\mathcal{F}_{\Phi}$, preserved by a $G$-isomorphism, then $E\to Z(G)S_E$  is also preserved by a
$G$-isomorphism since $Z(G)$ is preserved by every $G$-isomorphism.

We have described in subsection \ref{extmodZ} (equality \ref{Gisomod1}) how each compatible pair for $\Phi$
defines an automorphism of $T_{\Phi}$. Therefore it is possible to define the normalizer (or rather stabilizer)
in $Comp(\Phi)$ of a subgroup $S\subset T_{\Phi}$ and we denote it $N_{Comp(\Phi)}(S)$.

\begin{proposition}\label{Phi-isopres}
%Let $\Phi:H\to Out(G)$ be a homomorphism and let $\mathcal{F}_{\Phi}$ be the family of extensions of $G$
%by $H$ with coupling $\Phi$.
Let $S$ be a supplement to $Inn(G)_{\times 1_H}$ in $T_{\Phi}$ and for every
$E\in \mathcal{F}_{\Phi}$ let $S_E$ be the preimage of $S$ in $E$. Then % under $\gamma$.
%If $Comp(\Phi)$ $Inn(T_{\Phi})N_{Comp(\Phi)}(S)$ then the function $E\to S_E$ on $\mathcal{F}_{\Phi}$
\begin{enumerate}
\item The function $E\to S_E$ on $\mathcal{F}_{\Phi}$  is preserved by a $G$-isomorphism if and only if
%$N_{Comp(\Phi)}(S)$ is a supplement to $Ind(E)$ in $Comp(\Phi)$ for every $E\in\mathcal{F}_{\Phi}$.
$Comp(\Phi)=Ind(E)N_{Comp(\Phi)}(S)$ for every $E\in\mathcal{F}_{\Phi}$.
\item Consequently, if $N_{Comp(\Phi)}(S)$ is a supplement to $Inn(T_{\Phi})$ in $Comp(\Phi)$ then
the function $E\to S_E$ on $\mathcal{F}_{\Phi}$  is preserved by a $G$-isomorphism.
\end{enumerate}
\end{proposition}
\emph{Proof }.
\begin{enumerate}
\item  Let $E_1$ and $E_2$ be two extensions in $\mathcal{F}_{\Phi}$  and let $\gamma_1$ and $\gamma_2$
 be their projections on $T_{\Phi}$. Let $S$ be a supplement to $Inn(G)_{\times 1_{H}}$  in $T_{\Phi}$.
Observe that if $j:E_1\to E_2$ is a $G$-isomorphism that induces the compatible pair $\beta$, then
$j$ maps $\gamma_1 ^{-1}(S)$ to $\gamma_2 ^{-1}(S)$ if and only if $\beta$ normalizes $S$
(i.e. $S^{\beta}=S$) as shown by equalities \ref{Gisomod1} and \ref{Gisomod2}
(see figure \ref{GisoTphi}). For a given $E_1 \in \mathcal{F}_{\Phi}$, by Theorem \ref{cnsisoGam}, for every
compatible pair $\beta$ there exists $E_2 \in \mathcal{F}_{\Phi}$ and a $G$-isomorphism $j:E_1 \to E_2$ that
induces $\beta$ : if $E_1 =E(\xi,\varphi_1)$, it suffices to take $E_2=E(\xi,\varphi_2)$ where
$\varphi_2 \in \varphi_1 ^{\Gamma (\omega,\pi)}$. A function from $E_1$ onto $E_2$ is a $G$-isomorphism
if and only if it is equal to a $G$-automorphism of $E_1$ followed by $j$.
% The others $G$-isomorphism from $E_1$ onto $E_2$
%induce pairs that are in one-to-one correspondence with the coset $Ind(E_1)(\omega,\pi)$.
Therefore, there is a $G$-isomorphism that maps $S_{E_1}=\gamma_1 ^{-1} (S)$ onto $S_{E_2}=\gamma_2 ^{-1} (S)$
 if and only if there is an inducible pair $\alpha\in Ind(E_1)$ such that $S^{\alpha \beta}=S$. This means that
 $\alpha \beta$ belongs to the normalizer $N_{Comp(\Phi)}(S)$ of $S$ under the action of $Comp(\Phi)$ on $T_{\Phi}$.
Hence, the function $S_E$ is preserved by a $G$-isomorphism if and only if for every compatible pair $x$
and every $E_1 \in\mathcal{F}_{\Phi}$ there is $\alpha\in Ind(E_1)$ such that
$x\in \alpha^{-1}N_{Comp(\Phi)}(S)$. This is equivalent to $Comp(\Phi)=Ind(E_1)N_{Comp(\Phi)}(S)$ for every
$E_1\in\mathcal{F}_{\Phi}$.
\item It is an immediate consequence of (1), since $Inn(T_{\Phi})\subset Ind(E_1)$ for every
$E_1 \in\mathcal{F}_{\Phi}$. $\quad\Box$
\end{enumerate}

%%%%%%%%%%%%%%%%%%%%%%%%%%%%%%%
\subsubsection{Supplement functions reduce orbit calculations}
Let $S$ be a supplement to $Inn(G)_{\times 1_H}$ in $T_{\Phi}$ and for every
$E\in \mathcal{F}_{\Phi}$ let $S_E$ be the preimage of $S$ in $E$.
 Now, assume that the function $E\to S_E$ is preserved by a $G$-isomorphism.
 We explain here how $S_E$ can reduce the computation of
actions and orbits described in subsection \ref{computeorb}.

For $E\in\mathcal{F}_{\Phi}$, let us show that it is always possible to consider a canonical transversal $\{\overline{h}:h\in H\}$
 that belongs to $S_E$ and this can be achieved by a suitable choice of the function $\xi :H\to Aut(G)$ that is
 associated with all the canonical transversals in $\mathcal{F}_{\Phi}$. Let $\tilde{\rho}$ be the projection of $T_{\Phi}$
 onto $H$. Since $S$ is a supplement to $Ker \tilde{\rho}=Inn(G)_{\times 1_H}$ in $T_{\Phi}$, for every $h\in H$
 it is possible to find $s_h\in S$ such that $\tilde{\rho} (s_h)=h$. Let $\tilde{\xi} (s_h)$ be the projection of
 $s_h$ into $Aut(G)$ so that $s_h=(h,\tilde{\xi} (s_h))$. If $\gamma$ is the projection of $E$ onto $T_{\Phi}$,
 an element $e\in E$ belongs to $\gamma^{-1} (s_h)$  if and only if
 $\gamma (e)=s_h=(h,\tilde{\xi} (s_h))$. Hence, for every $h\in H$,  it is sufficient to choose
 $\xi (h)=\tilde{\xi} (s_h)$  to guarantee that $\{\overline{h}:h\in H\}\subset S_E:=\gamma^{-1} (S)$.

Let $E\in\mathcal{F}_{\Phi}$. In order to obtain representatives of the right cosets of $Ind(E)$ in $Comp(\Phi)$,
it is sufficient to compute right coset representatives of $Ind(E)\cap N_{Comp(\Phi)}(S)$ in $N_{Comp(\Phi)}(S)$,
since by Proposition \ref{Phi-isopres},  $N_{Comp(\Phi)}(S)$ is a supplement to $Ind(E)$ in $Comp(\Phi)$.
%for every $E\in\mathcal{F}_{\Phi}$.
The group $Ind(E)$ is generally unknown so that these representatives are
computed gradually (see subsection \ref{computeorb}), starting with right coset representatives of
$Inn(T_{\Phi}) \cap N_{Comp(\Phi)}(S)$ in $N_{Comp(\Phi)}(S)$.
This avoids to compute right cosets in the full group $Comp(\Phi)$.
Actually, %even if the full normalizer of $S$ is unknown,
it is sufficient to work with any subgroup of $N_{Comp(\Phi)}(S)$ that is a supplement to $Ind(E)$ in $Comp(\Phi)$.

Let $\varphi_1$ be the factor set of an extension $E_1=E(\xi,\varphi)\in\mathcal{F}_{\Phi}$ with canonical transversal
$\{\overline{h}:h\in H\}\subset S_{E_1}$. We have just explained that in order to compute the orbit of $\varphi_1$
under $\Gamma$, it is sufficient to consider images of $\varphi_1$ under elements of $N_{Comp(\Phi)}(S)$.
But if $x=(\omega,\pi) \in N_{Comp(\Phi)}(S)$, then for every $h\in H$, a preimage $\overdub{h}$ of
$(\gamma (\overline{h}))^{x^{-1}}$ under $\gamma$ belongs to $S_{E_1}$ because, as $x^{-1}$ normalizes $S$,
 we have $\gamma^{-1} (S^{x^{-1}})=\gamma^{-1} (S)=S_{E_1}$.
Consequently the element $\varphi_2\in \varphi_1 ^{\Gamma (x)}$ defined by
$\varphi_2 (h_1,h_2):=((\overdub{h_1 h_2})^{-1}\overdub{h_1}\,\,\overdub{h_2}\,)^{\pi}$  (for $h_1, h_2 \in H$)
 can be computed in $S_{E_{1}}$ only. Finally in order to compute the orbit of $\varphi_1$, it is sufficient to
 construct the supplement $S_{E_1}$ instead of the full extension $E(\xi,\varphi)$.  It is also sufficient to consider
 the restriction $\gamma |_{S_{E_1}}:S_{E_1}\to S$ instead of constructing the full homomorphism $\gamma$.
%% We assume that we have already an extension $E(\xi,\varphi )$ with coupling $\Phi$.
% first to construct one extension , it suffices to construct a supplement
%solvable if $G$ is finite and $H$ is solvable. a factor set with values in $U$ but not all factor sets are acceptable
%they must be compatible with $Aut(G)$ and not only with $Aut(U)$. It remains (why exactly ?) to
%show that the restriction of $\xi (h)$ on $U$ can be extended to a homomorphism from $UH$ onto $B<Aut(G)$
%see the example with $SL_2 (9)$

\emph{\bf{ Finally, for each step needed to compute the $\Gamma$-orbits, the existence of a supplement
function preserved by a $G$-isomorphism enables us to reduce the computations in substantially smaller
groups. If furthermore $H$ is solvable and $G$ is finite,  all computations can be performed in a solvable
group (by Proposition \ref{refinep}).}}

%%%%%%%%%%%%%%%%%%%%%%%%%%%
\section{Implementations and performances}
We have achieved a basic implementation of our methods using $GAP 4.2$, for the extensions of a finite group $G$
by a finite solvable group $H$. It computes a supplement function preserved by a
$G$-isomorphism and uses it to reduce the determination of the $\Gamma$-orbits into small solvable groups.
We say that this implementation is basic because we have not yet used all the computational tools that have
been developed for finite solvable groups. For instance,  the computation of second cohomology groups $H^{2}$
 and the "Orbit-Stabilizer algorithm" have specific implementations for polycyclic groups (see \cite{GAP4}) that
 we have not used yet. Hence, further performance improvements may be expected.
 Nevertheless, the performances are already very impressive. Our supplement method applied to the
 construction of nonsolvable groups completely outperforms the iterated cyclic extension method.
There is no mystery behind that fact. If instead of classifying groups of order $600,000.$ one classifies groups
of order $80$ everyone would expect such an impressive improvement (such a case occurs with the perfect
group $P$ of order $15,000$ such that $P/Z(P)\cong PerfectGroup(7500,1)$ see table \ref{tableout}).
We believe that what could be expected from a more complete implementation, is to construct nonsolvable
groups as fast as solvable groups.
%%%%%%%%%%%%%%%%%%%%%%%%%%%%%%%%%%%%%%%%%%%%%%%%%%%%%%%%%%%%%%%ùù
\subsection{Comparisons of the different methods}
\subsubsection{Cross-check of the two methods by $z\phi$-classes}
%%\emph{Extensions of $G=PerfectGroup(1080,1)$ by $H$}
\emph{\bf{Extensions of $P=PerfectGroup(1080,1)$ by $H$ :}}\\
We have $|Z(P)|=3$ and $Out(P)\cong 2^2$. The first column shows performances of the iterated cyclic
extension method. The second column do the same for the computation of the $\Gamma$-orbits by the
"reduction to supplement" method. The last column is the cross-check by the $z\phi$-classes method described
in chapter \ref{classiftheo}. The mention "SG" means that the Small Groups Library has been used.
 The supplement method used it to find the groups of order $H$.
The $z\phi$-class method used it to find the groups of order $3|H|$.

\begin{tabular}{|c|c|c|c|}
\hline
& iter. cyclic & $\Gamma$+supp.+SG & $z\phi$-classes + SG \\
\hline\hline
extensions  for $|H|=6$ & 12 & 12 & 12  \\
 time &  17 min. & 7 sec  & 1.3 sec.  \\
\hline
extensions  for $|H|=8$ & 34 & 34 & 34  \\
 time & 1 h 3 min.  & 1 min 20 sec. & 24 sec.  \\
\hline
extensions  for $|H|=12$ & 34  & 34& 34  \\
 time & 3 h. 9 min.  & 1 min. 40 sec.  & 29 sec.  \\
\hline
extensions  for $|H|=16$ & ? & 151 &  151 \\
 time & $?>10\,days$  & 1 h. 44 min.  &  2 min. 27 sec. \\
\hline
extensions  for $|H|=24$ & ? & 159 & 159  \\
 time & $?>10\,days$  & 1 h. 36 min. &  2 min. 30 sec. \\
\hline
\end{tabular}\label{tab3a6_ch5}

%%%%%%%%%%%%%%%%%%%%%%%%%%%%%%
\subsubsection{Supplements versus Cyclic extensions}
Extensions of $P=SL_{2}(9)$ by $H$\\\\

\begin{tabular}{|c | c | c |}
\hline
 & iter. cyclic & supplements  \\
\hline\hline
extensions  for $|H|=4$ &  22 & 22   \\
\hline
 time &  10 min & 6 sec   \\
\hline
extensions  for $|H|=8$ & 107 & 107   \\
\hline
 time &  8 hours & 2 min 5 sec   \\
 \hline
\end{tabular}
\\\\
Actually, the gap between these two methods gets deeper when larger group $P$ are considered.
The supplement method reduces the extensions of $P$ to the extensions of a small subgroup $U<P$.
But the information collected in table \ref{tableout} about perfect groups shows that $|U|$ is almost constant
as $|P|$ increases. Hence, the "reduction ratio" $|P|/|U|$ increases with $|P|$.
%Is it true for $|P|>15360$ ?
%Hence, our method clearly outperforms the method of successive cyclic extensions.
%\item Same set of extensions for $|P|.|H| \leq 2000$ \\
%as Eick \& Besche (year 2000)\Large{\textbf{!}}
\subsection*{example : 1535 extensions $H.SL_2 (9)$ of order $<23040$}
\index{EXTENSIONS $S_6$ BY $2^5$ REMAINS TO BE DONE !! No 32*720=23040
IT IS NOT NECESSARY TO EXPLAIN THE RESULT SINCE IT SUFFICES TO INCLUDE
THIS REFINEMENT IN THE GAP ALGORITHM FOR ELEMENTARY ABELIAN H=32 }
\begin{tabular}{|c|c|c|c|c|c|c|c|c|}
\hline
\textbf{$|H|$} & 2& 3 & 4 & 5&6 &7&8&9\\
\hline
\textbf{$ext$} &5  & 1& 22& 1 & 10&1&107&2\\
\hline\hline
\textbf{$|H|$} & 10& 11 & 12 & 13&14&15 & 16&17 \\
\hline
\textbf{$ext$} &10  &1 & 60& 1 & 10 & 1& 706 &1 \\%
\hline\hline
\textbf{$|H|$} & 18& 19 & 20 & 21& 22& 23 & 24 &25 \\%[17..23]); [ 1, 25, 1, 66, 2, 10, 1 ]
\hline
\textbf{$ext$} &25  &1 & 66& 2 & 10 & 1&394  &2 \\%
\hline\hline
\textbf{$|H|$} &26 &27  &28  &29 &30 &31  &  & \\% numbext720([26..31]);[ 10, 5, 58, 1, 20, 1 ]
\hline
\textbf{$ext$} &10  &5 & 58& 1 & 20 &1 &  & \\%
\hline
\end{tabular}\label{tab2a6}
%%%%%%%%%%%%%%%%%%%%%%%%%ù

\chapter{Appendix on perfect groups}\label{appendix}

%\begin{definition}
Let $\phi$ be an homomorphism of the group $M$ onto the group
$\phi (M)$ whose kernel is $K_{\phi}$. An automorphism $\alpha$
of $M$ such that $\alpha(K_{\phi})=K_{\phi}$ induces an automorphism
$\alpha_{\phi}$ of $\phi(M)$ described by
$$\alpha_{\phi}:\phi(m)\longrightarrow\phi(m^{\alpha})\quad \textrm{for every m}\in M.\quad (\star)$$
and we say that $\alpha$ is a \emph{lifting of $\alpha_{\phi}$ by $\phi$}.
The simplest examples of automorphisms for which a lifting exists, are the inner
automorphisms. Indeed, the conjugation by $x$ in $M$ is a lifting of the
conjugation by $\phi(x)$ in $\phi(M)$ since
$\phi(x)^{-1}\phi(m)\phi(x)=\phi(x^{-1}mx)$ and equation $(\star)$ is satisfied.

\begin{theorem}\label{plift}
Let $P$ be a perfect group, $Z$ a central characteristic subgroup of $P$ and
$\tilde\alpha$ an automorphism of $P/Z$. If there exists a lifting
$\alpha \in Aut(P)$ of $\tilde\alpha$ then $\alpha$ is unique.
\end{theorem}
\emph{Proof }.
Every $\alpha\in Aut(P)$ induces an automorphism on $P/Z$ since $\alpha$ preserves $Z$.
 If $\alpha_{1}$ and $\alpha_{2}$ are two automorphisms of $P$ that
induce the same automorphism on $P/Z$ then $k=\alpha_{1}\alpha_{2}^{-1}$ induces
the identity on $P/Z$.\\
Let us show that $k$ must be the identity automorphism of $P$.\\
For every $p_{i}\in P$ there is a unique $z_{k,p_{i}}\in Z$
such that $p_{i}^{k}=z_{k,p_{i}}p_{i}$. Hence
  $z_{k,p_{1}p_{2}}p_{1}p_{2}=(p_{1}p_{2})^{k}=p_{1}^{k}p_{2}^{k}=z_{k,p_{1}}p_{1}z_{k,p_{2}}p_{2}= z_{k,p_{1}}z_{k,p_{2}}p_{1}p_{2}$
 since $Z$ is a subgroup of the centre of $P$.
The equality $z_{k,p_{1}p_{2}}=z_{k,p_{1}}z_{k,p_{2}}$ shows that the function
$f_{k}:p\rightarrow z_{k,p}$ is a homomorphism from $P$ onto
 the abelian group $f_{k}(P)\leq Z$ and since $P$ is perfect, $f_{k}(P)=1\in Z$.
  As announced, $z_{k,p}=1$ for every $p\in P$
 and $k=\alpha_{1}\alpha_{2}^{-1}$ is the identity automorphism of $P.\quad\Box$

%%-true for infinite group ? : yes since an infinite perfect group can't have abelian quotient

Here are two important consequences of theorem \ref{plift}.
\begin{proposition}\label{autperf}Let $P$ be a perfect group whose centre is $Z$ :
\begin{enumerate}
\item $Aut(P)$ is isomorphic to a subgroup of $Aut(P/Z)$ containing $In(P/Z)$
\item the centre of $P/Z$ is the identity.
\end{enumerate}
\end{proposition}
\emph{Proof }. (1) Since $Z$ is preserved by $Aut(P)$, we have a canonical homomorphism
$f$ from $Aut(P)$ into $Aut(P/Z)$. The kernel of this homomorphism is precisely
the set of automorphisms of $P$ that are lifting of $1\in Aut(P/Z)$.
 Since $1\in Aut(P)$ is one such lifting and since the lifting is unique by theorem
 \ref{plift}, the kernel of $f$ is the identity. So $Aut(P)$ is
 isomorphic to its image $f(P)\leq Aut(P/Z)$ and we have shown previously that
 $In(P)$ induces $In(P/Z)$.\\
(2)  Let $x=Zp$ be an element of
$P/Z$ and let $\sigma_{x}$ be the conjugation by $x$ in $P/Z$. The conjugation
$\tilde\sigma_{p}$ by $p$ in $P$ is a lifting of $\sigma_{x}$. If $x$ is in the
 centre of $P/Z$ then $\sigma_{x}$ is the identity on $P/Z$ and so
 $1\in Aut(P)$ is  also a lifting of $\sigma_{x}$. Therefore, theorem \ref{plift}
 implies that  $\tilde\sigma_{p}$ is the identity on $P$ and $p$ must belong to the
 centre $Z$ of $P$. Thus $x=Zp=Z$ is the identity of $P/Z.\quad\Box$

Let $P$ be a perfect group such that $Z(P)=1$ and let us consider the family
$\mathcal{F}_{P}:=\{ \tilde{P}\, \textrm{is perfect and}\, \tilde{P}/Z(\tilde{P})\cong P \}$ 
( the family of perfect central extensions of $P$). Let $\tilde{P}\in \mathcal{F}_{P}$
 and let $L$ be a subgroup such that $Inn(\tilde{P})\leq L\leq Aut (\tilde{P})$. 
If we want to construct the extensions of $\tilde{P}$ associated with $L$ then the method 
described in our work needs a supplement $\tilde{B}$ to $Inn(\tilde{P})$ in $L$. Let 
$j$ be the injection of  $Aut(\tilde{P})$ into $Aut(P)$ described by Proposition 
\ref{autperf}. Then $j (\tilde{B})$ is a supplement to $j(Inn(\tilde{P}))=Inn(P)$ 
in $j(L)\leq Aut(P)$. This is true for every $\tilde{P}$ in  $\mathcal{F}_{P}$ and thus such 
supplements can be found once for all $\tilde{P}\in\mathcal{F}_{P}$, by just finding 
supplements in $Aut(P)$. 

It is also important to note that $Aut(P)$ can be strictly smaller than $Aut(P/Z)$.
The smallest counter example of this statement
can be found in \cite{Holt_Plesken} page 174. The perfect group $P$ of order $2688$
 identified by $[2688,2]$, has no outer automorphism. Nevertheless the group $P/Z$,
  which is a split extension of $2^{3}$ by $L_{3}(2)$, has an outer automorphism group
  of order $2$.
 
%We have found an incorrect version of this theorem in a paper (Peter Schmidt Aachen 1970)
%that stated that $Aut(P)\cong Aut(P/Z)$. 

%\include{tchap2}
%\include{tchap3}
%\include{tchap4}

%\include{tchap6}
%\include{tchap7}
%\include{tchap8}
%\include{summary}

\bibliographystyle{plain}
\bibliography{these}
%\printindex

\end{document}